\documentclass{article}
\usepackage[%
journal=FoCM,    
lang=british,   
]{ems-journal}

\makeatletter
\renewcommand*\ems@ps@titlepage@FoCM{%
  \def\@oddhead{}%
  \def\@evenhead{}%
}
\makeatother

\usepackage{lineno}
\usepackage{breqn}
\usepackage{derivative}
\usepackage{placeins}
\usepackage{caption, subcaption}
\usepackage{multirow}
\usepackage{geometry}
\usepackage{bm}
\usepackage{enumitem}
\usepackage[title]{appendix}
\usepackage{adjustbox}
\allowdisplaybreaks

\newcommand{\ud}{\,\mathrm{d}}
\newcommand{\mF}{\,\mathcal{F}}
\newcommand{\tu}{\,\tilde{u}}
\newcommand{\bu}{\,\overline{u}}
\newcommand{\bvu}{\,\overline{\mathbf{u}}}
\newcommand{\vu}{\,\mathbf{u}}
\newcommand{\vecv}{\,\mathbf{v}}
\newcommand{\vw}{\,\mathbf{w}}
\newcommand{\vf}{\,\mathbf{f}}
\newcommand{\vx}{\,\mathbf{x}}
\newcommand{\vn}{\,\mathbf{n}}
\newcommand{\hvf}{\,\hat{\mathbf{f}}}
\newcommand{\vH}{\,\mathbf{H}}
\newcommand{\vh}{\,\mathbf{h}}
\newcommand{\vz}{\,\mathbf{z}}
\newcommand{\bQ}{\,\mathbb{Q}}
\newcommand{\hf}{\,\hat{f}}
\theoremstyle{plain}
\newtheorem{theorem}{Theorem}[section]
\newtheorem{lemma}[theorem]{Lemma}
\newtheorem{proposition}[theorem]{Proposition}
\newtheorem{corollary}[theorem]{Corollary}

\theoremstyle{definition}
\newtheorem{definition}[theorem]{Definition}
\newtheorem{remark}[theorem]{Remark}

\newtheorem{expl}[theorem]{Example}

\numberwithin{equation}{section}

\begin{document}

\title{A class of high-order discontinuous-Galerkin methods satisfying infinitely many entropy conditions with provable error estimates and strong convergence for general nonlinear conservation laws \\[0.5em]
{\textnormal{\normalsize September 3, 2026}}}
\titlemark{A class of high-order DG methods satisfying infinitely many entropy conditions}


%

\emsauthor{1}{
  \givenname{Yuanzhe}
  \surname{Wei}
  \mrid{}
  \zblid{}
  \orcid{}}{Y.~Wei}
\emsauthor{2}{
  \givenname{Chi-Wang}
  \surname{Shu}
  \mrid{242268}
  \zblid{shu.chi-wang}
  \orcid{0000-0001-7720-9564}}{C.-W.~Shu}

\Emsaffil{1}{
  \department{Division of Applied Mathematics}
  \organisation{Brown University}
  \rorid{05gq02987}
  \address{Box F, 182 George Street}
  \zip{02912}
  \city{Providence, RI}
  \country{USA}
  \affemail{yuanzhe\_wei@brown.edu}}
\Emsaffil{2}{
  \department{Division of Applied Mathematics}
  \organisation{Brown University}
  \rorid{05gq02987}
  \address{Box F, 182 George Street}
  \zip{02912}
  \city{Providence, RI}
  \country{USA}
  \affemail{chi-wang\_shu@brown.edu}}


\begin{abstract}
We propose a novel framework for deriving semi-discrete discontinuous-Galerkin (DG) methods using operator semigroups for scalar conservation laws, then apply it to construct a class of high-order OFDG-type schemes \cite{OFDG} satisfying infinitely many local entropy inequalities with general E-fluxes on non-uniform meshes. Such schemes are further generalized to systems of conservation laws in any number of space dimensions by using entropy stable numerical fluxes in the sense of \cite{CHEN2017427}. Finally, we prove optimal error estimates for smooth solutions to nonlinear scalar conservation laws, and prove strong convergence for discontinuous solutions to strictly convex conservation laws via compensated compactness.
\end{abstract}

\maketitle

\section{Introduction}
Hyperbolic conservation laws arise throughout continuum mechanics and describe the propagation of nonlinear waves that may develop discontinuities even from smooth initial data. For such problems, a numerical method must reconcile several competing objectives: high-order accuracy in smooth regions, nonlinear stability near shocks, consistency with the entropy condition that selects the physically relevant weak solution, and convergence as the mesh is refined. Discontinuous Galerkin (DG) methods are particularly attractive in this setting because they combine high-order polynomial approximation, local conservation and geometric flexibility. Since the development of Runge--Kutta discontinuous Galerkin (RKDG) methods for nonlinear conservation laws \cite{CS2}, substantial effort has therefore been devoted to understanding their stability, entropy properties, and error estimates; see, for example, \cite{J,Z2005,CHEN2017427,WS1} and the references therein.

Hyperbolic conservation laws generally admit infinitely many weak solutions. For scalar conservation laws, the only physically relevant solution, called the \emph{entropy solution}, is a special weak solution uniquely characterized by a family of conditions known as \emph{entropy inequalities} \cite{Dafermos}. Any numerical method that preserves sufficiently many entropy inequalities on the discrete level is particularly favorable since it can often be shown to converge strongly to the entropy solution following DiPerna's argument\footnote{If strong convergence cannot be proved, a weaker result is still valid by the Lax–Wendroff theorem \cite{RJL}, which states that if the numerical solution converges strongly with bounded total variation, then the limit is the unique entropy solution.}, see \cite{DiPerna1983,CCL}. First order monotone schemes inherit entropy inequalities for a broad class of convex entropies \cite{Harten}, whereas the situation for high-order approximations is more delicate. A classical semi-discrete DG method with a monotone numerical flux satisfies a local entropy inequality for the square entropy \cite{J}, but quadratic entropy alone does not encode the full family of entropy conditions associated with the underlying conservation law. On the other hand, the entropy-stable DG method developed in \cite{CHEN2017427} admits an arbitrarily prescribed entropy other than the square entropy, but still cannot include more than one entropy. These developments motivate the following question: can one construct a high-order DG method that retains the familiar structure of the classical method while enforcing substantially more entropy inequalities, without sacrificing optimal accuracy for smooth solutions?

The starting point of this paper is an operator interpretation of the classical DG method. Let $S_t$ denote the entropy-solution semigroup of the conservation law and let $\pi$ denote the cellwise $L^2$-projection onto the DG space. We show that the classical semi-discrete DG method with the Godunov flux can be recovered from the infinitesimal action of the composite operator $\pi S_{\Delta t}$ as $\Delta t\to0$. This viewpoint separates two mechanisms that are usually intertwined in the standard weak-form derivation: the exact evolution $S_{\Delta t}$ carries the complete entropy structure of the conservation law, while the projection back to the finite-dimensional polynomial space determines which part of that structure survives in the discretization. In particular, the $L^2$-projection is compatible with the square entropy, so that the classical DG scheme admits the square entropy inequality, which was first proved in \cite{J}. This observation suggests modifying the projection in order to enforce additional entropy inequalities.

Following this idea, we introduce a nonlinear projection\footnote{This nonlinear projection resembles the limiter in \cite{Sunlimiter}.} that contracts both the square entropy and an arbitrarily prescribed strictly convex entropy. Its infinitesimal operator leads to a DG method consisting of the classical DG spatial operator plus a cellwise damping term acting only on the nonconstant modes. We then derive computable upper bounds for the damping coefficient. The resulting schemes preserve cell averages and conservation and, with general E-fluxes, satisfy local entropy inequalities for the prescribed entropy. More importantly, by choosing bounds that are uniform over suitable classes of convex entropies, the same numerical solution satisfies infinitely many local entropy inequalities. We also introduce simplified damping coefficients whose structure is closely related to the oscillation-free discontinuous Galerkin (OFDG) methodology of \cite{OFDG,OFDGsystem}, while retaining the entropy estimates needed in our analysis.

The operator construction extends beyond one-dimensional scalar equations. For one-dimensional systems, we formulate the method for a general strictly convex entropy pair and entropy-stable numerical fluxes in the sense of \cite{CHEN2017427}. We further extend the framework to systems in arbitrary space dimensions on general convex, shape-regular meshes. In this setting, the additional damping again acts on the deviation from the element average, and entropy stability at element interfaces is supplied by numerical fluxes satisfying the entropy-stability condition used in \cite{CHEN2017427}. Thus the same basic mechanism---entropy-compatible evolution followed by a suitably controlled projection---provides a unified construction for scalar equations and systems.

A central issue is whether the additional dissipation required by the entropy conditions degrades the high-order accuracy of DG methods. We show that it does not. For smooth solutions of general nonlinear scalar conservation laws on quasi-uniform meshes, the proposed semi-discrete schemes achieve the optimal $L^2$ error estimate $O(h^{k+1})$ when an upwind flux is used, and an $O(h^{k+1/2})$ estimate for a general monotone flux. The proof builds on the framework developed for nonlinear conservation laws in \cite{Z2005}, together with estimates showing that the new damping terms are sufficiently high order in smooth regions.

We also study convergence when the exact solution contains discontinuities. For strictly convex scalar fluxes, assuming a uniform $L^\infty$ bound on the numerical solutions and any E-flux, we prove strong convergence of the proposed approximations to the entropy solution. The argument uses the compensated-compactness framework of Tartar \cite{Tartar1979} and the convergence theory for approximate solutions of conservation laws developed by DiPerna \cite{DiPerna1983}. In particular, the entropy dissipation built into the schemes (which is absent from the classical DG or OFDG) provides the necessary estimates for proving the compactness of entropy productions in $H^{-1}_{\mathrm{loc}}$.

The main contributions of this paper can therefore be summarized as follows: we give an operator-semigroup derivation of the classical semi-discrete DG method; construct high-order DG schemes satisfying infinitely many local entropy inequalities; extend the construction to systems and multiple space dimensions; establish optimal-order error estimates for smooth scalar solutions; and prove strong convergence for discontinuous solutions of strictly convex scalar conservation laws. Numerical experiments are included to examine the accuracy and robustness of the proposed damping mechanisms and to illustrate the effect of the parameters in practice.

The class of DG schemes developed in this paper can be summarized as follows:
\begin{description}
    \item[\textbf{1D scalar:}] In Section~\ref{section: scheme 1D scalar} we introduce three schemes for solving 1D scalar conservation laws. They are named \eqref{scheme B}, \eqref{scheme C} and \eqref{scheme D} after their damping coefficients $B_j$, $C_j$ and $D_j$. If the damping coefficients are chosen according to Corollary \ref{corollary: infinite entropy 2}, \eqref{scheme B} and \eqref{scheme C} would satisfy infinitely many entropy inequalities. Moreover, \eqref{scheme D} also satisfies infinitely many entropy inequalities if the numerical solution remains bounded. Because \eqref{scheme D} is the easiest to implement, it is adopted in the numerical tests in Section \ref{section: numerical results}.

    \item[\textbf{1D system:}] In Section~\ref{section: scheme 1D system} we introduce three schemes for solving 1D systems of conservation laws. Exactly parallel to the 1D scalar case, they are named \eqref{scheme B'}, \eqref{scheme C'} and \eqref{scheme D'} after their damping coefficients $B'_j$, $C'_j$ and $D'_j$. We use \eqref{scheme D'} for the numerical tests in Section \ref{section: numerical results}.

    \item[\textbf{Multidimensional system:}] In Section~\ref{section: scheme dD system} we introduce two schemes for solving systems of conservation laws in higher space dimension. Same as before, they are named \eqref{scheme B''} and \eqref{scheme D''} after their damping coefficients $B''_j$ and $D''_j$. We use \eqref{scheme D''} for the numerical tests in Section \ref{section: numerical results}.
\end{description}

We shall use the general term \emph{Operator DG} to denote any of these schemes mentioned above, since they are derived from our operator framework. They are closely related to the OFDG methods in the sense that they all contain a cellwise damping term acting on the nonconstant modes; the only difference is that, in our schemes, the damping coefficients have an additional term for controlling the oscillations on the interior of each cell, which provides the entropy estimates needed in our analysis. The remainder of the paper is organized as follows. Section~\ref{section: Godunov DG} develops the operator framework and reviews the classical DG method. Section~\ref{section: scheme 1D scalar} constructs the Operator DG schemes for one-dimensional scalar conservation laws. Sections~\ref{section: scheme 1D system} and \ref{section: scheme dD system} extend the framework to one-dimensional and multidimensional systems. Section~\ref{section: error estimate} proves the error estimates for smooth scalar solutions, and Section~\ref{section: strong convergence} establishes strong convergence via compensated compactness. The numerical results are presented thereafter, followed by concluding remarks in Section~\ref{section: concluding remarks}, and supplementary materials in the appendix concerning the HLL-based operator construction and sharper estimates for constants appearing in the entropy bounds.

\section{An operator framework for the classical DG method}\label{section: Godunov DG}
Consider the Cauchy problem for a nonlinear scalar conservation law in one space dimension (with $I$ being any interval):
\begin{equation}\label{HCL 1D scalar}
\begin{cases} 
\mathcal{U}_t(x,t)+f(\mathcal{U}(x,t))_x=0,\quad x\in I, \quad t\ge0.\\
\mathcal{U}(x,0)=\mathcal{U}_0(x).
\end{cases} 
\end{equation}
To solve \eqref{HCL 1D scalar} numerically, the interval $I$ is divided into $N$ disjoint sub-intervals $I_1,I_2,\cdots,I_N$, where $I_j=[x_{j-\frac{1}{2}},x_{j+\frac{1}{2}}]$, $|I_j|=x_{j+\frac{1}{2}}-x_{j-\frac{1}{2}}=h_j$, $I=\cup_{j=1}^N I_j$. For any function $v(x)$ that is continuous on $I_j$, 
define $v^+_{j-\frac{1}{2}}:=v(x_{j-\frac{1}{2}}^+)$ and $v^-_{j+\frac{1}{2}}:=v(x_{j+\frac{1}{2}}^-)$ to be the left and right end-point limit of $v$ on $I_j$ respectively. Define for $k\ge1$ the space
\begin{equation*}
\mathbb{V}^k:=\{v \in L^2(I) : v\big|_{I_j}\in \mathbb{P}^k(I_j), 1\le j \le N \}
\end{equation*}
of all piecewise (possibly discontinuous) polynomials on $I$ that coincides with a $k$-th order polynomial on each cell $I_j$. From now on, we reserve the letter $k$ to denote the polynomial order. The classical semi-discrete DG method with 1D $\mathbb{P}^k$ elements for solving the nonlinear scalar conservation law \eqref{HCL 1D scalar} can be stated as follows: Seek $u(\cdot,t)\in \mathbb{V}^k$ such that
\begin{equation}\tag{$\textbf{Classical DG}$}\label{DG1}
\int_{I_j} u_t(x,t) v(x) \ud x -\int_{I_j} f(u(x,t)) v_x(x) \ud x +\hat{f}_{j+\frac{1}{2}}(t)v^-_{j+\frac{1}{2}}-\hat{f}_{j-\frac{1}{2}}(t) v^+_{j-\frac{1}{2}}=0
\end{equation}
holds for all $v\in \mathbb{V}^k$ and all $j=1,2,\cdots,N$, where the numerical flux $\hat{f}_{j+\frac{1}{2}}(t)=\hat{f}(u(x^-_{j+\frac{1}{2}},t),u(x^+_{j+\frac{1}{2}},t))$ is Lipschitz continuous and consistent ($\hf(u,u)=f(u)$). Clearly, \eqref{DG1} can be rewritten into the following form:
\begin{equation}\label{DG ODE}
    \pdv{u}{t}(x,t)=F[u(\cdot,t)](x),\quad \forall x\in I,
\end{equation}
where $F[v]\in \mathbb{V}^k$ for any $v \in \mathbb{V}^k$ is a piecewise polynomial whose coefficients depend on $v$. Moreover, \eqref{DG ODE} is an ODE system since $u(\cdot,t)$ lies in a finite dimensional space $\mathbb{V}^k$. The initial condition of \eqref{DG ODE} is taken as any projection of $\mathcal{U}_0(x)$ onto $\mathbb{V}^k$; this can be the $L^2$-projection, the Gauss-Radau projection (see Section \ref{section: error estimate}), or interpolation on point values. Let $\Phi_{t}: \mathbb{V}^k \mapsto \mathbb{V}^k$ denote the flow map of this ODE system, which maps any initial condition to the solution at time $t$. Since the right-hand side of the ODE is locally Lipschitz continuous, the solution is uniquely defined at least for $t$ sufficiently small; moreover, one can show that the solution is always bounded (see the discussion below), so that $\Phi_t$ is well-defined for all $t\ge 0$. Then, \eqref{DG ODE} can be equivalently expressed as \[\pdv{}{t}\Phi_t u=F[\Phi_t u], \quad u \in \mathbb{V}^k.\]

\eqref{DG1} is most commonly derived by multiplying the test function $v$ on both sides of \eqref{HCL 1D scalar}, then integrate by parts on $I_j$, and finally replace the flux in the boundary terms by the numerical flux (which makes the scheme conservative) \cite{CS2}. If the monotone flux is used, \cite{J} showed that the numerical solution of \eqref{DG1} is $L^2$-stable: \[\|\Phi_{t_2} u\|_{L^2(I)}\le \|\Phi_{t_1} u\|_{L^2(I)},\quad\forall 0\le t_1\le t_2, \quad\forall u\in\mathbb{V}^k,\] which is a property shared by the unique entropy solution of \eqref{HCL 1D scalar}.

We now introduce a very different approach to deriving \eqref{DG1}, which proves to be more natural and provides deeper insight into its fundamental structure. First, let us introduce the exact solution operator $S_t$, $t\ge0$, associated with \eqref{HCL 1D scalar}, which maps a given initial condition to the unique entropy solution at time $t$. In other words,
\[
S_t \mathcal{U}_0(x)=\mathcal{U}(x,t),\quad \forall t\ge0,
\]
where $\mathcal{U}_0(x)$ and $\mathcal{U}(x,t)$ are the same as in \eqref{HCL 1D scalar}. For 1D scalar conservation laws, $S_t$ is well-defined for all $t\ge 0$ as a map from $L^\infty(\mathbb{R})\cup L^1(\mathbb{R})$ to $L^\infty(\mathbb{R})\cup L^1(\mathbb{R})$ \cite{Dafermos}. Furthermore, let $\pi:L^2(I) \mapsto \mathbb{V}^k$ denote the $L^2$-projection onto the space of piecewise polynomials, namely, for $w\in L^2(I)$, $\pi w \in \mathbb{V}^k$ is characterized by
\begin{equation*}
    \int_{I} (\pi w - w)v \ud x=0 , \quad \forall v\in \mathbb{V}^k.
\end{equation*}
Then we have the following theorem, which provides the alternative derivation of \eqref{DG1} when $\hat f$ is the Godunov flux:
\begin{theorem}\label{theorem 1}
For any $u\in \mathbb{V}^k$, we have
\begin{equation}\label{beautiful 1}
    \lim_{\Delta t\rightarrow 0}\frac{\pi S_{\Delta t}u-u}{\Delta t}=F[u],
\end{equation}
uniformly on $I$, where $F[\cdot]\in \mathbb{V}^k$ is the right-hand side of \eqref{DG ODE} and $\hat f$ is the Godunov flux:
\begin{equation}\label{Godunov}
    \hat f(u^-,u^+)=f(u^*),
\end{equation}
where $u^*$ is the value at $\frac{x}{t}=0$ of the entropy solution to the Riemann problem with left and right states $u^-$ and $u^+$ at $x=0$. If $u^-$ and $u^+$ are connected by a stationary shock, $u^*$ may be taken as either $u^-$ or $u^+$, since the Rankine–Hugoniot condition yields $f(u^-)=f(u^+)$.
\end{theorem}
\begin{proof}
Since $u\in \mathbb{V}^k$, it is smooth on the interior of each $I_j$, but can have discontinuities at the cell interfaces. Let $\Delta t>0$ be sufficiently small such that the evolutions of discontinuities at different cell interfaces do
not interact, and no new discontinuity (including the discontinuity in derivatives of all orders) is formed. For the Riemann problem with left and right states $u^-_{j-\frac{1}{2}}$ and $u^+_{j-\frac{1}{2}}$, let $\lambda^-_{j-\frac{1}{2}}$ and $\lambda^+_{j-\frac{1}{2}}$ be lower and upper bounds, respectively, for all wave speeds in the Riemann solution. To be more conservative, we assume
\begin{equation}\label{lambda estimate scalar}
    \lambda^-_{j-\frac{1}{2}}\le \min_{0\le \theta \le 1} f'(\theta u^-_{j-\frac{1}{2}}+(1-\theta) u^+_{j-\frac{1}{2}}),\quad \lambda^+_{j-\frac{1}{2}}\ge \max_{0\le \theta \le 1} f'(\theta u^-_{j-\frac{1}{2}}+(1-\theta) u^+_{j-\frac{1}{2}}).
\end{equation}
In each $I_j$, denote the polluted regions on the left and right by
\begin{equation}\label{polluted regions}
I_j^+(\Delta t):=\left[x_{j-\frac{1}{2}},x_{j-\frac{1}{2}}+\max\{0,\lambda^+_{j-\frac{1}{2}}\}\Delta t \right],\quad
I_j^-(\Delta t):=\left[x_{j+\frac{1}{2}},x_{j+\frac{1}{2}}+\min\{0,\lambda^-_{j+\frac{1}{2}}\}\Delta t \right].
\end{equation}
These are the parts of $I_j$ that might be influenced by the discontinuity of $u$ or its derivatives at $x_{j-\frac{1}{2}}$ and $x_{j+\frac{1}{2}}$, respectively. Thus, $S_{\Delta t}u$ remains smooth on $I_j \backslash (I_j^+(\Delta t)\cup I_j^-(\Delta t))$. Around each $x_{j-\frac{1}{2}}$, the discontinuity of $u$ at $x_{j-\frac{1}{2}}$ can only influence the region $I_{j-1}^-(\Delta t)\cup I_{j}^+(\Delta t)$ by the time $\Delta t$. 

Then, for any $v\in\mathbb{V}^k$, we have
\begin{align*}
&\int_{I_j} \frac{\pi S_{\Delta t} u(x) - u(x)}{\Delta t}\cdot v(x) \ud x \\
=& \int_{I_j} \frac{S_{\Delta t} u(x) - u(x)}{\Delta t}\cdot v(x) \ud x \quad\text{(since $v\vert_{I_j}$ is a polynomial)}\\
=& \int_{I_j \backslash (I_j^+(\Delta t)\cup I_j^-(\Delta t))} \frac{S_{\Delta t} u(x) - u(x)}{\Delta t}\cdot v(x) \ud x\\
&+ \int_{I_j^+(\Delta t)} \frac{S_{\Delta t} u(x) - u(x)}{\Delta t}\cdot v(x) \ud x + \int_{I_j^-(\Delta t)} \frac{S_{\Delta t} u(x) - u(x)}{\Delta t}\cdot v(x) \ud x\\
=& -\int_{I_j} f(u)_x\cdot v(x) \ud x + \mathcal{O}(\Delta t)\\
&+ \int_{I_j^+(\Delta t)} \frac{S_{\Delta t} u(x) - u(x)}{\Delta t}\cdot v(x) \ud x + \int_{I_j^-(\Delta t)} \frac{S_{\Delta t} u(x) - u(x)}{\Delta t}\cdot v(x) \ud x,
\end{align*}
where the last equality follows from the fact that both $u$ and $S_{\Delta t} u$ are smooth on $I_j \backslash (I_j^+(\Delta t)\cup I_j^-(\Delta t))$. The notation $\mathcal{O}(\Delta t)$ stands for any term that can be upper bounded by $C\Delta t$, where $C>0$ is a constant depending on the norms of derivatives of $u$ on the interior of $I_j$. Since eventually we would let $\Delta t\rightarrow 0$ for a fixed $u$ and on a fixed mesh, the dependency of the constant $C$ on $u$ or $h_j$'s need not concern us.

By taking $\Delta t \rightarrow 0$ in the last expression, we obtain
\begin{align}
\lim_{\Delta t\rightarrow 0} \int_{I_j} \frac{\pi S_{\Delta t} u(x) - u(x)}{\Delta t}\cdot v(x) \ud x = &-\int_{I_j} f(u)_x\cdot v(x) \ud x + v_{j-\frac{1}{2}}^+\lim_{\Delta t\rightarrow 0} \int_{I_j^+(\Delta t)} \frac{S_{\Delta t} u(x) - u(x)}{\Delta t} \ud x \nonumber \\
&+ v_{j+\frac{1}{2}}^-\lim_{\Delta t\rightarrow 0} \int_{I_j^-(\Delta t)} \frac{S_{\Delta t} u(x) - u(x)}{\Delta t}\ud x . \label{theorem 1 eq 1}
\end{align}
To compute \[\lim_{\Delta t\rightarrow 0} \int_{I_j^+(\Delta t)} \frac{S_{\Delta t} u(x) - u(x)}{\Delta t} \ud x,\] we first observe that the limit would not change if we replace $u$ in a neighborhood of $x_{j-\frac{1}{2}}$ by
\begin{equation}\label{theorem 1 Riemann}
    \tilde u(x)=\begin{cases}
        u^-_{j-\frac{1}{2}}, & x<x_{j-\frac{1}{2}}\\
        u^+_{j-\frac{1}{2}}, & x>x_{j-\frac{1}{2}}
    \end{cases}
\end{equation}
since $S_{\Delta t} u(x)=S_{\Delta t} \tilde u(x) + \mathcal{O}(\Delta t)$ for $x\in I_j^+(\Delta t)$. But $S_{\Delta t} \tilde u(x)$ is simply the Riemann solution with left and right states $u^-_{j-\frac{1}{2}}$ and $u^+_{j-\frac{1}{2}}$ at $x_{j-\frac{1}{2}}$. Because the Riemann solution is self-similar, set
\begin{equation}\label{self-similar R}
    R(\xi):=S_{t} \tilde u(x),\quad \xi=\frac{x-x_{j-\frac{1}{2}}}{t}.
\end{equation}
By definition, all right-going waves of $S_{\Delta t} \tilde u(x)$ are contained in $I_j^+(\Delta t)=[x_{j-\frac{1}{2}},x_{j-\frac{1}{2}}+a\Delta t ]$, where $a=\max\{0,\lambda^+_{j-\frac{1}{2}}\}$. Then,
\begin{align}
    \lim_{\Delta t\rightarrow 0} \int_{I_j^+(\Delta t)} \frac{S_{\Delta t} u(x) - u(x)}{\Delta t} \ud x &= \lim_{\Delta t\rightarrow 0} \int_{x_{j-\frac{1}{2}}}^{x_{j-\frac{1}{2}}+a\Delta t} \frac{S_{\Delta t} u(x) - u(x)}{\Delta t} \ud x \nonumber\\
    &= \lim_{\Delta t\rightarrow 0} \int_{x_{j-\frac{1}{2}}}^{x_{j-\frac{1}{2}}+a\Delta t} \frac{S_{\Delta t} \tilde u(x) - \tilde u(x)}{\Delta t} \ud x \nonumber\\
    &=\int_{0}^a R(\xi) - u^+_{j-\frac{1}{2}} \ud \xi. \label{theorem 1 eq 2}
\end{align}
Using the conservation law for the self-similar Riemann solution, we have
\begin{equation*}
    -\xi R'(\xi)+\odv{}{\xi}f(R(\xi))=0
\end{equation*}
in the sense of distributions. Hence
\begin{equation}\label{Riemann eq}
    \odv{}{\xi}[\xi R(\xi)-f(R(\xi))]=R(\xi)
\end{equation}
also holds in the sense of distributions, and therefore,
\begin{equation*}
    \int_0^a R(\xi) \ud \xi =aR(a)-f(R(a))+f(R(0^+)).
\end{equation*}
Because $a$ lies to the right of all Riemann waves, \[R(a)=u^+_{j-\frac{1}{2}}.\] Furthermore, by the definition of the Godunov flux, \[f(R(0^+))=\hat f (u^-_{j-\frac{1}{2}},u^+_{j-\frac{1}{2}})=\hat f_{j-\frac{1}{2}}.\] Consequently,
\begin{equation*}
    \int_0^a R(\xi) \ud \xi =au^+_{j-\frac{1}{2}}-f(u^+_{j-\frac{1}{2}})+\hat f_{j-\frac{1}{2}},
\end{equation*}
which, together with \eqref{theorem 1 eq 2}, implies
\begin{equation}\label{theorem 1 eq 3}
    \lim_{\Delta t\rightarrow 0} \int_{I_j^+(\Delta t)} \frac{S_{\Delta t} u(x) - u(x)}{\Delta t} \ud x=\int_{0}^a R(\xi) - u^+_{j-\frac{1}{2}} \ud \xi = \hat f_{j-\frac{1}{2}}-f(u)^+_{j-\frac{1}{2}}.
\end{equation}
In the same way, we can also prove
\begin{equation}\label{theorem 1 eq 4}
    \lim_{\Delta t\rightarrow 0} \int_{I_j^-(\Delta t)} \frac{S_{\Delta t} u(x) - u(x)}{\Delta t} \ud x= -\hat f_{j+\frac{1}{2}}+f(u)^-_{j+\frac{1}{2}}.
\end{equation}
Using \eqref{theorem 1 eq 3} and \eqref{theorem 1 eq 4} in \eqref{theorem 1 eq 1} and integrate by parts on $\int_{I_j} f(u)_x\cdot v(x) \ud x$, we obtain \eqref{DG1}. Since \[\lim_{\Delta t\rightarrow 0}\frac{\pi S_{\Delta t}u-u}{\Delta t}\in\mathbb{V}^k\] and $v$ is an arbitrary element in $\mathbb{V}^k$, we have proved \eqref{beautiful 1} holds point-wise, thus also uniformly because both sides of \eqref{beautiful 1} are piecewise polynomials defined on the same grid.
\end{proof}
A simple corollary from the above theorem is that the limit $\lim_{\Delta t\rightarrow 0} \frac{\pi S_{\Delta t} u - u}{\Delta t}$ exists for all $u\in\mathbb{V}^k$. In particular, $\lim_{\Delta t\rightarrow 0} \pi S_{\Delta t} u = u$ uniformly on $I$. Since $S_{\Delta t}$ and $\pi$ both decrease the $L^2$-norm, we have $\|\pi S_{\Delta t} u\|_{L^2(I)}\le \|u\|_{L^2(I)}$, and the following corollary is immediate after passing to the limit as $\Delta t \rightarrow 0$:
\begin{corollary}\label{L2 stable Godunov}
    Under periodic or compactly supported boundary conditions, \eqref{DG1} with Godunov flux is $L^2$-stable in the sense that 
    \begin{equation*}
        \|\Phi_{t_2} u\|_{L^2(I)}\le \|\Phi_{t_1} u\|_{L^2(I)},\quad\forall 0\le t_1\le t_2, \quad\forall u\in\mathbb{V}^k.
    \end{equation*}
\end{corollary}
\begin{proof}
    It suffices to estimate the time derivative of $\|\Phi_t u\|_{L^2(I)}^2$:
    \begin{align*}
        \frac{1}{2}\odv{}{t}\|\Phi_t u\|^2_{L^2(I)}=&\int_{I} \Phi_t u(x) \cdot \pdv{}{t}\Phi_t u(x) \ud x=\int_{I} \Phi_t u(x) \cdot F[\Phi_t u](x) \ud x \\
        =& \lim_{\Delta t\rightarrow 0} \int_{I} \Phi_t u\cdot\frac{\pi S_{\Delta t} \Phi_t u - \Phi_t u}{\Delta t}\ud x \quad \text{(Theorem \ref{theorem 1})}\\
        =& \lim_{\Delta t\rightarrow 0} \int_{I} \frac{\pi S_{\Delta t} \Phi_t u + \Phi_t u}{2}\cdot\frac{\pi S_{\Delta t} \Phi_t u - \Phi_t u}{\Delta t}\ud x\\
        =& \frac{1}{2}\lim_{\Delta t\rightarrow 0} \int_{I} \frac{(\pi S_{\Delta t} \Phi_t u)^2 - (\Phi_t u)^2}{\Delta t}\ud x\\
        =& \frac{1}{2} \lim_{\Delta t\rightarrow 0} \frac{\|\pi S_{\Delta t} \Phi_tu\|_{L^2(I)}^2 - \|\Phi_tu\|_{L^2(I)}^2}{\Delta t} \le 0,
    \end{align*}
    where the last inequality follows from the fact that $\|\pi S_{\Delta t} \Phi_tu\|_{L^2(I)}^2\le \|\Phi_tu\|_{L^2(I)}^2$.
\end{proof}
In fact, since $S_{\Delta t} u$ is the entropy solution, it satisfies all entropy inequalities; and since $\pi$ decreases the $L^2$-norm on each cell, the scheme actually admits a local square-entropy inequality with boundary flux determined by the flux of the true solution $S_{\Delta t} u$, which was originally proved in \cite{J}. The next corollary is not needed in the subsequent discussion, but nevertheless might be conceptually helpful.
\begin{corollary}\label{operator limit}
Fix an arbitrary time $T>0$ and $u\in\mathbb{V}^k$. Let $n$ be a positive integer. Then,
    \begin{equation*}
        \lim_{n\rightarrow\infty}(\pi S_{\frac{T}{n}} )^n u = \Phi_T u
    \end{equation*}
    uniformly on $I$.
\end{corollary}
\begin{proof}
    The proof of Theorem \ref{theorem 1} implies that the discrete map $u\mapsto \pi S_{\frac{T}{n}} u$ is a first-order time discretization of the ODE \eqref{DG ODE} with uniform time step $\frac{T}{n}$. Since the ODE itself is Lipschitz continuous, and for a fixed $u\in\mathbb{V}^k$, the first-order time derivative of $\Phi_t u(x)$ is uniformly bounded for all $x\in I$ and $0\le t \le T$. The proof follows from a routine estimate through truncation errors.
\end{proof}
In light of Corollary \ref{operator limit}, the result of Corollary \ref{L2 stable Godunov} is strikingly clear since \[\|(\pi S_{\Delta t} )^n u\|_{L^2(I)} \le \|(\pi S_{\Delta t} )^m u\|_{L^2(I)},\quad\forall n\ge m,\quad \forall u\in\mathbb{V}^k.\] It is worth noting that, since $S_{\Delta t}$ is the exact solution operator, $\pi S_{\Delta t}$ in the limit can only generate \eqref{DG1} with Godunov flux; however, a slightly modified operator can generate \eqref{DG1} with a general HLL flux. The discussion of how to modify $S_{\Delta t}$ and the computation of the limit is contained in Appendix \ref{section:HLL classic}. This might be of interest if one wishes to generalize this operator framework to other types of PDEs.

\section{The Operator DG method for 1D scalar conservation laws}\label{section: scheme 1D scalar}
Let $(U,\mF)$ be any fixed entropy pair. Assume $U\in C^3$ and strictly convex ($U''>0$ on $\mathbb{R}$). The entropy flux $\mF$ is defined by
\begin{equation}\label{def entropy flux}
    \mF'( u)=U'( u) f'( u).
\end{equation}
In addition, let 
\begin{equation}\label{def psi}
    \psi(u):=U'( u) f( u)-\mF( u),
\end{equation}
so that $\psi'(u)=U''(u)f(u).$ Our next goal is to construct a scheme that satisfies two entropy inequalities: one for the square-entropy, and the other for $U(\cdot)$. In Section \ref{section: Godunov DG}, we have shown that \eqref{DG1} with the Godunov flux can be generated by the operator $\pi S_{\Delta t}$ as $\Delta t \to 0$. The reason that it can only admit the square-entropy inequality is because $\pi$ is the $L^2$-projection, which can only decrease the integral of the square function on each cell but not for other convex functions. But $S_{\Delta t}$ admits all entropy pairs, whose potential is lost after applying $\pi$. Therefore, to realize the full potential of $S_{\Delta t}$, we should replace $\pi$ by a different projection operator $\tilde \pi$ such that it can decrease both the integral of the square function and $U(\cdot)$. With this $\tilde\pi$, we shall derive a new scheme generated by $\tilde\pi S_{\Delta t}$, which will look very similar to the classical DG scheme with Godunov flux and satisfy entropy inequalities for both the square-entropy and $U(\cdot)$. Then, we will further modify it such that: (1) the Godunov flux can be replaced by any E-flux; (2) the scheme satisfies infinitely many entropy conditions; (3) it still retains its optimal order of accuracy.

In order to derive the scheme generated by the operator $\tilde\pi S_{\Delta t}$, we need to compute:
\begin{align}
&\int_{I_j} \frac{\tilde\pi S_{\Delta t}  u(x) -  u(x)}{\Delta t}\cdot  v(x) \ud x \nonumber\\
=& \int_{I_j} \frac{\pi S_{\Delta t}  u(x) -  u(x)}{\Delta t}\cdot  v(x) \ud x + \int_{I_j} \frac{\tilde\pi S_{\Delta t}  u(x) - \pi S_{\Delta t}  u(x)}{\Delta t}\cdot  v(x) \ud x \nonumber \\
=& \int_{I_j} \frac{ S_{\Delta t}  u(x) -  u(x)}{\Delta t}\cdot  v(x) \ud x + \int_{I_j} \frac{\tilde\pi S_{\Delta t}  u(x) - \pi S_{\Delta t}  u(x)}{\Delta t}\cdot  v(x) \ud x. \label{modified scheme 1}
\end{align}
For any function $v\in L^2(I)$, let $\overline{v_j}$ and $\overline{(\pi  v)}_j$ denote the average of $v$ and its $L^2$-projection on $I_j$, respectively. Now, we define the new (nonlinear) projection $\tilde\pi$ as follows:\footnote{This is the same limiter used in \cite{Sunlimiter}, where this limiter is simply applied after each time-stepping of the fully-discrete classical RKDG scheme. In their paper, the error estimate is done heuristically, and the accuracy of their scheme is observed to be sub-optimal in certain numerical tests. In our paper, however, we incorporate this limiter into the DG framework in a very different way so that optimal error estimates can be proved rigorously and also verified in numerical tests (see Section \ref{section: error estimate}).}
\begin{equation*}
\tilde\pi  v(x):=\theta_j[ v] \pi  v(x) + (1-\theta_j[ v]) \overline{(\pi  v)}_j, \quad \forall x \in I_j,
\end{equation*}
where
\begin{equation}\label{def theta}
    \theta_j[ v]:=\min\left\{ \frac{ \int_{I_j} U\left(  v(x) \right) \ud x - U\left( \overline{v}_j \right) h_j }{\int_{I_j} U\left( \pi  v(x) \right) \ud x - U\left( \overline{v}_j \right) h_j } , 1 \right\}.
\end{equation}
If the denominator in \eqref{def theta} is zero, we simply set $\theta_j[ v]:=1$. Firstly, from the above definition we can see that $\tilde\pi  v= v$ for all $ v\in\mathbb{V}^k$, since in which case $\pi v= v$ so that $\theta_j[ v]=1$. Therefore, our new operator $\tilde\pi$ is indeed a projection. Note that $\tilde \pi$ still decreases the $L^2$-norm on each cell, since it preserves the average and decreases the high frequency modes. The next lemma shows that $\tilde\pi$ also decreases the integral of $U(\cdot)$:
\begin{lemma}\label{stability of projection}
For all $v\in L^2(I)$, we have $0\le \theta_j[v] \le 1$ and
\begin{equation}\label{theta}
\int_{I_j} U\left(\tilde\pi v(x) \right) \ud x=\int_{I_j} U\left(\theta_j[v] \pi v(x) + (1-\theta_j[v]) \overline{(\pi v)}_j \right) \ud x \le \int_{I_j} U\left( v(x) \right) \ud x.
\end{equation}
\end{lemma}
\begin{proof}
When the denominator in \eqref{def theta} is zero, $\pi v (x) \equiv \overline{v}_j$ for all $x \in I_j$, so that \eqref{theta} holds trivially; so we may assume otherwise. In order to apply Jensen's inequality to the left hand side of \eqref{theta}, we need $0\le \theta_j[v] \le 1$. $\theta_j[v]\le 1$ is already clear from the definition \eqref{def theta}, so we only need to show
\begin{equation*}
    \frac{ \int_{I_j} U\left( v(x) \right) \ud x - U\left( \overline{v}_j \right) h_j }{\int_{I_j} U\left( \pi v(x) \right) \ud x - U\left( \overline{v}_j \right) h_j }\ge 0.
\end{equation*}
This follows easily from applying Jensen's inequality to both the numerator and denominator and noting that $\overline{v}_j=\overline{(\pi v)}_j$:
\begin{align*}
    & U\left( \overline{v}_j \right)  \le \frac{1}{h_j}\int_{I_j} U\left( v(x) \right) \ud x,\\
    & U\left( \overline{v}_j \right)= U\left( \overline{( \pi v)}_j \right) \le \frac{1}{h_j} \int_{I_j} U\left(\pi v(x) \right) \ud x.
\end{align*}
Thus, $0\le \theta_j[v] \le 1$ is proved. Now, by Jensen's inequality again,
\begin{align*}
\int_{I_j} U\left(\theta_j[v] \pi v(x) + (1-\theta_j[v]) \overline{(\pi v)}_j \right) \ud x \le& \theta_j[v] \int_{I_j} U\left( \pi v(x) \right) \ud x + (1-\theta_j[v])U\left( \overline{(\pi v)}_j \right) h_j \\
=& \theta_j[v] \int_{I_j} U\left( \pi v(x) \right) \ud x + (1-\theta_j[v])U\left( \overline{v}_j \right) h_j.
\end{align*}
The right hand side is less than or equal to $\int_{I_j} U\left( v(x) \right) \ud x$ if and only if
\begin{equation*}
    \theta_j\le \frac{ \int_{I_j} U\left( v(x) \right) \ud x - U\left( \overline{v}_j \right) h_j }{\int_{I_j} U\left( \pi v(x) \right) \ud x - U\left( \overline{v}_j \right) h_j },
\end{equation*}
which proves the claim.
\end{proof}
Applying this special projection $\tilde\pi$ to \eqref{modified scheme 1}, we then have
\begin{align}
&\int_{I_j} \frac{\tilde\pi S_{\Delta t}  u(x) -  u(x)}{\Delta t}\cdot  v(x) \ud x \nonumber \\
=& \int_{I_j} \frac{ S_{\Delta t}  u(x) -  u(x)}{\Delta t}\cdot  v(x) \ud x -\frac{1-\theta_j[S_{\Delta t}  u(x)]}{\Delta t} \int_{I_j} (\pi S_{\Delta t}  u(x) - \overline{(\pi S_{\Delta t}  u)}_j) \cdot  v(x) \ud x \nonumber \\
=& \int_{I_j} \frac{ S_{\Delta t}  u(x) -  u(x)}{\Delta t}\cdot  v(x) \ud x -\frac{1-\theta_j[S_{\Delta t}  u(x)]}{\Delta t} \int_{I_j} (S_{\Delta t}  u(x) - \overline{(S_{\Delta t}  u)}_j) \cdot  v(x) \ud x. \label{modified scheme before limit}
\end{align}
Taking the limit of \eqref{modified scheme before limit} as $\Delta t \rightarrow 0$ in the same way as in \eqref{theorem 1 eq 1}, we obtain the following semi-discrete scheme: seek $u\in\mathbb{V}^k$ such that for all $v\in\mathbb{V}^k$ we have\footnote{Here we follow the convention in the DG literature: when the context is clear, we suppress the variable $t$ in the notation $u(\cdot,t)\in\mathbb{V}^{k}$, while keeping in mind that $u$ is an element in $\mathbb{V}^{k}$ parameterized by $t\ge 0$.}
\begin{align}
\quad\int_{I_j}  u_t  v \ud x =& \int_{I_j} \lim_{\Delta t \rightarrow 0}\frac{\tilde\pi S_{\Delta t}  u(x) -  u(x)}{\Delta t}\cdot  v(x) \ud x \nonumber\\
=& -\int_{I_j}  f(u)_x v \ud x + H^+_{j-\frac{1}{2}}\cdot  v^+_{j-\frac{1}{2}} + H^-_{j+\frac{1}{2}}\cdot  v^-_{j+\frac{1}{2}} - A_j[u;U]\int_{I_j}  ( u - \bu_j)  v \ud x, \tag{$\textbf{Scheme A}$}\label{scheme A}
\end{align}
where $\bu_j$ denotes the average of $u$ on $I_j$,
\begin{equation}\label{def A}
    A_j[u;U]=\max\left\{ \lim_{\Delta t\rightarrow 0} \frac{1}{\Delta t} \frac{\int_{I_j}U(\pi S_{\Delta t} u) \ud x-\int_{I_j}U(S_{\Delta t} u) \ud x}{ \int_{I_j}U(\pi S_{\Delta t} u)\ud x -h_j U\left(\frac{1}{h_j}\int_{I_j}S_{\Delta t} u \ud x\right)},0\right\},
\end{equation}
and 
\begin{align}
H^+_{j-\frac{1}{2}}=& \lim_{\Delta t\rightarrow 0} \frac{1}{\Delta t}\int_{I_j^+(\Delta t)}S_{\Delta t} u(x) -  u(x)\ud x=\hf_{j-\frac{1}{2}}-f(u)^+_{j-\frac{1}{2}}, \label{limitH1}\\
H^-_{j+\frac{1}{2}}=& \lim_{\Delta t\rightarrow 0} \frac{1}{\Delta t}\int_{I_j^-(\Delta t)}S_{\Delta t} u(x) -  u(x)\ud x=-\hf_{j+\frac{1}{2}}+f(u)^-_{j+\frac{1}{2}}. \label{limitH2}
\end{align}
The expressions \eqref{limitH1} and \eqref{limitH2} are already proved in \eqref{theorem 1 eq 3} and \eqref{theorem 1 eq 4}.

Although the explicit expression of $A_j[u;U]$ defined in \eqref{def A} can generally be obtained, the resulting formula is too complicated to be implemented in practice. To make \eqref{scheme A} computationally efficient, the following proposition is useful:
\begin{proposition}[Monotonicity of Remainder]\label{monotone}
If there exists $B_j[u;U]$ such that $A_j[u;U] \le B_j[u;U]$ for all $ u\in \mathbb{V}^k$ and $j=1,\cdots,N,$ then replacing $A_j[u;U]$ by $B_j[u;U]$ in \eqref{scheme A} can only decrease its entropy production rate $\odv{}{t}\int_{I_j} U( u)(x,t) \ud x$ on each cell $I_j$. We call this property ``monotonicity of remainder."
\end{proposition}
\begin{proof}
Let $ u\in \mathbb{V}^k$ and $ z:=\pi[U'(u)] \in \mathbb{V}^k$. Taking $ v=z $ in \eqref{scheme A}, we obtain the entropy production rate for $A_j$ as
\begin{align*}
    \left(\odv{}{t}\int_{I_j} U( u) \ud x\right)_A =& \int_{I_j} u_t U'( u) \ud x = \int_{I_j} u_t \pi [U'(u)] \ud x = \int_{I_j} u_t z \ud x\\
    =&-\int_{I_j}  f(u)_x\cdot  z \ud x + H^+_{j-\frac{1}{2}}\cdot  z^+_{j-\frac{1}{2}} + H^-_{j+\frac{1}{2}}\cdot  z^-_{j+\frac{1}{2}} - A_j[u;U]\int_{I_j}  ( u - \bu_j) \cdot  z \ud x.
\end{align*}
Similarly, the entropy production rate for $B_j$ is
\begin{align*}
    \left(\odv{}{t}\int_{I_j} U( u) \ud x\right)_B =&-\int_{I_j}  f(u)_x\cdot  z \ud x + H^+_{j-\frac{1}{2}}\cdot  z^+_{j-\frac{1}{2}} + H^-_{j+\frac{1}{2}}\cdot  z^-_{j+\frac{1}{2}} - B_j[u;U]\int_{I_j}  ( u - \bu_j) \cdot  z \ud x.
\end{align*}
Therefore,
\begin{align*}
    &\left(\odv{}{t}\int_{I_j} U( u) \ud x\right)_B - \left(\odv{}{t}\int_{I_j} U( u) \ud x\right)_A = (A_j[u;U]- B_j[u;U])\int_{I_j}  ( u - \bu_j) \cdot  \pi[U'(u)] \ud x\\
    =& (A_j[u;U]- B_j[u;U])\int_{I_j}  ( u - \bu_j) \cdot  U'(u) \ud x \quad \text{(since $u-\bu_j \in \mathbb{P}^k(I_j)$)}\\
    =& (A_j[u;U]- B_j[u;U])\int_{I_j}  ( u - \bu_j) \cdot  (U'(u)-U'(\bu_j)) \ud x \le 0,
\end{align*}
where in the  last equality, we used the fact that $u-\bu_j$ is orthogonal to the constant $U'(\bu_j)$. The last inequality follows from $A_j[u;U]\le B_j[u;U]$ and $( u(x) - \bu_j) \cdot (U'( u(x))-U'(\bu_j))\ge 0$ since $U(\cdot)$ is convex. Thus, the proposition is proved.
\end{proof}

It follows from Proposition \ref{monotone} that we can instead derive an upper bound $B_j[u;U]$ for $A_j[u;U]$, and the scheme with $B_j$ satisfies the entropy condition if $A_j$ does. Moreover, as we will see shortly, if $B_j$ is a bound uniform in an infinite class of entropies, the modified scheme will satisfy infinitely many entropy inequalities instead of just two (the prescribed $U(\cdot)$ and the original square-entropy). In the rest of this section, we will find a bound on $A_j[u;U]$ as tight as possible, so that the resulting scheme is still of optimal order of accuracy. To this end, we let $P_{i,j}(x)$ denote the $i$th Legendre polynomial normalized on $I_j$ satisfying $\|P_{i,j}(x)\|_{L^2(I_j)}=1$. Then,
\begin{align*}
&\pi S_{\Delta t} u(x)- u(x)\\
=& \sum_{i=0}^k P_{i,j}(x) \int_{I_j} P_{i,j}(x)S_{\Delta t} u(x)\ud x - u(x)\\
=& \sum_{i=0}^k P_{i,j}(x) \int_{I_j} P_{i,j}(x)(S_{\Delta t} u(x)- u(x))\ud x \\
=& \Delta t \sum_{i=0}^k P_{i,j}(x) \bigg\{ -\int_{I_j} f( u)_x P_{i,j}(x)\ud x + P_{i,j}(x^-_{j+\frac{1}{2}})H^-_{j+\frac{1}{2}} +P_{i,j}(x^+_{j-\frac{1}{2}})H^+_{j-\frac{1}{2}} \bigg\} +\mathcal{O}(\Delta t^2)\\
=& -\Delta t \pi \circ [ f( u)_x](x)  + \Delta t \sum_{i=0}^k P_{i,j}(x) P_{i,j}(x^-_{j+\frac{1}{2}})H^-_{j+\frac{1}{2}} + \Delta t \sum_{i=0}^k P_{i,j}(x) P_{i,j}(x^+_{j-\frac{1}{2}})H^+_{j-\frac{1}{2}} +\mathcal{O}(\Delta t^2),
\end{align*}
Therefore, 
\begin{align}
&\int_{I_j}U(\pi S_{\Delta t} u(x)) \ud x -\int_{I_j} U( u(x)) \ud x \nonumber\\
=&  \int_{I_j} U'( u(x))\cdot(\pi S_{\Delta t} u(x)- u(x)) \ud x + \mathcal{O}(\Delta t^2) \nonumber\\
=& \int_{I_j} U'( u(x))\cdot\bigg\{-\Delta t \pi \circ [ f( u)_x](x)  + \Delta t \sum_{i=0}^k P_{i,j}(x) P_{i,j}(x^-_{j+\frac{1}{2}})H^-_{j+\frac{1}{2}} \nonumber\\
&\qquad\qquad + \Delta t \sum_{i=0}^k P_{i,j}(x) P_{i,j}(x^+_{j-\frac{1}{2}})H^+_{j-\frac{1}{2}} \bigg\} \ud x +\mathcal{O}(\Delta t^2) \nonumber\\
=& -\Delta t\int_{I_j} U'( u(x)) \cdot \pi \circ [ f( u)_x](x)\ud x +\Delta t \pi \circ [U'( u)](x^-_{j+\frac{1}{2}})\cdot H^-_{j+\frac{1}{2}} \nonumber\\
&+\Delta t \pi \circ [U'( u)](x^+_{j-\frac{1}{2}})\cdot H^+_{j-\frac{1}{2}} +\mathcal{O}(\Delta t^2).\label{FirstTerm}
\end{align}
On the other hand,
\begin{align}
&\int_{I_j}U(S_{\Delta t} u(x)) \ud x -\int_{I_j} U( u(x)) \ud x \nonumber\\
=&  \int_{I_j \backslash (I_j^+(\Delta t)\cup I_j^-(\Delta t))}U(S_{\Delta t} u(x)) - U( u(x)) \ud x+\int_{I_j^+(\Delta t)}U(S_{\Delta t} u(x)) - U( u(x)) \ud x \nonumber\\
&+\int_{I_j^-(\Delta t)}U(S_{\Delta t} u(x)) - U( u(x)) \ud x \nonumber\\
=& -\Delta t \int_{I_j} U'( u(x)) \cdot  f( u)_x \ud x +\Delta t G^+_{j-\frac{1}{2}} +\Delta t G^-_{j+\frac{1}{2}} +\mathcal{O}(\Delta t^2), \label{SecondTerm}
\end{align}
where
\begin{align}
    &G^+_{j-\frac{1}{2}}:=\lim_{\Delta t\rightarrow 0} \frac{1}{\Delta t}\int_{I_j^+(\Delta t)}U(S_{\Delta t} u(x)) - U( u(x))\ud x,\label{limitG1}\\
    &G^-_{j+\frac{1}{2}}:=\lim_{\Delta t\rightarrow 0} \frac{1}{\Delta t}\int_{I_j^-(\Delta t)}U(S_{\Delta t} u(x)) - U( u(x))\ud x.\label{limitG2}
\end{align}
The exact expressions of $G^+_{j-\frac{1}{2}}$ and $G^-_{j+\frac{1}{2}}$ are not important, since they will be eliminated in the subsequent computations.

It follows from \eqref{FirstTerm} and \eqref{SecondTerm} that
\begin{align}
&\lim_{\Delta t\rightarrow 0} \frac{\int_{I_j}U(\pi S_{\Delta t} u(x)) \ud x-\int_{I_j}U(S_{\Delta t} u(x)) \ud x}{\Delta t} \nonumber\\
=& \int_{I_j} U'( u) \cdot ( f( u)_x - \pi \circ [ f( u)_x])\ud x + \pi \circ [U'( u)](x^-_{j+\frac{1}{2}}) \cdot H^-_{j+\frac{1}{2}} \nonumber\\
&+ \pi \circ [U'( u)](x^+_{j-\frac{1}{2}}) \cdot H^+_{j-\frac{1}{2}} - G^-_{j+\frac{1}{2}} - G^+_{j-\frac{1}{2}} \nonumber\\
=& \int_{I_j} U'( u) \cdot ( f( u)_x - \pi \circ [ f( u)_x])\ud x \label{a1}\\
& + \pi \circ [U'( u)- U'( u^+_{j-\frac{1}{2}})](x^+_{j-\frac{1}{2}}) \cdot H^+_{j-\frac{1}{2}} \label{a2}\\
& + \pi \circ [U'( u)- U'( u^-_{j+\frac{1}{2}})](x^-_{j+\frac{1}{2}}) \cdot H^-_{j+\frac{1}{2}} \label{a3}\\
& + (U'( u^+_{j-\frac{1}{2}})\cdot H^+_{j-\frac{1}{2}}-G^+_{j-\frac{1}{2}}) \label{a4}\\
& + (U'( u^-_{j+\frac{1}{2}})\cdot H^-_{j+\frac{1}{2}}-G^-_{j+\frac{1}{2}}). \label{a5}
\end{align}

This is the explicit expression for the numerator in \eqref{def A}. We first prove that both \eqref{a4} and \eqref{a5} in the last expression are non-positive.
\begin{lemma}\label{estimate1}
\begin{align*}
    U'( u^+_{j-\frac{1}{2}})\cdot H^+_{j-\frac{1}{2}}-G^+_{j-\frac{1}{2}} \le 0,\quad U'( u^-_{j+\frac{1}{2}})\cdot H^-_{j+\frac{1}{2}}-G^-_{j+\frac{1}{2}}\le 0.
\end{align*}
\end{lemma}
\begin{proof}
It follows from \eqref{limitH1} and \eqref{limitH2} that
\begin{align*}
U'( u^+_{j-\frac{1}{2}})\cdot H^+_{j-\frac{1}{2}}=& \lim_{\Delta t\rightarrow 0} \frac{1}{\Delta t}\int_{I_j^+(\Delta t)}U'( u(x))\cdot (S_{\Delta t} u(x) -  u(x))\ud x, \\
U'( u^-_{j+\frac{1}{2}})\cdot H^-_{j+\frac{1}{2}}=& \lim_{\Delta t\rightarrow 0} \frac{1}{\Delta t}\int_{I_j^-(\Delta t)}U'( u(x))\cdot (S_{\Delta t} u(x) -  u(x))\ud x. 
\end{align*}
From the definitions of $G^+_{j-\frac{1}{2}}$ and $G^-_{j+\frac{1}{2}}$ in \eqref{limitG1} and \eqref{limitG2}, we see that it suffices to prove 
\begin{align*}
\int_{I_j^+(\Delta t)}U'( u(x))\cdot (S_{\Delta t} u(x) -  u(x))\ud x - \int_{I_j^+(\Delta t)}U(S_{\Delta t} u(x)) - U( u(x))\ud x \le 0, \\
\int_{I_j^-(\Delta t)}U'( u(x))\cdot (S_{\Delta t} u(x) -  u(x))\ud x - \int_{I_j^-(\Delta t)}U(S_{\Delta t} u(x)) - U( u(x))\ud x \le 0, 
\end{align*}
since taking the limit does not change sign. By the convexity of $U(\cdot)$, the above two inequalities are trivial since
\begin{equation*}
U'( u(x))\cdot (S_{\Delta t} u(x) -  u(x)) - \left(U(S_{\Delta t} u(x)) - U( u(x))\right) \le 0, \quad \forall x \in I.
\end{equation*}
\end{proof}

Next, we estimate \eqref{a2} and \eqref{a3}.
\begin{lemma}\label{estimate2}
For any $u \in \mathbb V^k$, we have
\begin{align*}
&|\pi \circ [U'( u)- U'( u^+_{j-\frac{1}{2}})](x^+_{j-\frac{1}{2}})|\le \frac{\beta_k}{2h_j}\max_{x\in I_j}|U'''( u(x))| \int_{I_j} | u(x)-\bu_j|^2 \ud x,\\
&|\pi \circ [U'( u)- U'( u^-_{j+\frac{1}{2}})](x^-_{j+\frac{1}{2}})|\le \frac{\beta_k}{2h_j}\max_{x\in I_j}|U'''( u(x))| \int_{I_j} | u(x)-\bu_j|^2 \ud x,
\end{align*}
where $\beta_k\le (k+1)^4$. 
\end{lemma}

\begin{proof}
Since $U''( u^+_{j-\frac{1}{2}})( u(x)- u^+_{j-\frac{1}{2}})$ is a $k$-th order polynomial on $I_j$ and vanishes at $x=x^+_{j-\frac{1}{2}}$,
\begin{align*}
&|\pi \circ [U'( u)- U'( u^+_{j-\frac{1}{2}})](x^+_{j-\frac{1}{2}})|\\
=& |\pi \circ [U'( u)- U'( u^+_{j-\frac{1}{2}})- U''( u^+_{j-\frac{1}{2}})( u- u^+_{j-\frac{1}{2}})](x^+_{j-\frac{1}{2}})|\\
=& \bigg|\sum_{i=0}^k P_{i,j}(x^+_{j-\frac{1}{2}}) \int_{I_j} P_{i,j}(x)[U'( u(x))- U'( u^+_{j-\frac{1}{2}})- U''( u^+_{j-\frac{1}{2}})( u(x)- u^+_{j-\frac{1}{2}})]\ud x\bigg|\\
\le& \sum_{i=0}^k \|P_{i,j}(\cdot)\|_{L^\infty(I_j)}^2 \int_{I_j} |U'( u(x))- U'( u^+_{j-\frac{1}{2}})- U''( u^+_{j-\frac{1}{2}})( u(x)- u^+_{j-\frac{1}{2}})|\ud x \\
\le& \sum_{i=0}^k \frac{2i+1}{h_j} \frac{\max_{x\in I_j}|U'''( u(x))|}{2} \int_{I_j} | u(x)- u^+_{j-\frac{1}{2}}|^2 \ud x \\
=& \frac{(k+1)^2}{2h_j} \max_{x\in I_j}|U'''( u(x))| \int_{I_j} | u(x)- u^+_{j-\frac{1}{2}}|^2 \ud x\\
\le& \frac{(k+1)^4}{2h_j} \max_{x\in I_j}|U'''( u(x))| \int_{I_j} | u(x)-\bu_j|^2 \ud x,
\end{align*}
where in the last step, the following norm equivalence on the space of all $k$-th order polynomials with zero average on $I_j$ has been used:
\begin{equation*}
    \| p-p^+_{j-\frac{1}{2}} \|_{L^2(I_j)}\le (k+1) \| p \|_{L^2(I_j)}, \quad \forall p\in \mathbb{P}^k(I_j)\text{ with }\overline{p}_j=0.
\end{equation*}
The estimate for $|\pi \circ [U'( u)- U'( u^-_{j+\frac{1}{2}})](x^-_{j+\frac{1}{2}})|$ is proved in the same way.
\end{proof}

\begin{remark}
    A substantially improved estimate on $\beta_k$ for small $k$ can be found in Appendix \ref{section:estimate beta_k}. In Proposition \ref{prop:improved beta} we show that $\beta_k \le 2 \mathcal{A}_k$, where the first few values of $\mathcal{A}_k$ are tabulated in Table \ref{tab:improved beta} therein.
\end{remark}

Finally, we estimate \eqref{a1}.
\begin{lemma}\label{estimate3}
For any $\vu\in \mathbb V^k$, we have
\begin{align*}
    \bigg|\int_{I_j} U'( u)\cdot &( f( u)_x - \pi \circ [ f( u)_x])\ud x\bigg|\\
    &\le \frac{ C_k}{2\sqrt{h_j}} \max_{x\in I_j}|U'''( u(x))| \| f( u)_x - \pi \circ [ f( u)_x]\|_{L^2(I_j)}\int_{I_j} | u(x)-\bu_j|^2 \ud x 
\end{align*}
where $C_k>0$ is the optimal constant such that
\begin{equation}\label{C_k}
    \|p\|_{L^4(I_j)}\le \frac{\sqrt{C_k}}{h_j^{1/4}} \|p\|_{L^2(I_j)},\quad \forall p\in \mathbb{P}^k(I_j).
\end{equation}
\end{lemma}
\begin{proof}
Note that $U'(\bu_j)+ U''(\bu_j)( u(x)-\bu_j)$ is a $k$-th order polynomial on $I_j$, which is orthogonal to $ f( u)_x - \pi \circ [ f( u)_x]$. Therefore,
\begin{align*}
    &\bigg|\int_{I_j} U'( u)\cdot ( f( u)_x - \pi \circ [ f( u)_x])\ud x\bigg|\\
    =& \bigg|\int_{I_j} (U'( u)- U'(\bu_j)- U''(\bu_j)( u-\bu_j))\cdot ( f( u)_x - \pi \circ [ f( u)_x])\ud x\bigg| \\
    \le& \frac{\max_{x\in I_j}| U'''( u(x))|}{2}\int_{I_j}  | u-\bu_j|^2 | f( u)_x - \pi \circ [ f( u)_x]|\ud x \\
    \le&  \frac{\max_{x\in I_j}| U'''( u(x))|}{2} \| f( u)_x - \pi \circ [ f( u)_x]\|_{L^2(I_j)} \sqrt{\int_{I_j} | u-\bu_j|^4 \ud x}  \\
    \stackrel{\eqref{C_k}}{\le}& \frac{C_k}{2\sqrt{h_j}} \max_{x\in I_j}| U'''( u(x))| \| f( u)_x - \pi \circ [ f( u)_x]\|_{L^2(I_j)}\int_{I_j} | u(x)-\bu_j|^2 \ud x.
\end{align*}
\end{proof}

\begin{remark}
The first few values of $C_k$ in \eqref{C_k} are
\FloatBarrier
\begin{table}[htbp]
\centering
\begin{tabular}{ c|c|c|c|c|c|c|c|c } 
 $k$ & 0 & 1 & 2 & 3 & 4 & 5 & 6 & 7 \\ 
 \hline
 $C_k$ & 1 & 1.5 & 2.156 & 2.834 & 3.520 & 4.209 & 4.901 & 5.594
\end{tabular}
\caption{Optimal constants $C_k$.}
\label{tab:Ck}
\end{table}
For a general $k\ge 0$ we have the convenient but non-sharp bound
\begin{equation*}
    C_k \le k+1,
\end{equation*}
since
\begin{equation*}
    \|p\|_{L^4(I_j)}^2 \le \|p\|_{L^\infty(I_j)} \|p\|_{L^2(I_j)} \quad \text{and} \quad \|p\|_{L^\infty(I_j)} \le \frac{k+1}{\sqrt{h_j}}\|p\|_{L^2(I_j)}.
\end{equation*}
\end{remark}

Combining Lemma \ref{estimate1}, \ref{estimate2} and \ref{estimate3}, we obtain
\begin{align}
&\lim_{\Delta t\rightarrow 0} \frac{\int_{I_j}U(\pi S_{\Delta t} u(x))\ud x-\int_{I_j}U(S_{\Delta t} u(x))\ud x}{\Delta t} \nonumber\\
\le& \frac{ C_k}{2\sqrt{h_j}} \max_{x\in I_j}| U'''( u(x))| \| f( u)_x - \pi \circ [ f( u)_x]\|_{L^2(I_j)}\int_{I_j} | u(x)-\bu_j|^2 \ud x \nonumber\\
&+\frac{\beta_k}{2h_j}\max_{x\in I_j}| U'''( u(x))| \left(|H^+_{j-\frac{1}{2}}|+|H^-_{j+\frac{1}{2}}|\right) \int_{I_j} | u(x)-\bu_j|^2 \ud x. \label{nominator estimate}
\end{align}

By Hölder’s Defect Formula \cite{Steele_2004} (also known as \emph{the Jensen Gap}), we also have
\begin{align}
    &\lim_{\Delta t\rightarrow 0} \int_{I_j}U(\pi S_{\Delta t} u)\ud x -h_j U\left(\frac{1}{h_j}\int_{I_j}S_{\Delta t} u \ud x\right) \nonumber\\
    =&\int_{I_j}U( u)\ud x -h_j U(\bu_j)\ge \frac{1}{2}\min_{x\in I_j}| U''( u(x))| \int_{I_j} | u(x)-\bu_j|^2, \label{Jensen}
\end{align}
It follows from \eqref{nominator estimate} and \eqref{Jensen} that
\begin{align}
    A_j[u;U]=&\max\left\{ \lim_{\Delta t\rightarrow 0} \frac{1}{\Delta t} \frac{\int_{I_j}U(\pi S_{\Delta t} u) \ud x-\int_{I_j}U(S_{\Delta t} u) \ud x}{ \int_{I_j}U(\pi S_{\Delta t} u)\ud x -h_j U\left(\frac{1}{h_j}\int_{I_j}S_{\Delta t} u \ud x\right)},0\right\}\nonumber\\
    \le& \frac{\max_{x\in I_j}| U'''( u(x))|}{\min_{x\in I_j}|U''( u(x))|} \bigg\{ \frac{C_k}{\sqrt{h_j}} \| f( u)_x - \pi \circ [ f( u)_x]\|_{L^2(I_j)} +\frac{\beta_k}{h_j}\left(|H^+_{j-\frac{1}{2}}|+|H^-_{j+\frac{1}{2}}|\right) \bigg\}.\label{bound A_j}
\end{align}
In the above, we implicitly used the following inequality:
\begin{equation*}
    \max\{a,0\}\le\max\{b,0\},\quad\forall a<b.
\end{equation*}
As the next theorem shows, we only need half of the bound in \eqref{bound A_j} to satisfy the entropy condition for $U(\cdot)$. Therefore, we define
\begin{equation}\label{def B_j}
    B_j[u;U]:= \frac{\max_{x\in I_j}| U'''( u(x))|}{\min_{x\in I_j}|U''( u(x))|} \bigg\{ \frac{C_k}{2 \sqrt{h_j}} \| f( u)_x - \pi \circ [ f( u)_x]\|_{L^2(I_j)}+\frac{\beta_k}{2 h_j}\left(|H^+_{j-\frac{1}{2}}|+|H^-_{j+\frac{1}{2}}|\right) \bigg\}.
\end{equation}
Now, replacing $A_j[u;U]$ in \eqref{scheme A} with the $B_j[u;U]$ defined above yields the following new scheme: seek $u\in\mathbb{V}^k$ such that for all $v\in\mathbb{V}^k$ we have
\begin{align}
\int_{I_j}  u_t  v \ud x = -\int_{I_j}  f(u)_x  v \ud x + H^+_{j-\frac{1}{2}}\cdot  v^+_{j-\frac{1}{2}} + H^-_{j+\frac{1}{2}}\cdot  v^-_{j+\frac{1}{2}} - B_j[u;U]\int_{I_j}  ( u - \bu_j) v \ud x, \tag{$\textbf{Scheme B}$}\label{scheme B}
\end{align}
where $H^+_{j-\frac{1}{2}}$ and $H^-_{j+\frac{1}{2}}$ are defined as
\begin{equation}\label{def H}
    H^+_{j-\frac{1}{2}}:=\hf_{j-\frac{1}{2}}-f(u)^+_{j-\frac{1}{2}},\quad H^-_{j+\frac{1}{2}}:=-\hf_{j+\frac{1}{2}}+f(u)^-_{j+\frac{1}{2}},
\end{equation}
according to \eqref{limitH1} and \eqref{limitH2}.
After performing integration by parts on $\int_{I_j}  f(u)_x\cdot  v \ud x$, we get the equivalent weak formulation of \eqref{scheme B} as
\begin{align}
\int_{I_j}  u_t  v \ud x - \int_{I_j}  f(u)  v_x \ud x + \hf_{j+\frac{1}{2}}\cdot  v^-_{j+\frac{1}{2}} - \hf_{j-\frac{1}{2}}\cdot  v^+_{j-\frac{1}{2}}= - B_j[u;U]\int_{I_j}  ( u - \bu_j)  v \ud x, \label{weak scheme B}
\end{align}
which is really \eqref{DG1} plus a remainder term on the right-hand side. Note that the numerical flux $\hf$ computed using the expressions \eqref{limitH1}--\eqref{limitH2} is the Godunov flux. This follows from the fact that $S_{\Delta t}$ is the exact solution operator. Now, however, \eqref{scheme B} has nothing to do with $S_{\Delta t}$ if we just define $H^+_{j-\frac{1}{2}}$ and $H^-_{j+\frac{1}{2}}$ by \eqref{def H} instead of using \eqref{limitH1}--\eqref{limitH2}; thus, \eqref{scheme B} is still well-defined if any other numerical flux is used in the definitions \eqref{def H}. 
\begin{definition}\label{def E-flux}
    $\hf$ is an E-flux \cite{Osher} if it is consistent ($\hf(u,u)=f(u)$), locally Lipschitz continuous, and satisfies
    \begin{equation*}
        (\hf(u^-,u^+)-f(u))(u^+-u^-)\le 0
    \end{equation*}
    for all $u$ between $u^-$ and $u^+$.
\end{definition}

All monotone fluxes, such as the Godunov flux, local Lax-Friedrichs flux, or the more general class of HLL fluxes with suitable lower and upper wave-speed bounds (see \eqref{HLL}), are E-fluxes. The next theorem shows that \eqref{scheme B} satisfies local entropy inequality for $U(\cdot)$, provided that $\hf$ is any E-flux.
\begin{theorem}[Local Entropy Inequality for \eqref{scheme B}]\label{theorem: Local Entropy Inequality B}
Suppose that $\hf$ in \eqref{def H} is a consistent E-flux. Then \eqref{scheme B} satisfies the local semi-discrete entropy inequality
\begin{equation}\label{scheme B entropy inequality U}
\odv{}{t}\int_{I_j} U( u) \ud x + \hat\mF_{j+\frac{1}{2}} - \hat\mF_{j-\frac{1}{2}}\le 0,\quad \forall u\in\mathbb{V}^k,\quad \forall j=1,\cdots,N,
\end{equation}
where $\hat\mF_{j+\frac{1}{2}}=\hat\mF( u^-_{j+\frac{1}{2}}, u^+_{j+\frac{1}{2}})$ and
\begin{equation}\label{entropy flux}
\hat\mF( u^-, u^+):=U'(u^-)\hf(u^-,u^+)-\psi(u^-),
\end{equation}
with $\psi$ defined by \eqref{def psi}.
\end{theorem}
\begin{proof}
Set
\[
    z:=\pi[U'(u)]\in\mathbb V^k.
\]
Since $u_t\in\mathbb V^k$ and $\pi$ is the $L^2$-projection onto $\mathbb V^k$, we have
\begin{equation}\label{B proof time projection}
\int_{I_j}u_tz\,\ud x
=\int_{I_j}u_tU'(u)\,\ud x
=\odv{}{t}\int_{I_j}U(u)\,\ud x.
\end{equation}
Taking $v=z$ in \eqref{scheme B} therefore gives
\begin{align}
\odv{}{t}\int_{I_j}U(u)\,\ud x
={}&-\int_{I_j}f(u)_x z\,\ud x
+H^+_{j-\frac12}z^+_{j-\frac12}
+H^-_{j+\frac12}z^-_{j+\frac12} \nonumber\\
&-B_j[u;U]\int_{I_j}(u-\bu_j)z\,\ud x. \label{B proof DG identity}
\end{align}

We first rewrite the volume term. By the self-adjointness of the $L^2$-projection,
\[
\int_{I_j}f(u)_x z\,\ud x
=\int_{I_j}\pi[f(u)_x]U'(u)\,\ud x.
\]
Hence, using $\mF'(u)=U'(u)f'(u)$,
\begin{align}
-\int_{I_j}f(u)_x z\,\ud x
={}&-\int_{I_j}U'(u)f(u)_x\,\ud x
+\int_{I_j}U'(u)\bigl(f(u)_x-\pi[f(u)_x]\bigr)\,\ud x \nonumber\\
={}&\mF(u^+_{j-\frac12})-\mF(u^-_{j+\frac12})+R_{j,1}, \label{B proof volume}
\end{align}
where
\begin{equation}\label{B proof R1}
R_{j,1}:=
\int_{I_j}U'(u)\bigl(f(u)_x-\pi[f(u)_x]\bigr)\,\ud x.
\end{equation}
Also, since $u-\bu_j\in\mathbb V^k$ and has zero average,
\begin{align}
\int_{I_j}(u-\bu_j)z\,\ud x
&=\int_{I_j}(u-\bu_j)U'(u)\,\ud x \nonumber\\
&=\int_{I_j}(u-\bu_j)\bigl(U'(u)-U'(\bu_j)\bigr)\,\ud x
=:D_j. \label{B proof cell dissipation}
\end{align}
By the mean value theorem and strict convexity of $U$,
\begin{equation}\label{B proof Dj lower bound}
D_j\ge m_j\int_{I_j}|u-\bu_j|^2\,\ud x\ge0,
\end{equation}
where
\begin{equation*}
    m_j:= \min_{x \in I_j} |U''(u(x))|.
\end{equation*}

We next treat the interface terms. At $x_{j+\frac12}$, it follows from \eqref{def psi} and \eqref{entropy flux} that
\begin{align*}
\hat\mF_{j+\frac12}
&=U'(u^-_{j+\frac12})\hf_{j+\frac12}-\psi(u^-_{j+\frac12})\\
&=\mF(u^-_{j+\frac12})
+U'(u^-_{j+\frac12})\bigl(\hf_{j+\frac12}-f(u^-_{j+\frac12})\bigr)\\
&=\mF(u^-_{j+\frac12})-U'(u^-_{j+\frac12})H^-_{j+\frac12}.
\end{align*}
Consequently,
\begin{align}
&-\mF(u^-_{j+\frac12})
+H^-_{j+\frac12}z^-_{j+\frac12}
+\hat\mF_{j+\frac12} \nonumber\\
&\qquad
=H^-_{j+\frac12}
\pi\bigl[U'(u)-U'(u^-_{j+\frac12})\bigr](x^-_{j+\frac12})
=:R_{j,3}. \label{B proof right interface}
\end{align}

At $x_{j-\frac12}$, put
\[
    u^-:=u^-_{j-\frac12},\qquad
    u^+:=u^+_{j-\frac12},\qquad
    h:=\hf(u^-,u^+).
\]
For convenience define the same numerical entropy-flux expression evaluated with the right state by
\[
    \hat\mF^R(u^-,u^+):=U'(u^+)h-\psi(u^+).
\]
Since $H^+_{j-\frac12}=h-f(u^+)$,
\begin{align}
&\mF(u^+_{j-\frac12})
+H^+_{j-\frac12}z^+_{j-\frac12}
-\hat\mF^R(u^-,u^+) \nonumber\\
&\qquad
=H^+_{j-\frac12}
\pi\bigl[U'(u)-U'(u^+_{j-\frac12})\bigr](x^+_{j-\frac12})
=:R_{j,2}. \label{B proof left projection}
\end{align}
On the other hand,
\begin{align}
\hat\mF(u^-,u^+)-\hat\mF^R(u^-,u^+)
&=h\bigl(U'(u^-)-U'(u^+)\bigr)-\bigl(\psi(u^-)-\psi(u^+)\bigr)\nonumber\\
&=\int_{u^-}^{u^+}U''(s)\bigl(f(s)-h\bigr)\,\ud s
=:\mathcal D(u^-,u^+). \label{B proof interface entropy dissipation}
\end{align}
Due to the property of E-flux and $U''\ge0$, we have
\begin{equation}\label{B proof interface entropy nonnegative}
    \mathcal D(u^-,u^+)\ge0.
\end{equation}
Thus, by \eqref{B proof left projection},
\begin{align}
&\mF(u^+_{j-\frac12})
+H^+_{j-\frac12}z^+_{j-\frac12}
-\hat\mF_{j-\frac12}
=R_{j,2}-\mathcal D_{j-\frac12}, \label{B proof left interface}
\end{align}
where
\[
    \mathcal D_{j-\frac12}
    :=\mathcal D(u^-_{j-\frac12},u^+_{j-\frac12})\ge0.
\]

Adding $\hat\mF_{j+\frac12}-\hat\mF_{j-\frac12}$ to \eqref{B proof DG identity} and using \eqref{B proof volume}, \eqref{B proof right interface}, \eqref{B proof left interface}, and \eqref{B proof cell dissipation}, we obtain the exact identity
\begin{align}
&\odv{}{t}\int_{I_j}U(u)\,\ud x
+\hat\mF_{j+\frac12}-\hat\mF_{j-\frac12} \nonumber\\
&\qquad
=R_{j,1}+R_{j,2}+R_{j,3}
-\mathcal D_{j-\frac12}-B_j[u;U]D_j. \label{B proof master identity}
\end{align}

It remains to estimate the three projection errors. Lemma \ref{estimate3} gives
\begin{align}
|R_{j,1}|
\le{}&\frac{C_k}{2\sqrt{h_j}}
\max_{x\in I_j}|U'''(u(x))|
\|f(u)_x-\pi[f(u)_x]\|_{L^2(I_j)}
\int_{I_j}|u-\bu_j|^2\,\ud x, \label{B proof R1 estimate}
\end{align}
and Lemma \ref{estimate2} gives
\begin{align}
|R_{j,2}|+|R_{j,3}|
\le{}&\frac{\beta_k}{2h_j}
\max_{x\in I_j}|U'''(u(x))|
\bigl(|H^+_{j-\frac12}|+|H^-_{j+\frac12}|\bigr)
\int_{I_j}|u-\bu_j|^2\,\ud x. \label{B proof R23 estimate}
\end{align}
By the definition of $B_j[u;U]$, \eqref{B proof R1 estimate}--\eqref{B proof R23 estimate} imply
\begin{align}
R_{j,1}+R_{j,2}+R_{j,3}
&\le |R_{j,1}|+|R_{j,2}|+|R_{j,3}|\nonumber\\
&\le B_j[u;U]m_j
\int_{I_j}|u-\bu_j|^2\,\ud x\nonumber\\
&\le B_j[u;U]D_j, \label{B proof total error bound}
\end{align}
where the last inequality follows from \eqref{B proof Dj lower bound}. Substituting \eqref{B proof total error bound} into \eqref{B proof master identity} yields
\begin{align*}
\odv{}{t}\int_{I_j}U(u)\,\ud x
+\hat\mF_{j+\frac12}-\hat\mF_{j-\frac12}
&\le -\mathcal D_{j-\frac12} \le0,
\end{align*}
which proves \eqref{scheme B entropy inequality U}.
\end{proof}

\begin{remark}
    The above proof is very much parallel to the proof for the square entropy $U(u)=\frac{u^2}{2}$ given in \cite{J}, the only difference being that the test function $v$ is simply taken as $u=U'(u)$ therein. When $U(\cdot)$ is not the square entropy, however, we cannot take $v=U'(u)$ since $U'(u)$ is not a piecewise polynomial. Thus, the only way is by setting $v=\pi[U'(u)]$ and then estimating the projection error, which can be done easily thanks to Lemmas \ref{estimate2} and \ref{estimate3}.
\end{remark}

Note that in \eqref{scheme B}, $U$ only appears in the quotient
\begin{equation}\label{entropy quotient}
    q_j[u;U]:=\frac{\max_{x\in I_j}| U'''( u(x))|}{\min_{x\in I_j}|U''( u(x))|}
\end{equation}
in the expression of $B_j[u;U]$, and nowhere else! By monotonicity of remainder (Proposition \ref{monotone}), \eqref{scheme B} still preserves the entropy inequality with a larger quotient. 

\begin{corollary}\label{corollary: also L^2}
    \eqref{scheme B} also satisfies the local entropy inequality for the square-entropy $V(u)=\frac{u^2}{2}$, that is, inequality \eqref{scheme B entropy inequality U} with $(U,\mF)$ replaced by $V$ and the corresponding square-entropy flux. This is nothing other than the well-known result for \eqref{DG1} proved in \cite{J}. Thus, \eqref{scheme B} with any $U(u)\neq \frac{u^2}{2}$ satisfies at least two local entropy inequalities.
\end{corollary}
\begin{proof}
    This follows from $B_j[u;U]\ge B_j[u;V]\equiv 0$ and monotonicity of remainder.
\end{proof}
Moreover, if $U$ is uniformly convex, i.e. $\min_{w\in\mathbb{R}}U''(w)>0$, then we can replace \eqref{entropy quotient} in $B_j[u;U]$ by the larger quotient
\begin{equation}\label{entropy quotient weaker}
    \tilde q_j[u;U]:=\frac{\max_{x\in I_j}| U'''( u(x))|}{\min_{w\in\mathbb R}U''(w)} \ge q_j[u;U].
\end{equation}
Denote the new $B_j[u;U]$ with $\tilde q_j[u;U]$ by $\tilde B_j[u;U]$. Clearly, $\tilde B_j[u;U] \ge B_j[u;U]$, so Theorem \ref{theorem: Local Entropy Inequality B} also holds for $\tilde B_j[u;U]$ by monotonicity of remainder.
\begin{corollary}\label{corollary: infinite entropy}
    Suppose $\min_{w\in\mathbb{R}}U''(w)=a>0$ and we use $\tilde B_j[u;U]$ in \eqref{scheme B}. Then, for all uniformly convex function $V(\cdot)$ satisfying
    \begin{equation}\label{other entropy condition}
     a\le V''(w),\quad |V'''(w)| \le |U'''(w)|, \quad \forall w\in\mathbb{R},
    \end{equation}
    \eqref{scheme B} also satisfies the local entropy inequality for $V$.
\end{corollary}
\begin{proof}
    This follows from $\tilde B_j[u;U]\ge \tilde B_j[u;V]$ and monotonicity of remainder.
\end{proof}
The next corollary states that for a special class of $U$, \eqref{scheme B} with $\tilde B_j[u;U]$ can satisfy infinitely many local entropy inequalities.
\begin{corollary}\label{corollary: infinite entropy 2}
    Suppose we use $\tilde B_j[u;U]$ in \eqref{scheme B} and take $U(u)=\sum_{i=1}^\infty a_i u^{2i}$ to be analytic on $\mathbb{R}$, such that $a_1>0$, $a_i\ge 0$ for all $i>1$ and not all zero. Then, \eqref{scheme B} also satisfies local entropy inequalities for all $V$ defined by
    \begin{equation*}
     V(u)=\sum_{i=1}^\infty b_i u^{2i},\quad \text{with $b_1\ge a_1$ and $0 \le b_i\le a_i$ for all $i>1$.} 
    \end{equation*}
    Thus, with this $U$, \eqref{scheme B} satisfies infinitely many entropy inequalities.
\end{corollary}
\begin{proof}
    $V$ satisfies \eqref{other entropy condition}, so the statement follows from Corollary \ref{corollary: infinite entropy}.
\end{proof}

The term $\| f( u)_x - \pi \circ [ f( u)_x]\|_{L^2(I_j)}$ in $B_j[u;U]$ (or $\tilde B_j[u;U]$) is easy to compute if $f(\cdot)$ is a polynomial. In particular, it vanishes when $f(\cdot)$ is linear. When $f(\cdot)$ is a general nonlinear function, however, this term might be computationally expensive. Since $f(u(x))$ is smooth on each compact interval $I_j$, by Bramble–Hilbert lemma and best approximation property of $\pi$,
\begin{equation}\label{bound on bad term}
    \| f( u)_x - \pi \circ [ f( u)_x]\|_{L^2(I_j)} \le C(k,r) h_j^{r-1} \|\partial_x^{r}f(u)\|_{L^2(I_j)},\quad k+1 \le r \le k+2,
\end{equation}
where $C(k,r)$ is a constant only depending on $k$ and $r$. The right-hand side is still zero if $f(\cdot)$ is linear, and when $f(\cdot)$ is nonlinear it can be computed with relative ease if the derivatives of $f(\cdot)$ have analytical formulae. In light of Proposition \ref{monotone}, we can replace $\| f( u)_x - \pi \circ [ f( u)_x]\|_{L^2(I_j)}$ in $B_j[u;U]$ by this upper bound and get the new constant
\begin{equation}\label{def C_j}
    C_j[u;U,r]:= q_j[u;U] \bigg\{ \frac{C_k C(k,r) h_j^{r-\frac{3}{2}}}{2} \|\partial_x^{r}f(u)\|_{L^2(I_j)} +\frac{\beta_k}{2h_j}\left(|H^+_{j-\frac{1}{2}}|+|H^-_{j+\frac{1}{2}}|\right) \bigg\}.
\end{equation}
Replacing $B_j[u;U]$ by $C_j[u;U,r]\ge B_j[u;U]$ in the weak formulation of \eqref{scheme B}, we get yet another scheme that satisfies the local entropy inequality of $U$: seek $u\in\mathbb{V}^k$ such that for all $v\in\mathbb{V}^k$ we have
\begin{align}
\int_{I_j}  u_t  v \ud x - \int_{I_j}  f(u)  v_x \ud x + \hf_{j+\frac{1}{2}}\cdot  v^-_{j+\frac{1}{2}} - \hf_{j-\frac{1}{2}}\cdot  v^+_{j-\frac{1}{2}}= - C_j[u;U,r]\int_{I_j}  ( u - \bu_j)  v \ud x. \tag{$\textbf{Scheme C}$}\label{scheme C}
\end{align}
Clearly, by monotonicity of remainder, Theorem \ref{theorem: Local Entropy Inequality B} and Corollaries \ref{corollary: also L^2}--\ref{corollary: infinite entropy 2} also hold for \eqref{scheme C}. In particular, \eqref{scheme C} admits infinitely 
many entropy inequalities if $q_j[u;U]$ is replaced by $\tilde q_j[u;U]$ with $U$ defined as in Corollary \ref{corollary: infinite entropy 2}. The following table lists the optimal $C(k,r)$ in \eqref{bound on bad term} with $1 \le k \le 4$ and $k+1 \le r \le k+2$.
\begin{table}[htbp]
\centering
\begin{tabular}{c|c}
\hline
$(k,r)$ & $C(k,r)$ \\
\hline
$(1,2)$ & $\frac{1}{2\pi} \approx 0.159154943$ \\
$(1,3)$ & $0.0446961624$ \\
$(2,3)$ & $0.0162145975$ \\
$(2,4)$ & $0.00403144180$ \\
$(3,4)$ & $0.00119363570$ \\
$(3,5)$ & $0.000267584017$ \\
$(4,5)$ & $0.0000684049992$ \\
$(4,6)$ & $0.0000140442003$ \\
\hline
\end{tabular}
\caption{Optimal constants $C(k,r)$.}
\label{tab:Ckr}
\end{table}

For practical reasons, if we do not care about the choice of $U$, we may let
\begin{equation*}
    U(u)=U_\epsilon(u)=\frac{u^2}{2}+\epsilon V(u),
\end{equation*}
where $\epsilon>0$ is any small constant and $V(u)$ is any convex function. Then, $q_j[u;U_\epsilon]\le \epsilon \max |V'''(u)|$. In numerical implementations we can really replace $\max |V'''(u)|$ with a uniform constant whenever the numerical solution is observed to be uniformly bounded; therefore, $q_j[u;U_\epsilon]$ can be made arbitrarily small depending on how small $\epsilon$ is. As for the term $|H^+_{j-\frac{1}{2}}|+|H^-_{j+\frac{1}{2}}|$, they can clearly be bounded by $|[\![u]\!]_{j-\frac{1}{2}}|+|[\![u]\!]_{j+\frac{1}{2}}|$ multiplied by the Local Lipschitz constant of $\hf$ depending on $\|u\|_{L^\infty(I)}$. Therefore, we may replace $C_j[u;U,r]$ by
\begin{equation}\label{def D_j}
    D_j[u;r,\epsilon_1,\epsilon_2] := \epsilon_1 h_j^{r-\frac{3}{2}} \|\partial_x^{r}f(u)\|_{L^2(I_j)} + \epsilon_2 \frac{|[\![u]\!]_{j-\frac{1}{2}}|+|[\![u]\!]_{j+\frac{1}{2}}|}{h_j},
\end{equation}
where $\epsilon_1,\epsilon_2 >0$ are arbitrarily chosen. Then, we can define our last scheme in this section: seek $u\in\mathbb{V}^k$ such that for all $v\in\mathbb{V}^k$ we have
\begin{align}
\int_{I_j}  u_t  v \ud x - \int_{I_j}  f(u)  v_x \ud x + \hf_{j+\frac{1}{2}}\cdot  v^-_{j+\frac{1}{2}} - \hf_{j-\frac{1}{2}}\cdot  v^+_{j-\frac{1}{2}}= - D_j[u;r,\epsilon_1,\epsilon_2]\int_{I_j}  ( u - \bu_j)  v \ud x. \tag{$\textbf{Scheme D}$}\label{scheme D}
\end{align}
This scheme is computationally efficient. Except for the first term in the definition of $D_j[u;r,\epsilon_1,\epsilon_2]$ involving $\|\partial_x^{r}f(u)\|_{L^2(I_j)}$, which is needed to bound the variation of $u$ on the interior of each cell, \eqref{scheme D} is in striking similarity with the OFDG scheme proposed in \cite{OFDG}. From Proposition \ref{monotone} and Corollary \ref{corollary: infinite entropy 2}, we deduce that \eqref{scheme D} still admits infinitely many entropy inequalities (although they might be very close to the square-entropy) if the numerical solution remains bounded. In addition, more entropies are included as we increase $\epsilon_1$ and $\epsilon_2$.

The optimal error estimates of \eqref{scheme B}, \eqref{scheme C} and \eqref{scheme D} for smooth solutions will be proved in Section \ref{section: error estimate}, and the proof of strong convergence for general discontinuous data will be given in Section \ref{section: strong convergence}. In the next two sections, we will provide generalizations of these schemes to one and higher dimensional systems of conservation laws.

\section{The Operator DG method for 1D systems of conservation laws}\label{section: scheme 1D system}
Consider an $m \times m$ system of conservation laws in one space dimension:
\begin{equation}\label{HCL 1D system}
\vu_t(x,t)+\vf(\vu(x,t))_x=\mathbf{0},\quad x\in I, \quad t\ge0.
\end{equation}
where $\vf:\Omega'\mapsto\mathbb{R}^m$ is a smooth function defined on some open set $\Omega' \subset \mathbb{R}^m$. In what follows, $\vu$ and $\vf$ are viewed as column vectors. Let $(U,\mF)$ be any fixed entropy pair for \eqref{HCL 1D system}. Assume $U\in C^3$ and strictly convex ($\nabla^2U >0$ on $\Omega'$). The entropy flux $\mF$ is defined by
\begin{equation}\label{def entropy flux 1D system}
    \nabla \mF( \vu)^T= \nabla U( \vu)^T \nabla \vf( \vu),
\end{equation}
where the gradients $\nabla \mF( \vu)$ and $\nabla U( \vu)$ are also viewed as column vectors, and $\nabla \vf$ denotes the Jacobian matrix of $\vf$. In addition, let 
\begin{equation}\label{def psi 1D system}
    \psi(\vu):=\nabla U( \vu)^T \vf( \vu) -\mF( \vu),
\end{equation}
so that $\nabla \psi(\vu)^T=\vf(\vu)^T\nabla^2 U(\vu).$ For clarity, we also adopt the notation $\mathbf{a}\cdot\mathbf{b}=\mathbf{a}^T \mathbf{b}$ for the dot product of two column vectors $\mathbf{a}$ and $\mathbf{b}$. Let $|\mathbf{a}|$ denote the Euclidean norm of the vector $\mathbf{a}$. Define $(\mathbb V^k)^m$ as the space of all vector-valued functions $\vecv:I\mapsto\mathbb{R}^m$ with each component belonging to $\mathbb V^k$.
As in Section \ref{section: scheme 1D scalar}, we have the following estimates analogous to Lemmas \ref{estimate2} and \ref{estimate3}.
\begin{lemma}\label{estimate4}
For any $\vu\in (\mathbb V^k)^m$, we have
\begin{align*}
&|\pi \circ [\nabla U(\vu)- \nabla U(\vu^+_{j-\frac{1}{2}})](x^+_{j-\frac{1}{2}})|\le \frac{m^{3/2}\beta_k}{2h_j}\max_{\vw\in \mathrm{co}(\vu(I_j))}|\partial^3 U(\vw)| \int_{I_j} |\vu(x)-\bvu_j|^2 \ud x,\\
&|\pi \circ [\nabla U(\vu)- \nabla U(\vu^-_{j+\frac{1}{2}})](x^-_{j+\frac{1}{2}})|\le \frac{m^{3/2}\beta_k}{2h_j}\max_{\vw\in \mathrm{co}(\vu(I_j))}|\partial^3 U(\vw)| \int_{I_j} |\vu(x)-\bvu_j|^2 \ud x,
\end{align*}
where $\beta_k$ may be chosen to be the same constant as in Lemma \ref{estimate2}; in particular,
\begin{equation*}
    \beta_k=2\mathcal{A}_k
\end{equation*}
is admissible, with $\mathcal{A}_k$ defined in \eqref{eq:def Ak beta}. Moreover,
\[\max_{\vw\in \mathrm{co}(\vu(I_j))}|\partial^3 U(\vw)|=\max_{\substack{ 1\le i_1,i_2,i_3 \le m \\ \vw\in \mathrm{co}(\vu(I_j)) }} \left| \frac{\partial^3 U(\vw)}{\partial u_{i_1} \partial u_{i_2} \partial u_{i_3}} \right|\]
is the maximum over all possible third order partial derivatives of $U$ in the convex hull of the states $\vu(I_j)=\{\vu(x):x\in I_j\}$.
\end{lemma}

\begin{proof}
We prove the estimate at $x^+_{j-\frac12}$; the estimate at the other endpoint follows by symmetry. Introduce the reference coordinate $\xi\in[-1,1]$ by
\begin{equation*}
    x=x_j+\frac{h_j}{2}\xi,
\end{equation*}
and define the vector-valued polynomial
\begin{equation*}
    \mathbf{p}(\xi):=\vu\left(x_j+\frac{h_j}{2}\xi\right)-\bvu_j.
\end{equation*}
Then each component of $\mathbf{p}$ belongs to $\mathbb P^k([-1,1])$ and has zero average. In addition,
\begin{equation*}
    \vu(x)-\vu^+_{j-\frac12}=\mathbf{p}(\xi)-\mathbf{p}(-1).
\end{equation*}

Set
\begin{equation*}
    M_j:=\max_{\vw\in\mathrm{co}(\vu(I_j))}|\partial^3U(\vw)|.
\end{equation*}
For each component $r=1,\ldots,m$, Taylor's theorem with integral remainder gives
\begin{align*}
&\big[\nabla U(\vu)-\nabla U(\vu^+_{j-\frac12})
-\nabla^2U(\vu^+_{j-\frac12})(\vu-\vu^+_{j-\frac12})\big]_r\\
&\qquad=\int_0^1(1-s)\sum_{a,b=1}^m
\frac{\partial^3U}{\partial u_r\partial u_a\partial u_b}
\left(\vu^+_{j-\frac12}+s(\vu-\vu^+_{j-\frac12})\right)
(\vu_a-\vu^+_{a,j-\frac12})(\vu_b-\vu^+_{b,j-\frac12})\,\ud s.
\end{align*}
Since the line segment joining $\vu^+_{j-\frac12}$ and $\vu(x)$ is contained in $\mathrm{co}(\vu(I_j))$, it follows that
\begin{align*}
&\left|\nabla U(\vu)-\nabla U(\vu^+_{j-\frac12})
-\nabla^2U(\vu^+_{j-\frac12})(\vu-\vu^+_{j-\frac12})\right|\\
&\qquad\le \frac{m^{3/2}M_j}{2}|\vu-\vu^+_{j-\frac12}|^2.
\end{align*}
Indeed, for each component the absolute value is bounded by
\[\frac{M_j}{2}(\sum_{a=1}^m|\vu_a-\vu^+_{a,j-\frac12}|)^2\le
\frac{mM_j}{2}|\vu-\vu^+_{j-\frac12}|^2,\] and taking the Euclidean norm over the $m$ components yields the factor $m^{3/2}$.

Each component of
\begin{equation*}
    \nabla^2U(\vu^+_{j-\frac12})(\vu-\vu^+_{j-\frac12})
\end{equation*}
belongs to $\mathbb P^k(I_j)$ and vanishes at $x^+_{j-\frac12}$. Since $\pi$ acts componentwise, its projected value at that endpoint is therefore zero. Using the endpoint reproducing kernel $K_k^-$ from \eqref{eq:left projection kernel}, we obtain
\begin{align*}
&\left|\pi\circ[\nabla U(\vu)-\nabla U(\vu^+_{j-\frac12})](x^+_{j-\frac12})\right|\\
&\quad\le \int_{-1}^1|K_k^-(\xi)|
\left|\nabla U(\vu)-\nabla U(\vu^+_{j-\frac12})
-\nabla^2U(\vu^+_{j-\frac12})(\vu-\vu^+_{j-\frac12})\right|\ud\xi\\
&\quad\le \frac{m^{3/2}M_j}{2}
\int_{-1}^1|K_k^-(\xi)|\,|\mathbf{p}(\xi)-\mathbf{p}(-1)|^2\ud\xi.
\end{align*}

The scalar inequality defining $\mathcal{A}_k$ extends to vector-valued polynomials componentwise. This can be seen by writing $\mathbf{p}=(p_1,\ldots,p_m)^T$ and using \eqref{eq:def Ak beta},
\begin{align*}
\int_{-1}^1|K_k^-(\xi)|\,|\mathbf{p}(\xi)-\mathbf{p}(-1)|^2\ud\xi
&=\sum_{r=1}^m\int_{-1}^1|K_k^-(\xi)|\,|p_r(\xi)-p_r(-1)|^2\ud\xi\\
&\le \mathcal{A}_k\sum_{r=1}^m\int_{-1}^1|p_r(\xi)|^2\ud\xi\\
&=\mathcal{A}_k\int_{-1}^1|\mathbf{p}(\xi)|^2\ud\xi.
\end{align*}
Therefore,
\begin{align*}
&\left|\pi\circ[\nabla U(\vu)-\nabla U(\vu^+_{j-\frac12})](x^+_{j-\frac12})\right|\\
&\quad\le \frac{m^{3/2}M_j\mathcal{A}_k}{2}
\int_{-1}^1|\mathbf{p}(\xi)|^2\ud\xi\\
&\quad=\frac{m^{3/2}\mathcal{A}_k}{h_j}M_j
\int_{I_j}|\vu(x)-\bvu_j|^2\ud x\\
&\quad=\frac{m^{3/2}\beta_k}{2h_j}M_j
\int_{I_j}|\vu(x)-\bvu_j|^2\ud x,
\end{align*}
where in the last line we take $\beta_k=2\mathcal{A}_k$.

For the right endpoint, use $K_k^+(\xi)=K_k^-(-\xi)$ and the transformation $\mathbf{p}(\xi)\mapsto\mathbf{p}(-\xi)$. The zero-average condition and the vector $L^2$ norm are preserved, so exactly the same constant $\mathcal{A}_k$, and hence the same $\beta_k$, applies.
\end{proof}

\begin{lemma}\label{estimate5}
For any $\vu\in (\mathbb V^k)^m$, we have
\begin{align*}
    \bigg|\int_{I_j} & \nabla U(\vu)\cdot (\vf(\vu)_x - \pi \circ [\vf(\vu)_x])\ud x\bigg|\\
    &\le \frac{m^{3/2} C_k}{2\sqrt{h_j}} \max_{\vw\in \mathrm{co}(\vu(I_j))}|\partial^3 U(\vw)| \|\vf(\vu)_x - \pi \circ [\vf(\vu)_x]\|_{L^2(I_j)}\int_{I_j} |\vu(x)-\bvu_j|^2 \ud x 
\end{align*}
where $C_k>0$ is the optimal constant from \eqref{C_k}, and $\max_{\vw\in \mathrm{co}(\vu(I_j))}|\partial^3 U(\vw)|$ is the same maximum from Lemma \ref{estimate4}.
\end{lemma}
\begin{proof}
Note that $\nabla U(\bvu_j)+\nabla^2 U(\bvu_j)(\vu(x)-\bvu_j)$ is a $k$-th order polynomial on $I_j$, which is orthogonal to $\vf(\vu)_x - \pi \circ [\vf(\vu)_x]$. Therefore,
\begin{align*}
    &\bigg|\int_{I_j} \nabla U(\vu)\cdot (\vf(\vu)_x - \pi \circ [\vf(\vu)_x])\ud x\bigg|\\
    =& \bigg|\int_{I_j} (\nabla U(\vu)- \nabla U(\bvu_j)-\nabla^2 U(\bvu_j)(\vu-\bvu_j))\cdot (\vf(\vu)_x - \pi \circ [\vf(\vu)_x])\ud x\bigg| \\
    \le& \frac{m^{3/2}}{2} \max_{\vw\in \mathrm{co}(\vu(I_j))}|\partial^3 U(\vw)| \int_{I_j} |\vu-\bvu_j|^2 |\vf(\vu)_x - \pi \circ [\vf(\vu)_x]|\ud x \\
    \le& \frac{m^{3/2}}{2} \max_{\vw\in \mathrm{co}(\vu(I_j))}|\partial^3 U(\vw)| \|\vf(\vu)_x - \pi \circ [\vf(\vu)_x]\|_{L^2(I_j)} \sqrt{\int_{I_j} |\vu-\bvu_j|^4 \ud x}  \\
    \stackrel{\eqref{C_k}}{\le}& \frac{m^{3/2} C_k}{2\sqrt{h_j}} \max_{\vw\in \mathrm{co}(\vu(I_j))}|\partial^3 U(\vw)| \|\vf(\vu)_x - \pi \circ [\vf(\vu)_x]\|_{L^2(I_j)}\int_{I_j} |\vu(x)-\bvu_j|^2 \ud x.
\end{align*}
\end{proof}
In addition, Hölder’s Defect Formula also holds in $\mathbb{R}^m$:
\begin{align}
    \int_{I_j}U(\vu)\ud x -h_j U(\bvu_j)\ge \frac{1}{2}\min_{\vw\in \mathrm{co}(\vu(I_j))}|\sigma_{\min}(\nabla^2 U(\vw))| \int_{I_j} |\vu(x)-\bvu_j|^2, \label{Jensen system}
\end{align}
where $\sigma_{\min} (\nabla^2 U)$ denotes the (positive) minimum eigenvalue of the Hessian matrix $\nabla^2 U$, and $\min_{\vw\in \mathrm{co}(\vu(I_j))}$ means taking the minimum over the convex hull of the states $\vu(I_j)=\{\vu(x):x\in I_j\}$. Therefore, define
\begin{align*}
    B'_j[\vu;U]:= \frac{m^{3/2}}{2}&\frac{\max_{\vw\in \mathrm{co}(\vu(I_j))}|\partial^3 U(\vw)|}{\min_{\vw\in \text{co}(\vu(I_j))}|\sigma_{\min}(\nabla^2 U(\vw))|} \times\\
    &\bigg\{ \frac{C_k}{\sqrt{h_j}} \|\vf(\vu)_x - \pi \circ [\vf(\vu)_x]\|_{L^2(I_j)} +\frac{ \beta_k}{h_j}\left(|\vH^+_{j-\frac{1}{2}}|+|\vH^-_{j+\frac{1}{2}}|\right) \bigg\},
\end{align*}
where 
\begin{equation}\label{def H system}
    \vH^+_{j-\frac{1}{2}}:=\hvf_{j-\frac{1}{2}}-\vf(\vu)^+_{j-\frac{1}{2}},\quad \vH^-_{j+\frac{1}{2}}:=-\hvf_{j+\frac{1}{2}}+\vf(\vu)^-_{j+\frac{1}{2}},
\end{equation}
and $\hvf$ is the numerical flux. Then, \eqref{scheme B} can be directly generalized from 1D scalar equation to 1D system as follows: seek $\vu\in(\mathbb{V}^k)^m$ such that for all $\vecv\in(\mathbb{V}^k)^m$ we have
\begin{align}
\int_{I_j}  \vu_t \cdot  \vecv \ud x = -\int_{I_j}  \vf(\vu)_x \cdot \vecv \ud x + \vH^+_{j-\frac{1}{2}}\cdot  \vecv^+_{j-\frac{1}{2}} + \vH^-_{j+\frac{1}{2}}\cdot \vecv^-_{j+\frac{1}{2}} - B'_j[\vu;U]\int_{I_j}  ( \vu - \bvu_j) \cdot \vecv \ud x. \tag{$\textbf{Scheme B'}$}\label{scheme B'}
\end{align}
After performing integration by parts on $\int_{I_j}  \vf(\vu)_x\cdot  \vecv \ud x$, we also get the equivalent weak formulation of \eqref{scheme B'} as
\begin{align}
\int_{I_j}  \vu_t \cdot \vecv \ud x - \int_{I_j}  \vf(\vu) \cdot \vecv_x \ud x + \hvf_{j+\frac{1}{2}}\cdot  \vecv^-_{j+\frac{1}{2}} - \hvf_{j-\frac{1}{2}}\cdot  \vecv^+_{j-\frac{1}{2}}= - B'_j[\vu;U]\int_{I_j}  ( \vu - \bvu_j) \cdot \vecv \ud x. \label{weak scheme B'}
\end{align}
Next, we borrow the definition of entropy stable numerical flux from \cite{CHEN2017427}:
\begin{definition}
A consistent two-point numerical flux $\hvf(\vu^-,\vu^+)$ is entropy stable for a given entropy function $U$ if
\begin{equation}\label{def entropy stable flux}
    \hvf(\vu^-,\vu^+)\cdot \left(\nabla U (\vu^+) - \nabla U (\vu^-) \right) - (\psi(\vu^+)-\psi(\vu^-))\le 0.
\end{equation}
\end{definition}
It is shown in \cite{CHEN2017427} that the class of HLL flux is entropy stable, which includes the local Lax–Friedrichs flux. The next theorem shows that \eqref{scheme B'} satisfies local entropy inequality for $U(\cdot)$, provided that $\hvf$ is any entropy stable flux.
\begin{theorem}[Local Entropy Inequality for \eqref{scheme B'}]\label{theorem: Local Entropy Inequality B'}
Suppose that $\hvf$ in \eqref{def H system} is an entropy stable flux. Then \eqref{scheme B'} satisfies the local semi-discrete entropy inequality
\begin{equation}\label{scheme B' entropy inequality U}
\odv{}{t}\int_{I_j} U( \vu) \ud x + \hat\mF_{j+\frac{1}{2}} - \hat\mF_{j-\frac{1}{2}}\le 0, \quad \forall \vu\in(\mathbb{V}^k)^m,\quad \forall j=1,\cdots,N,
\end{equation}
where $\hat\mF_{j+\frac{1}{2}}=\hat\mF( \vu^-_{j+\frac{1}{2}}, \vu^+_{j+\frac{1}{2}})$ and
\begin{equation}\label{entropy flux system}
\hat\mF( \vu^-, \vu^+):=\nabla U(\vu^-)\cdot \hvf(\vu^-,\vu^+)-\psi(\vu^-),
\end{equation}
with $\psi$ defined by \eqref{def psi 1D system}.
\end{theorem}
\begin{proof}
    This is essentially the same as the proof of Theorem \ref{theorem: Local Entropy Inequality B}. The only differences are that now we use the definition \eqref{def entropy stable flux} of the entropy stable flux to show the positivity of \eqref{B proof interface entropy dissipation}, and use Lemmas \ref{estimate4}--\ref{estimate5} to bound $|R_{j,1}|$, $|R_{j,2}|$ and $|R_{j,3}|$.
\end{proof}

Of course, similar to the 1D scalar case, we may use Bramble-Hilbert lemma \eqref{bound on bad term} on the tricky term $\|\vf(\vu)_x - \pi \circ [\vf(\vu)_x]\|_{L^2(I_j)}$ in the definition of $B_j'[\vu;U]$ and derive the following simplified schemes: seek $\vu\in(\mathbb{V}^k)^m$ such that for all $\vecv\in(\mathbb{V}^k)^m$ we have
\begin{align}
\int_{I_j}  \vu_t \cdot \vecv \ud x - \int_{I_j}  \vf(\vu) \cdot \vecv_x \ud x +& \hvf_{j+\frac{1}{2}}\cdot  \vecv^-_{j+\frac{1}{2}} - \hvf_{j-\frac{1}{2}}\cdot  \vecv^+_{j-\frac{1}{2}} \nonumber\\
=& - C'_j[\vu;U,r]\int_{I_j}  ( \vu - \bvu_j) \cdot \vecv \ud x \tag{$\textbf{Scheme C'}$}\label{scheme C'}
\end{align}
or
\begin{align}
\int_{I_j}  \vu_t \cdot \vecv \ud x - \int_{I_j}  \vf(\vu) \cdot \vecv_x \ud x +& \hvf_{j+\frac{1}{2}}\cdot  \vecv^-_{j+\frac{1}{2}} - \hvf_{j-\frac{1}{2}}\cdot  \vecv^+_{j-\frac{1}{2}} \nonumber\\
=& - D'_j[\vu;r,\epsilon_1,\epsilon_2]\int_{I_j}  ( \vu - \bvu_j) \cdot \vecv \ud x, \tag{$\textbf{Scheme D'}$}\label{scheme D'}
\end{align}
where the damping coefficients are
\begin{align*}
    C'_j[\vu;U,r]:= \frac{m^{3/2}}{2}&\frac{\max_{\vw\in \mathrm{co}(\vu(I_j))}|\partial^3 U(\vw)|}{\min_{\vw\in \text{co}(\vu(I_j))}|\sigma_{\min}(\nabla^2 U(\vw))|} \times\\
    &\bigg\{ C_k C(k,r) h_j^{r-\frac32} \|\partial_x^{r}\vf(\vu)\|_{L^2(I_j)} +\frac{ \beta_k}{h_j}\left(|\vH^+_{j-\frac{1}{2}}|+|\vH^-_{j+\frac{1}{2}}|\right) \bigg\}
\end{align*}
and
\begin{align*}
    D'_j[\vu;r,\epsilon_1,\epsilon_2]:= \epsilon_1 h_j^{r-\frac32} \|\partial_x^{r}\vf(\vu)\|_{L^2(I_j)} +\epsilon_2 \frac{ |[\![\vu]\!]_{j-\frac{1}{2}}|+|[\![\vu]\!]_{j+\frac{1}{2}}| }{h_j}
\end{align*}
respectively with $k+1 \le r \le k+2$.

Because monotonicity of remainder (Proposition \ref{monotone}) still holds for systems of conservation laws, Theorem \ref{theorem: Local Entropy Inequality B'} also holds true for \eqref{scheme C'}, and also for \eqref{scheme D'} with sufficiently large $\epsilon_1,\epsilon_2$ if the numerical solution is uniformly bounded.

\section{The Operator DG method for higher dimensional systems of conservation laws}\label{section: scheme dD system}
Consider an $m \times m$ system of conservation laws in $d$ space dimensions:
\begin{equation}\label{HCL dD system}
\pdv{\vu(\vx,t)}{t}+\sum_{i=1}^d \pdv{\vf^{(i)}(\vu(\vx,t))}{x^{(i)}}=\mathbf{0},\quad \vx=(x^{(1)},\cdots,x^{(d)})\in \Omega\subset \mathbb{R}^d, \quad t\ge0.
\end{equation}
where $\vf^{(i)}:\Omega'\mapsto\mathbb{R}^m$ for $i=1,\cdots,d$ are smooth functions defined on some open set $\Omega' \subset \mathbb{R}^m$. In what follows, $\vu$ and $\vf^{(i)}$'s are viewed as column vectors. Let $U$ be any fixed entropy for \eqref{HCL dD system} with entropy fluxes $\{\mF^{(i)}\}_{i=1}^d$ defined by
\begin{equation}\label{def entropy flux dD system}
    \nabla \mF^{(i)}( \vu)^T= \nabla U( \vu)^T \nabla \vf^{(i)}( \vu),\quad i=1\cdots,d.
\end{equation}
Assume $U\in C^3$ and strictly convex. In addition, let 
\begin{equation}\label{def psi dD system}
    \psi^{(i)}(\vu):=\nabla U( \vu)^T \vf^{(i)}( \vu) -\mF^{(i)}( \vu), \quad i=1\cdots,d,
\end{equation}
so that $\nabla \psi^{(i)}(\vu)^T=\vf^{(i)}(\vu)^T\nabla^2 U(\vu)$.

Let $\mathcal{T}_h$ be a partition of $\Omega$. Assume every $K\in\mathcal{T}_h$ is convex and let $\rho_K$ denote the diameter of the largest sphere inscribed in $K$. Let $|K|$ be the volume of $K$ and
\begin{equation*}
    h=\max_{K\in \mathcal{T}_h} h_K,\quad h_K=\operatorname{diam}\, K,\quad \gamma_K=\frac{h_K}{\rho_K}.
\end{equation*}
Define for $k\ge1$ the space
\begin{equation*}
(\mathbb{V}^k_d)^m:=\{\vecv \in (L^2(\Omega))^m : \vecv\big|_{K}\in (\mathbb{P}^k(K))^m, \forall K\in \mathcal{T}_h \}
\end{equation*}
of all piecewise (possibly discontinuous) polynomials on $\Omega$ that coincides with a $k$-th order polynomial on each element $K$. For any function $\vecv$ continuous on $\overline{K}$ and $\vx\in \partial K$, denote the limit of $\vecv$ when approaching $\vx$ from the interior of $K$ by $\vecv(\vx^{\text{int}\,K})$, and the limit from outside of $K$ by $\vecv(\vx^{\text{ext}\,K})$. Let $\pi_h$ be the $L^2$-projection onto $(\mathbb V_d^k)^m$, and denote the average of $\vu$ on $K$ by $\bvu_K$. With similar proofs, one can show the following generalizations of Lemmas \ref{estimate4} and \ref{estimate5} to any element $K\in \mathcal{T}_h$:
\begin{lemma}\label{estimate6}
For any $K\in \mathcal{T}_h$ and $\vu\in (\mathbb V^k_d)^m$, we have
\begin{equation*}
\left|\pi_h \circ [\nabla U(\vu)- \nabla U(\vu(\vx^{\text{int}\,K}))](\vx^{\text{int}\,K})\right|\le \frac{C(k,m,d,\gamma_K)}{|K|}\max_{\vw\in \mathrm{co}(\vu(K))}|\partial^3 U(\vw)| \int_{K} |\vu(x)-\bvu_K|^2 \ud \vx
\end{equation*}
for all $\vx\in\partial K$, and
\begin{align*}
    &\bigg|\sum_{i=1}^d\int_{K} \nabla U(\vu)\cdot \bigg(\pdv{\vf^{(i)}(\vu)}{x^{(i)}} - \pi_h \circ \bigg[\pdv{\vf^{(i)}(\vu)}{x^{(i)}}\bigg] \bigg)\ud x\bigg|\\
    \le& \frac{C(k,m,d,\gamma_K)}{\sqrt{|K|}} \max_{\vw\in \mathrm{co}(\vu(K))}|\partial^3 U(\vw)| \bigg\|\sum_{i=1}^d\pdv{\vf^{(i)}(\vu)}{x^{(i)}} - \pi_h \circ \bigg[\sum_{i=1}^d \pdv{\vf^{(i)}(\vu)}{x^{(i)}}\bigg]\bigg\|_{L^2(K)}\int_{K} |\vu(x)-\bvu_K|^2 \ud \vx,
\end{align*}
where $C(k,m,d,\gamma_K)>0$ is constant depending only on $k,m,d,\gamma_K$, and \[\max_{\vw\in \mathrm{co}(\vu(K))}|\partial^3 U(\vw)|=\max_{\substack{ 1\le i_1,i_2,i_3 \le m \\ \vw\in \mathrm{co}(\vu(K)) }} \left| \frac{\partial^3 U (\vw) }{\partial u_{i_1} \partial u_{i_2} \partial u_{i_3}} \right|\] is the maximum over all possible third order partial derivatives of $U$ over the convex hull of the states $\vu(K)=\{\vu(\vx):\vx\in K\}$.
\end{lemma}

For $\vu\in(\mathbb{V}^k_d)^m$ and using the constant $C(k,m,d,\gamma_K)$ from Lemma \ref{estimate6}, define
\begin{align*}
    B''_K[\vu;U]:=& C(k,m,d,\gamma_K)\frac{\max_{\vw\in \mathrm{co}(\vu(K))}|\partial^3 U(\vw)|}{\min_{\vw\in \mathrm{co}(\vu(K))}|\sigma_{\min}(\nabla^2 U(\vw))|} \times\\
    &\bigg\{ \frac{1}{\sqrt{|K|}} \bigg\|\sum_{i=1}^d\pdv{\vf^{(i)}(\vu)}{x^{(i)}} - \pi_h \circ \bigg[\sum_{i=1}^d \pdv{\vf^{(i)}(\vu)}{x^{(i)}}\bigg]\bigg\|_{L^2(K)}+\frac{1}{|K|} \int_{\partial K} \left|\vH_K(\vx,\vn)\right| \ud s(\vx) \bigg\},
\end{align*}
where $\vn=(n_1,\cdots,n_d)^T$ is the unit outward normal with respect to $\partial K$,
\begin{equation}\label{def H system dD}
    \vH_K(\vx,\vn):=-\hvf(\vu(\vx^{\text{int}\,K}),\vu(\vx^{\text{ext}\,K}),\vn)+ \sum_{i=1}^d n_i\vf^{(i)}(\vu(\vx^{\text{int}\,K})),
\end{equation}
and $\hvf$ is the directional numerical flux. Then, \eqref{scheme B'} can be directly generalized from 1D system to $d$-dimension as follows: seek $\vu\in(\mathbb{V}^k_d)^m$ such that for all $\vecv\in(\mathbb{V}^k_d)^m$ we have
\begin{align}
\int_{K}  \vu_t \cdot  \vecv \ud \vx =& -\sum_{i=1}^d\int_{K}  \pdv{\vf^{(i)}(\vu)}{x^{(i)}} \cdot \vecv \ud \vx + \int_{\partial K} \vH_K(\vx,\vn)\cdot \vecv(\vx^{\text{int}\,K}) \ud s(\vx) \nonumber\\
& - B''_K[\vu;U]\int_{K}  ( \vu - \bvu_K) \cdot \vecv \ud \vx, \tag{$\textbf{Scheme B''}$}\label{scheme B''}
\end{align}
After using divergence theorem on $\sum_{i=1}^d \int_{K}  \pdv{\vf^{(i)}(\vu)}{x^{(i)}} \cdot \vecv \ud \vx$, we also get the equivalent weak formulation of \eqref{scheme B''} as
\begin{align}
\int_{K}  \vu_t \cdot \vecv \ud \vx - \sum_{i=1}^d \int_{K} \vf^{(i)}(\vu) \cdot \pdv{\vecv}{x^{(i)}} \ud \vx + & \int_{\partial K} \hvf(\vu(\vx^{\text{int}\,K}),\vu(\vx^{\text{ext}\,K}),\vn)\cdot \vecv(\vx^{\text{int}\,K}) \ud s(\vx) \nonumber \\
=& - B''_K[\vu;U]\int_{K}  ( \vu - \bvu_K) \cdot \vecv \ud \vx. \label{weak scheme B''}
\end{align}
Now, the definition \eqref{def entropy stable flux} of entropy stable numerical flux is generalized to $d$-dimension following \cite{CHEN2017427}.
\begin{definition}\label{def entropy stable flux dD}
    Given a normal vector $\vn\in\mathbb{R}^d$, a directional numerical flux $\hvf(\vu^{\text{int}\,K},\vu^{\text{ext}\,K},\vn)$ is consistent if
    \begin{equation}\label{eq: def consistent numerical flux dD}
        \hvf(\vu,\vu,\vn)=\sum_{i=1}^d n_i \vf^{(i)}(\vu).
    \end{equation}
    It is called conservative if
    \begin{equation}\label{eq: def conservative numerical flux dD}
        \hvf(\vu^{\text{ext}\,K},\vu^{\text{int}\,K},-\vn)=-\hvf(\vu^{\text{int}\,K},\vu^{\text{ext}\,K},\vn).
    \end{equation}
    A directional numerical flux is \emph{entropy stable} for a given entropy $U(\cdot)$ if it is consistent and conservative and satisfies
    \begin{equation}\label{eq: def entropy stable flux dD}
        \hvf(\vu^{\text{int}\,K},\vu^{\text{ext}\,K},\vn)\cdot\left(\nabla U (\vu^{\text{ext}\,K}) - \nabla U (\vu^{\text{int}\,K}) \right) - \sum_{i=1}^d n_i (\psi^{(i)}(\vu^{\text{ext}\,K})-\psi^{(i)}(\vu^{\text{int}\,K})) \le 0.
    \end{equation}
\end{definition} 
The class of HLL fluxes, including the local Lax-Friedrichs flux, are entropy stable. The next theorem shows that \eqref{scheme B''} satisfies local entropy inequality for $U(\cdot)$, provided that $\hvf$ is any entropy stable directional flux.
\begin{theorem}[Local Entropy Inequality for \eqref{scheme B''}]\label{theorem: Local Entropy Inequality B''}
Suppose that $\hvf$ in \eqref{def H system dD} is an entropy stable directional flux. Then \eqref{scheme B''} satisfies the local semi-discrete entropy inequality
\begin{equation}\label{scheme B'' entropy inequality U}
\odv{}{t}\int_{K} U( \vu) \ud \vx + \int_{\partial K} \hat\mF(\vu(\vx^{\text{int}\,K}),\vu(\vx^{\text{ext}\,K}),\vn) \ud s(\vx)\le 0,\quad \forall \vu\in (\mathbb{V}^k_d)^m, \quad \forall K\in\mathcal{T}_h,
\end{equation}
where
\begin{align}
\hat\mF( \vu^{\text{int}\,K}, \vu^{\text{ext}\,K},\vn):=& \frac{\nabla U(\vu^{\text{int}\,K})+\nabla U(\vu^{\text{ext}\,K})}{2} \cdot \hvf(\vu^{\text{int}\,K},\vu^{\text{ext}\,K},\vn) \nonumber \\
&- \sum_{i=1}^d n_i \frac{\psi^{(i)}(\vu^{\text{int}\,K})+\psi^{(i)}(\vu^{\text{ext}\,K})}{2}, \label{entropy flux system dD}
\end{align}
with $\psi^{(i)}$ defined by \eqref{def psi dD system}. Clearly, $\hat\mF$ is consistent and conservative.
\end{theorem}
\begin{proof}
Fix $K\in\mathcal T_h$ and set
\[
    \vz:=\pi_h[\nabla U(\vu)]\in(\mathbb V_d^k)^m.
\]
Since $\vu_t|_K\in(\mathbb P^k(K))^m$ and $\pi_h$ is the $L^2$-projection,
\begin{equation}\label{Bpp proof time projection}
\int_K\vu_t\cdot\vz\,\ud\vx
=\int_K\vu_t\cdot\nabla U(\vu)\,\ud\vx
=\odv{}{t}\int_KU(\vu)\,\ud\vx.
\end{equation}
Taking $\vecv=\vz$ in \eqref{scheme B''} gives
\begin{align}
\odv{}{t}\int_KU(\vu)\,\ud\vx
={}&-\sum_{i=1}^d\int_K\pdv{\vf^{(i)}(\vu)}{x^{(i)}}\cdot\vz\,\ud\vx
+\int_{\partial K}\vH_K\cdot\vz(\vx^{\mathrm{int}\,K})\,\ud s \nonumber\\
&-B''_K[\vu;U]\int_K(\vu-\bvu_K)\cdot\vz\,\ud\vx.
\label{Bpp proof DG identity}
\end{align}

By self-adjointness of the $L^2$-projection,
\begin{align}
&-\sum_{i=1}^d\int_K\pdv{\vf^{(i)}(\vu)}{x^{(i)}}\cdot\vz\,\ud\vx \nonumber\\
={}&-\sum_{i=1}^d\int_K\nabla U(\vu)\cdot\pdv{\vf^{(i)}(\vu)}{x^{(i)}}\,\ud\vx+R_{K,1} \nonumber\\
={}&-\int_{\partial K}\sum_{i=1}^dn_i\mF^{(i)}(\vu(\vx^{\mathrm{int}\,K}))\,\ud s+R_{K,1},
\label{Bpp proof volume}
\end{align}
where
\begin{equation}\label{Bpp proof R1}
R_{K,1}:=\int_K\nabla U(\vu)\cdot\left\{\sum_{i=1}^d\pdv{\vf^{(i)}(\vu)}{x^{(i)}}-
\pi_h\left[\sum_{i=1}^d\pdv{\vf^{(i)}(\vu)}{x^{(i)}}\right]\right\}\,\ud\vx.
\end{equation}
The last equality follows from \eqref{def entropy flux dD system} and the divergence theorem.

Moreover, since $\vu-\bvu_K\in(\mathbb P^k(K))^m$ and has zero average,
\begin{align}
\int_K(\vu-\bvu_K)\cdot\vz\,\ud\vx
&=\int_K(\vu-\bvu_K)\cdot\bigl(\nabla U(\vu)-\nabla U(\bvu_K)\bigr)\,\ud\vx=:D_K.
\label{Bpp proof cell dissipation}
\end{align}
Put
\[
m_K:=\min_{\vw\in\operatorname{co}(\vu(K))}\sigma_{\min}(\nabla^2U(\vw)).
\]
Since $\bvu_K\in\operatorname{co}(\vu(K))$, the segment joining $\bvu_K$ and $\vu(\vx)$ lies in $\operatorname{co}(\vu(K))$. Hence
\begin{align*}
&(\vu-\bvu_K)\cdot(\nabla U(\vu)-\nabla U(\bvu_K))\\
&\quad=\int_0^1(\vu-\bvu_K)^T\nabla^2U(\bvu_K+\theta(\vu-\bvu_K))(\vu-\bvu_K)\,\ud\theta
\ge m_K|\vu-\bvu_K|^2,
\end{align*}
and therefore
\begin{equation}\label{Bpp proof DK lower bound}
D_K\ge m_K\int_K|\vu-\bvu_K|^2\,\ud\vx\ge0.
\end{equation}

For $\vx\in\partial K$, write
\[
\vu^-:=\vu(\vx^{\mathrm{int}\,K}),\qquad \vu^+:=\vu(\vx^{\mathrm{ext}\,K}),\qquad
\vh:=\hvf(\vu^-,\vu^+,\vn),
\]
and set
\[
\mF_{\vn}(\vu):=\sum_{i=1}^dn_i\mF^{(i)}(\vu),\qquad
\psi_{\vn}(\vu):=\sum_{i=1}^dn_i\psi^{(i)}(\vu).
\]
By \eqref{def H system dD}, \eqref{def psi dD system}, and \eqref{entropy flux system dD},
\begin{align}
&-\mF_{\vn}(\vu^-)+\vH_K\cdot\nabla U(\vu^-)+\hat\mF(\vu^-,\vu^+,\vn) \nonumber\\
={}&\frac12\left\{\vh\cdot(\nabla U(\vu^+)-\nabla U(\vu^-))-
(\psi_{\vn}(\vu^+)-\psi_{\vn}(\vu^-))\right\}
=:\frac12\mathcal E_K(\vx).
\label{Bpp proof interface identity}
\end{align}
The entropy stability condition \eqref{eq: def entropy stable flux dD} implies
\begin{equation}\label{Bpp proof interface entropy sign}
\mathcal E_K(\vx)\le0\qquad\text{on }\partial K.
\end{equation}
Since constant vectors are preserved by $\pi_h$,
\[
\vz(\vx^{\mathrm{int}\,K})-\nabla U(\vu^-)
=\pi_h[\nabla U(\vu)-\nabla U(\vu^-)](\vx^{\mathrm{int}\,K}).
\]
Define
\begin{equation}\label{Bpp proof R2}
R_{K,2}:=\int_{\partial K}\vH_K(\vx,\vn)\cdot
\pi_h[\nabla U(\vu)-\nabla U(\vu^-)](\vx^{\mathrm{int}\,K})\,\ud s.
\end{equation}
Integrating \eqref{Bpp proof interface identity} over $\partial K$, then adding \eqref{Bpp proof R2} to both sides yields
\begin{align}
&-\int_{\partial K}\mF_{\vn}(\vu^-)\,\ud s
+\int_{\partial K}\vH_K\cdot\vz(\vx^{\mathrm{int}\,K})\,\ud s
+\int_{\partial K}\hat\mF(\vu^-,\vu^+,\vn)\,\ud s \nonumber\\
&\qquad=R_{K,2}+\frac12\int_{\partial K}\mathcal E_K(\vx)\,\ud s.
\label{Bpp proof boundary identity}
\end{align}
Combining \eqref{Bpp proof DG identity}, \eqref{Bpp proof volume}, \eqref{Bpp proof cell dissipation}, and \eqref{Bpp proof boundary identity}, we obtain
\begin{align}
&\odv{}{t}\int_KU(\vu)\,\ud\vx+\int_{\partial K}\hat\mF(\vu^-,\vu^+,\vn)\,\ud s \nonumber\\
&\qquad=R_{K,1}+R_{K,2}+\frac12\int_{\partial K}\mathcal E_K(\vx)\,\ud s-B''_K[\vu;U]D_K.
\label{Bpp proof master identity}
\end{align}

By Lemma \ref{estimate6},
\begin{align}
|R_{K,1}|\le \frac{C(k,m,d,\gamma_K)}{\sqrt{|K|}}\max_{\vw\in \mathrm{co}(\vu(K))}|\partial^3U(\vw)|
&\left\|\sum_{i=1}^d\pdv{\vf^{(i)}(\vu)}{x^{(i)}}-
\pi_h\left[\sum_{i=1}^d\pdv{\vf^{(i)}(\vu)}{x^{(i)}}\right]\right\|_{L^2(K)} \nonumber\\
&\times\int_K|\vu-\bvu_K|^2\,\ud\vx,
\label{Bpp proof R1 estimate}
\end{align}
and
\begin{align}
|R_{K,2}|\le{}&\frac{C(k,m,d,\gamma_K)}{|K|}\max_{\vw\in \mathrm{co}(\vu(K))}|\partial^3U(\vw)|
\left(\int_{\partial K}|\vH_K(\vx,\vn)|\,\ud s\right)
\int_K|\vu-\bvu_K|^2\,\ud\vx.
\label{Bpp proof R2 estimate}
\end{align}
Thus, the definition of $B''_K[\vu;U]$ and \eqref{Bpp proof DK lower bound} give
\begin{equation}\label{Bpp proof remainder bound}
R_{K,1}+R_{K,2}\le |R_{K,1}|+|R_{K,2}|
\le B''_K[\vu;U]m_K\int_K|\vu-\bvu_K|^2\,\ud\vx
\le B''_K[\vu;U]D_K.
\end{equation}
Substituting this into \eqref{Bpp proof master identity} and using \eqref{Bpp proof interface entropy sign}, we conclude that
\[
\odv{}{t}\int_KU(\vu)\,\ud\vx+\int_{\partial K}\hat\mF(\vu^-,\vu^+,\vn)\,\ud s
\le\frac12\int_{\partial K}\mathcal E_K(\vx)\,\ud s\le0,
\]
which is precisely \eqref{scheme B'' entropy inequality U}.
\end{proof}
Similar as in the 1D system case, we may use Bramble-Hilbert lemma \eqref{bound on bad term} on the term \[ \bigg\|\sum_{i=1}^d\pdv{\vf^{(i)}(\vu)}{x^{(i)}} - \pi_h \circ \bigg[\sum_{i=1}^d \pdv{\vf^{(i)}(\vu)}{x^{(i)}}\bigg]\bigg\|_{L^2(K)}\] and derive the following simplified scheme: seek $\vu\in(\mathbb{V}^k_d)^m$ such that for all $\vecv\in(\mathbb{V}^k_d)^m$ we have
\begin{align}
\int_{K}  \vu_t \cdot \vecv \ud \vx - \sum_{i=1}^d \int_{K} \vf^{(i)}(\vu) \cdot \pdv{\vecv}{x^{(i)}} \ud \vx + & \int_{\partial K} \hvf(\vu(\vx^{\text{int}\,K}),\vu(\vx^{\text{ext}\,K}),\vn)\cdot \vecv(\vx^{\text{int}\,K}) \ud s(\vx) \nonumber \\
=& - D''_K[\vu;r,\epsilon_1,\epsilon_2]\int_{K}  ( \vu - \bvu_K) \cdot \vecv \ud \vx, \tag{$\textbf{Scheme D''}$}\label{scheme D''}
\end{align}
where
\begin{align*}
D''_K[\vu;r,\epsilon_1,\epsilon_2]
:={}&
\epsilon_1
\frac{h_K^{r-1}}{\sqrt{|K|}}
\left(
\sum_{i=1}^d
|\vf^{(i)}(\vu)|_{H^r(K)}^2
\right)^{1/2}
+
\frac{\epsilon_2}{|K|}
\int_{\partial K} \left|\vH_K(\vx,\vn)\right| \ud s(\vx), \quad k+1\le r \le k+2.
\end{align*}
Note that here we still integrate $\left|\vH_K(\vx,\vn)\right|$ along the boundary of $K$ instead of $|[\![\vu]\!]|$. As before, \eqref{scheme D''} also satisfies Theorem \ref{theorem: Local Entropy Inequality B''} for sufficiently large $\epsilon_1,\epsilon_2$ if the numerical solution is uniformly bounded.

\section{Optimal error estimates to smooth solutions of general nonlinear scalar conservation laws}\label{section: error estimate}
In this section, we show the optimal error estimates to smooth solutions of \eqref{scheme B}, \eqref{scheme C} and \eqref{scheme D} for general nonlinear scalar conservation laws. For the partition of the interval $I=\cup_{j=1}^N I_j$ with $|I_j|=h_j$, define \[h:=\max_{1\le j \le N} h_j.\] Assume the mesh is quasi-uniform, namely there exists a constant $C>0$ such that \[\max_{1\le j \le N} h_j \le C \min_{1\le j \le N} h_j.\]
\begin{theorem}\label{theorem: error estimate}
    Let $\tu(x,t)$ denote the exact solution of the 1D nonlinear scalar conservation law \eqref{HCL 1D scalar} with initial condition $\tu_0(x)$ and periodic or compactly supported boundary conditions. Suppose $f(\cdot)$ is sufficiently smooth, and $\exists T>0$ such that $\tu$ is smooth for all $t\in[0, T]$. For any given integer $k\ge 1$, let $u(\cdot,t)\in \mathbb{V}^k$ denote the numerical solution of \eqref{scheme B}, \eqref{scheme C} or \eqref{scheme D} with a quasi-uniform mesh and initial condition $\pi \tu_0\in\mathbb{V}^k$. If an upwind numerical flux is used, then for small enough $h$ there holds the following optimal error estimate:
    \begin{equation}\label{upwind flux error estimate}
        \max_{0\le t \le T} \|\tu(\cdot,t)-u(\cdot,t)\|_{L^2(I)} \le C h^{k+1}.
    \end{equation}
    Moreover, if a general monotone flux is used, then for small enough $h$ there holds
    \begin{equation}\label{monotone flux error estimate}
        \max_{0\le t \le T} \|\tu(\cdot,t)-u(\cdot,t)\|_{L^2(I)} \le C h^{k+\frac{1}{2}}.
    \end{equation}
    Here, $C>0$ depends on the uniform bound on $\tu$ and its derivatives in $I\times [0,T]$, but not on $u$ and $h$.
\end{theorem}

\begin{proof}
As it will be clear at the end, it suffices to show \eqref{upwind flux error estimate} and \eqref{monotone flux error estimate} for \eqref{scheme D} with any $k+1 \le r \le k+2$ and any $\epsilon_1,\epsilon_2>0$. Recall that \eqref{scheme D} is defined as follows: seek $u(\cdot,t)\in\mathbb{V}^k$ such that for all $v\in \mathbb{V}^k$ we have
\begin{align}
\int_{I_j}  u_t(x,t)  v(x) \ud x & - \int_{I_j}  f(u(x,t))  v_x(x) \ud x + \hf_{j+\frac{1}{2}}(t)\cdot  v^-_{j+\frac{1}{2}} - \hf_{j-\frac{1}{2}}(t)\cdot  v^+_{j-\frac{1}{2}} \nonumber \\
&= - D_j[u(\cdot,t);r,\epsilon_1,\epsilon_2] \int_{I_j}  ( u(x,t) - \bu_j(t))  v(x) \ud x, \label{scheme D again}
\end{align}
where $\bu_j(t)$ is the average of $u(x,t)$ on the cell $I_j$, $\hat{f}_{j+\frac{1}{2}}(t)=\hat{f}(u(x^-_{j+\frac{1}{2}},t),u(x^+_{j+\frac{1}{2}},t))$, and
\begin{equation}\label{def D_j again}
    D_j[u(\cdot,t);r,\epsilon_1,\epsilon_2] = \epsilon_1 h_j^{r-\frac{3}{2}} \|\partial_x^{r}f(u)(\cdot,t)\|_{L^2(I_j)} + \epsilon_2 \frac{|[\![u(\cdot,t)]\!]_{j-\frac{1}{2}}|+|[\![u(\cdot,t)]\!]_{j+\frac{1}{2}}|}{h_j}
\end{equation}
is defined in \eqref{def D_j}. Set $e(x,t):=\tu(x,t)-u(x,t)$ and denote
\begin{equation}\label{def xi eta}
    e=\xi-\eta,\quad \xi:= \bQ \tu - u,\quad \eta:=\bQ \tu - \tu,
\end{equation}
where the projection $\bQ$ is
\begin{enumerate}[label=(\arabic*)]
    \item the $L^2$-projection $\pi$ if a general monotone flux is used;
    \item if an upwind flux is used, $\bQ$ is a modified Gauss-Radau projection that depends on the initial condition $\tu_0$: for any continuous function $w$ on $I$, define $\bQ w\in\mathbb{V}^k$ such that
\begin{equation*}
    \begin{cases}
    \displaystyle
        \int_{I_j} (\bQ w - w) v \ud x = 0,& \forall v \in \mathbb{P}^{k-1}(I_j),\\
        \bQ w (x^-_{j+\frac12}) = w (x_{j+\frac12}), & \text{if $f'(\tu_0)>0$ on $I_j$,} \\
        \bQ w (x^+_{j-\frac12}) = w (x_{j-\frac12}), & \text{if $f'(\tu_0)<0$ on $I_j$,} \\
        \bQ w = \pi w, & \text{if $f'(\tu_0)$ has at least one zero on $I_j$.}
    \end{cases}
\end{equation*}
\end{enumerate}
Because the points where $f'(\tilde u)=0$ do not move with time as long as $\tu$ remains smooth, the definition of $\bQ$ would not change even if we replace $f'(\tu_0)$ by the time-dependent $f'(\tu)$ in the above three conditions, which is the original definition adopted in \cite{Z2005}. Since for all $v\in\mathbb{V}^k$, the exact solution $\tu$ satisfies
\begin{equation}\label{tu equation}
\int_{I_j}  \tu_t(x,t)  v(x) \ud x - \int_{I_j}  f(\tu(x,t))  v_x(x) \ud x + f(\tu(x_{j+\frac12},t))\cdot  v^-_{j+\frac{1}{2}} - f(\tu(x_{j-\frac12},t)) \cdot  v^+_{j-\frac{1}{2}}= 0,
\end{equation}
subtracting \eqref{scheme D again} from \eqref{tu equation} yields the following equation for $e$:
\begin{align}
&\int_{I_j}  e_t(x,t)  v(x) \ud x - \int_{I_j}  (f(\tu(x,t))-f(u(x,t)))  v_x(x) \ud x + (f(\tu(x_{j+\frac12},t))-\hf_{j+\frac12}(t))\cdot  v^-_{j+\frac{1}{2}} \nonumber \\
&- (f(\tu(x_{j-\frac12},t))-\hf_{j-\frac12}(t)) \cdot  v^+_{j-\frac{1}{2}} = D_j[u(\cdot,t);r,\epsilon_1,\epsilon_2]\int_{I_j}  ( u(x,t) - \bu_j(t))  v(x) \ud x, \label{error equation}
\end{align}
which holds for all $v\in\mathbb{V}^k$. At any time $t\ge 0$, by taking $v(x)=\xi(x,t)$ in \eqref{error equation}, we obtain
\begin{align}
    &\frac{1}{2}\odv{}{t}\|\xi(\cdot,t)\|^2_{L^2(I_j)}=\int_{I_j}  \xi_t(x,t)  \xi(x,t) \ud x \nonumber\\
    =& Z_j[\tu(\cdot,t),u(\cdot,t)] +D_j[u(\cdot,t);r,\epsilon_1,\epsilon_2]\int_{I_j}  ( u(x,t) - \bu_j(t))  \xi(x,t) \ud x, \label{error equation 2}
\end{align}
where
\begin{equation*}
    Z_j[\tu,u]:= \int_{I_j}  \eta_t  \xi \ud x + \int_{I_j}  (f(\tu)-f(u))  \xi_x \ud x - (f(\tu)_{j+\frac12}-\hf_{j+\frac12})\cdot  \xi^-_{j+\frac{1}{2}} + (f(\tu)_{j-\frac12}-\hf_{j-\frac12}) \cdot  \xi^+_{j-\frac{1}{2}}
\end{equation*}
is a functional acting on the functions $u,\tu$ at time $t$. It has been shown in \cite{Z2005} (see Lemmas 5.4, 5.5 and 5.7 therein) that 
\begin{enumerate}[label=(\arabic*)]
\item if $\hf$ is a general monotone flux,
\begin{equation}\label{monotone Z estimate}
    \sum_{j} Z_j[\tu,u]\le 
        C\left( 1+\|\xi\|_{L^2(I)}+\frac{1}{h} \|e\|_{L^2(I)}^2 \right) \|\xi\|_{L^2(I)}^2 + C\left( 1+\frac{1}{h} \|e\|_{L^2(I)}^2 \right) h^{2k+1};
\end{equation}

\item if $\hf$ is an upwind flux,
\begin{equation}\label{upwind Z estimate}
    \sum_{j} Z_j[\tu,u]\le 
        C\left( 1 + \frac{1}{h}\|e\|_{L^2(I)}+\frac{1}{h^2} \|e\|_{L^2(I)}^2 \right) \|\xi\|_{L^2(I)}^2 + C\left( 1+\frac{1}{h^2} \|e\|_{L^2(I)}^2 \right) h^{2k+2}.
\end{equation}
\end{enumerate}
In \eqref{monotone Z estimate} and \eqref{upwind Z estimate}, $C$ depends on the uniform bound on $\tu$ and its derivatives in $I\times [0,T]$, but not on $t$, $u$ and $h$.

In order to estimate the last term in \eqref{error equation 2}, we need the following lemma:

\begin{lemma}\label{lemma: second term estimate}
Under the assumptions of Theorem \ref{theorem: error estimate}, we have
\begin{align}
    \sum_{j} D_j[u(\cdot,t);r,\epsilon_1,\epsilon_2]\int_{I_j}  ( u(x,t) - \bu_j(t))  \xi(x,t) \ud x \le C(\|u(\cdot,t)\|_{L^\infty(I)}) \left(  \|\xi\|^2_{L^2(I)} + h^{2k+2} \right), \label{second term estimate}
\end{align}
where $C(\|u(\cdot,t)\|_{L^\infty(I)})<\infty$ is a constant which depends on $t$ only through the $L^\infty$-norm of $u$ at time $t$; it also depends on the uniform bound on $\tu$ and its derivatives in $I\times [0,T]$, but not on $h$.
\end{lemma}

\begin{proof}
Within the proof of this lemma, we adopt the notation $C_\star$ to denote any such constant $C(\|u(\cdot,t)\|_{L^\infty(I)})$ described in the statement above, and suppress the variable $t$ as we used to, since we may assume $t$ is fixed. We still use $C$ to denote those constants which are independent of $\|u(\cdot,t)\|_{L^\infty(I)}$ and $h$, but still depends on the uniform bound on $\tu$ and its derivatives in $I\times [0,T]$. Substituting $u=\bQ \tu - \xi$ and $\bu_j=\overline{(\bQ \tu)}_j - \overline{\xi}_j$ on the left-hand side of \eqref{second term estimate}, we get
\begin{align}
    &\sum_{j} D_j[u;r,\epsilon_1,\epsilon_2]\int_{I_j}  ( u - \bu_j)  \xi \ud x \nonumber \\
    = & \sum_{j} D_j[u;r,\epsilon_1,\epsilon_2]\int_{I_j}  ( \bQ \tu - \overline{(\bQ \tu)}_j)  \xi \ud x - \sum_{j} D_j[u;r,\epsilon_1,\epsilon_2]\int_{I_j}  ( \xi - \overline{\xi}_j)  \xi \ud x \nonumber \\
    = & \sum_{j} D_j[u;r,\epsilon_1,\epsilon_2]\int_{I_j}  ( \bQ \tu - \overline{(\bQ \tu)}_j)  \xi \ud x - \sum_{j} D_j[u;r,\epsilon_1,\epsilon_2]\int_{I_j}  ( \xi - \overline{\xi}_j)^2 \ud x \nonumber \\
    \le & \sum_{j} D_j[u;r,\epsilon_1,\epsilon_2]\int_{I_j}  ( \bQ \tu - \overline{(\bQ \tu)}_j)  \xi \ud x \nonumber \\
    \le & \sum_{j} D_j[u;r,\epsilon_1,\epsilon_2] \|\bQ \tu-\overline{(\bQ \tu)}_j\|_{L^2(I_j)} \|\xi\|_{L^2(I_j)}. \nonumber \\
    \le & C \sum_{j} h_j^{3/2} D_j[u;r,\epsilon_1,\epsilon_2] \|\xi\|_{L^2(I_j)}, \label{second term estimate 2}
\end{align}
where the last inequality follows from
\begin{align*}
\|\bQ \tu-\overline{(\bQ \tu)}_j\|_{L^2(I_j)}\le \|\bQ \tu-\tu\|_{L^2(I_j)}+\|\tu-\overline{(\tu)}_j\|_{L^2(I_j)}+\|\overline{(\tu-\bQ \tu)}_j\|_{L^2(I_j)}\le Ch_j^{3/2}.
\end{align*}
Now we proceed to estimate $D_j[u;r,\epsilon_1,\epsilon_2]$. By triangle inequality,
\begin{align}
\left\| \partial_x^{r}f(u)\right\|_{L^2(I_j)} \le& \left\| \partial_x^{r}f(u)-\partial_x^{r}f(\bQ \tu) \right\|_{L^2(I_j)} \nonumber \\
&+ \left\| \partial_x^{r}f(\bQ \tu)-\partial_x^{r}f(\tu) \right\|_{L^2(I_j)} \nonumber \\
&+ \left\| \partial_x^{r}f(\tu) \right\|_{L^2(I_j)}. \label{error triangle}
\end{align}
Set $a_i=\partial_x^i u /i!$, $b_i=\partial_x^i (\bQ \tu) /i!$ and $c_i=\partial_x^i \tu /i!$.  By Faà di Bruno's formula \cite{Hairer},
\begin{equation}\label{Faà di Bruno's formula}
    \partial_x^{r}f(u)=\sum_{\substack{m_1,\cdots,m_{r}\ge 0 \\ m_1+2m_2+\cdots+rm_r=r}} \frac{r !}{m_1 ! \cdots m_{r} !} f^{(m_1+\cdots+m_{r})}(u) \prod_{i=1}^{r}\left( a_i \right)^{m_i},
\end{equation}
so that
\begin{align}
    &\partial_x^{r}f(u)-\partial_x^{r}f(\bQ \tu) \nonumber\\
    =& \sum_{\substack{m_1,\cdots,m_{r}\ge 0 \\ m_1+2m_2+\cdots+rm_{r}=r}} \frac{r !}{m_1 ! \cdots m_{r} !} \left(f^{(m_1+\cdots+m_{r})}(u) - f^{(m_1+\cdots+m_{r})}(\bQ \tu) \right)\prod_{i=1}^{r}\left( a_i \right)^{m_i} \nonumber\\
    & + \sum_{\substack{m_1,\cdots,m_{r}\ge 0 \\ m_1+2m_2+\cdots+rm_{r}=r}} \frac{r !}{m_1 ! \cdots m_{r} !} f^{(m_1+\cdots+m_{r})}(\bQ \tu) \left(  \prod_{i=1}^{r}\left( a_i \right)^{m_i} - \prod_{i=1}^{r}\left( b_i \right)^{m_i} \right). \label{u-Qtu 1}
\end{align}
Note that
\begin{align}
\prod_{i=1}^{r}(a_i)^{m_i}-\prod_{i=1}^{r}(b_i)^{m_i}=& \sum_{i=1}^{r} ((a_i)^{m_i}-(b_i)^{m_i})\prod_{l=1}^{i-1} (b_l)^{m_l} \prod_{l=i+1}^{r} (a_l)^{m_l} \nonumber\\
=& \sum_{i=1}^{r} (a_i-b_i)\left( \sum_{l=0}^{m_i-1} a_i^l b_i^{m_i-1-l} \right)\prod_{l=1}^{i-1} (b_l)^{m_l} \prod_{l=i+1}^{r} (a_l)^{m_l}, \label{diff prod ai bi}
\end{align}
and also
\begin{align*}
&\|a_i-b_i\|_{L^2(I_j)} \le \frac{C}{h_j^i} \|\xi\|_{L^2(I_j)},\\
&\|a_i\|_{L^\infty(I_j)} \le \frac{C}{h_j^i} \|u\|_{L^\infty(I_j)} \le \frac{C_\star}{h_j^i},\\
&\|b_i\|_{L^\infty(I_j)} \le \frac{C}{h_j^i} \|\bQ \tu\|_{L^\infty(I_j)} \le \frac{C}{h_j^i}.
\end{align*}
Therefore, the exponent of $\frac{1}{h_j}$ in \eqref{diff prod ai bi} is
\begin{align*}
    &i+i(m_i-1)+(m_1+2m_2+\cdots+(i-1)m_{i-1})+((i+1)m_{i+1}+\cdots+rm_{r})\\
    =&m_1+2m_2+\cdots+rm_{r}=r,
\end{align*}
which implies
\begin{align*}
\left\| \prod_{i=1}^{r}(a_i)^{m_i}-\prod_{i=1}^{r}(b_i)^{m_i} \right\|_{L^2(I_j)}\le \frac{C_\star}{ h_j^{r}} \|\xi\|_{L^2(I_j)} .
\end{align*}
Also, by mean value theorem and the $L^\infty$-bound on $u$ and $\bQ \tu$,
\begin{align*}
\left\| f^{(m_1+\cdots+m_{r})}(u) - f^{(m_1+\cdots+m_{r})}(\bQ \tu) \right\|_{L^2(I_j)} \le C_\star \|\xi\|_{L^2(I_j)}.
\end{align*}
Hence, by \eqref{u-Qtu 1} and the two inequalities above,
\begin{equation}
\left\| \partial_x^{r}f(u)-\partial_x^{r}f(\bQ \tu) \right\|_{L^2(I_j)}\le \frac{C_\star}{ h_j^{r}}\|\xi\|_{L^2(I_j)}.\label{u-Qtu 2}
\end{equation}
To estimate $\left\|\partial_x^{r}f(\bQ \tu)-\partial_x^{r}f(\tu)\right\|_{L^2(I_j)}$, we proceed similarly. First, we compute
\begin{align}
    &\partial_x^{r}f(\bQ \tu)-\partial_x^{r}f(\tu) \nonumber\\
    =& \sum_{\substack{m_1,\cdots,m_{r}\ge 0 \\ m_1+2m_2+\cdots+rm_{r}=r}} \frac{r !}{m_1 ! \cdots m_{r} !} \left(f^{(m_1+\cdots+m_{r})}(\bQ \tu) - f^{(m_1+\cdots+m_{r})}(\tu) \right)\prod_{i=1}^{r}\left( c_i \right)^{m_i} \nonumber\\
    & + \sum_{\substack{m_1,\cdots,m_{r}\ge 0 \\ m_1+2m_2+\cdots+rm_{r}=r}} \frac{r !}{m_1 ! \cdots m_{r} !} f^{(m_1+\cdots+m_{r})}(\bQ \tu) \left(  \prod_{i=1}^{r}\left( b_i \right)^{m_i} - \prod_{i=1}^{r}\left( c_i \right)^{m_i} \right). \label{Qtu-tu 1}
\end{align}
Next, we have
\begin{align}
\prod_{i=1}^{r}(b_i)^{m_i}-\prod_{i=1}^{r}(c_i)^{m_i}= \sum_{i=1}^{r} (b_i-c_i)\left( \sum_{l=0}^{m_i-1} b_i^l c_i^{m_i-1-l} \right)\prod_{l=1}^{i-1} (c_l)^{m_l} \prod_{l=i+1}^{r} (b_l)^{m_l}, \label{diff prod bi ci}
\end{align}
It follows from Bramble–Hilbert lemma that
\begin{align*}
&\|b_i-c_i\|_{L^2(I_j)} \le C h_j^{k+1-i} |\tu|_{H^{k+1}(I_j)}\le C h_j^{k+\frac{3}{2}-i},\\
&\|b_i-c_i\|_{L^\infty(I_j)} \le C h_j^{k+1-i} \|\partial_x^{k+1} \tu \|_{L^\infty(I_j)}\le C h_j^{k+1-i},\\
&\|c_i\|_{L^\infty(I_j)} \le C,\\
&\|b_i\|_{L^\infty(I_j)} \le \|b_i-c_i\|_{L^\infty(I_j)} + \|c_i\|_{L^\infty(I_j)} \le C h_j^{\min\{0,k+1-i\}}.
\end{align*}
Recall that $i\le r\le k+2$ and $m_r\le 1$. Therefore,
\begin{align*}
\left\| \prod_{i=1}^{r}(b_i)^{m_i}-\prod_{i=1}^{r}(c_i)^{m_i} \right\|_{L^2(I_j)}\le C h_j^{k+\frac{3}{2}-r}.
\end{align*}
By mean value theorem and the $L^\infty$-bound on $u$ and $\bQ \tu$,
\begin{align*}
\left\| f^{(m_1+\cdots+m_{r})}(\bQ \tu) - f^{(m_1+\cdots+m_{r})}(\tu) \right\|_{L^2(I_j)} \le C \|\eta\|_{L^2(I_j)}.
\end{align*}
Hence, by \eqref{Qtu-tu 1} and the two inequalities above,
\begin{equation}
\left\| \partial_x^{r}f(\bQ \tu)-\partial_x^{r}f(\tu) \right\|_{L^2(I_j)}\le  C\|\eta\|_{L^2(I_j)} + C h_j^{k+\frac{3}{2}-r}.\label{Qtu-tu 2}
\end{equation}
Finally, since $\tu$ is smooth, we have the trivial bound
\begin{equation}
\left\| \partial_x^{r}f(\tu) \right\|_{L^2(I_j)}\le  C h_j^{\frac{1}{2}}.\label{only tu}
\end{equation}
Combining \eqref{error triangle}, \eqref{u-Qtu 2}, \eqref{Qtu-tu 2} and \eqref{only tu}, we get
\begin{align*}
h_j^{r}\left\| \partial_x^{r}f(u)\right\|_{L^2(I_j)} \le& C_\star \|\xi\|_{L^2(I_j)} + h_j^{r} C\|\eta\|_{L^2(I_j)} + C h_j^{\min\{k+\frac32,r+\frac12\}} \\
\le& C_\star \|\xi\|_{L^2(I_j)} + h_j^{r} C\|\eta\|_{L^2(I_j)} + C h_j^{k+\frac32},
\end{align*}
where the last inequality follows from $k+1\le r \le k+2$. This takes care of the first term in \eqref{def D_j again}. The second term in \eqref{def D_j} can be easily bounded by
\begin{align*}
|[\![u]\!]|_{j-\frac{1}{2}}+|[\![u]\!]|_{j+\frac{1}{2}}\le& |[\![\xi]\!]|_{j-\frac{1}{2}}+|[\![\eta]\!]|_{j-\frac{1}{2}}+|[\![\xi]\!]|_{j+\frac{1}{2}}+|[\![\eta]\!]|_{j+\frac{1}{2}}
\end{align*}
since the jumps of $\tu$ at cell interfaces are all zero. Therefore, taking into account $\|\eta\|_{L^2(I)} \le C h^{k+1}$, \eqref{second term estimate 2} can be estimated by
\begin{align*}
&C \sum_{j} h_j^{3/2} D_j[u;r,\epsilon_1,\epsilon_2] \|\xi\|_{L^2(I_j)} \\
\le& C \sum_j h_j^{r} \|\partial_x^{r}f(u)\|_{L^2(I_j)} \|\xi\|_{L^2(I_j)} + C \sqrt{h} \sum_j \left( |[\![u]\!]_{j-\frac{1}{2}}|+|[\![u]\!]_{j+\frac{1}{2}}| \right) \|\xi\|_{L^2(I_j)} \\
\le& C_\star \|\xi\|^2_{L^2(I)} + C h^r \|\eta\|^2_{L^2(I)} + C h^{2k+2} + C \sqrt{h} \sqrt{ \sum_j [\![\xi]\!]^2_{j+\frac{1}{2}} + \sum_j [\![\eta]\!]^2_{j+\frac{1}{2}} } \|\xi\|_{L^2(I)} \\
\le& C_\star \|\xi\|^2_{L^2(I)} + C h^{2k+2} + C \left( \|\xi\|_{L^2(I)} +\|\eta\|_{L^2(I)} \right) \|\xi\|_{L^2(I)}\\
\le& C_\star \|\xi\|^2_{L^2(I)} + C h^{2k+2}
\end{align*}
and the lemma is proved.
\end{proof}

Now, we begin to prove \eqref{upwind flux error estimate} and \eqref{monotone flux error estimate} using a bootstrap argument. We first suppose 
\begin{equation}\tag{Bootstrap Hypothesis}\label{Bootstrap Hypothesis}
    \|e(\cdot,t)\|_{L^2(I)}\le h \quad \text{for all $t\in[0,T]$.}
\end{equation}
Under this hypothesis, $\|u(\cdot,t)\|_{L^\infty(I)}\le C$ for all $t\in[0,T]$. By Lemma \ref{lemma: second term estimate}, 
\begin{align}\label{second term estimate induction}
    \sum_{j} D_j[u(\cdot,t);r,\epsilon_1,\epsilon_2]\int_{I_j}  ( u(x,t) - \bu_j(t))  \xi(x,t) \ud x \le C \|\xi\|^2_{L^2(I)} + C h^{2k+2} , 
\end{align}
where the constant $C$ no longer depends on $t$. Moreover, it follows from \eqref{monotone Z estimate} and \eqref{upwind Z estimate} that
\begin{enumerate}[label=(\arabic*)]
\item if $\hf$ is a general monotone flux,
\begin{equation*}
    \sum_{j} Z_j[\tu,u]\le 
        C\|\xi\|_{L^2(I)}^2 + Ch^{2k+1}.
\end{equation*}
Therefore, by \eqref{error equation 2} and \eqref{second term estimate induction},
\begin{align*}
    \frac{1}{2}\odv{}{t}\|\xi(\cdot,t)\|^2_{L^2(I)} \le C\|\xi\|_{L^2(I)}^2 + Ch^{2k+1}.
\end{align*}
Applying Gronwall's inequality, we obtain
\begin{align*}
    \|\xi(\cdot,t)\|^2_{L^2(I)} \le \|\xi(\cdot,0)\|^2_{L^2(I)}+Ch^{2k+1}\le Ch^{2k+1}, \quad \forall t\in[0,T],
\end{align*}
so that
\begin{align*}
    \|e(\cdot,t)\|_{L^2(I)} \le \|\xi(\cdot,t)\|_{L^2(I)}+\|\eta(\cdot,t)\|_{L^2(I)}\le Ch^{k+\frac12}. \quad \forall t\in[0,T].
\end{align*}
Now, we can take $h>0$ sufficiently small so that $Ch^{k+\frac12}\le h$, which completes the bootstrap step and the proof.
\item If $\hf$ is an upwind flux,
\begin{equation*}
    \sum_{j} Z_j[\tu,u]\le 
        C \|\xi\|_{L^2(I)}^2 + C h^{2k+2}.
\end{equation*}
We proceed as the monotone case and obtain
\begin{align*}
    \|\xi(\cdot,t)\|^2_{L^2(I)} \le \|\xi(\cdot,0)\|^2_{L^2(I)}+Ch^{2k+2}\le Ch^{2k+2}, \quad \forall t\in[0,T],
\end{align*}
and
\begin{align*}
    \|e(\cdot,t)\|_{L^2(I)} \le \|\xi(\cdot,t)\|_{L^2(I)}+\|\eta(\cdot,t)\|_{L^2(I)}\le Ch^{k+1}. \quad \forall t\in[0,T].
\end{align*}
Now, we can take $h>0$ sufficiently small so that $Ch^{k+1}\le h$, which completes the bootstrap step and the proof.
\end{enumerate} 

Now, we have proved Theorem \ref{theorem: error estimate} for \eqref{scheme D}. Since $u$ is uniformly bounded for all $t\in[0,T]$ under \eqref{Bootstrap Hypothesis}, we can find $\epsilon_1,\epsilon_2$ such that $D_j[u;r,\epsilon_1,\epsilon_2]\ge B_j[u;U]$ and $D_j[u;r,\epsilon_1,\epsilon_2]\ge C_j[u;U,r]$ (recall this is how we constructed $D_j$ so that we could use monotonicity of remainder). Thus, our estimate for $D_j[u;r,\epsilon_1,\epsilon_2]$ in Lemma \ref{lemma: second term estimate} can be directly applied to $B_j[u;U]$ and $C_j[u;U,r]$, and in the same way we can prove Theorem \ref{theorem: error estimate} for \eqref{scheme B} and \eqref{scheme C}.

\end{proof}

\subsection{Numerical tests for optimal order of convergence}
\begin{expl}
We first test the optimal order of convergence proved in Theorem \ref{theorem: error estimate} for \eqref{scheme D} with polynomial degrees $k=1,2,3$ and $k+1 \le r \le k+2$. Consider the nonlinear Burgers equation
\begin{equation*}
    u_t + \left( \frac{u^2}{2} \right)_x = 0
\end{equation*}
with the initial condition $u_0(x)=\sin(x)+0.5$. The computational domain is $I=[0, 2\pi]$ with periodic boundary conditions. This is the same example from \cite{OFDG}. The final time is taken to be $T=0.6$, so that the exact solution remains smooth. We use the classic fourth order Runge--Kutta method for time discretizations. The CFL condition is $\Delta t = \mathcal{O}(h)$. In all tests, we adopt the local Lax-Friedrichs numerical flux and set $\epsilon_1=\epsilon_2=1$.

We test the accuracy on a series of uniform meshes, with the total number of cells equal to $N=16,32,\cdots,512$. The $L^1$, $L^2$ and $L^\infty$ errors and their orders of convergence for the cases $r=k+1$ and $r=k+2$ are shown in Tables \ref{tab:burgers_convergence k+1} and \ref{tab:burgers_convergence k+2}, respectively. We can see that the numerical results matches with the statement of Theorem \ref{theorem: error estimate}.

\begin{table}[htbp]
\begin{adjustbox}{width=\columnwidth,center}
\begin{tabular}{c|r|cc|cc|cc}
\hline
 & $N$
& $L^1$ error & Order
& $L^2$ error & Order
& $L^\infty$ error & Order \\
\hline
\multirow{6}{*}{$\mathbb{P}^1$} & 16  & $5.858936\times10^{-2}$ & --     & $3.913706\times10^{-2}$ & --     & $9.937107\times10^{-2}$ & --     \\
 & 32  & $1.375636\times10^{-2}$ & 2.0905 & $9.450200\times10^{-3}$ & 2.0501 & $2.300587\times10^{-2}$ & 2.1108 \\
 & 64  & $3.271946\times10^{-3}$ & 2.0719 & $2.399096\times10^{-3}$ & 1.9779 & $7.577038\times10^{-3}$ & 1.6023 \\
 & 128 & $7.845710\times10^{-4}$ & 2.0602 & $5.862501\times10^{-4}$ & 2.0329 & $2.055943\times10^{-3}$ & 1.8818 \\
 & 256 & $1.921616\times10^{-4}$ & 2.0296 & $1.471028\times10^{-4}$ & 1.9947 & $5.553945\times10^{-4}$ & 1.8882 \\
 & 512 & $4.746975\times10^{-5}$ & 2.0172 & $3.662547\times10^{-5}$ & 2.0059 & $1.346929\times10^{-4}$ & 2.0438 \\
\hline
\multirow{6}{*}{$\mathbb{P}^2$} & 16  & $5.579540\times10^{-2}$ & --     & $6.537560\times10^{-2}$ & --     & $2.125512\times10^{-1}$ & --     \\
 & 32  & $6.356686\times10^{-3}$ & 3.1338 & $9.017231\times10^{-3}$ & 2.8580 & $3.935147\times10^{-2}$ & 2.4333 \\
 & 64  & $4.968015\times10^{-4}$ & 3.6775 & $8.178954\times10^{-4}$ & 3.4627 & $6.172782\times10^{-3}$ & 2.6724 \\
 & 128 & $3.620624\times10^{-5}$ & 3.7784 & $5.919958\times10^{-5}$ & 3.7883 & $4.116191\times10^{-4}$ & 3.9065 \\
 & 256 & $2.977955\times10^{-6}$ & 3.6038 & $5.265680\times10^{-6}$ & 3.4909 & $6.562058\times10^{-5}$ & 2.6491 \\
 & 512 & $2.648582\times10^{-7}$ & 3.4910 & $4.668005\times10^{-7}$ & 3.4957 & $6.501662\times10^{-6}$ & 3.3353 \\
\hline
\multirow{6}{*}{$\mathbb{P}^3$} & 16  & $3.098316\times10^{-1}$ & --     & $2.214555\times10^{-1}$ & --     & $3.821968\times10^{-1}$ & --     \\
 & 32  & $6.271633\times10^{-2}$ & 2.3046 & $8.178122\times10^{-2}$ & 1.4372 & $1.780029\times10^{-1}$ & 1.1024 \\
 & 64  & $1.298156\times10^{-3}$ & 5.5943 & $3.586751\times10^{-3}$ & 4.5110 & $1.591189\times10^{-2}$ & 3.4837 \\
 & 128 & $1.287967\times10^{-5}$ & 6.6552 & $2.873509\times10^{-5}$ & 6.9637 & $1.597338\times10^{-4}$ & 6.6383 \\
 & 256 & $2.857467\times10^{-7}$ & 5.4942 & $4.912621\times10^{-7}$ & 5.8702 & $2.823307\times10^{-6}$ & 5.8221 \\
 & 512 & $9.106452\times10^{-9}$ & 4.9717 & $1.517097\times10^{-8}$ & 5.0171 & $9.072239\times10^{-8}$ & 4.9598 \\
\hline
\end{tabular}
\end{adjustbox}
\caption{Errors and convergence orders for \eqref{scheme D} with polynomial degrees $k=1,2,3$ and $r=k+1$.}
\label{tab:burgers_convergence k+1}
\end{table}

\begin{table}[htbp]
\begin{adjustbox}{width=\columnwidth,center}
\begin{tabular}{c|r|cc|cc|cc}
\hline
 & $N$
& $L^1$ error & Order
& $L^2$ error & Order
& $L^\infty$ error & Order \\
\hline
\multirow{6}{*}{$\mathbb{P}^1$} & 16  & $4.294674\times10^{-2}$ & --     & $2.754407\times10^{-2}$ & --     & $5.637620\times10^{-2}$ & --     \\
 & 32  & $1.099892\times10^{-2}$ & 1.9652 & $7.523647\times10^{-3}$ & 1.8722 & $2.174760\times10^{-2}$ & 1.3742 \\
 & 64  & $2.785689\times10^{-3}$ & 1.9813 & $2.031890\times10^{-3}$ & 1.8886 & $6.226255\times10^{-3}$ & 1.8044 \\
 & 128 & $7.101160\times10^{-4}$ & 1.9719 & $5.367464\times10^{-4}$ & 1.9205 & $1.692266\times10^{-3}$ & 1.8794 \\
 & 256 & $1.802998\times10^{-4}$ & 1.9777 & $1.393244\times10^{-4}$ & 1.9458 & $4.389260\times10^{-4}$ & 1.9469 \\
 & 512 & $4.559686\times10^{-5}$ & 1.9834 & $3.566824\times10^{-5}$ & 1.9657 & $1.115395\times10^{-4}$ & 1.9764 \\
\hline
\multirow{6}{*}{$\mathbb{P}^2$} & 16  & $1.854099\times10^{-2}$ & --     & $1.849425\times10^{-2}$ & --     & $6.238941\times10^{-2}$ & --     \\
 & 32  & $1.645980\times10^{-3}$ & 3.4937 & $1.991278\times10^{-3}$ & 3.2153 & $6.517240\times10^{-3}$ & 3.2590 \\
 & 64  & $1.080701\times10^{-4}$ & 3.9289 & $1.303109\times10^{-4}$ & 3.9337 & $5.678523\times10^{-4}$ & 3.5207 \\
 & 128 & $9.718192\times10^{-6}$ & 3.4751 & $1.100091\times10^{-5}$ & 3.5663 & $6.984597\times10^{-5}$ & 3.0233 \\
 & 256 & $1.091877\times10^{-6}$ & 3.1539 & $1.259977\times10^{-6}$ & 3.1262 & $9.076374\times10^{-6}$ & 2.9440 \\
 & 512 & $1.322708\times10^{-7}$ & 3.0452 & $1.559972\times10^{-7}$ & 3.0138 & $1.170951\times10^{-6}$ & 2.9544 \\
\hline
\multirow{6}{*}{$\mathbb{P}^3$} & 16  & $2.167622\times10^{-1}$ & --     & $2.113418\times10^{-1}$ & --     & $4.010803\times10^{-1}$ & --     \\
 & 32  & $5.608432\times10^{-2}$ & 1.9504 & $7.943122\times10^{-2}$ & 1.4118 & $1.981113\times10^{-1}$ & 1.0176 \\
 & 64  & $6.803013\times10^{-5}$ & 9.6872 & $9.882538\times10^{-5}$ & 9.6506 & $3.113658\times10^{-4}$ & 9.3135 \\
 & 128 & $1.058955\times10^{-6}$ & 6.0055 & $1.542502\times10^{-6}$ & 6.0015 & $8.865795\times10^{-6}$ & 5.1342 \\
 & 256 & $2.144179\times10^{-8}$ & 5.6261 & $3.149298\times10^{-8}$ & 5.6141 & $2.284965\times10^{-7}$ & 5.2780 \\
 & 512 & $6.680798\times10^{-10}$ & 5.0043 & $1.000724\times10^{-9}$ & 4.9759 & $8.485364\times10^{-9}$ & 4.7511 \\
\hline
\end{tabular}
\end{adjustbox}
\caption{Errors and convergence orders for \eqref{scheme D} with polynomial degrees $k=1,2,3$ and $r=k+2$.}
\label{tab:burgers_convergence k+2}
\end{table}

\end{expl}

\FloatBarrier
\begin{expl}
Although Theorem \ref{theorem: error estimate} only covers 1D scalar conservation laws, we nevertheless proceed to the case of 2D scalar conservation laws. Consider the nonlinear 2D Burgers equation
\begin{equation*}
    u_t + \left( u^2 \right)_x + \left( u^2 \right)_y = 0
\end{equation*}
with the initial condition $u_0(x,y)=0.5 \sin(2\pi (x+y) )$. We test the optimal order of convergence for \eqref{scheme D''} with polynomial degrees $k=1,2,3$ and $k+1 \le r \le k+2$. The computational domain is $I=[0, 1]\times [0,1]$ with periodic boundary conditions. This is the same example from \cite{CHEN2017427}. The final time is taken to be $T=0.05$, so that the exact solution remains smooth. We use the classic fourth order Runge--Kutta method for time discretizations. The CFL condition is $\Delta t = \mathcal{O}(h)$. In all tests, we adopt the local Lax-Friedrichs numerical flux and set $\epsilon_1=\epsilon_2=1$.

We test the accuracy on a series of uniform rectangular meshes with $N_x=N_y=16,32,\cdots,512$. The $L^1$, $L^2$ and $L^\infty$ errors and their orders of convergence for the cases $r=k+1$ and $r=k+2$ are shown in Tables \ref{tab:burgers-convergence-Dpp-k-plus-1} and \ref{tab:burgers-convergence-Dpp-k-plus-2}, respectively. The convergence orders are indeed optimal. Note that the $L^1$ and $L^2$ errors are evaluated using a tensor-product 10-point Gauss–Legendre rule on every rectangular cell, while the $L^\infty$ error uses a separate uniform point set that includes cell boundaries; super-convergence in $L^1$ and $L^2$ errors can be observed for $\mathbb{P}^3$ elements.

\begin{table}[htbp]
\begin{adjustbox}{width=\columnwidth,center}
\begin{tabular}{c|r|cc|cc|cc}
\hline
 & $N$
& $L^1$ error & Order
& $L^2$ error & Order
& $L^\infty$ error & Order \\
\hline
\multirow{6}{*}{$\mathbb{P}^1$}
 & 16  & $8.873109\times10^{-3}$ & --    & $1.516498\times10^{-2}$ & --    & $1.535248\times10^{-1}$ & --    \\
 & 32  & $2.219736\times10^{-3}$ & 1.999 & $4.294241\times10^{-3}$ & 1.820 & $5.490943\times10^{-2}$ & 1.483 \\
 & 64  & $5.276571\times10^{-4}$ & 2.073 & $1.079269\times10^{-3}$ & 1.992 & $1.371005\times10^{-2}$ & 2.002 \\
 & 128 & $1.250893\times10^{-4}$ & 2.077 & $2.612186\times10^{-4}$ & 2.047 & $2.932631\times10^{-3}$ & 2.225 \\
 & 256 & $3.006254\times10^{-5}$ & 2.057 & $6.402969\times10^{-5}$ & 2.028 & $8.263882\times10^{-4}$ & 1.827 \\
 & 512 & $7.332548\times10^{-6}$ & 2.036 & $1.593880\times10^{-5}$ & 2.006 & $2.199474\times10^{-4}$ & 1.910 \\
\hline
\multirow{6}{*}{$\mathbb{P}^2$}
 & 16  & $3.851715\times10^{-3}$ & --    & $8.276021\times10^{-3}$ & --    & $1.020042\times10^{-1}$ & --    \\
 & 32  & $5.841907\times10^{-4}$ & 2.721 & $1.750315\times10^{-3}$ & 2.241 & $2.323799\times10^{-2}$ & 2.134 \\
 & 64  & $6.805218\times10^{-5}$ & 3.102 & $2.180787\times10^{-4}$ & 3.005 & $3.479090\times10^{-3}$ & 2.740 \\
 & 128 & $6.682944\times10^{-6}$ & 3.348 & $2.203113\times10^{-5}$ & 3.307 & $4.685234\times10^{-4}$ & 2.893 \\
 & 256 & $6.231348\times10^{-7}$ & 3.423 & $2.062058\times10^{-6}$ & 3.417 & $6.001436\times10^{-5}$ & 2.965 \\
 & 512 & $6.120003\times10^{-8}$ & 3.348 & $2.041695\times10^{-7}$ & 3.336 & $7.359222\times10^{-6}$ & 3.028 \\
\hline
\multirow{6}{*}{$\mathbb{P}^3$}
 & 16  & $4.058691\times10^{-2}$ & --    & $6.319777\times10^{-2}$ & --    & $3.127916\times10^{-1}$ & --    \\
 & 32  & $1.042123\times10^{-2}$ & 1.961 & $2.798347\times10^{-2}$ & 1.175 & $1.871199\times10^{-1}$ & 0.741 \\
 & 64  & $2.347751\times10^{-3}$ & 2.150 & $1.071306\times10^{-2}$ & 1.385 & $9.616082\times10^{-2}$ & 0.960 \\
 & 128 & $7.096482\times10^{-5}$ & 5.048 & $8.266929\times10^{-4}$ & 3.696 & $1.557553\times10^{-2}$ & 2.626 \\
 & 256 & $3.059249\times10^{-7}$ & 7.858 & $2.532318\times10^{-6}$ & 8.351 & $6.318595\times10^{-5}$ & 7.945 \\
 & 512 & $1.004042\times10^{-8}$ & 4.929 & $6.985113\times10^{-8}$ & 5.180 & $2.634123\times10^{-6}$ & 4.584 \\
\hline
\end{tabular}
\end{adjustbox}
\caption{Errors and convergence orders for \eqref{scheme D''} with polynomial degrees
$k=1,2,3$ and $r=k+1$.}
\label{tab:burgers-convergence-Dpp-k-plus-1}
\end{table}

\begin{table}[htbp]
\begin{adjustbox}{width=\columnwidth,center}
\begin{tabular}{c|r|cc|cc|cc}
\hline
 & $N$
& $L^1$ error & Order
& $L^2$ error & Order
& $L^\infty$ error & Order \\
\hline
\multirow{6}{*}{$\mathbb{P}^1$}
 & 16  & $6.745922\times10^{-3}$ & --    & $1.206151\times10^{-2}$ & --    & $1.380574\times10^{-1}$ & --    \\
 & 32  & $1.762432\times10^{-3}$ & 1.936 & $3.512284\times10^{-3}$ & 1.780 & $4.127613\times10^{-2}$ & 1.742 \\
 & 64  & $4.419675\times10^{-4}$ & 1.996 & $9.224084\times10^{-4}$ & 1.929 & $1.129637\times10^{-2}$ & 1.869 \\
 & 128 & $1.108997\times10^{-4}$ & 1.995 & $2.388377\times10^{-4}$ & 1.949 & $3.236580\times10^{-3}$ & 1.803 \\
 & 256 & $2.788790\times10^{-5}$ & 1.992 & $6.146162\times10^{-5}$ & 1.958 & $8.705389\times10^{-4}$ & 1.894 \\
 & 512 & $7.020283\times10^{-6}$ & 1.990 & $1.567670\times10^{-5}$ & 1.971 & $2.256230\times10^{-4}$ & 1.948 \\
\hline
\multirow{6}{*}{$\mathbb{P}^2$}
 & 16  & $3.458565\times10^{-3}$ & --    & $8.039311\times10^{-3}$ & --    & $9.947845\times10^{-2}$ & --    \\
 & 32  & $3.873141\times10^{-4}$ & 3.159 & $1.131736\times10^{-3}$ & 2.829 & $1.898286\times10^{-2}$ & 2.390 \\
 & 64  & $3.181137\times10^{-5}$ & 3.606 & $9.140592\times10^{-5}$ & 3.630 & $2.517905\times10^{-3}$ & 2.914 \\
 & 128 & $3.184988\times10^{-6}$ & 3.320 & $9.412806\times10^{-6}$ & 3.280 & $3.038137\times10^{-4}$ & 3.051 \\
 & 256 & $3.869309\times10^{-7}$ & 3.041 & $1.192756\times10^{-6}$ & 2.980 & $3.953387\times10^{-5}$ & 2.942 \\
 & 512 & $4.895042\times10^{-8}$ & 2.983 & $1.557787\times10^{-7}$ & 2.937 & $5.673613\times10^{-6}$ & 2.801 \\
\hline
\multirow{6}{*}{$\mathbb{P}^3$}
 & 16  & $4.589640\times10^{-2}$ & --    & $6.740899\times10^{-2}$ & --    & $3.427352\times10^{-1}$ & --    \\
 & 32  & $9.132882\times10^{-3}$ & 2.329 & $2.828390\times10^{-2}$ & 1.253 & $2.050638\times10^{-1}$ & 0.741 \\
 & 64  & $3.319369\times10^{-5}$ & 8.104 & $1.220094\times10^{-4}$ & 7.857 & $1.763706\times10^{-3}$ & 6.861 \\
 & 128 & $8.950958\times10^{-7}$ & 5.213 & $3.976124\times10^{-6}$ & 4.939 & $6.332885\times10^{-5}$ & 4.800 \\
 & 256 & $2.294565\times10^{-8}$ & 5.286 & $1.062076\times10^{-7}$ & 5.226 & $1.641234\times10^{-6}$ & 5.270 \\
 & 512 & $7.024232\times10^{-10}$ & 5.030 & $2.994068\times10^{-9}$ & 5.149 & $1.017521\times10^{-7}$ & 4.012 \\
\hline
\end{tabular}
\end{adjustbox}
\caption{Errors and convergence orders for \eqref{scheme D''} with polynomial degrees
$k=1,2,3$ and $r=k+2$.}
\label{tab:burgers-convergence-Dpp-k-plus-2}
\end{table}

\end{expl}

\section{Strong convergence through compensated compactness for general strictly convex conservation laws}\label{section: strong convergence}

In this section we restrict attention to the scalar equation \eqref{HCL 1D scalar}
with a strictly convex flux ($f''>0$ on $\mathbb{R}$). Let $h:=\max_jh_j$ as in Section \ref{section: error estimate}, but denote the numerical solution by $u_h$ (instead of $u$) to emphasize its role as a sequence parameterized by $h$. We show in Theorem \ref{theorem: strong convergence} that, for any initial data \(\mathcal U_0\in L^\infty(I)\) with $u_h(\cdot,0)\to \mathcal U_0$ strongly in \(L^2(I)\), the sequence of numerical solutions $\{u_h\}$ generated by \eqref{scheme D} with $\epsilon_1>0$ converges strongly to the unique entropy solution as $h\rightarrow 0$ if $\{u_h\}$ is uniformly bounded in $L^\infty$. We prove it using the compensated compactness argument
\cite{Murat1978,Tartar1979,DiPerna1983}. 

We make the following assumptions. The meshes are quasi-uniform, $k\ge1$ is fixed, and $u_h$ denotes a solution of \eqref{scheme D} subject to periodic or compactly supported boundary conditions. For convenience, we only consider the case 
\begin{equation}\label{SC choice r}
    r=k+1,\quad \epsilon_1>0.
\end{equation}
The same argument applies to $r=k+2$, but $\epsilon_1>0$ must be satisfied. We assume that, for every
$T>0$,
\begin{equation}\label{SC Linfinity}
    \|u_h\|_{L^\infty(I\times(0,T))}\le M_T
\end{equation}
uniformly in $h$. We also assume that $f\in C^{k+2}$ and
\begin{equation}\label{SC uniform convexity}
    f''(s)\ge c_f>0,\qquad |s|\le M_T.
\end{equation}
Note that $c_f$ might depend on $M_T$ if $f(.)$ is not uniformly convex. Finally, $\hf$ is assumed to be a consistent locally Lipschitz E-flux whose
square-entropy dissipation is quantitatively coercive: if
\begin{equation}\label{SC square-entropy flux}
    \eta(s):=\frac{s^2}{2},\qquad q'(s):=s f'(s),
\end{equation}
denotes the square-entropy pair\footnote{Here we adopt the notation $(\eta,q)$ instead of $(U,\mF)$ as the entropy-entropy flux pair to avoid confusion with the prescribed entropy $U$ in the definition of \eqref{scheme B} and \eqref{scheme C} from the previous sections.}, then its interface dissipation $\mathcal D_\eta$ satisfies
\begin{equation}\label{SC interface dissipation definition}
    \mathcal D_\eta(a,b)
    :=\int_a^b\bigl(f(s)-\hf(a,b)\bigr)\,\ud s.
\end{equation}
This is \eqref{B proof interface entropy dissipation} specialized to the
square-entropy $U$, for which $U''\equiv1$. We assume that
\begin{equation}\label{SC cubic dissipation assumption}
    \mathcal D_\eta(a,b)\ge c_{\hf}|a-b|^3,
    \qquad |a|,|b|\le M_T.
\end{equation}
This is crucial for deriving the weak-BV estimate \eqref{SC weak BV} below. The following lemma shows that every E-flux (see Definition \ref{def E-flux}) satisfies \eqref{SC cubic dissipation assumption}.
\begin{lemma}
Assume that \eqref{SC uniform convexity} holds. Then every E-flux
satisfies \eqref{SC cubic dissipation assumption}; in particular, one may
take
\[
c_{\hf}=\frac{c_f}{24}.
\]
\end{lemma}
\begin{proof}
Let $a,b\in[-M_T,M_T]$.  We prove the estimate with
\begin{equation}\label{SC Eflux cubic constant}
    c_{\hf}:=\frac{c_f}{24},
\end{equation}
which is independent of $a$, $b$, and of the particular E-flux.  The case
$a=b$ is immediate, so it remains to consider the two possible orientations
of the interval.

Suppose first that $a<b$.  Definition~\ref{def E-flux} gives
\[
    \hf(a,b)\le f(s)\qquad\text{for every }s\in[a,b].
\]
Consequently,
\begin{equation}\label{SC Eflux below minimum}
    \hf(a,b)\le \min_{s\in[a,b]}f(s).
\end{equation}
Let $m\in[a,b]$ be a point at which this minimum is attained.  We claim that
\begin{equation}\label{SC convex distance from minimum}
    f(s)-f(m)\ge \frac{c_f}{2}|s-m|^2,
    \qquad s\in[a,b].
\end{equation}
Indeed, uniform convexity gives
\[
    f(s)\ge f(m)+f'(m)(s-m)+\frac{c_f}{2}|s-m|^2.
\]
If $m\in(a,b)$, then $f'(m)=0$.  If $m=a$, then $f'(m)\ge0$ and
$s-m\ge0$; if $m=b$, then $f'(m)\le0$ and $s-m\le0$.  Thus in every case
$f'(m)(s-m)\ge0$, which proves \eqref{SC convex distance from minimum}.
Using \eqref{SC interface dissipation definition},
\eqref{SC Eflux below minimum}, and
\eqref{SC convex distance from minimum}, we obtain
\begin{align*}
    \mathcal D_\eta(a,b)
    &=\int_a^b\bigl(f(s)-\hf(a,b)\bigr)\,\ud s\\
    &\ge \int_a^b\bigl(f(s)-f(m)\bigr)\,\ud s\\
    &\ge \frac{c_f}{2}\int_a^b|s-m|^2\,\ud s\\
    &=\frac{c_f}{6}\left((m-a)^3+(b-m)^3\right).
\end{align*}
For $x,y\ge0$, convexity of $z\mapsto z^3$ gives
\[
    x^3+y^3\ge 2\left(\frac{x+y}{2}\right)^3
    =\frac{(x+y)^3}{4}.
\]
Taking $x=m-a$ and $y=b-m$ therefore yields
\begin{equation}\label{SC Eflux increasing orientation}
    \mathcal D_\eta(a,b)
    \ge \frac{c_f}{24}(b-a)^3.
\end{equation}

Suppose next that $a>b$.  In this case Definition~\ref{def E-flux} gives
\[
    \hf(a,b)\ge f(s) \qquad\text{for every }s\in[b,a],
\]
and hence
\begin{equation}\label{SC Eflux above maximum}
    \hf(a,b)\ge\max\{f(a),f(b)\}.
\end{equation}
Let $\ell$ be the affine function joining the endpoint values of $f$ on
$[b,a]$:
\[
    \ell(s):=\frac{a-s}{a-b}f(b)+\frac{s-b}{a-b}f(a).
\]
Since $\ell(s)$ is a convex combination of $f(a)$ and $f(b)$,
\begin{equation}\label{SC chord below endpoint maximum}
    \ell(s)\le\max\{f(a),f(b)\}.
\end{equation}
Furthermore, $f''\ge c_f$ implies
\begin{equation}\label{SC strong convexity chord estimate}
    \ell(s)-f(s)\ge\frac{c_f}{2}(s-b)(a-s),
    \qquad s\in[b,a].
\end{equation}
For completeness, this follows by applying the ordinary convexity inequality
to the convex function $g(s):=f(s)-c_fs^2/2$: writing
$s=(1-\theta)b+\theta a$, where
$\theta=(s-b)/(a-b)$, gives
\[
    g(s)\le(1-\theta)g(b)+\theta g(a),
\]
and rearrangement is exactly
\eqref{SC strong convexity chord estimate}.

Because the integral in \eqref{SC interface dissipation definition} is now
oriented from $a$ to $b$, equations
\eqref{SC Eflux above maximum}--\eqref{SC strong convexity chord estimate}
give
\begin{align*}
    \mathcal D_\eta(a,b)
    &=\int_b^a\bigl(\hf(a,b)-f(s)\bigr)\,\ud s\\
    &\ge\int_b^a\bigl(\ell(s)-f(s)\bigr)\,\ud s\\
    &\ge\frac{c_f}{2}\int_b^a(s-b)(a-s)\,\ud s\\
    &=\frac{c_f}{12}(a-b)^3\\
    &\ge\frac{c_f}{24}|a-b|^3.
\end{align*}
Together with \eqref{SC Eflux increasing orientation}, this proves
\[
    \mathcal D_\eta(a,b)
    \ge c_{\hf}|a-b|^3,
    \qquad |a|,|b|\le M_T,
\]
with $c_{\hf}=c_f/24$, and hence proves
\eqref{SC cubic dissipation assumption} for every E-flux.
\end{proof}

In what follows, we denote by $C$ any constant that may depend on $k,T$ but independent of $h$. We first record the estimates used in the proof; they cannot be obtained without the first term in the definition of $D_j$ \eqref{def D_j}, which is why we need to assume $\epsilon_1>0$ in \eqref{SC choice r}.

\begin{lemma}[Entropy dissipation and consistency]\label{lemma: SC estimates}
Under the preceding assumptions, the numerical solutions satisfy
\begin{align}
    &\int_0^T\sum_j
    D_j[u_h;k+1,\epsilon_1,\epsilon_2]
    \|u_h-\bu_{h,j}\|_{L^2(I_j)}^2\,\ud t\le C,
    \label{SC cell dissipation}\\
    &\int_0^T\sum_j
    |[\![u_h]\!]_{j+\frac12}|^3\,\ud t\le C,
    \label{SC weak BV}\\
    &\int_0^T\sum_j
    D_j[u_h;k+1,\epsilon_1,\epsilon_2](h_j)^3\,\ud t
    \le C h. \label{SC damping consistency}
\end{align}
Moreover, with the piecewise smooth function $\rho$ defined on $I$ by
\begin{equation}\label{SC projection remainder}
    \rho(x):=f(u_h)_x(x)-\pi[f(u_h)_x](x)\quad \forall x \in I_j,
\end{equation}
one has
\begin{equation}\label{SC rho estimate}
    \|\rho\|_{L^2(I_j)}
    \le C(h_j)^{1/2}D_j.
\end{equation}
Here and below $D_j=D_j[u_h;k+1,\epsilon_1,\epsilon_2]$.
\end{lemma}

\begin{proof}
Assume $I=[x_L,x_R]$. Taking $v=u_h$ in \eqref{scheme D}, proceeding as in \cite{J} or in the proof of Theorem \ref{theorem: Local Entropy Inequality B} and using
\[
    \int_{I_j}(u_h-\bu_{h,j})u_h\,\ud x
    =\|u_h-\bu_{h,j}\|_{L^2(I_j)}^2,
\]
we obtain the following local square-entropy inequality:
\begin{equation*}
    \odv{}{t} \int_{I_j} \eta(u_h) \ud x + \hat{q}_{j+\frac12}(t) - \hat{q}_{j-\frac12}(t) + \frac12 \left( \mathcal D_{\eta,j-\frac12} + \mathcal D_{\eta,j+\frac12} \right) + D_j \|u_h-\bu_{h,j}\|_{L^2(I_j)}^2 \le 0,
\end{equation*}
where $\mathcal D_{\eta,j + \frac12}=\mathcal D_{\eta}(u^-_{j+\frac12}, u^+_{j+\frac12})$ and \[\hat q_{j+\frac12}(t) = \frac{1}{2}(u^-_{j+\frac12}+u^+_{j+\frac12}) \hf_{j+\frac12} - \frac{1}{2}\left( \int^{u^-_{j+\frac12}} f(\xi) \ud \xi + \int^{u^+_{j+\frac12}} f(\xi) \ud \xi \right).\]
Summing it over $j$ yields the global square-entropy inequality
\begin{equation*}
    \odv{}{t} \int_{I} \eta(u_h) \ud x + \hat{q}_{R}(t) - \hat{q}_{L}(t) + \sum_j \frac12 \left( \mathcal D_{\eta,j-\frac12} + \mathcal D_{\eta,j+\frac12} \right) + \sum_j D_j \|u_h-\bu_{h,j}\|_{L^2(I_j)}^2 \le 0.
\end{equation*}
Integrate in time and apply the assumption $\|u_h\|_{L^\infty(I\times(0,T))}\le M_T$,
\begin{align}
\int_0^T \sum_j \frac12 \left( \mathcal D_{\eta,j-\frac12} + \mathcal D_{\eta,j+\frac12} \right) \ud t + \int_0^T \sum_j D_j \|u_h-\bu_{h,j}\|_{L^2(I_j)}^2 \ud t \nonumber \\
\le \int_{I} \eta(u_h(x,0)) \ud x + \int_0^T (\hat{q}_{L}(t) - \hat{q}_{R}(t)) \ud t \le C. \label{tight square-entropy inequality}
\end{align}
By \eqref{SC cubic dissipation assumption}, \[\int_0^T \sum_j \frac12 \left( \mathcal D_{\eta,j-\frac12} + \mathcal D_{\eta,j+\frac12} \right) \ud t \ge \frac{c_{\hf}}{2} \int_0^T\sum_j |[\![u_h]\!]_{j+\frac12}|^3\,\ud t,\]
so that
\begin{align*}
\frac{c_{\hf}}{2} \int_0^T\sum_j |[\![u_h]\!]_{j+\frac12}|^3\,\ud t + \int_0^T \sum_j D_j \|u_h-\bu_{h,j}\|_{L^2(I_j)}^2 \ud t \le C.
\end{align*}
This proves \eqref{SC cell dissipation} and \eqref{SC weak BV}.

For the first term in the definition of $D_j$, map $I_j$ affinely to the reference interval.
The set of polynomials of degree at most $k$ that are bounded by $M_T$ is a
bounded subset of a finite-dimensional space. Hence by Faà di Bruno's formula \eqref{Faà di Bruno's formula}, the
boundedness of the derivatives of $f$ on $[-M_T,M_T]$, and scaling back to
$I_j$ give
\begin{align*}
 \|\partial_x^{k+1}f(u_h)\|_{L^2(I_j)}
 &\le C (h_j)^{-(k+1)}\sum_{i=1}^{k+1}\|f^{(i)}(u_h)\|_{L^2(I_j)}\le C (h_j)^{-k-\frac12}.
\end{align*}
Since the jumps are uniformly bounded by \eqref{SC Linfinity}, the definition
\eqref{def D_j} consequently implies
\begin{equation}\label{SC Dj inverse bound}
    0\le D_j\le C (h_j)^{-1}.
\end{equation}
Quasi-uniformity now yields
\[
 \int_0^T\sum_jD_j(h_j)^3\,\ud t
 \le C \sum_jh_j^2\le C h,
\]
which is \eqref{SC damping consistency}. Finally, \eqref{SC rho estimate} follows from \eqref{bound on bad term} with $r=k+1$ and
\[
 D_j\ge\epsilon_1(h_j)^{k-\frac12}
 \|\partial_x^{k+1}f(u_h)\|_{L^2(I_j)}.
\]
\end{proof}

\begin{lemma}[Compact entropy productions]\label{lemma: SC Hminus1}
The two sequences
\begin{equation}\label{SC two productions}
 \partial_tu_h+\partial_xf(u_h),\qquad
 \partial_t \eta(u_h)+\partial_xq(u_h)
\end{equation}
are relatively compact in $H^{-1}_{\mathrm{loc}}(I\times (0,T))$.
\end{lemma}

\begin{proof}
Fix a bounded Lipschitz open set $\mathcal O\Subset I\times(0,T)$. All estimates below are
restricted to the finitely many space-time cells meeting $\mathcal O$. We use the distributional
convention
\begin{align}
 \left\langle\partial_tu_h+\partial_xf(u_h),\varphi\right\rangle
 &:=-\iint_{\mathcal O}
 \bigl(u_h\varphi_t+f(u_h)\varphi_x\bigr)\,\ud x\ud t,
 \label{SC definition L1}\\
 \left\langle\partial_t\eta(u_h)+\partial_xq(u_h),\varphi\right\rangle
 &:=-\iint_{\mathcal O}
 \bigl(\eta(u_h)\varphi_t+q(u_h)\varphi_x\bigr)\,\ud x\ud t
 \label{SC definition L2}
\end{align}
for every $\varphi\in C_c^\infty(\mathcal O)$. In what follows, $\mathcal M (\mathcal O)$ denotes the space of finite signed Radon measures on the open set $\mathcal O$. Equivalently, $\mathcal M (\mathcal O)=C_0(\mathcal O)^*$. For any $\mu\in \mathcal M (\mathcal O)$, define its norm as $\|\mu\|_{\mathcal M (\mathcal O)}:=|\mu|(\mathcal O)$.

\medskip
\noindent\emph{Step 1: the decomposition for the residual of \eqref{SC definition L1} on each cell.}
Let $\pi_j$ denote the $L^2(I_j)$-projection onto $\mathbb P^k(I_j)$, and
define the two polynomials $\ell_j^-,\ell_j^+\in
\mathbb P^k(I_j)$ by
\begin{equation}\label{SC lifting definition}
 \int_{I_j}\ell_j^-v\,\ud x=v(x_{j-\frac12}^+),\qquad
 \int_{I_j}\ell_j^+v\,\ud x=v(x_{j+\frac12}^-)
 \quad\forall v\in\mathbb P^k(I_j).
\end{equation}
One can actually obtain the explicit expressions for $\ell_j^-,\ell_j^+$ as
\[
\ell_j^-(x)
:=\sum_{i=0}^k \sqrt{\frac{2i+1}{h_j}} (-1)^iP_{i,j}(x),
\qquad
\ell_j^+(x)
:=\sum_{i=0}^k \sqrt{\frac{2i+1}{h_j}} P_{i,j}(x).
\]
where $P_{i,j}(x)$ is the $i$th Legendre polynomial normalized on $I_j$ with $\|P_{i,j}\|_{L^2(I_j)}=1$. Therefore, $\int_{I_j}\ell_j^-\varphi\,\ud x=\pi_{j}\varphi(x_{j-\frac12}^+)$ and $\int_{I_j}\ell_j^+\varphi\,\ud x=\pi_j\varphi(x_{j+\frac12}^-)$. Recall 
\begin{equation}\label{SC alpha definition}
 H_{j-\frac12}^+
 :=\hf_{j-\frac12}-f(u_h)^+_{j-\frac12},\qquad
 H_{j+\frac12}^-
 :=f(u_h)^-_{j+\frac12}-\hf_{j+\frac12}.
\end{equation}
Using the strong form of \eqref{scheme D} and the defining property
\eqref{SC lifting definition}, we obtain the following identity in
$\mathbb P^k(I_j)$:
\begin{equation}\label{SC exact cell residual}
 (u_h)_t\vert_{I_j}+\pi_j[f(u_h)_x\vert_{I_j}]
 =H_{j-\frac12}^+\ell_j^-
  +H_{j+\frac12}^-\ell_j^+
  -D_j(u_h-\bu_{h,j}).
\end{equation}
Therefore, with
\[
 \rho_j:=\rho\vert_{I_j}= f(u_h)_x\vert_{I_j}-\pi_j[f(u_h)_x\vert_{I_j}] \in C^{k+1}(I_j),
\]
the classical residual in the interior of $I_j$ is
\begin{equation}\label{SC classical cell residual}
 \left((u_h)_t+f(u_h)_x\right)\vert_{I_j}
 =\rho_j+H_{j-\frac12}^+\ell_j^-
  +H_{j+\frac12}^-\ell_j^+
  -D_j(u_h-\bu_{h,j}).
\end{equation}

Because $f(u_h)$ is discontinuous at the interfaces, its distributional
derivative contains additionally the term
\[
 \sum_j[\![f(u_h)]\!]_{j+\frac12}\delta_{x_{j+\frac12}}.
\]
Combining this fact with \eqref{SC classical cell residual} gives the exact
decomposition on $I$:
\begin{equation}\label{SC L1 exact decomposition}
 \partial_tu_h+\partial_xf(u_h)
 =R_h^{\mathrm{vol}}+R_h^{\mathrm{damp}}
  +R_h^{\mathrm{int}}
\end{equation}
in $\mathcal D'(\mathcal O)$, where
\begin{align}
 R_h^{\mathrm{vol}}&:=\sum_j\rho_j\mathbf 1_{I_j},
 \label{SC Rvol definition}\\
 R_h^{\mathrm{damp}}&:=-\sum_jD_j(u_h-\bu_{h,j})\mathbf 1_{I_j},
 \label{SC Rdamp definition}\\
 R_h^{\mathrm{int}}&:=\sum_j \left(H_{j-\frac12}^+\ell_j^-
  +H_{j+\frac12}^-\ell_j^+ \right)\mathbf 1_{I_j} + \sum_j[\![f(u_h)]\!]_{j+\frac12}\delta_{x_{j+\frac12}}.
\end{align}

\medskip
\noindent\emph{Step 2: $R_h^{\mathrm{vol}}$ tends to zero in
$H^{-1}(\mathcal O)$.}
In addition to \eqref{SC rho estimate}, we need a bound containing the
within-cell oscillation. Set
\[
 v_j:=u_h-\bu_{h,j}.
\]
Taylor expansion at $\bu_{h,j}$ gives
\begin{equation}\label{SC Taylor remainder definition}
 f(u_h)=f(\bu_{h,j})+f'(\bu_{h,j})v_j+\mathcal R_j,
\end{equation}
where the integral form of the remainder is
\begin{equation}\label{SC Taylor remainder integral}
 \mathcal R_j
 =v_j^2\int_0^1(1-\theta)
 f''(\bu_{h,j}+\theta v_j)\,\ud\theta.
\end{equation}
On the bounded range \eqref{SC Linfinity}, put
\[
 C_f:=\max_{|s|\le M_T}f''(s).
\]
Together with \eqref{SC uniform convexity},
\eqref{SC Taylor remainder integral} implies the pointwise comparison
\begin{equation}\label{SC remainder equivalence}
 \frac{c_f}{2}|v_j|^2
 \le\mathcal R_j
 \le\frac{C_f}{2}|v_j|^2.
\end{equation}

The first two terms on the right-hand side of
\eqref{SC Taylor remainder definition} belong to $\mathbb P^k(I_j)$.
Their derivatives therefore belong to
$\mathbb P^{k-1}(I_j)\subset\mathbb P^k(I_j)$ and are reproduced exactly by
$\pi_j$. Consequently,
\begin{equation}\label{SC rho remainder derivative}
 \rho_j=(\mathcal R_j)_x-\pi_j[(\mathcal R_j)_x].
\end{equation}
Since $\pi_j$ is the $L^2$-orthogonal projection,
\begin{equation}\label{SC rho best approximation}
 \|\rho_j\|_{L^2(I_j)}
 \le\|(\mathcal R_j)_x\|_{L^2(I_j)}.
\end{equation}

We next establish an inverse-type estimate for $\|(\mathcal R_j)_x\|_{L^2(I_j)}$. Differentiating
\eqref{SC Taylor remainder definition} in $x$ gives
\[
 (\mathcal R_j)_x
 =\bigl(f'(u_h)-f'(\bu_{h,j})\bigr)(v_j)_x.
\]
By the fundamental theorem of calculus,
\[
 f'(u_h)-f'(\bu_{h,j})
 =v_j\int_0^1f''(\bu_{h,j}+\theta v_j)\,\ud\theta,
\]
and hence
\begin{equation}\label{SC remainder derivative bound}
 |(\mathcal R_j)_x|
 \le C_f|v_j(v_j)_x|
 =\frac{C_f}{2}|(v_j^2)_x|.
\end{equation}
Although $\mathcal R_j$ is generally not a polynomial, $v_j^2$ belongs to
$\mathbb P^{2k}(I_j)$. The polynomial inverse inequality can therefore be
applied to $v_j^2$, giving
\[
 \|(v_j^2)_x\|_{L^2(I_j)}
 \le C h_j^{-1}\|v_j^2\|_{L^2(I_j)}.
\]
Moreover, the lower bound in \eqref{SC remainder equivalence} yields
\[
 \|v_j^2\|_{L^2(I_j)}
 \le\frac{2}{c_f}\|\mathcal R_j\|_{L^2(I_j)}.
\]
Combining these two inequalities with
\eqref{SC rho best approximation} and
\eqref{SC remainder derivative bound}, we obtain
\begin{equation}\label{SC first nonlinear rho inequality}
 \|\rho_j\|_{L^2(I_j)}
 \le C h_j^{-1}\|\mathcal R_j\|_{L^2(I_j)},
\end{equation}
where $C$ depends on $k$ and $C_f/c_f$, but is independent of $h_j$.

Finally, the upper bound in \eqref{SC remainder equivalence} and the polynomial
inverse estimate
\[
 \|u_h-\bu_{h,j}\|_{L^\infty(I_j)}
 \le C h_j^{-1/2}\|u_h-\bu_{h,j}\|_{L^2(I_j)}
\]
give
\begin{align}
 \|\rho_j\|_{L^2(I_j)}
 &\le C h_j^{-1}\|\mathcal R_j\|_{L^2(I_j)}\nonumber\\
 &\le C h_j^{-1}
 \|u_h-\bu_{h,j}\|_{L^\infty(I_j)}
 \|u_h-\bu_{h,j}\|_{L^2(I_j)}\nonumber\\
 &\le C(h_j)^{-\frac32}
 \|u_h-\bu_{h,j}\|_{L^2(I_j)}^2.
 \label{SC nonlinear rho estimate}
\end{align}
Multiplying \eqref{SC rho estimate} and
\eqref{SC nonlinear rho estimate} yields
\begin{equation}\label{SC rho product estimate}
 h_j^2\|\rho_j\|_{L^2(I_j)}^2
 \le Ch_jD_j
 \|u_h-\bu_{h,j}\|_{L^2(I_j)}^2.
\end{equation}
Moreover, $\rho_j\perp \mathbb{P}^0(I_j)$. Thus, if
$\varphi\in C_c^\infty(\mathcal O)$ and
$\bar\varphi_j(t)$ is its average over $I_j$, then
\begin{align*}
 |\langle R_h^{\mathrm{vol}},\varphi\rangle|
 &=\left|\int_0^T\sum_j\int_{I_j}
 \rho_j(\varphi-\bar\varphi_j)\,\ud x\ud t\right|\\
 &\le
 C\left(\int_0^T\sum_jh_j^2
 \|\rho_j\|_{L^2(I_j)}^2\,\ud t\right)^{1/2}
 \|\varphi_x\|_{L^2(\mathcal O)}\\
 &\le C h^{1/2}
 \|\varphi\|_{H^1(\mathcal O)}
\end{align*}
by the cellwise Poincar\'e inequality, \eqref{SC cell dissipation}, \eqref{SC rho product estimate}, and
quasi-uniformity. Consequently,
\begin{equation}\label{SC Rvol Hminus1}
 R_h^{\mathrm{vol}}\longrightarrow0
 \quad\text{strongly in }H^{-1}(\mathcal O).
\end{equation}

\medskip
\noindent\emph{Step 3: $R_h^{\mathrm{damp}}$ tends to zero in $H^{-1}(\mathcal O)$.}
Using again the zero mean of $u_h-\bu_{h,j}$, the cellwise Poincar\'e inequality,
\eqref{SC cell dissipation}, and \eqref{SC Dj inverse bound}, we obtain
\begin{align*}
 |\langle R_h^{\mathrm{damp}},\varphi\rangle|
 &\le\int_0^T\sum_jD_j
 \|u_h-\bu_{h,j}\|_{L^2(I_j)}
 \|\varphi-\bar\varphi_j\|_{L^2(I_j)}\,\ud t\\
 &\le
 C\left(\int_0^T\sum_jD_j
 \|u_h-\bu_{h,j}\|_{L^2(I_j)}^2\,\ud t\right)^{1/2}\\
 &\qquad\times
 \left(\int_0^T\sum_jD_j h_j^2
 \|\varphi_x\|_{L^2(I_j)}^2\,\ud t\right)^{1/2}\\
 &\le C h^{1/2}\|\varphi\|_{H^1(\mathcal O)}.
\end{align*}
Hence
\begin{equation}\label{SC Rdamp Hminus1}
 R_h^{\mathrm{damp}}\longrightarrow0
 \quad\text{strongly in }H^{-1}(\mathcal O).
\end{equation}

\medskip
\noindent\emph{Step 4: the conservation-law interface remainder $R_h^{\mathrm{int}}$.}
For clarity of notation, let $u_{j+\frac12}^\pm=u_h(x_{j+\frac12}^\pm,t)$ and
$\varphi_{j+\frac12}=\varphi(x_{j+\frac12},t)$. From \eqref{SC lifting definition}, the complete
contribution at $x_{j+\frac12}$ is
\begin{align}
 \langle R_{h,{j+\frac12}}^{\mathrm{int}},\varphi\rangle
 ={}&\int_0^T\Bigl\{
 (f(u_{j+\frac12}^-)-\hf_{j+\frac12})(\pi_j\varphi)(x_{j+\frac12}^-)
 +(\hf_{j+\frac12}-f(u_{j+\frac12}^+))(\pi_{j+1}\varphi)(x_{j+\frac12}^+)\nonumber\\
 &\qquad +(f(u_{j+\frac12}^+)-f(u_{j+\frac12}^-))\varphi_{j+\frac12}\Bigr\}\,\ud t \nonumber\\
 = &\int_0^T\Bigl\{
 (f(u_{j+\frac12}^-)-\hf_{j+\frac12})
 [(\pi_j\varphi)(x_{j+\frac12}^-)-\varphi_{j+\frac12}]\nonumber\\
 &\qquad+(\hf_{j+\frac12}-f(u_{j+\frac12}^+))
 [(\pi_{j+1}\varphi)(x_{j+\frac12}^+)-\varphi_{j+\frac12}]\Bigr\}\,\ud t.
 \label{SC conservation interface identity}
\end{align}
Indeed, the terms multiplying $\varphi_{j+\frac12}$ cancel because
\[
 f(u_{j+\frac12}^-)-\hf_{j+\frac12}+\hf_{j+\frac12}-f(u_{j+\frac12}^+)
 +f(u_{j+\frac12}^+)-f(u_{j+\frac12}^-)=0.
\]
Consistency and local Lipschitz continuity imply
\begin{equation}\label{SC alpha jump bound}
 |f(u_{j+\frac12}^-)-\hf_{j+\frac12}|+|\hf_{j+\frac12}-f(u_{j+\frac12}^+)|
 \le C|u_{j+\frac12}^+-u_{j+\frac12}^-|.
\end{equation}
For $p'>2$, we claim that the following inequality holds
\begin{equation}\label{SC endpoint W1p estimate}
 |(\pi_j\varphi)(x_{j+\frac12}^\pm)-\varphi_{j+\frac12}|
 \le C h_j^{1-\frac1{p'}}
 \|\varphi_x(\cdot,t)\|_{L^{p'}(I_j)},
\end{equation}
where $C$ is independent of $\varphi$. Here is the detailed scaling argument for proving \eqref{SC endpoint W1p estimate}. Let
\[
 F_j(\xi):=x_j+\frac{h_j}{2}\xi,\qquad -1\le\xi\le1,
\]
and set
\[
 \hat\varphi(\xi,t):=\varphi(F_j(\xi),t).
\]
If $\hat\pi$ denotes the $L^2(-1,1)$-projection onto
$\mathbb P^k(-1,1)$, then the affine invariance of the $L^2$-projection gives
\begin{equation}\label{SC projection affine invariance}
 (\pi_j\varphi)(F_j(\xi),t)
 =(\hat\pi\hat\varphi)(\xi,t).
\end{equation}
Since $x_{j+\frac12}=F_j(1)$ and $\hat\pi$ reproduces constants,
\begin{align}
 |(\pi_j\varphi)(x_{j+\frac12}^-,t)-\varphi_{j+\frac12}|
 &=
 \left|
 \hat\pi [\hat\varphi(\cdot,t)-\hat\varphi(1,t)](1)
 \right|\nonumber\\
 &\le C
 \|\hat\varphi(\cdot,t)-\hat\varphi(1,t)\|_{L^\infty(-1,1)}.
 \label{SC endpoint projection stability}
\end{align}
The last inequality follows because
$g\mapsto(\hat\pi g)(1)$ is a bounded linear functional on
$L^\infty(-1,1)$.
For every $\xi\in[-1,1]$, the fundamental theorem of calculus and H\"older's
inequality give
\begin{align*}
 |\hat\varphi(\xi,t)-\hat\varphi(1,t)|
 \le\int_\xi^1|\partial_\zeta\hat\varphi(\zeta,t)|\,\ud\zeta \le2^{1-\frac1{p'}}
 \|\partial_\xi\hat\varphi(\cdot,t)\|_{L^{p'}(-1,1)}.
\end{align*}
Since
\[
 \partial_\xi\hat\varphi
 =\frac{h_j}{2}\varphi_x\circ F_j
\]
and $dx=(h_j/2)d\xi$, scaling yields
\begin{equation}\label{SC derivative scaling}
 \|\partial_\xi\hat\varphi(\cdot,t)\|_{L^{p'}(-1,1)}
 =
 \left(\frac{h_j}{2}\right)^{1-\frac1{p'}}
 \|\varphi_x(\cdot,t)\|_{L^{p'}(I_j)}.
\end{equation}
Substituting \eqref{SC derivative scaling} into
\eqref{SC endpoint projection stability} proves
\eqref{SC endpoint W1p estimate}. The proof at the left endpoint is
identical. Notice that this one-dimensional estimate only requires $p'>1$;
we choose $p'>2$ because its conjugate exponent $p=p'/(p'-1)$ then satisfies
$1<p<2$, which is the range used below.

Put
\[
 b_{j+\frac12}(t):=|[\![u_h]\!]_{j+\frac12}|,\qquad
 a_j(t):=\|\varphi_x(\cdot,t)\|_{L^{p'}(I_j)}.
\]
At the interface $x_{j+\frac12}$, equations
\eqref{SC conservation interface identity},
\eqref{SC alpha jump bound}, and
\eqref{SC endpoint W1p estimate} imply
\begin{align}
 |\langle R_{h,{j+\frac12}}^{\mathrm{int}},\varphi\rangle|
 &\le C\int_0^T b_{j+\frac12}(t)
 \left[
 h_j^{1-\frac1{p'}}a_j(t)
 +h_{j+1}^{1-\frac1{p'}}a_{j+1}(t)
 \right]\,\ud t.\label{SC one interface estimate}
\end{align}
Define
\[
 A_{j+\frac12}(t):=a_j(t)+a_{j+1}(t).
\]
Summing \eqref{SC one interface estimate} over all interfaces contained in $\mathcal O$ gives
\begin{equation}\label{SC interface sum before Holder}
 |\langle R_h^{\mathrm{int}},\varphi\rangle|
 \le Ch^{1-\frac1{p'}}
 \int_0^T\sum_{(x_{j+\frac12},t)\in \mathcal O}
 b_{j+\frac12}(t)A_{j+\frac12}(t)\,\ud t.
\end{equation}
In what follows we shall simply write $\sum_j$ instead of $\sum_{(x_{j+\frac12},t)\in \mathcal O}$.
We now apply H\"older's inequality with exponents $3$ and $3/2$ on the
product of Lebesgue measure in time and counting measure on the interface
set:
\begin{align}
 \int_0^T\sum_j b_{j+\frac12} A_{j+\frac12}\,\ud t
 &\le
 \left(\int_0^T\sum_j b_{j+\frac12}^3\,\ud t\right)^{1/3}
 \left(\int_0^T\sum_j A_{j+\frac12}^{3/2}\,\ud t\right)^{2/3}.
 \label{SC space time Holder}
\end{align}
By \eqref{SC weak BV}, the first factor is bounded uniformly in $h$.
Moreover,
\[
 A_{j+\frac12}^{3/2}
 \le 2^{1/2}\bigl(a_j^{3/2}+a_{j+1}^{3/2}\bigr)
\]
implies
\begin{equation}\label{SC A to a}
 \sum_j A_{j+\frac12}(t)^{3/2}\le C\sum_ja_j(t)^{3/2}.
\end{equation}

For completeness, let $N_h$ be the number of cells meeting the spatial
projection of $\mathcal O$. Quasi-uniformity gives $N_h\le Ch^{-1}$.
Because $p'>2>3/2$, discrete H\"older's inequality with conjugate exponents
$2p'/3$ and $2p'/(2p'-3)$ yields
\begin{align}
 \sum_ja_j(t)^{3/2} \le
 \left(\sum_ja_j(t)^{p'}\right)^{\frac{3}{2p'}}
 N_h^{1-\frac{3}{2p'}} \le
 Ch^{-1+\frac{3}{2p'}}
 \left(\sum_ja_j(t)^{p'}\right)^{\frac{3}{2p'}}.
 \label{SC discrete a estimate}
\end{align}
Set
\[
 B(t):=\sum_ja_j(t)^{p'}
 =\sum_j\int_{I_j}|\varphi_x(x,t)|^{p'}\,\ud x.
\]
Since $\gamma:=3/(2p')<1$, H\"older's inequality in time gives
\begin{align}
 \int_0^TB(t)^\gamma\,\ud t \le T^{1-\gamma}
 \left(\int_0^TB(t)\,\ud t\right)^\gamma =T^{1-\frac{3}{2p'}}
 \|\varphi_x\|_{L^{p'}(\mathcal O)}^{3/2}.
 \label{SC Holder in time}
\end{align}
Combining \eqref{SC A to a}--\eqref{SC Holder in time} and then raising to
the power $2/3$, we obtain
\begin{align}
 \left(\int_0^T\sum_j A_{j+\frac12}(t)^{3/2}\,\ud t\right)^{2/3}
 &\le
 Ch^{-\frac23+\frac1{p'}}
 T^{\frac23-\frac1{p'}}
 \|\varphi_x\|_{L^{p'}(\mathcal O)}.
 \label{SC second Holder factor}
\end{align}
Substituting \eqref{SC space time Holder} and
\eqref{SC second Holder factor} into
\eqref{SC interface sum before Holder}, and using
\eqref{SC weak BV}, yields
\begin{align}
 |\langle R_h^{\mathrm{int}},\varphi\rangle|
 \le
 Ch^{1-\frac1{p'}}
 h^{-\frac23+\frac1{p'}}
 \|\varphi_x\|_{L^{p'}(\mathcal O)} =Ch^{1/3}\|\varphi_x\|_{L^{p'}(\mathcal O)} \le Ch^{1/3}\|\varphi\|_{W^{1,p'}(\mathcal O)}. \label{SC Rint quantitative}
\end{align}
Thus, with $1/p+1/p'=1$,
\begin{equation}\label{SC Rint Wminus1p}
 R_h^{\mathrm{int}}\longrightarrow0
 \quad\text{strongly in }W^{-1,p}(\mathcal O),
 \qquad 1<p<2.
\end{equation}
Combining \eqref{SC L1 exact decomposition},
\eqref{SC Rvol Hminus1}, \eqref{SC Rdamp Hminus1}, and
\eqref{SC Rint Wminus1p}, we conclude that the first production in
\eqref{SC two productions} is relatively compact in
$W^{-1,p}(\mathcal O)$ for every $1<p<2$. Its upgrade to
$H^{-1}(\mathcal O)$ is carried out in Step 6.

\medskip
\noindent\emph{Step 5: the decomposition of the square-entropy production \eqref{SC definition L2}.}
Since $\eta'(u_h)=u_h$, the chain rule is valid in each open cell and
\[
 \left( \partial_t\eta(u_h)+\partial_xq(u_h) \right)\vert_{I_j}
 =\left(u_h\bigl((u_h)_t+f(u_h)_x\bigr)\right)\vert_{I_j}.
\]
Multiplying \eqref{SC classical cell residual} by $u_h$ and adding the
distributional interface term
$[\![q(u_h)]\!]_{j+\frac12}\delta_{x_{j+\frac12}}$ gives
\begin{equation}\label{SC L2 exact decomposition}
 \partial_t\eta(u_h)+\partial_xq(u_h)
 =-\mu_h^{\mathrm{cell}}-\mu_h^{\mathrm{int}}
  +S_h^{\mathrm{vol}}+S_h^{\mathrm{damp}}
  +S_h^{\mathrm{int}},
\end{equation}
where the two nonnegative measures are
\begin{align}
 \mu_h^{\mathrm{cell}}
 &:=\sum_jD_j|u_h-\bu_{h,j}|^2\mathbf1_{I_j}\,\ud x\ud t,
 \label{SC cell measure}\\
 \mu_h^{\mathrm{int}}
 &:=\sum_j \mathcal D_\eta(u_{j+\frac12}^-,u_{j+\frac12}^+)
 \delta_{x_{j+\frac12}}\,\ud t.
 \label{SC interface measure}
\end{align}
and the space-time distributions are
\begin{align}
 S_h^{\mathrm{vol}}&:=\sum_j u_h\rho_j\mathbf1_{I_j},\\
 S_h^{\mathrm{damp}}&:=-\sum_jD_j\bu_{h,j}(u_h-\bu_{h,j})\mathbf1_{I_j},\\
 S_h^{\mathrm{int}}&:=\sum_j u_h \left(H_{j-\frac12}^+\ell_j^-
  +H_{j+\frac12}^-\ell_j^+ \right)\mathbf 1_{I_j} + \sum_j \left( [\![q(u_h)]\!]_{j+\frac12} + \mathcal D_\eta(u_{j+\frac12}^-,u_{j+\frac12}^+) \right) \delta_{x_{j+\frac12}}.
\end{align}
Note that the total contribution at the interface $x_{j+\frac12}$ is
\begin{align}
 & (f(u_{j+\frac12}^-)-\hf_{j+\frac12})
       [\pi_j(u_h\varphi)(x_{j+\frac12}^-)-u_{j+\frac12}^-\varphi_{j+\frac12}]\nonumber\\
 &\quad+(\hf_{j+\frac12}-f(u_{j+\frac12}^+))
       [\pi_{j+1}(u_h\varphi)(x_{j+\frac12}^+)-u_{j+\frac12}^+\varphi_{j+\frac12}]\nonumber\\
 &\quad+\Bigl\{
 u_{j+\frac12}^-[f(u_{j+\frac12}^-)-\hf_{j+\frac12}]
 +u_{j+\frac12}^+[\hf_{j+\frac12}-f(u_{j+\frac12}^+)]
 +q(u_{j+\frac12}^+)-q(u_{j+\frac12}^-)\Bigr\}\varphi_{j+\frac12}.
 \label{SC entropy interface preliminary}
\end{align}
By $q'(s)=sf'(s)$, the
quantity in braces equals
\begin{align*}
 &(u_{j+\frac12}^+-u_{j+\frac12}^-)\hf_{j+\frac12}-\int_{u_{j+\frac12}^-}^{u_{j+\frac12}^+}f(s)\,\ud s
 =-\mathcal D_\eta(u_{j+\frac12}^-,u_{j+\frac12}^+).
\end{align*}
This is precisely the measure $-\mu_h^{\mathrm{int}}$. The remaining two
terms define $S_h^{\mathrm{int}}$. In addition, $\mu_h^{\mathrm{cell}}$ and $S_h^{\mathrm{damp}}$ arise from $
 -D_ju_h(u_h-\bu_{h,j})
 =-D_j|u_h-\bu_{h,j}|^2-D_j\bu_{h,j}(u_h-\bu_{h,j})=-\mu_h^{\mathrm{cell}}+S_h^{\mathrm{damp}}.$

Inequalities \eqref{tight square-entropy inequality} and
\eqref{SC cell dissipation}, which are the direct consequences of the global square-entropy inequality, imply
\begin{equation}\label{SC measure bound}
 \|\mu_h^{\mathrm{cell}}\|_{\mathcal M(\mathcal O)}
 +\|\mu_h^{\mathrm{int}}\|_{\mathcal M(\mathcal O)}
 \le C.
\end{equation}
We now identify and estimate the three remaining terms in
\eqref{SC L2 exact decomposition}. First, because $|u_h|\le M_T$,
\[
 S_h^{\mathrm{vol}}:=\sum_j u_h\rho_j\mathbf1_{I_j}
\]
can be estimated in the same way as the proof of \eqref{SC Rvol Hminus1}:
\begin{align}
|\langle S_h^{\mathrm{vol}},\varphi\rangle|= &\left|\int_0^T\sum_j\int_{I_j}
 u_h \rho_j(\varphi-\bar\varphi_j)\,\ud x\ud t\right| \nonumber\\
 &\le
 C\left(\int_0^T\sum_jh_j^2
 \|\rho_j\|_{L^2(I_j)}^2\,\ud t\right)^{1/2}
 \|\varphi_x\|_{L^2(\mathcal O)} \nonumber\\
 &\le C h^{1/2}
 \|\varphi\|_{H^1(\mathcal O)}.\label{SC entropy projection vanishes}
\end{align}
Thus $S_h^{\mathrm{vol}}\to0$ strongly in $H^{-1}(\mathcal O)$.
Also, $S_h^{\mathrm{damp}}
 :=-\sum_jD_j\bu_{h,j}(u_h-\bu_{h,j})\mathbf1_{I_j}$
tends to zero strongly in $H^{-1}(\mathcal O)$ by Step 3 and
$|\bu_{h,j}|\le M_T$.

Finally, we estimate $S_h^{\mathrm{int}}$. By \eqref{SC entropy interface preliminary}, we have
\begin{align*}
\langle S_h^{\mathrm{int}},\varphi\rangle = \int_0^T \sum_j \bigg\{ & (f(u_{j+\frac12}^-)-\hf_{j+\frac12}) [\pi_j(u_h\varphi)(x_{j+\frac12}^-)-u_{j+\frac12}^-\varphi_{j+\frac12}] \\
&\quad+(\hf_{j+\frac12}-f(u_{j+\frac12}^+)) [\pi_{j+1}(u_h\varphi)(x_{j+\frac12}^+)-u_{j+\frac12}^+\varphi_{j+\frac12}] \bigg\} \ud t
\end{align*}
Since $\pi_j$ reproduces $u_h$,
the $L^\infty$-stability of projection gives
\begin{align}
 |\pi_j(u_h\varphi)(x_{j+\frac12}^\pm)-u_{j+\frac12}^\pm\varphi_{j+\frac12}|
 \le& C \|u_h(\varphi-\varphi_{j+\frac12})\|_{L^\infty(I_j)} \nonumber \\
 \le& Ch_j^{1-\frac{1}{p'}}\|u_h\|_{L^\infty(I_j)}
 \|\varphi_x\|_{L^{p'}(I_j)}\label{SC entropy endpoint estimate}
\end{align}
for $p'>1$. Now, repeating the explicit H\"older
estimates leading to \eqref{SC Rint quantitative}, gives
\begin{equation}\label{SC Sint Wminus1p}
 S_h^{\mathrm{int}}\longrightarrow0
 \quad\text{strongly in }W^{-1,p}(\mathcal O),
 \qquad1<p<2.
\end{equation}

\medskip
\noindent\emph{Step 6: application of Murat's lemma.}
We use the following form of Murat's lemma \cite{Murat1978}.

\smallskip
\noindent\textbf{Murat's lemma.}
Let $\mathcal O\subset\mathbb R^2$ be bounded. Suppose that a sequence
$\{L_n\}$ satisfies:
\begin{enumerate}[label=\textnormal{(\roman*)}]
 \item $\{L_n\}$ is bounded in $W^{-1,r}(\mathcal O)$ for some $r>2$;
 \item for some $1<p<2$, $\{L_n\}$ is relatively compact in
 $W^{-1,p}(\mathcal O)$.
\end{enumerate}
Then $\{L_n\}$ is relatively compact in $H^{-1}(\mathcal O)$.
A frequently used sufficient condition for \textnormal{(ii)} is a
decomposition
\[
 L_n=K_n+M_n,
\]
where $\{K_n\}$ is relatively compact in $W^{-1,p}(\mathcal O)$ and
$\{M_n\}$ is bounded in $\mathcal M(\mathcal O)$. Indeed, in two space-time
dimensions,
\begin{equation}\label{SC measure compact embedding}
 \mathcal M(\mathcal O)\hookrightarrow W^{-1,p}(\mathcal O)
 \quad\text{compactly for every }1<p<2.
\end{equation}
After extracting a subsequence, both $K_n$ and $M_n$ therefore converge in
$W^{-1,p}(\mathcal O)$, which proves \textnormal{(ii)}. Finally,
interpolation between $W^{-1,p}$ and $W^{-1,r}$ gives
\[
 \|L_n-L_m\|_{H^{-1}(\mathcal O)}
 \le C
 \|L_n-L_m\|_{W^{-1,p}(\mathcal O)}^\theta
 \|L_n-L_m\|_{W^{-1,r}(\mathcal O)}^{1-\theta},
\]
where $0<\theta<1$ is determined by
\[
 \frac12=\frac{\theta}{p}+\frac{1-\theta}{r}.
\]
Thus every subsequence that is Cauchy in $W^{-1,p}$ is Cauchy in $H^{-1}$,
which proves the lemma.

We now verify these hypotheses. Define
\[
 L_{1,h}:=\partial_tu_h+\partial_xf(u_h),\qquad
 L_{2,h}:=\partial_t\eta(u_h)+\partial_xq(u_h).
\]
For $\varphi\in C_c^\infty(\mathcal O)$,
\begin{align*}
 |\langle L_{1,h},\varphi\rangle|
 &\le
 \|u_h\|_{L^\infty(\mathcal O)}\|\varphi_t\|_{L^1(\mathcal O)}
 +\|f(u_h)\|_{L^\infty(\mathcal O)}
  \|\varphi_x\|_{L^1(\mathcal O)},\\
 |\langle L_{2,h},\varphi\rangle|
 &\le
 \|\eta(u_h)\|_{L^\infty(\mathcal O)}
  \|\varphi_t\|_{L^1(\mathcal O)}
 +\|q(u_h)\|_{L^\infty(\mathcal O)}
  \|\varphi_x\|_{L^1(\mathcal O)}.
\end{align*}
The uniform bound \eqref{SC Linfinity} and boundedness of $\mathcal O$ imply
that, for every finite $r>2$,
\begin{equation}\label{SC production Wminus1r}
 \|L_{1,h}\|_{W^{-1,r}(\mathcal O)}
 +\|L_{2,h}\|_{W^{-1,r}(\mathcal O)}
 \le C_{r}.
\end{equation}
This verifies condition \textnormal{(i)} for both entropy productions.

For $L_{1,h}$, Steps 2 and 3 give terms converging strongly in $H^{-1}$,
while Step 4 gives a term converging strongly in $W^{-1,p}$.
Thus \eqref{SC L1 exact decomposition} shows directly that
$\{L_{1,h}\}$ is relatively compact in $W^{-1,p}(\mathcal O)$.

For $L_{2,h}$, define
\[
 K_h:=S_h^{\mathrm{vol}}+S_h^{\mathrm{damp}}
      +S_h^{\mathrm{int}},\qquad
 M_h:=-\mu_h^{\mathrm{cell}}-\mu_h^{\mathrm{int}}.
\]
Equations \eqref{SC entropy projection vanishes} and
\eqref{SC Sint Wminus1p} show that $\{K_h\}$ is relatively compact in
$W^{-1,p}(\mathcal O)$, while \eqref{SC measure bound} shows that
$\{M_h\}$ is bounded in $\mathcal M(\mathcal O)$. By
\eqref{SC measure compact embedding}, $\{M_h\}$ is relatively compact in
$W^{-1,p}(\mathcal O)$. Hence
\eqref{SC L2 exact decomposition} implies that $\{L_{2,h}\}$ is relatively
compact in $W^{-1,p}(\mathcal O)$, proving condition \textnormal{(ii)}.

Murat's lemma and \eqref{SC production Wminus1r} now imply that both
$\{L_{1,h}\}$ and $\{L_{2,h}\}$ are relatively compact in
$H^{-1}(\mathcal O)$. Since $\mathcal O\Subset I\times(0,T)$ was arbitrary,
the two sequences in \eqref{SC two productions} are relatively compact in
$H^{-1}_{\mathrm{loc}}(I\times(0,T))$.
\end{proof}

We can now state the strong-convergence result.

\begin{theorem}[Strong convergence for a uniformly convex flux]
\label{theorem: strong convergence}
Let $k\ge1$ and let $u_h$ be the solution of \eqref{scheme D} with $\epsilon_1>0$ subject to periodic or compactly supported boundary conditions. Assume \eqref{SC choice r}--\eqref{SC cubic dissipation assumption}
hold, and that the initial projections converge strongly to
$\mathcal U_0$ in $L^2(I)$, i.e.,
\[
    u_h(\cdot,0)\longrightarrow\mathcal U_0
    \quad\text{strongly in }L^2(I).
\]
Then every sequence $h\to0$ contains a subsequence,
still denoted by $u_h$, and a function
$\mathcal U\in L^\infty(I\times(0,T))$ such that
\begin{equation}\label{SC strong conclusion}
    u_h\longrightarrow \mathcal U
    \quad\text{strongly in }L^p_{\mathrm{loc}}(I\times(0,T))
    \quad\text{for every }1\le p<\infty.
\end{equation}
The limit is a distributional solution of \eqref{HCL 1D scalar} and satisfies
the limiting square-entropy inequality. Since $f(\cdot)$ is strictly convex, the square-entropy inequality is enough to guarantee uniqueness of solution \cite{GR}, so that the entire family $u_h$ converges to the unique entropy solution of \eqref{HCL 1D scalar}.
\end{theorem}

\begin{proof}
By \eqref{SC Linfinity}, we can extract a subsequence of $u_h$ and the associated family of Young measures
$\nu_{x,t}$ supported in $[-M_T,M_T]$ such that $g(u_h)\rightharpoonup\langle\nu_{x,t},g\rangle$ weakly-* in \(L^\infty\), for any continuous function $g$ \cite{DiPerna1983,Dafermos}. By Lemma \ref{lemma: SC Hminus1}, we can apply the
div--curl lemma \cite{Tartar1979} to the two
entropy pairs
\[
    \eta_1(s)=s,\qquad q_1(s)=f(s),
\]
\[
    \eta_2(s)=\frac{s^2}{2},\qquad
    q_2(s)=q(s),\qquad q'(s)=s f'(s),
\]
and obtain
\begin{equation*}
    \langle \nu, \eta_1 \rangle \langle \nu, q_2 \rangle - \langle \nu, \eta_2 \rangle \langle \nu, q_1 \rangle = \langle \nu, \eta_1 q_2 - \eta_2 q_1 \rangle.
\end{equation*}
Equivalently, this can be written in its symmetrized form as
\begin{equation}\label{SC commutation}
 \iint_{\mathbb R^2}\mathcal C(a,b)\,\ud\nu_{x,t}(a)\ud\nu_{x,t}(b)=0,
\end{equation}
where
\begin{align}
    \mathcal C(a,b)
 =& (a-b)(q(a)-q(b))
 -\frac{a^2-b^2}{2}(f(a)-f(b)) \nonumber\\
 =& \frac12\int_b^a\int_b^a
 (s-r)(f'(s)-f'(r))\,\ud r\ud s. \label{SC covariance identity}
\end{align}
By \eqref{SC uniform convexity},
\[
 (s-r)(f'(s)-f'(r))\ge c_f(s-r)^2.
\]
Hence,
\begin{equation}\label{SC commutator lower bound}
    \mathcal C(a,b)\ge\frac{c_f}{12}|a-b|^4.
\end{equation}
Equations \eqref{SC commutation} and \eqref{SC commutator lower bound} imply
\[
 \iint_{\mathbb R^2}|a-b|^4\,
 \ud\nu_{x,t}(a)\ud\nu_{x,t}(b)=0.
\]
Thus $\nu_{x,t}$ is a Dirac mass for almost every $(x,t)$. Hence
$\nu_{x,t}=\delta_{\mathcal U(x,t)}$ and $u_h\to\mathcal U$
locally in measure. Since \(\{ |u_h|^p\}\) is uniformly integrable on every compact subset,
Vitali's theorem gives
\[
u_h\to\mathcal U
\quad\text{strongly in }L^p_{\mathrm{loc}},
\qquad 1\le p<\infty.
\]

In Lemma \ref{lemma: SC Hminus1} we showed that
$(u_h)_t+f(u_h)_x\to0$ in
$H^{-1}_{\mathrm{loc}}(I\times(0,T))$. Together with the strong convergence
of $u_h$, this implies
$\mathcal U_t+f(\mathcal U)_x=0$ in distributions, and the
square-entropy inequality can be passed to the limit as well. It remains to
identify the initial trace. Let
$\varphi\in C_c^\infty(I\times[0,T))$. The estimates in Steps~2--4 of
Lemma~\ref{lemma: SC Hminus1} remain valid for such a test function, since
they only use spatial derivatives of $\varphi$ and estimates integrated over
$0<t<T$. Hence, writing
$R_h:=(u_h)_t+f(u_h)_x$, we have
\[
 \left|\langle R_h,\varphi\rangle\right|
 \le C_\varphi\bigl(h^{1/2}+h^{1/3}\bigr)\longrightarrow0.
\]
Because $u_h$ is differentiable in time and $\varphi(\cdot,T)=0$,
integration by parts in time gives
\[
 \langle R_h,\varphi\rangle
 =-\int_0^T\!\!\int_I
 \bigl(u_h\varphi_t+f(u_h)\varphi_x\bigr)\,\ud x\ud t
 -\int_Iu_h(x,0)\varphi(x,0)\,\ud x.
\]
Fix $\delta>0$. On $I\times(\delta,T)$, the strong convergence already
proved above and continuity of $f$ on $[-M_T,M_T]$ allow passage to the
limit in the space--time integral. On $I\times(0,\delta)$, the uniform
$L^\infty$ bound on $u_h$ and $f(u_h)$ makes the absolute value of that
integral at most $C_\varphi\delta$, uniformly in $h$. Thus, first letting
$h\to0$ and then $\delta\downarrow0$, and using
$u_h(\cdot,0)\to\mathcal U_0$ strongly in $L^2(I)$, we obtain
\[
 \int_0^T\!\!\int_I
 \bigl(\mathcal U\varphi_t+f(\mathcal U)\varphi_x\bigr)\,\ud x\ud t
 +\int_I\mathcal U_0(x)\varphi(x,0)\,\ud x=0.
\]
Therefore $\mathcal U$ has initial value $\mathcal U_0$ in the
distributional trace sense, and the limit is the unique entropy solution to
\eqref{HCL 1D scalar}.
\end{proof}

\section{Numerical results}\label{section: numerical results}
Recall the definitions \eqref{def C_j} and \eqref{def D_j} of the damping coefficients in \eqref{scheme C} and \eqref{scheme D}:
\begin{align*}
    C_j[u;U,r]:=& q_j[u;U] \bigg\{ \frac{C_k C(k,r) h_j^{r-\frac{3}{2}}}{2} \|\partial_x^{r}f(u)\|_{L^2(I_j)} +\frac{\beta_k}{2h_j}\left(|H^+_{j-\frac{1}{2}}|+|H^-_{j+\frac{1}{2}}|\right) \bigg\}, \\
    D_j[u;r,\epsilon_1,\epsilon_2] :=& \epsilon_1 h_j^{r-\frac{3}{2}} \|\partial_x^{r}f(u)\|_{L^2(I_j)} + \epsilon_2 \frac{|[\![u]\!]_{j-\frac{1}{2}}|+|[\![u]\!]_{j+\frac{1}{2}}|}{h_j},
\end{align*}
where $k+1 \le r \le k+2$. The first term on the right-hand side of either $C_j$ and $D_j$ hinders the efficiency of the numerical scheme, but is nevertheless indispensable in the proof of local entropy inequalities and the strong convergence. In Section \ref{section: strong convergence}, we proved strong convergence only for strictly convex conservation laws; but for general non-convex conservation laws, the convergence of the scheme usually depends on the choice of $\epsilon_1,\epsilon_2$. This section discusses the effects of $\epsilon_1$ on numerical results and the question of whether it can be set to zero.

First of all, it can be seen from Tables \ref{tab:Ck}, \ref{tab:Ckr} and \ref{tab:improved beta} that the constant $C_k C(k,r)$ is much smaller in magnitude compared to $\beta_k=2 \mathcal A_k$ as $k$ increases. Indeed, the ratio $C_k C(k,r)/\beta_k$ is recorded in the following table for the first few $k$:
\FloatBarrier
\begin{table}[htbp]
\centering
\begin{tabular}{c|c|c|c}
\hline
$(k,r)$ & $C_k C(k,r)$ & $\beta_k=2\mathcal{A}_k$ 
& $C_k C(k,r)/\beta_k$ \\
\hline
(1,2) & $0.2387324145$   & $4.3703703704$  & $5.46252135\times10^{-2}$ \\
(1,3) & $0.06704424360$  & $4.3703703704$  & $1.53406320\times10^{-2}$ \\
(2,3) & $0.03495228188$  & $12.1950289930$ & $2.86610896\times10^{-3}$ \\
(2,4) & $0.008690199690$ & $12.1950289930$ & $7.12601806\times10^{-4}$ \\
(3,4) & $0.003382471073$  & $28.6965574262$ & $1.17870274\times10^{-4}$ \\
(3,5) & $0.0007582675327$ & $28.6965574262$ & $2.64236410\times10^{-5}$ \\
(4,5) & $0.0002407636609$ & $52.7920626087$ & $4.56060341\times10^{-6}$ \\
(4,6) & $0.00004943108133$& $52.7920626087$ & $9.36335481\times10^{-7}$ \\
\hline
\end{tabular}
\caption{The ratio $C_k C(k,r)/\beta_k$ for $k+1\le r\le k+2$.}
\label{tab:ratio-CkCkr-beta}
\end{table}
Therefore, assuming that our constant $\beta_k$ is close to its optimal value, and taking into account that $|H^\pm_{j+\frac12}|\approx |f'(u)|\cdot|[\![u]\!]_{j+\frac12}|$ we deduce that $\epsilon_1$ and $\epsilon_2$ should also satisfy
\begin{equation}\label{quotient rule}
    \frac{\epsilon_1}{\epsilon_2}\approx \frac{C_k C(k,r)}{\beta_k \max|f'(u)|},
\end{equation}
or at least both sides of \eqref{quotient rule} should be asymptotically in the same order of magnitude. This suggests that, for small values of $\epsilon_2$, we may take $\epsilon_1$ to be equally small during computations if the numerical solution itself is observed to be uniformly bounded. This is observed in numerical tests for non-convex conservation laws shown below. Moreover, it has also been observed that if $\epsilon_2$ is sufficiently large, the scheme would always converge for non-convex tests with $\epsilon_1=0$. The really interesting case is when we use a moderate value of $\epsilon_2$, say around 10, then the smallest value of $\epsilon_1$ for the scheme to converge satisfies \eqref{quotient rule}. Indeed, this is demonstrated by the following example.

\begin{expl}
Consider the Buckley-Leverett equation with the non-convex flux
\begin{equation*}
    f(u)=\frac{4 u^2}{4 u^2 + (1-u)^2}.
\end{equation*}
The computational domain is $I=[-1,1]$ and the initial condition is
\begin{equation*}
    u_0(x)=\begin{cases}
        1,& -\frac{1}{2}\le x \le 0,\\
        0,& \text{otherwise.}
    \end{cases}
\end{equation*}
We evolve the solution up to $T=0.3$. The reference solution is shown in Figure \ref{fig:ex1.0}, which was computed by the first-order monotone scheme with Godunov flux and a uniform mesh of size $N=20000$.
\begin{figure}[htbp]
     \centering
         \includegraphics[width=0.47\textwidth]{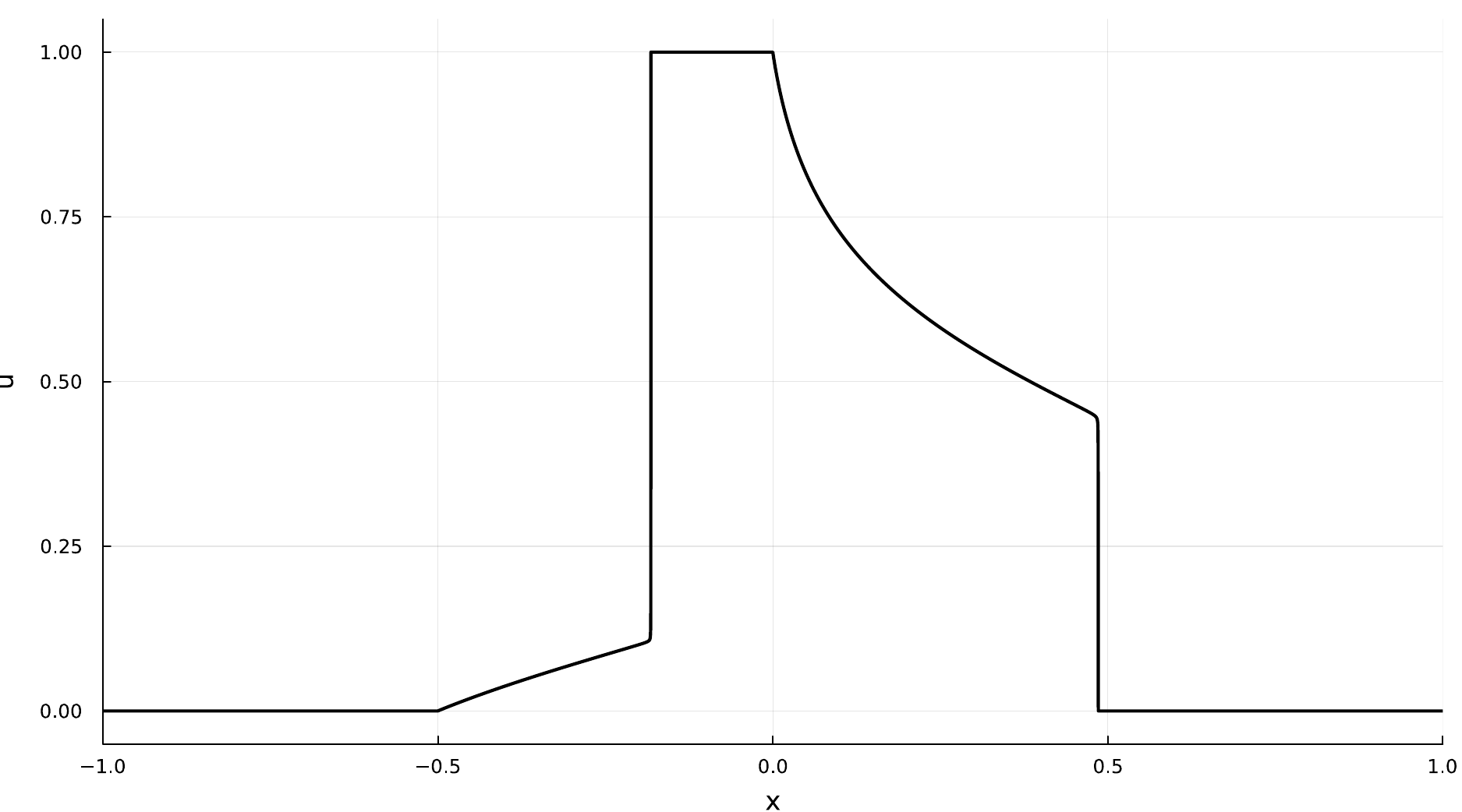}
     \caption{}
     \label{fig:ex1.0}
\end{figure}

Next, we compute the solution using \eqref{scheme D} with a uniform mesh of size $N=275$. The numerical solutions with $\epsilon_1=0$, $\epsilon_2=10$ and $k=1,2,3,4$ are shown in Figure \ref{fig:OFDG ex1}. The global Lax–Friedrichs flux with $\max|f'(u)|\approx 2.34$ and classical RK4 were used in all tests. Note that none of them converges to the correct solution.
\begin{figure}[htbp]
     \centering
     \begin{subfigure}[b]{0.2\textwidth}
         \centering
         \includegraphics[width=\textwidth]{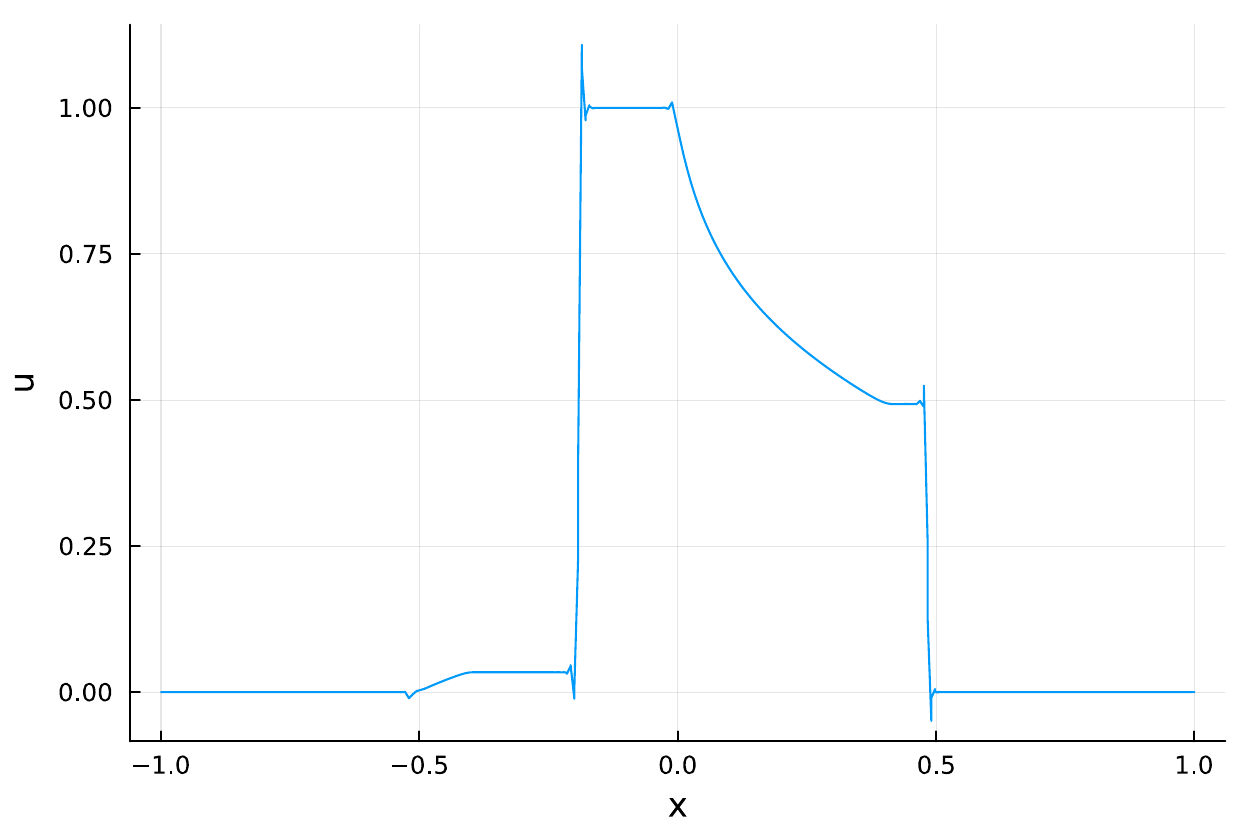}
         \caption*{$\mathbb{P}^1$}
     \end{subfigure}
     \hfill
     \begin{subfigure}[b]{0.2\textwidth}
         \centering
         \includegraphics[width=\textwidth]{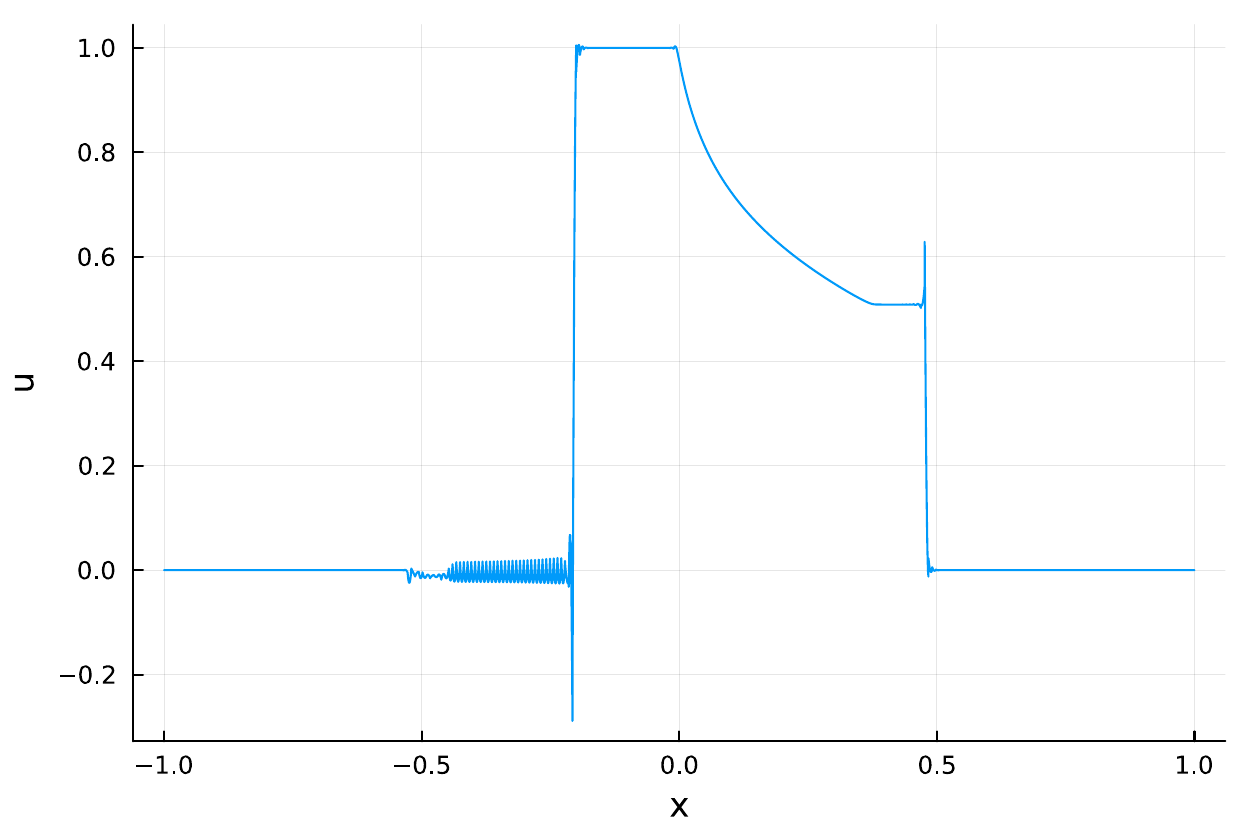}
         \caption*{$\mathbb{P}^2$}
     \end{subfigure}
     \hfill
     \begin{subfigure}[b]{0.2\textwidth}
         \centering
         \includegraphics[width=\textwidth]{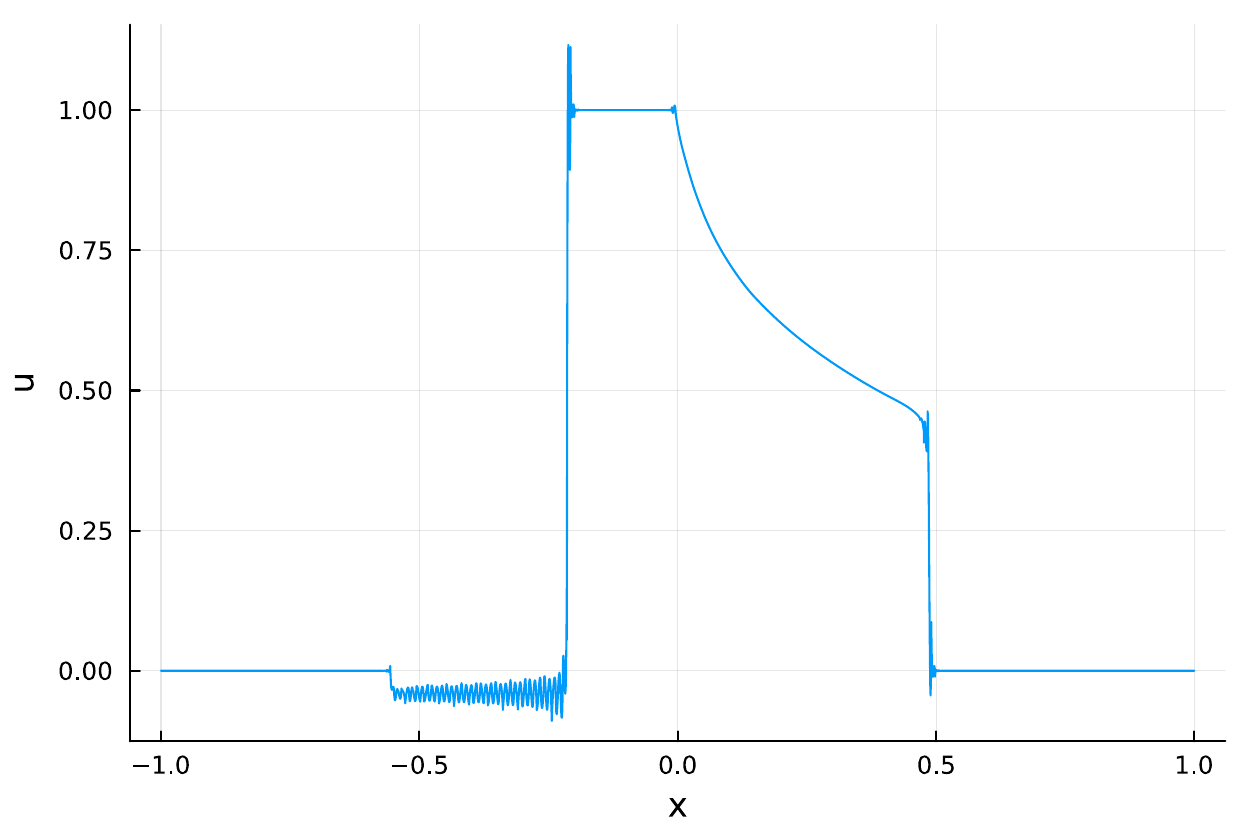}
         \caption*{$\mathbb{P}^3$}
     \end{subfigure}
     \hfill
     \begin{subfigure}[b]{0.2\textwidth}
         \centering
         \includegraphics[width=\textwidth]{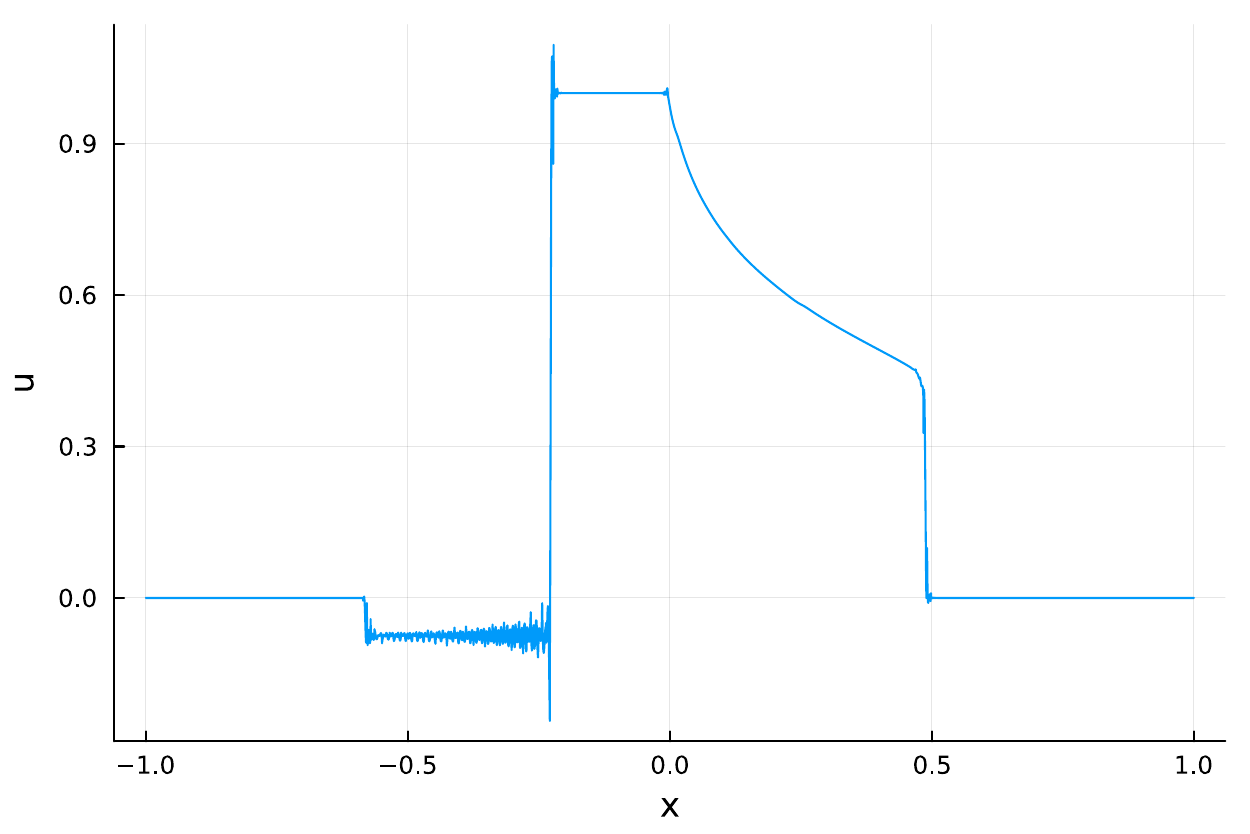}
         \caption*{$\mathbb{P}^4$}
     \end{subfigure}
     \caption{\small Solutions of \eqref{scheme D} with $\epsilon_1=0$, $\epsilon_2=10$ and $k=1,2,3,4$. We use a uniform mesh with $N=275$. None of the above solutions converges to the correct solution.}
     \label{fig:OFDG ex1}
\end{figure}
\FloatBarrier
Now, keep $\epsilon_2=10$ and start slightly increasing $\epsilon_1$. We observe that when $\epsilon_1$ reaches a critical value, the scheme would converge correctly. These numerical solutions for \eqref{scheme D} with $r=k+1$ are shown in Figure \ref{fig:OPDG k+1 ex1}, and those with $r=k+2$ are shown in Figure \ref{fig:OPDG k+2 ex1}.
\begin{figure}[htbp]
     \centering
     \begin{subfigure}[b]{0.2\textwidth}
         \centering
         \includegraphics[width=\textwidth]{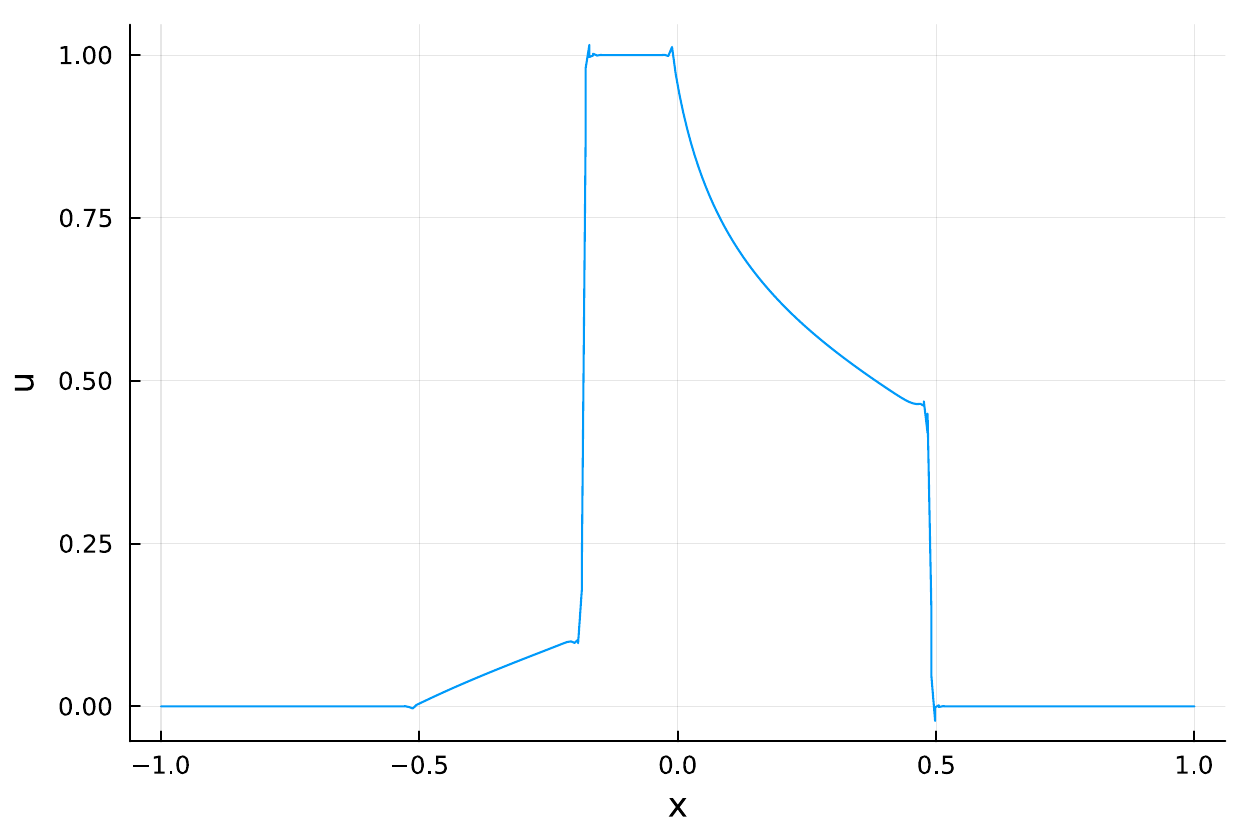}
         \caption*{$\mathbb{P}^1$ with $\epsilon_1=2$, $\epsilon_2=10$.}
     \end{subfigure}
     \hfill
     \begin{subfigure}[b]{0.2\textwidth}
         \centering
         \includegraphics[width=\textwidth]{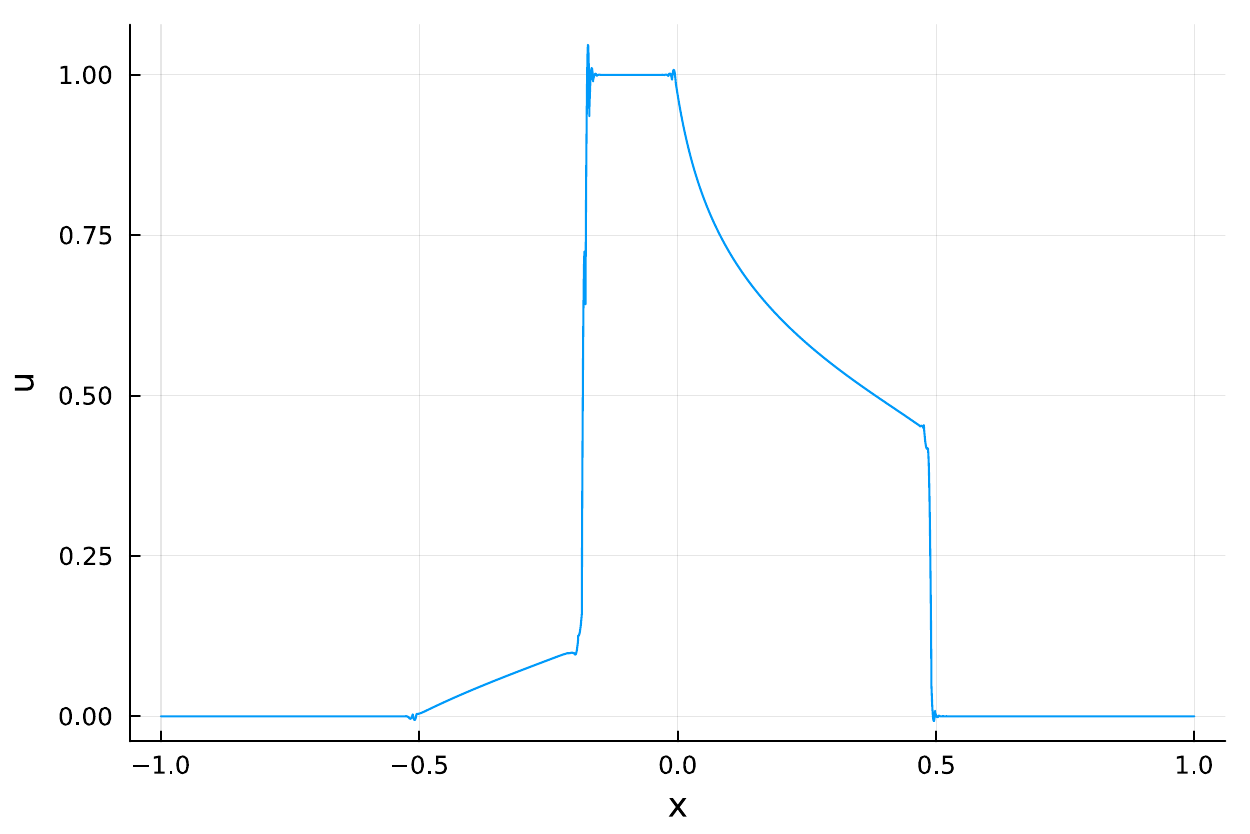}
         \caption*{$\mathbb{P}^2$ with $\epsilon_1=0.1$, $\epsilon_2=10$.}
     \end{subfigure}
     \hfill
     \begin{subfigure}[b]{0.2\textwidth}
         \centering
         \includegraphics[width=\textwidth]{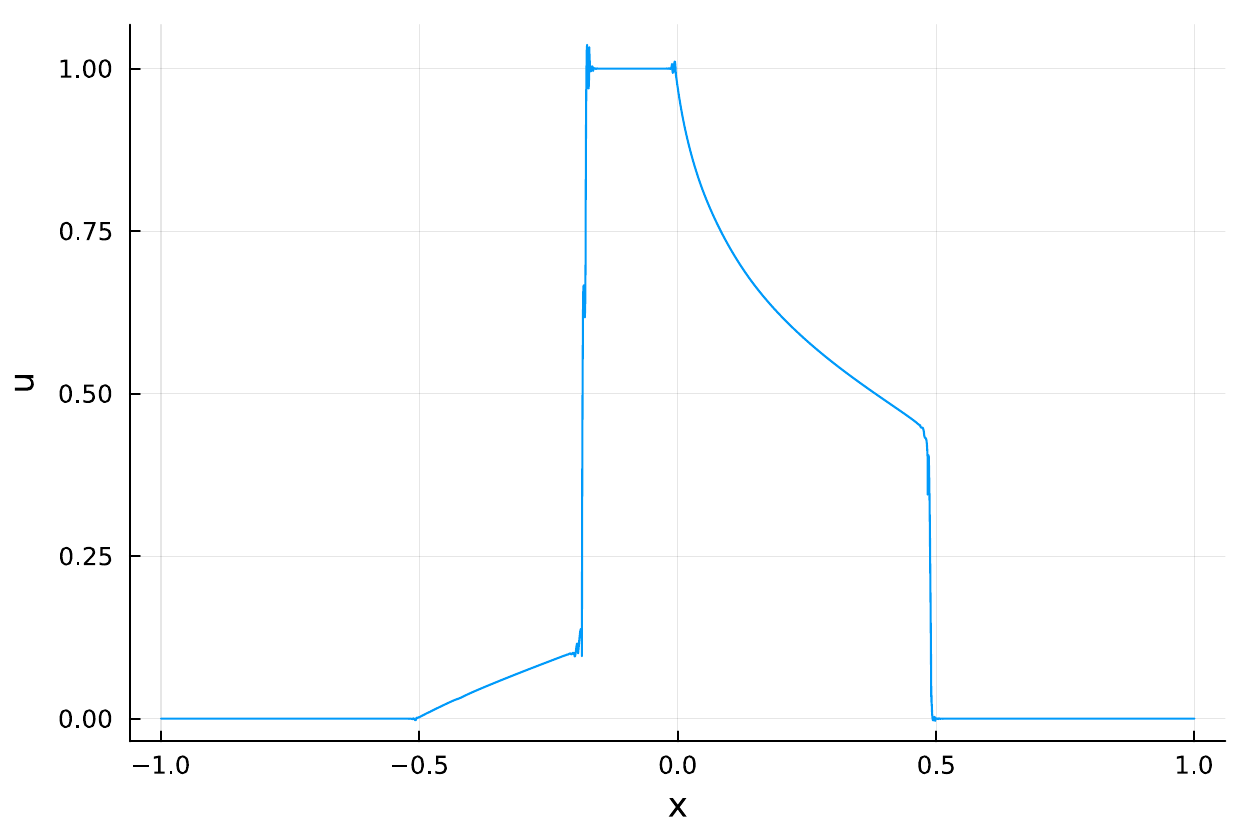}
         \caption*{$\mathbb{P}^3$ with $\epsilon_1=0.001$, $\epsilon_2=10$.}
     \end{subfigure}
     \hfill
     \begin{subfigure}[b]{0.2\textwidth}
         \centering
         \includegraphics[width=\textwidth]{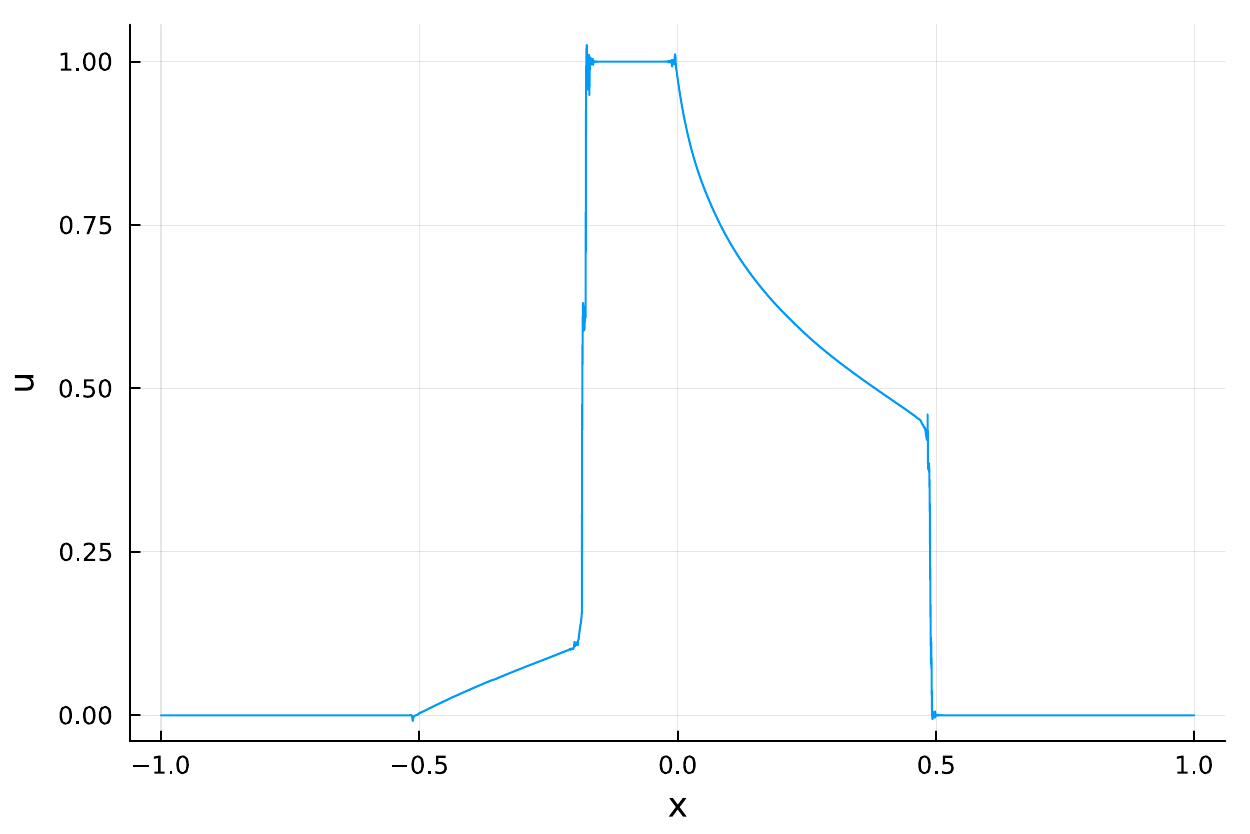}
         \caption*{$\mathbb{P}^4$ with $\epsilon_1=2.0\times 10^{-5}$, $\epsilon_2=10$.}
     \end{subfigure}
     \caption{\small Solutions of \eqref{scheme D} with $k=1,2,3,4$ and $r=k+1$. We still use a uniform mesh with $N=275$. Now, all of the above solutions agree with the correct solution.}
     \label{fig:OPDG k+1 ex1}
\end{figure}

\begin{figure}[htbp]
     \centering
     \begin{subfigure}[b]{0.2\textwidth}
         \centering
         \includegraphics[width=\textwidth]{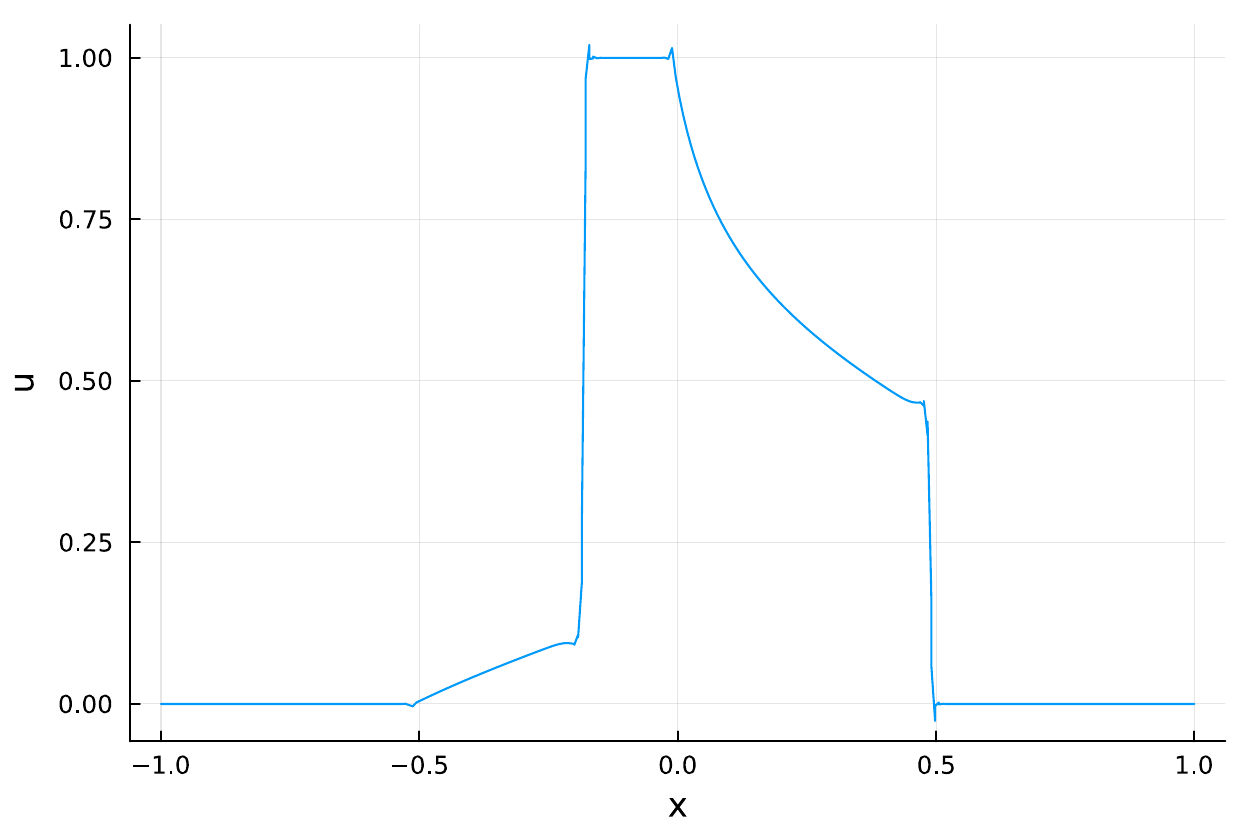}
         \caption*{$\mathbb{P}^1$ with $\epsilon_1=1$, $\epsilon_2=10$.}
     \end{subfigure}
     \hfill
     \begin{subfigure}[b]{0.2\textwidth}
         \centering
         \includegraphics[width=\textwidth]{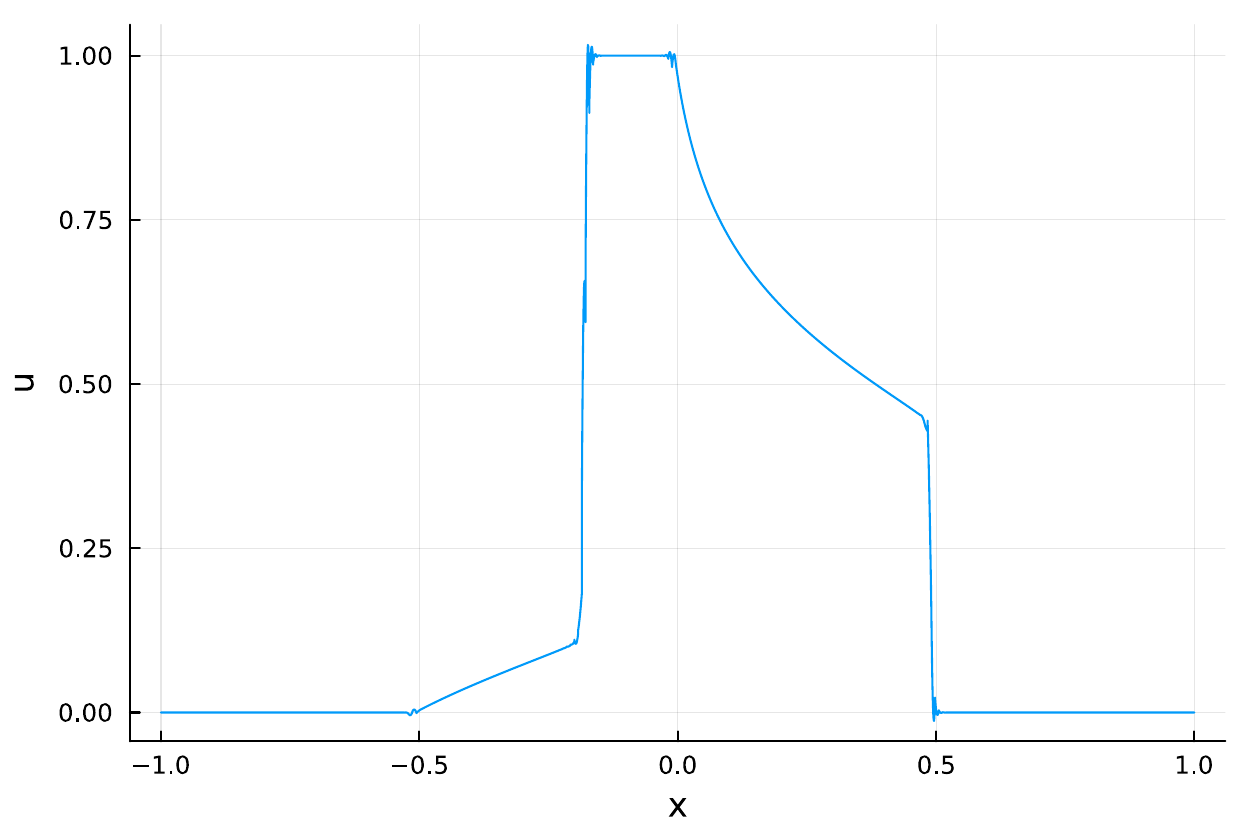}
         \caption*{$\mathbb{P}^2$ with $\epsilon_1=0.075$, $\epsilon_2=10$.}
     \end{subfigure}
     \hfill
     \begin{subfigure}[b]{0.2\textwidth}
         \centering
         \includegraphics[width=\textwidth]{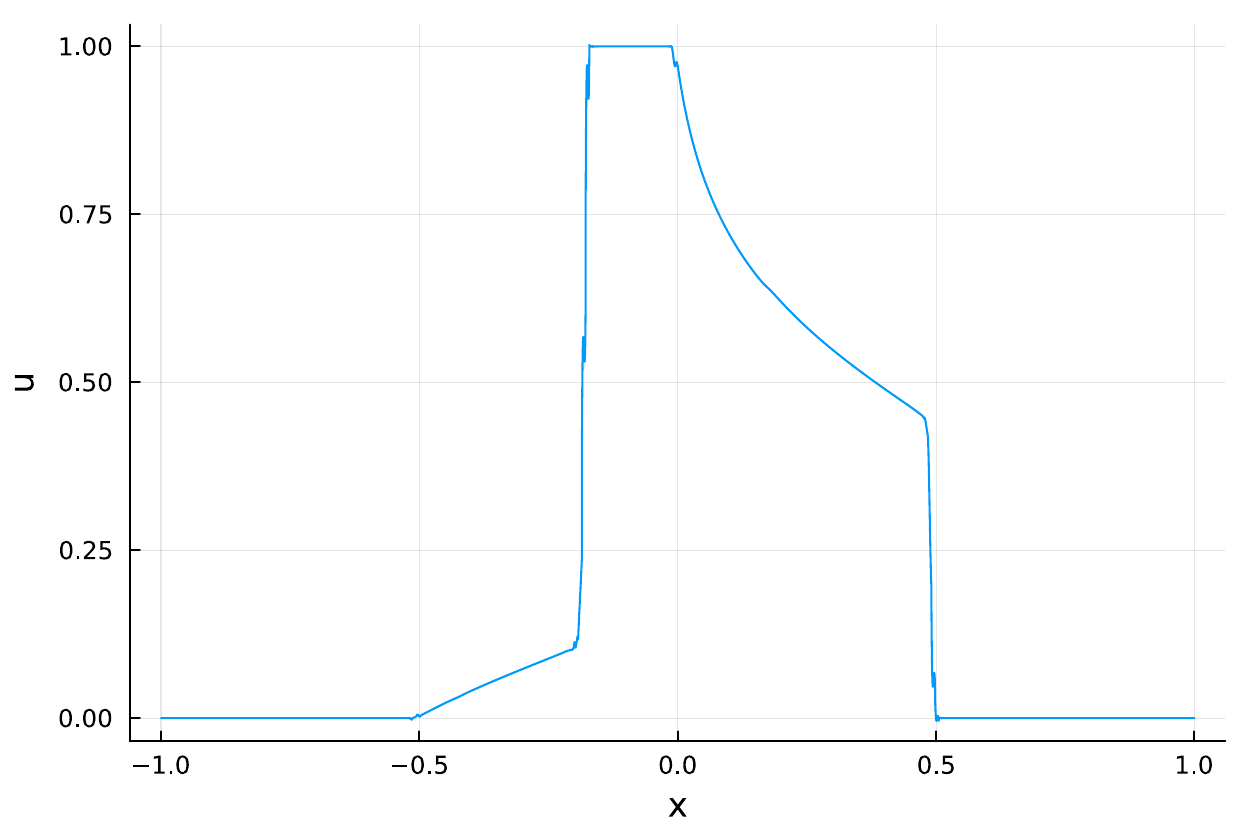}
         \caption*{$\mathbb{P}^3$ with $\epsilon_1=0.005$, $\epsilon_2=10$.}
     \end{subfigure}
     \hfill
     \begin{subfigure}[b]{0.2\textwidth}
         \centering
         \includegraphics[width=\textwidth]{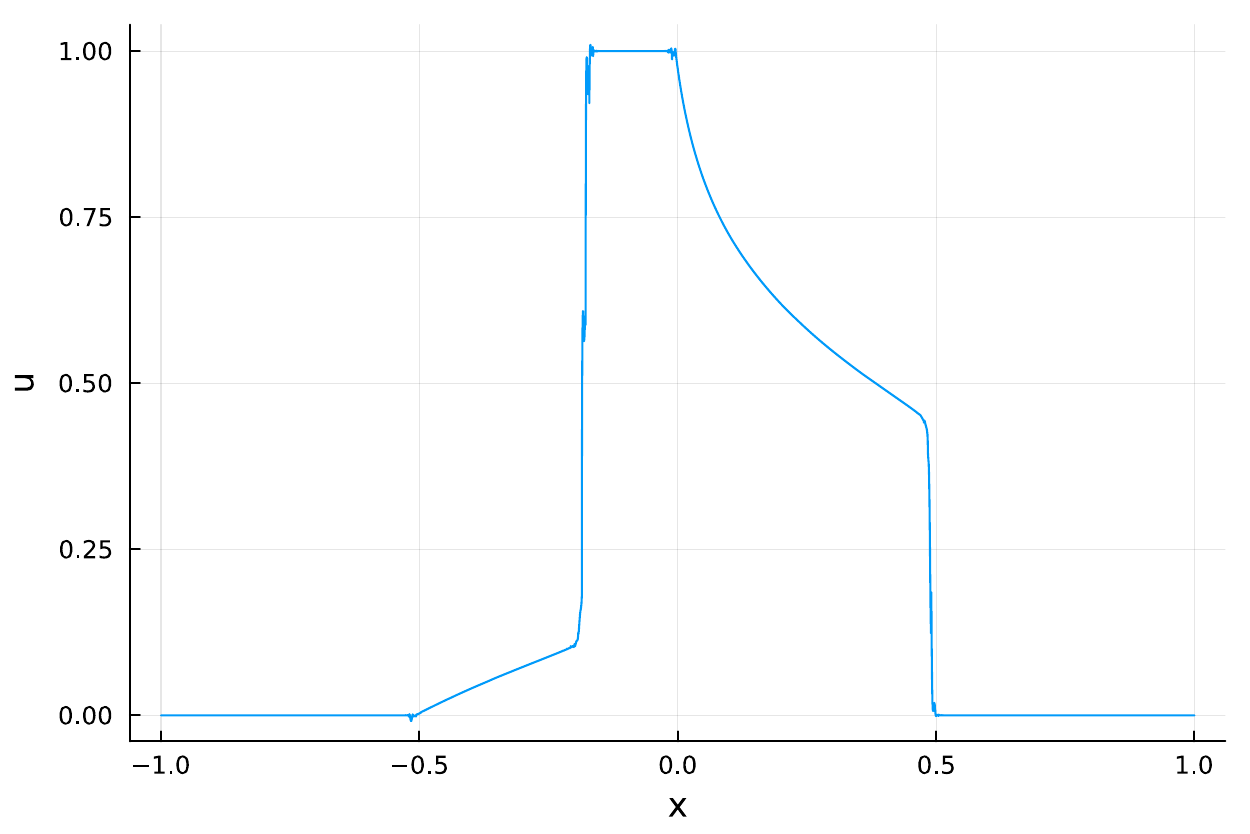}
         \caption*{$\mathbb{P}^4$ with $\epsilon_1=5.0\times 10^{-6}$, $\epsilon_2=10$.}
     \end{subfigure}
     \caption{\small Solutions of \eqref{scheme D} with $k=1,2,3,4$ and $r=k+2$. We still use a uniform mesh with $N=275$. All of the above solutions agree with the correct solution.}
     \label{fig:OPDG k+2 ex1}
\end{figure}
\FloatBarrier
If we compute the ratio of the observed $\epsilon_1$ and $\epsilon_2$ from Figures \ref{fig:OPDG k+1 ex1} and \ref{fig:OPDG k+2 ex1}, and then compare it with the value on the right-hand side of \eqref{quotient rule}, we would find that they indeed share the same order of magnitude asymptotically, as is shown in Table \ref{tab:ratio example 1}. 
\begin{table}[htbp]
\begin{tabular}{c|c|c}
\hline
$(k,r)$ & $\frac{\epsilon_1}{\epsilon_2}$ & $\frac{C_k C(k,r)}{\beta_k \max|f'(u)|}$ \\
\hline
(1,2) & $2.0\times 10^{-1}$   & $2.3\times10^{-2}$ \\
(1,3) & $1.0\times 10^{-1}$  & $6.6\times10^{-3}$ \\
(2,3) & $1.0\times 10^{-2}$  & $1.2\times10^{-3}$ \\
(2,4) & $7.5\times 10^{-3}$ & $3.0\times10^{-4}$ \\
(3,4) & $1.0\times 10^{-4}$  & $5.0\times10^{-5}$ \\
(3,5) & $5.0\times 10^{-4}$ & $1.1\times10^{-5}$ \\
(4,5) & $2.0\times 10^{-6}$ & $1.9\times10^{-6}$ \\
(4,6) & $5.0\times 10^{-7}$ & $4.0\times10^{-7}$ \\
\hline
\end{tabular}
\caption{}
\label{tab:ratio example 1}
\end{table}
With this smallest $\epsilon_1$, we can reduce $\epsilon_2$ significantly such that the scheme still converges, as shown in Figure \ref{fig:OPDG k+2 small constants ex1} for \eqref{scheme D} with $r=k+2$. In particular, when $k=4$ we may set $\epsilon_1=5.0\times 10^{-6}$ and $\epsilon_2=1.0\times 10^{-6}$.
\begin{figure}[htbp]
     \centering
     \begin{subfigure}[b]{0.2\textwidth}
         \centering
         \includegraphics[width=\textwidth]{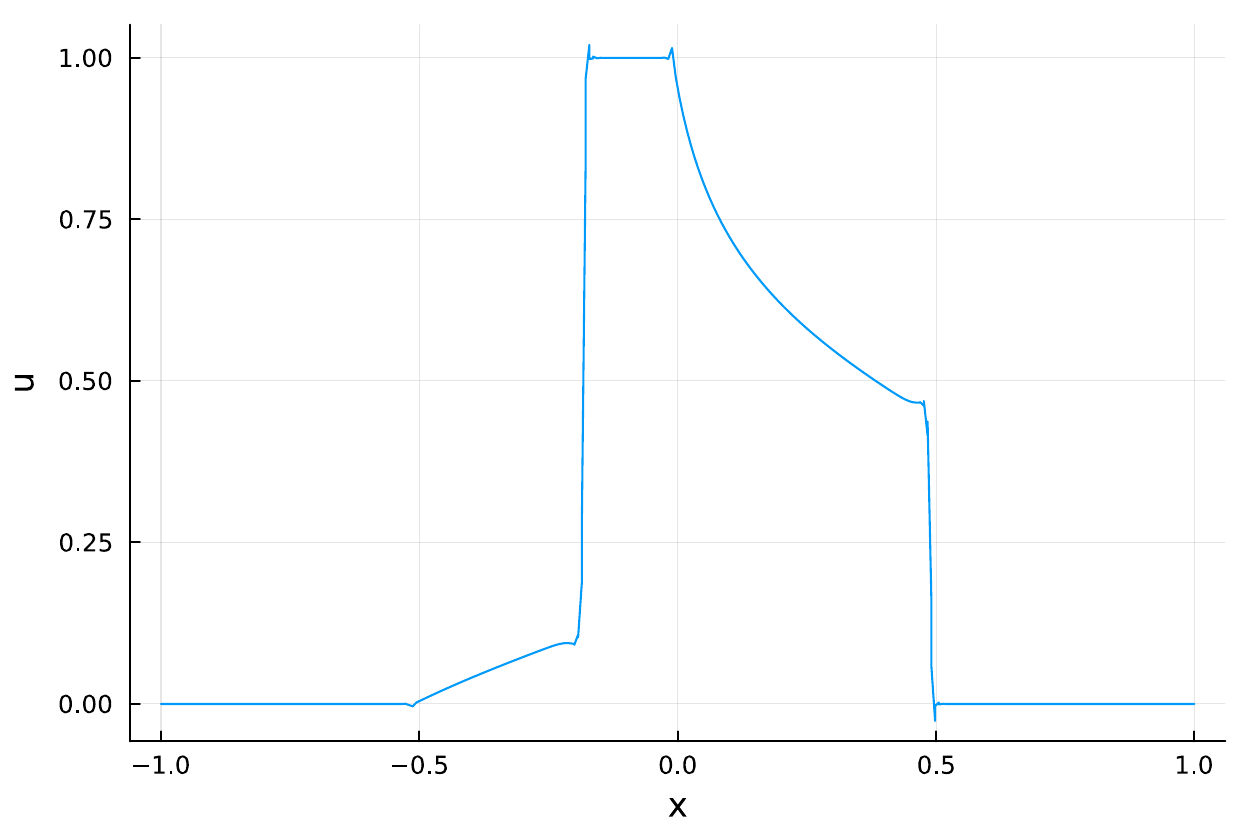}
         \caption*{$\mathbb{P}^1$ with $\epsilon_1=1$, $\epsilon_2=10$.}
     \end{subfigure}
     \hfill
     \begin{subfigure}[b]{0.2\textwidth}
         \centering
         \includegraphics[width=\textwidth]{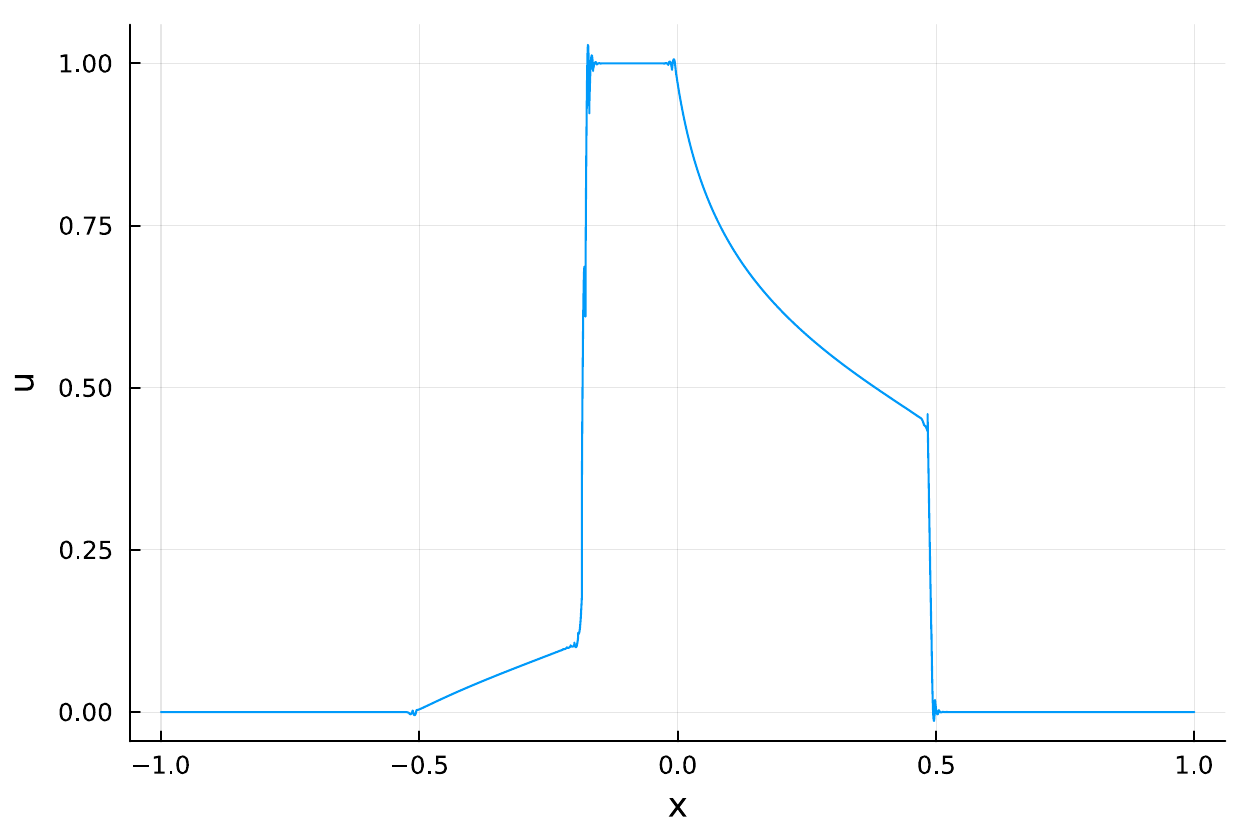}
         \caption*{$\mathbb{P}^2$ with $\epsilon_1=0.075$, $\epsilon_2=0.01$.}
     \end{subfigure}
     \hfill
     \begin{subfigure}[b]{0.2\textwidth}
         \centering
         \includegraphics[width=\textwidth]{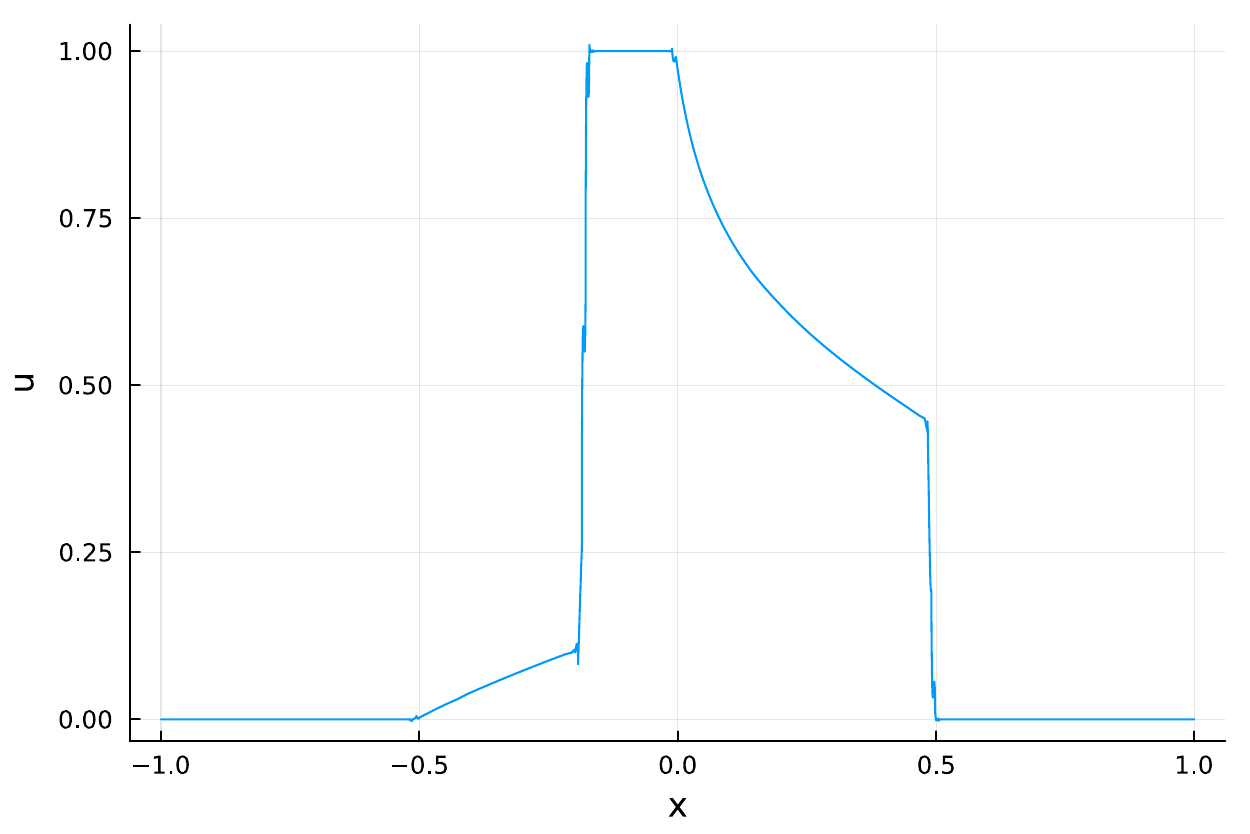}
         \caption*{$\mathbb{P}^3$ with $\epsilon_1=0.005$, $\epsilon_2=0.001$.}
     \end{subfigure}
     \hfill
     \begin{subfigure}[b]{0.2\textwidth}
         \centering
         \includegraphics[width=\textwidth]{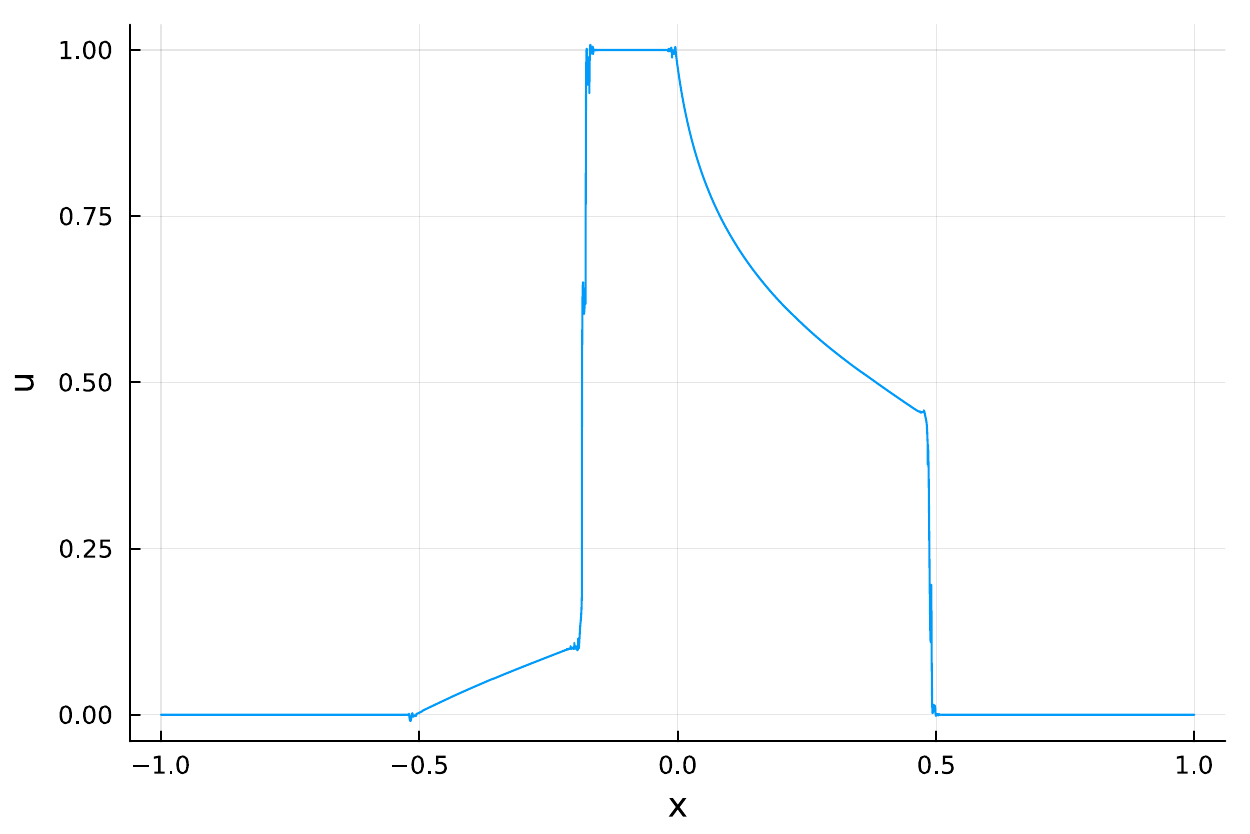}
         \caption*{$\mathbb{P}^4$ with $\epsilon_1=5\times 10^{-6}$, $\epsilon_2=1\times 10^{-6}$.}
     \end{subfigure}
     \caption{\small Solutions of \eqref{scheme D} with $k=1,2,3,4$ and $r=k+2$. We still use a uniform mesh with $N=275$. All of the above solutions agree with the correct solution.}
     \label{fig:OPDG k+2 small constants ex1}
\end{figure}
\end{expl}

\FloatBarrier
\begin{expl}
Consider the non-convex flux
\begin{equation*}
    f(u)=\sin(u).
\end{equation*}
The computational domain is $I=[-5,5]$ and the initial condition is
\begin{equation*}
    u_0(x)=\begin{cases}
        \frac{\pi}{64},& x < 0,\\
        \frac{255 \pi}{64},& x\ge 0.
    \end{cases}
\end{equation*}
We evolve the solution up to $T=4$. The reference solution is shown in Figure \ref{fig:ex sine ref}, which was computed by the first-order monotone scheme with Godunov flux and a uniform mesh of size $N=20000$. This example is borrowed from \cite{Qiu}, where it is used to demonstrate the failure of convergence for the classical DG scheme.
\begin{figure}[htbp]
     \centering
         \includegraphics[width=0.47\textwidth]{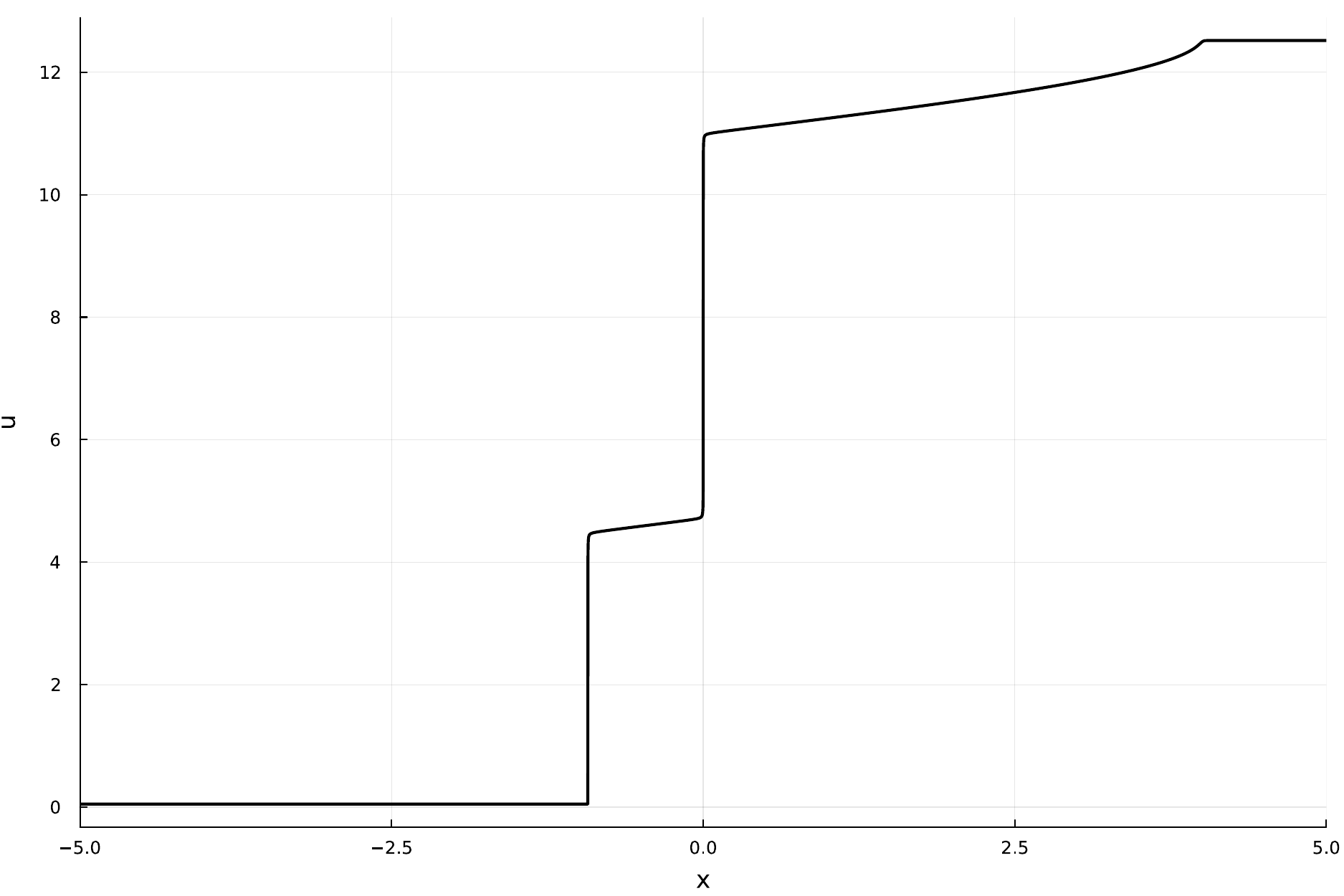}
     \caption{}
     \label{fig:ex sine ref}
\end{figure}

Next, we compute the solution using \eqref{scheme D} with a uniform mesh of size $N=275$. The numerical solutions with $\epsilon_1=0$, $\epsilon_2=1$ and $k=1,2,3,4$ are shown in Figure \ref{fig:OFDG ex sine}. The global Lax–Friedrichs flux with $\max|f'(u)|=1$ and classical RK4 were used in all tests. None of them converges to the correct solution.
\begin{figure}[htbp]
     \centering
     \begin{subfigure}[b]{0.2\textwidth}
         \centering
         \includegraphics[width=\textwidth]{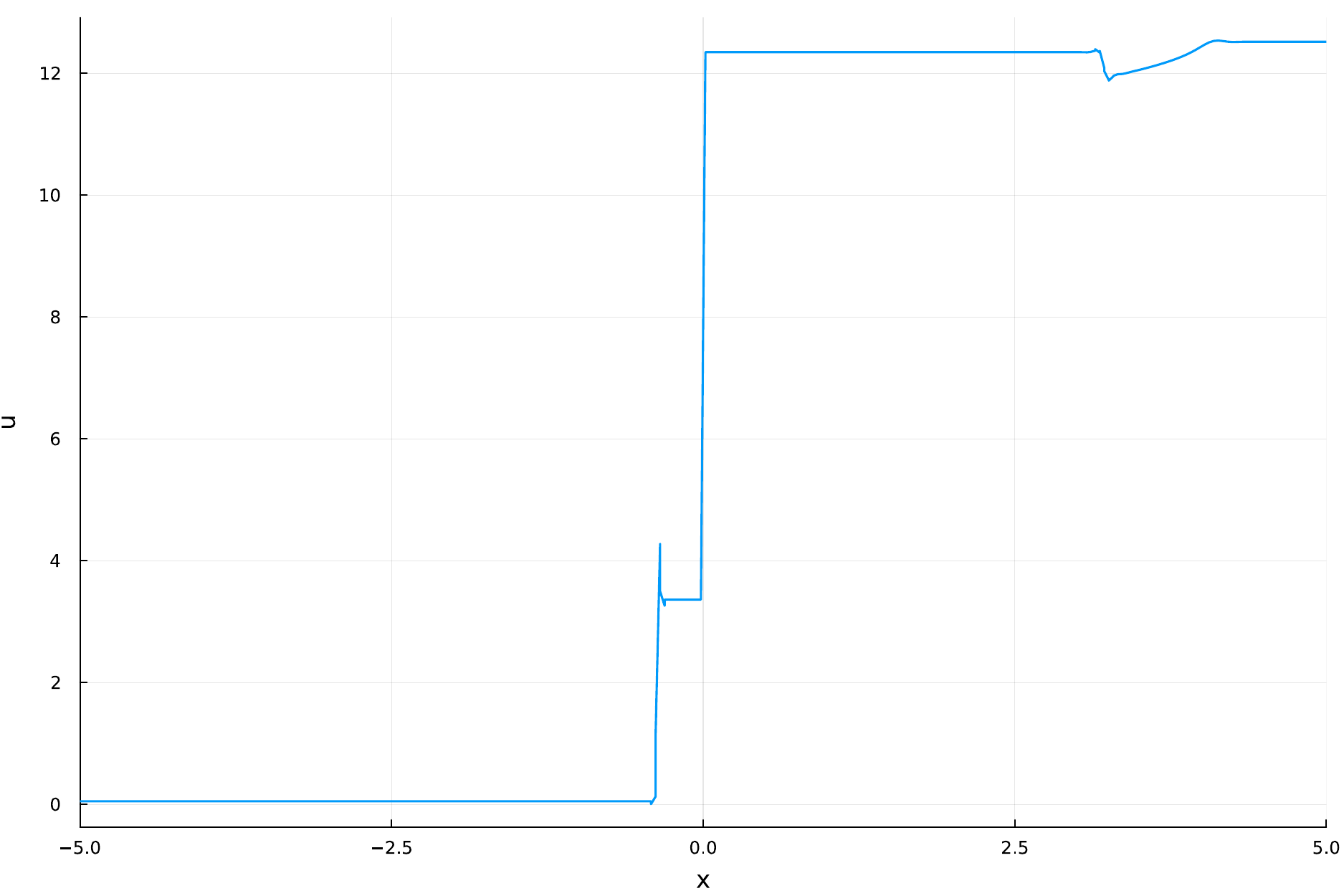}
         \caption*{$\mathbb{P}^1$}
     \end{subfigure}
     \hfill
     \begin{subfigure}[b]{0.2\textwidth}
         \centering
         \includegraphics[width=\textwidth]{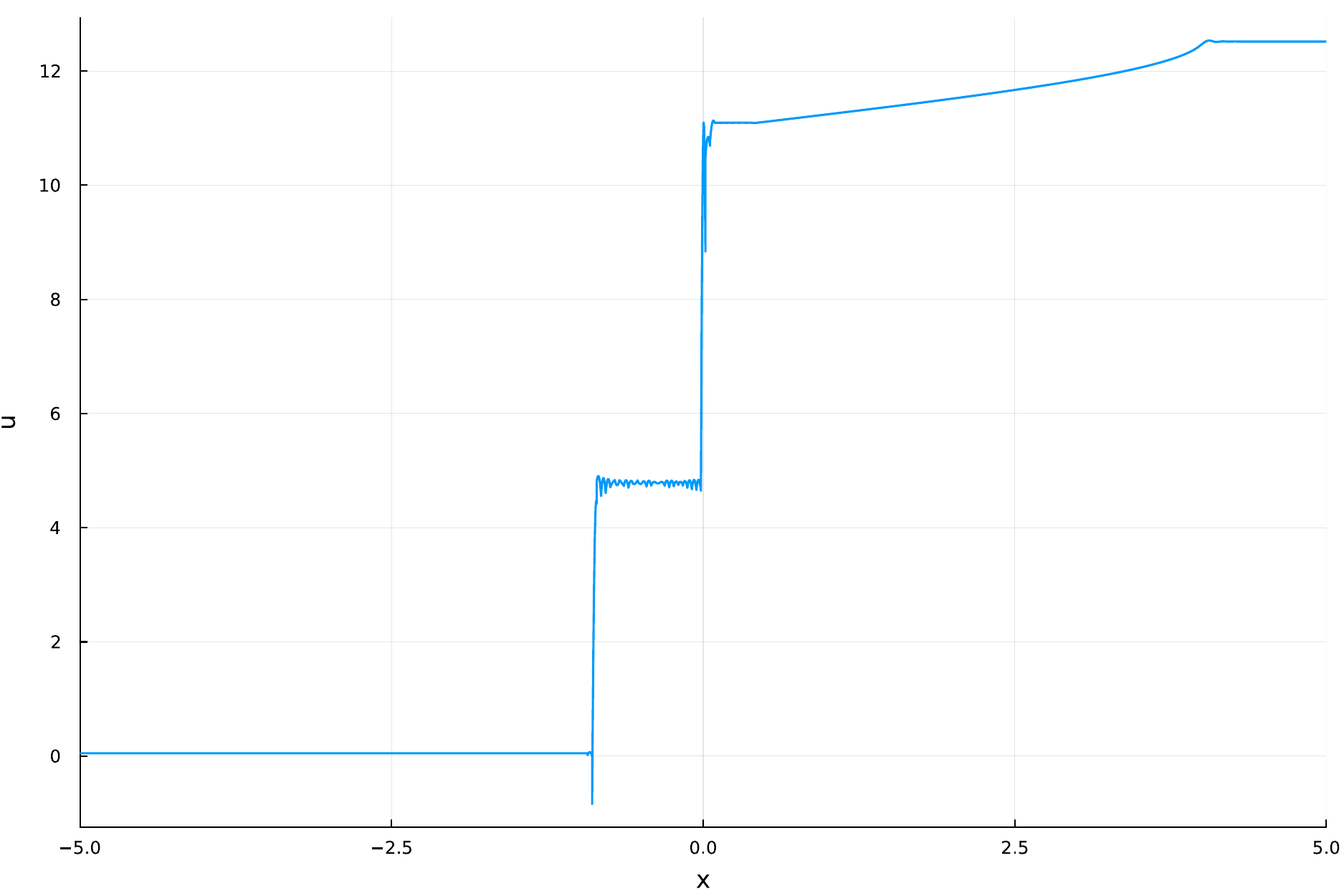}
         \caption*{$\mathbb{P}^2$}
     \end{subfigure}
     \hfill
     \begin{subfigure}[b]{0.2\textwidth}
         \centering
         \includegraphics[width=\textwidth]{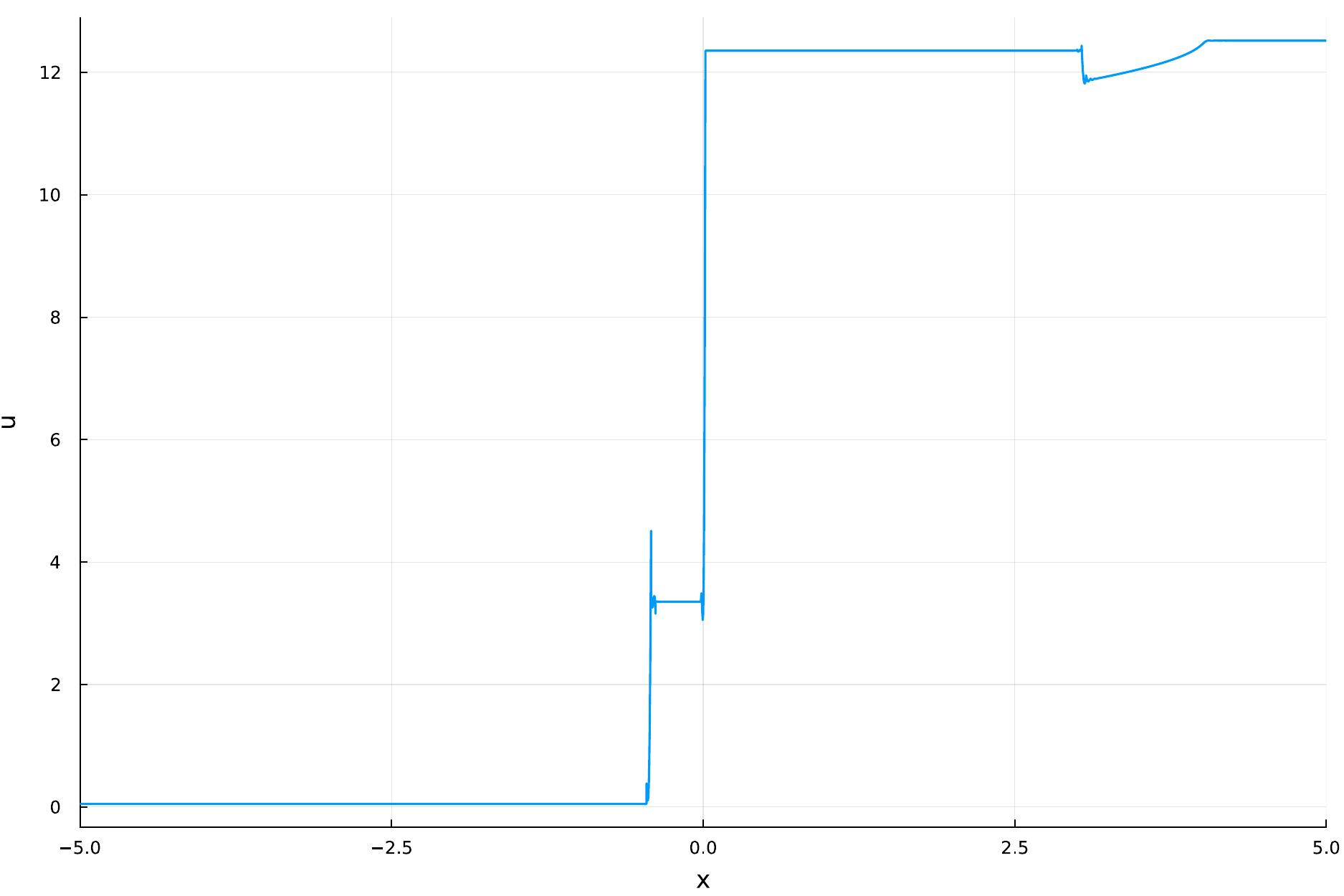}
         \caption*{$\mathbb{P}^3$}
     \end{subfigure}
     \hfill
     \begin{subfigure}[b]{0.2\textwidth}
         \centering
         \includegraphics[width=\textwidth]{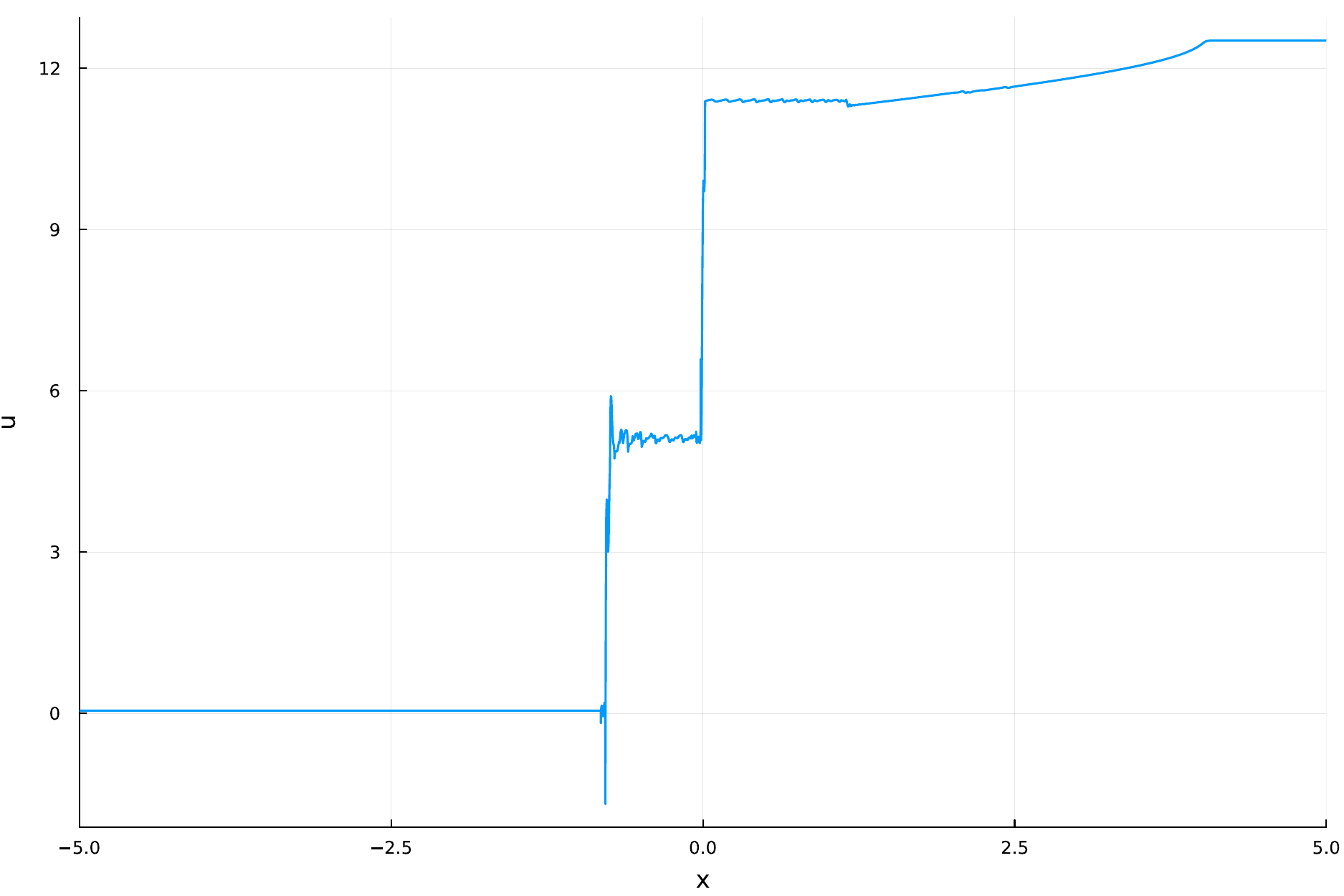}
         \caption*{$\mathbb{P}^4$}
     \end{subfigure}
     \caption{\small Solutions of \eqref{scheme D} with $\epsilon_1=0$, $\epsilon_2=1$ and $k=1,2,3,4$. We use a uniform mesh with $N=275$. None of the above solutions converges to the correct solution.}
     \label{fig:OFDG ex sine}
\end{figure}

Then, while keeping $\epsilon_2=1$, we set $\epsilon_1=1$ for $k=1,2$, $\epsilon_1=0.1$ for $k=3$, and $\epsilon_1=0.01$ for $k=4$. The results for $r=k+1$ and $r=k+2$ are shown in Figures \ref{fig:OPDG k+1 ex sine} and \ref{fig:OPDG k+2 ex sine} respectively. The scheme now converges correctly to the entropy solution.
\begin{figure}[htbp]
     \centering
     \begin{subfigure}[t]{0.2\textwidth}
         \centering
         \includegraphics[width=\textwidth]{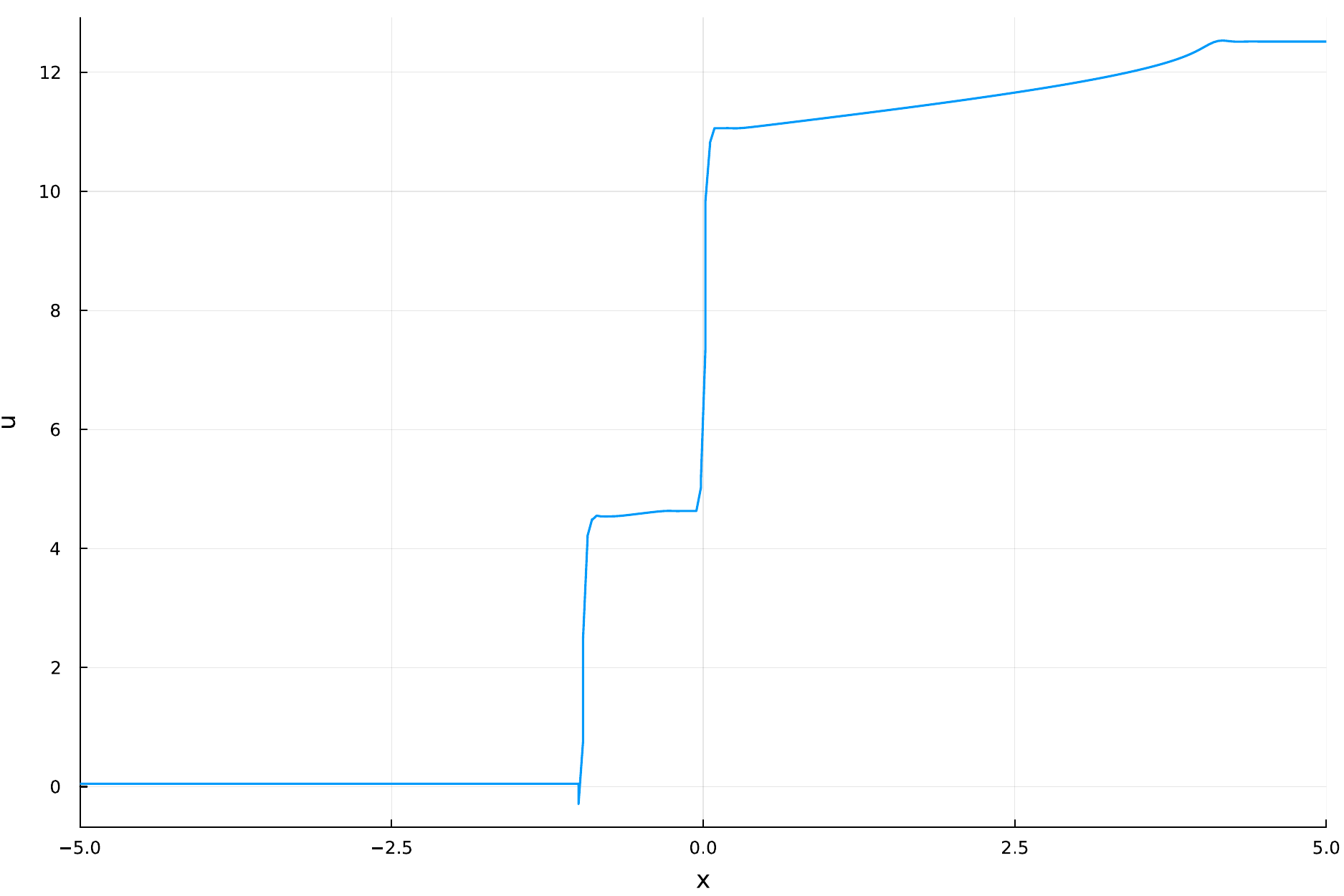}
         \caption*{$\mathbb{P}^1$ with $\epsilon_1=1$, $\epsilon_2=1$.}
     \end{subfigure}
     \hfill
     \begin{subfigure}[t]{0.2\textwidth}
         \centering
         \includegraphics[width=\textwidth]{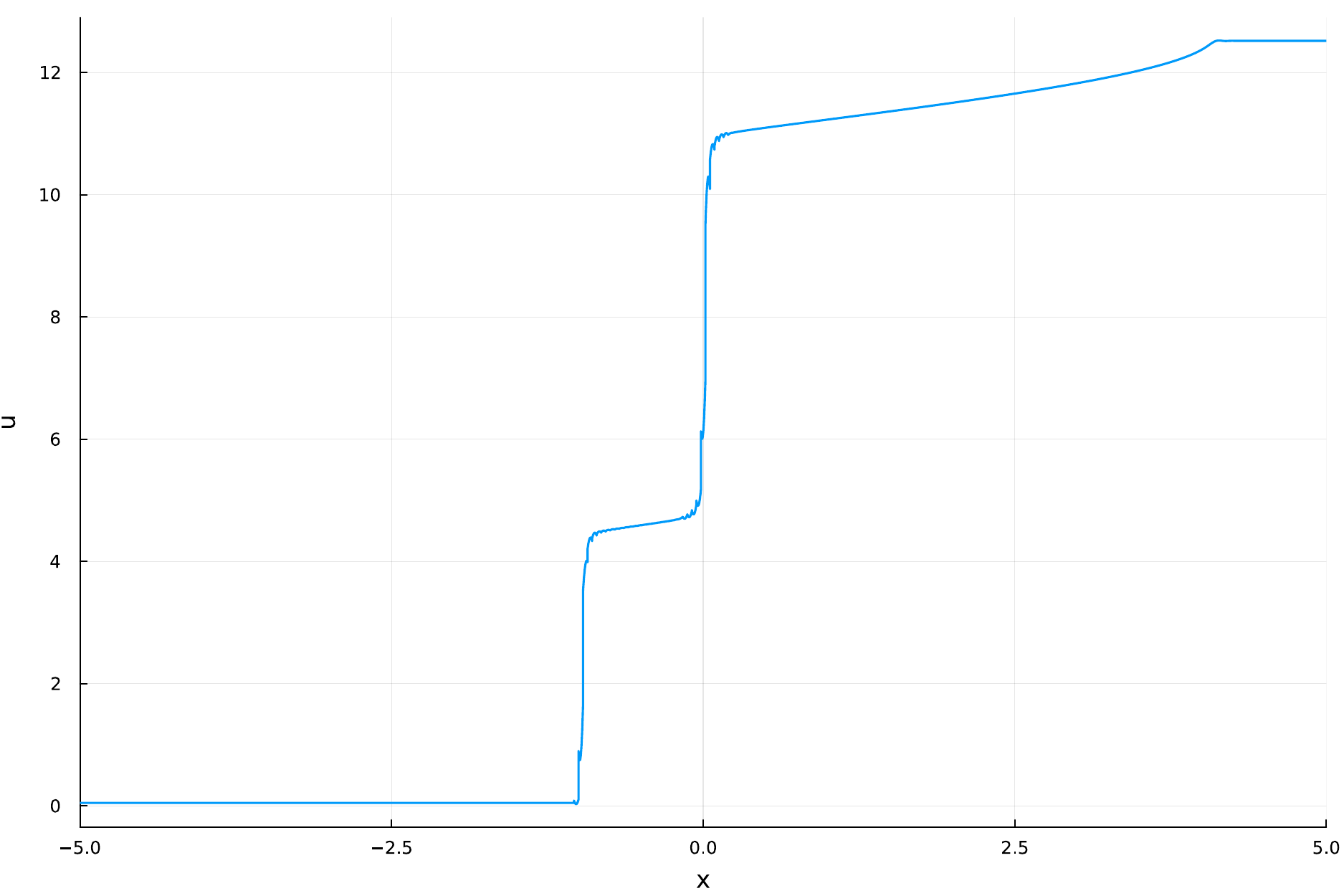}
         \caption*{$\mathbb{P}^2$ with $\epsilon_1=1$, $\epsilon_2=1$.}
     \end{subfigure}
     \hfill
     \begin{subfigure}[t]{0.2\textwidth}
         \centering
         \includegraphics[width=\textwidth]{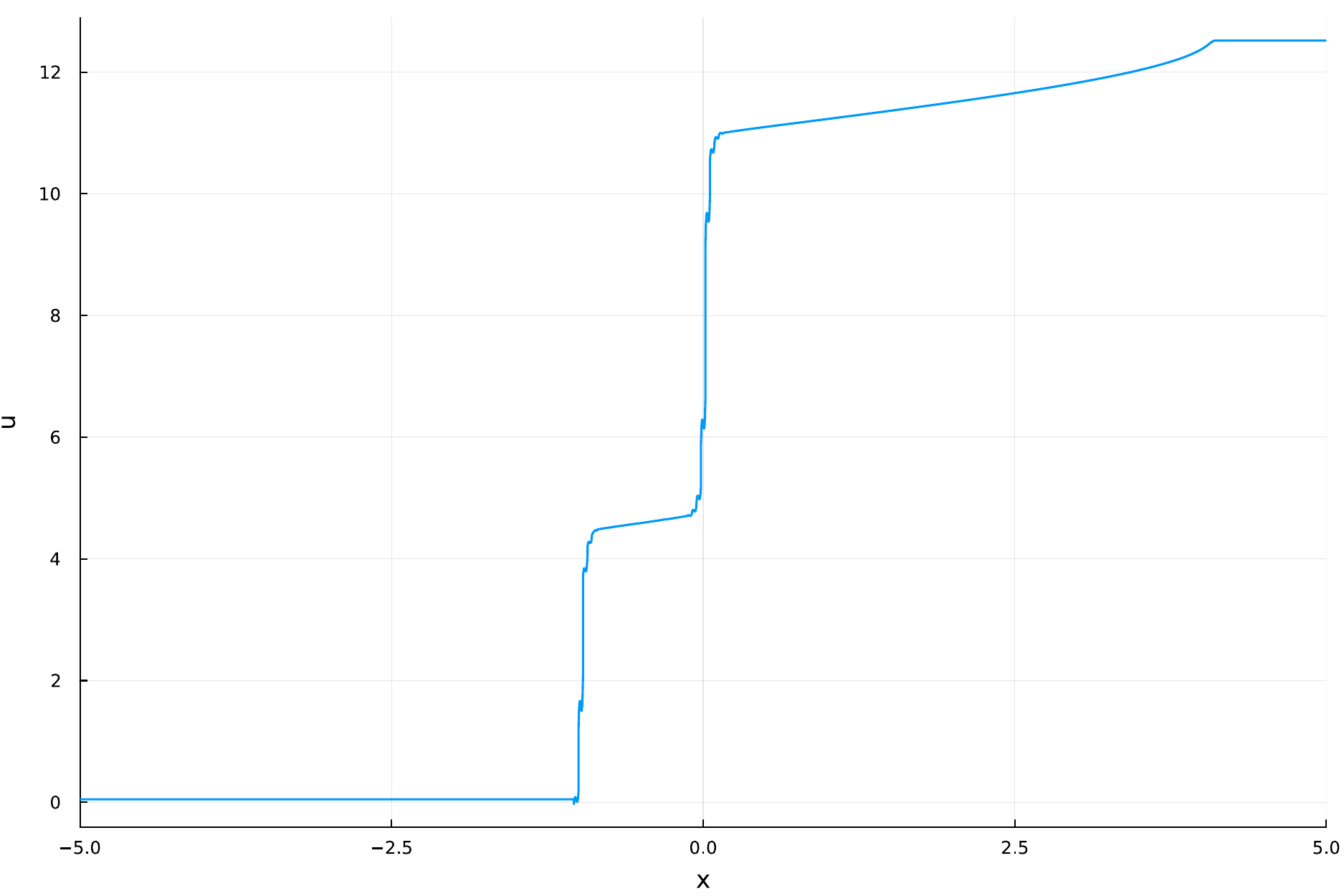}
         \caption*{$\mathbb{P}^3$ with $\epsilon_1=0.1$, $\epsilon_2=1$.}
     \end{subfigure}
     \hfill
     \begin{subfigure}[t]{0.2\textwidth}
         \centering
         \includegraphics[width=\textwidth]{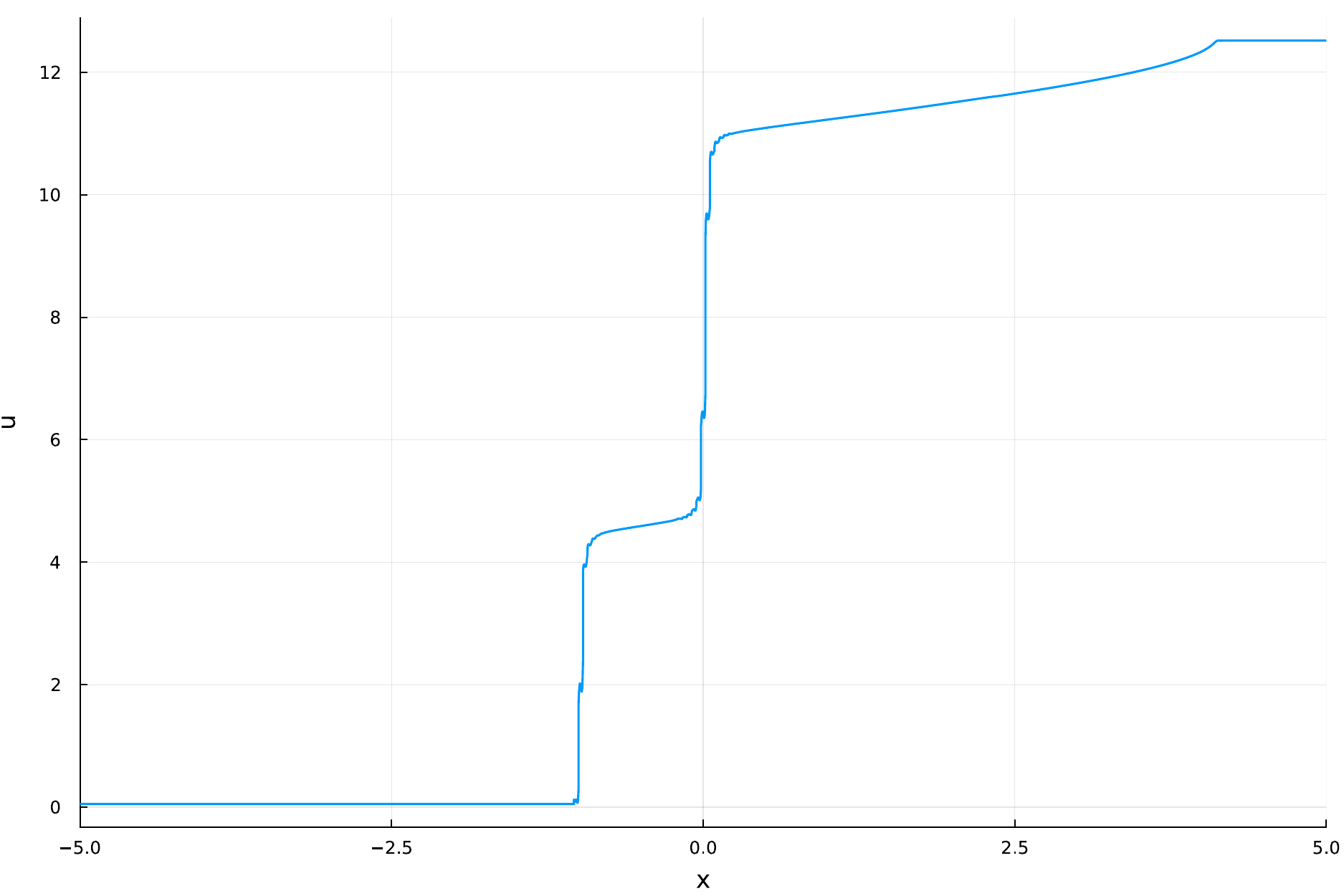}
         \caption*{$\mathbb{P}^4$ with $\epsilon_1=0.01$, $\epsilon_2=1$.}
     \end{subfigure}
     \caption{\small Solutions of \eqref{scheme D} for $k=1,2,3,4$ and $r=k+1$. We still use a uniform mesh with $N=275$. Now we have convergence to the correct solution.}
     \label{fig:OPDG k+1 ex sine}
\end{figure}

\begin{figure}[htbp]
     \centering
     \begin{subfigure}[t]{0.2\textwidth}
         \centering
         \includegraphics[width=\textwidth]{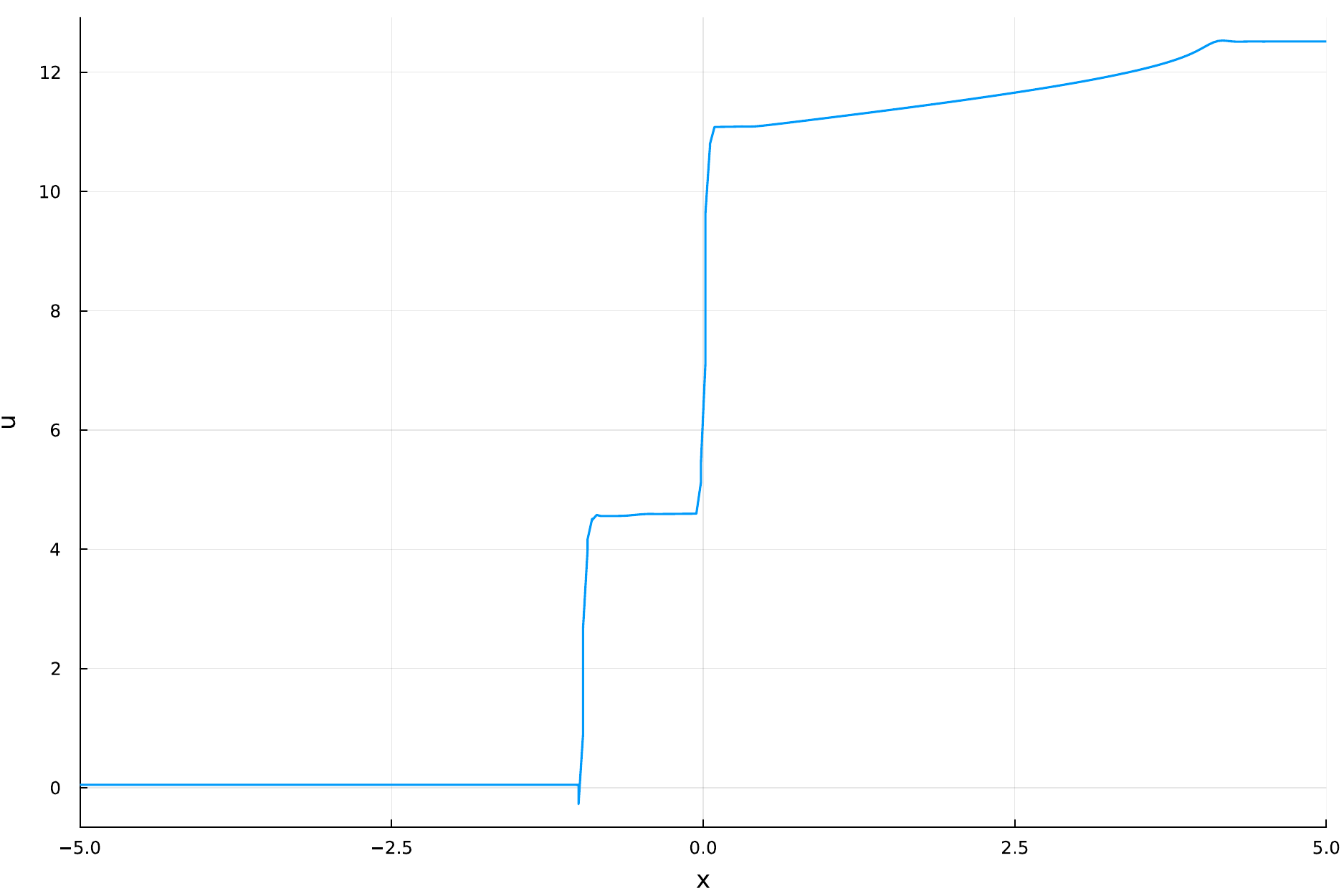}
         \caption*{$\mathbb{P}^1$ with $\epsilon_1=1$, $\epsilon_2=1$.}
     \end{subfigure}
     \hfill
     \begin{subfigure}[t]{0.2\textwidth}
         \centering
         \includegraphics[width=\textwidth]{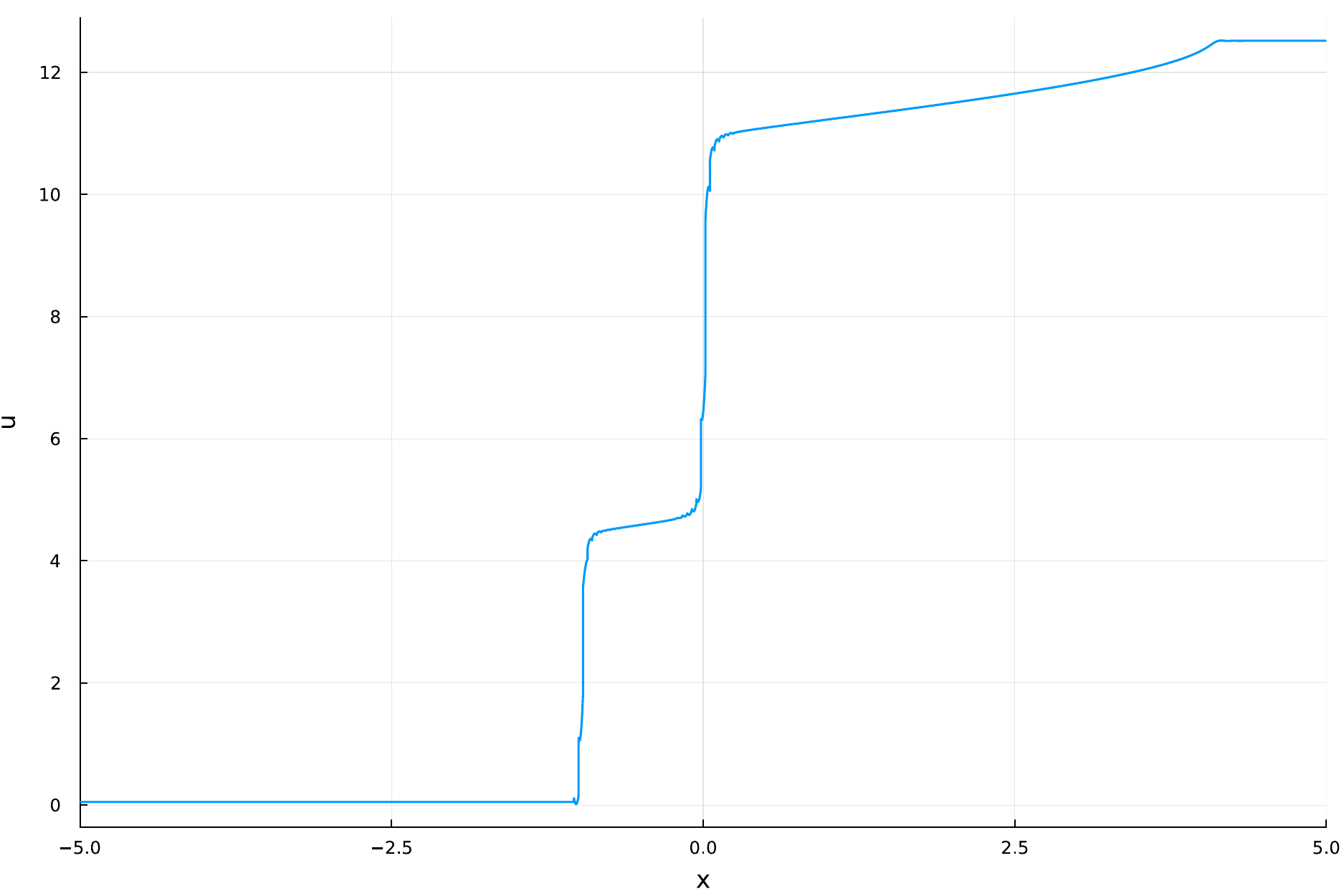}
         \caption*{$\mathbb{P}^2$ with $\epsilon_1=1$, $\epsilon_2=1$.}
     \end{subfigure}
     \hfill
     \begin{subfigure}[t]{0.2\textwidth}
         \centering
         \includegraphics[width=\textwidth]{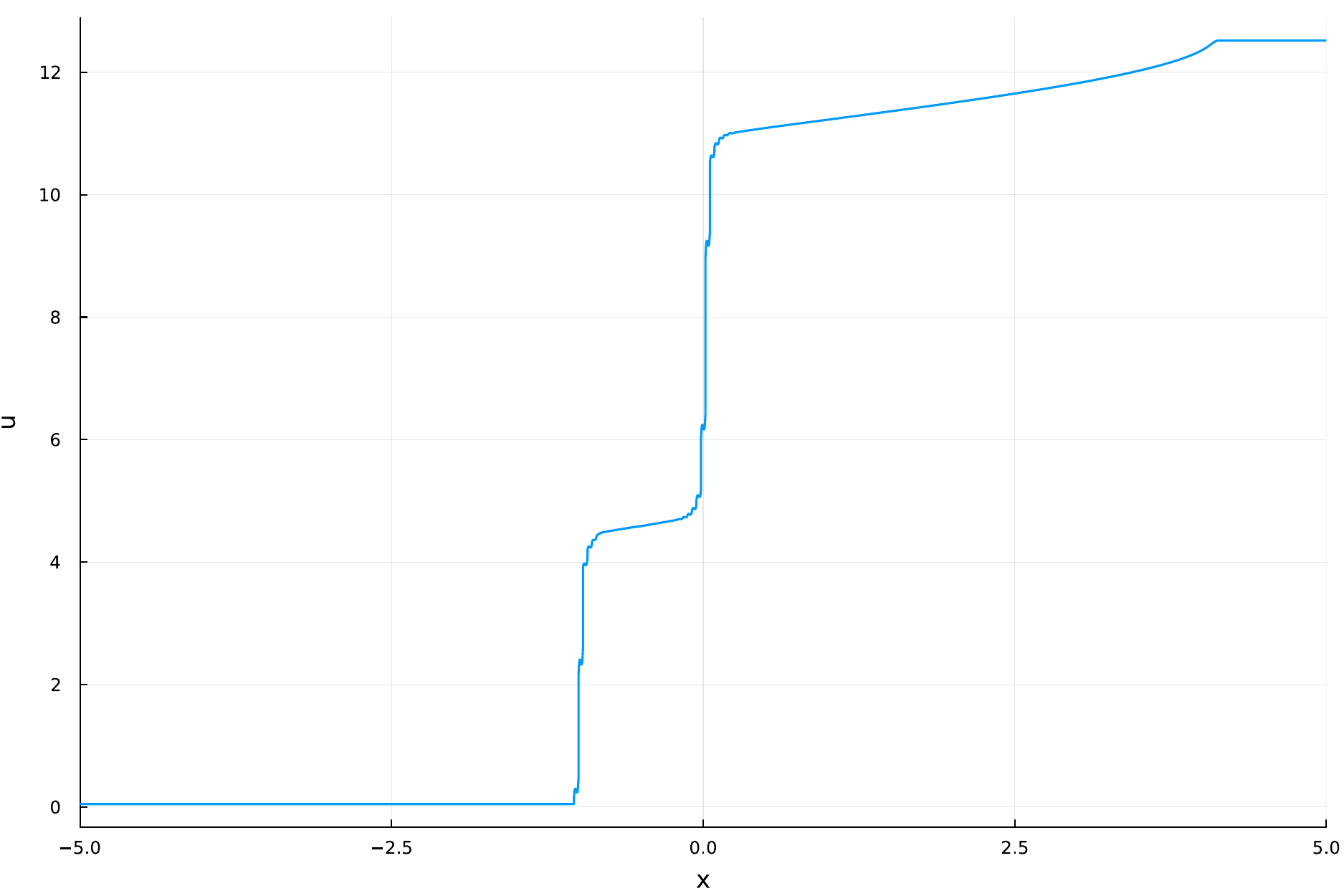}
         \caption*{$\mathbb{P}^3$ with $\epsilon_1=0.1$, $\epsilon_2=1$.}
     \end{subfigure}
     \hfill
     \begin{subfigure}[t]{0.2\textwidth}
         \centering
         \includegraphics[width=\textwidth]{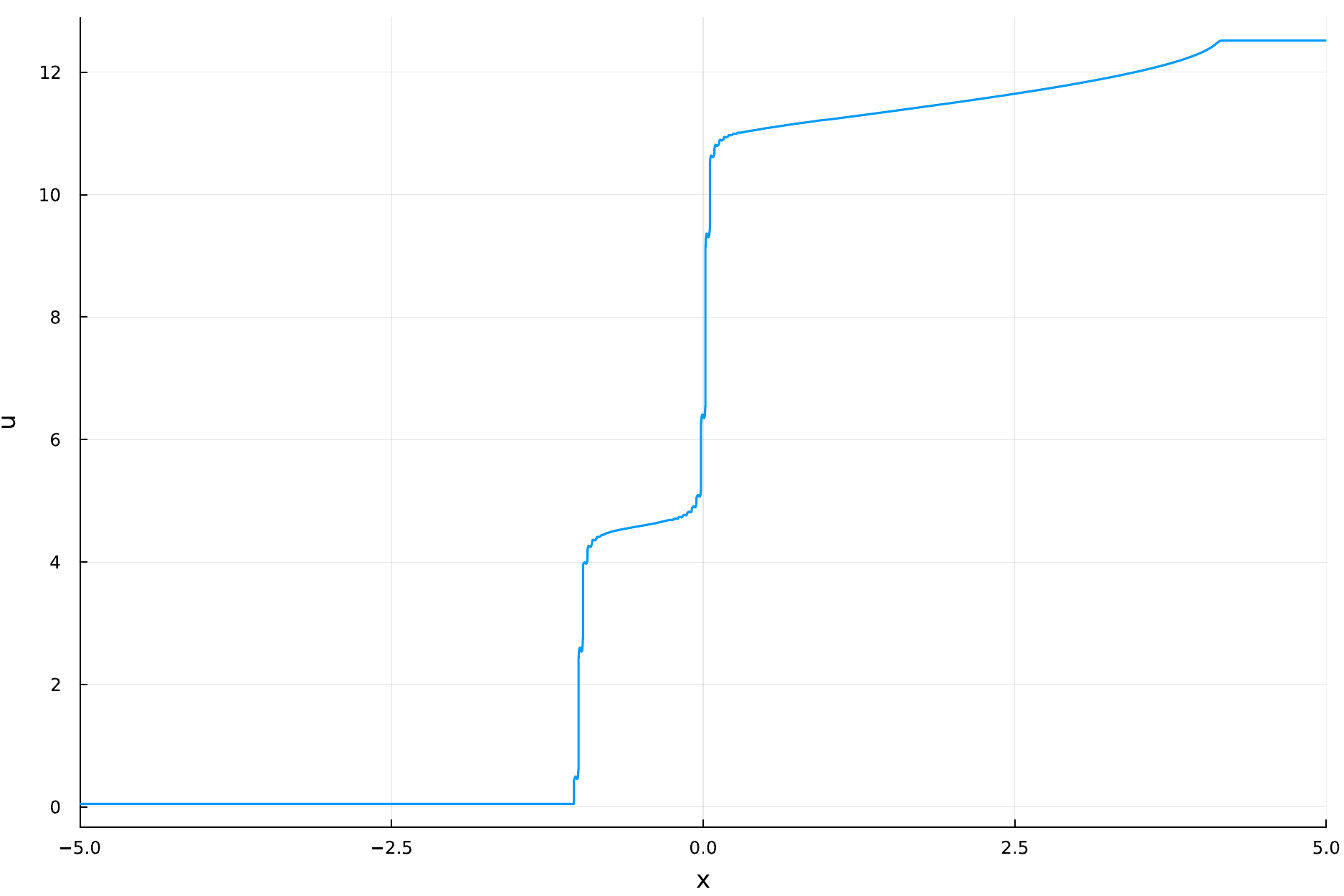}
         \caption*{$\mathbb{P}^4$ with $\epsilon_1=0.01$, $\epsilon_2=1$.}
     \end{subfigure}
     \caption{\small Solutions of \eqref{scheme D} for $k=1,2,3,4$ and $r=k+2$. We still use a uniform mesh with $N=275$. Now we have convergence to the correct solution.}
     \label{fig:OPDG k+2 ex sine}
\end{figure}

\end{expl}

\begin{expl}
Consider a Riemann problem for one-dimensional Euler equations known as the \emph{Sod shock tube problem}. The heat capacity ratio is taken as $\gamma=1.4$. The initial condition is
\begin{equation*}
    (\rho,u,p)=\begin{cases}
        (1,0,1),& x < 0.5,\\
        (0.125,0,0.1),& x \ge 0.5.
    \end{cases}
\end{equation*}
The computational domain is $I=[0,1]$ with transmissive boundary conditions. We use a uniform mesh of size $N=100$. The solution is computed up to $T=0.2$ using \eqref{scheme D'} with $k=1,2,3$ and $r=k+1$. The local Lax–Friedrichs flux and SSP--RK3 were used in all tests. We set $\epsilon_2=10$ for all tests. In Figure \ref{fig:Sod ex2}, the solutions on the left are obtained with $\epsilon_1=0$ for all $k=1,2,3$. On the right, we set $\epsilon_1=10$ for $k=1$, $\epsilon_1=1$ for $k=2$, and $\epsilon_1=0.01$ for $k=3.$ In this example, no huge difference can be noted between $\epsilon_1=0$ and $\epsilon_1>0$. The reference solution is computed by the fifth order finite difference WENO method with N = 4096.
\begin{figure}[htbp]
     \centering
     \begin{subfigure}[b]{0.47\textwidth}
         \centering
         \includegraphics[width=\textwidth]{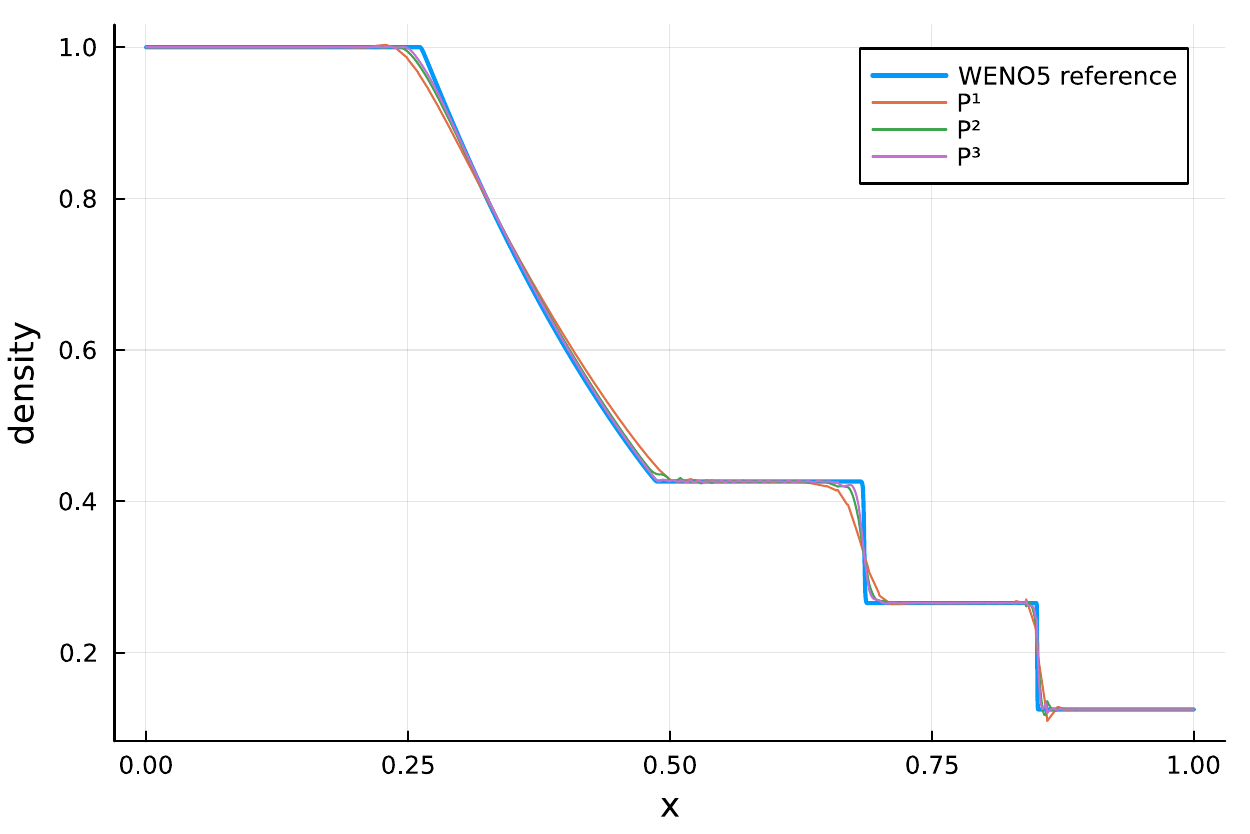}
         \caption*{Density $\rho$.}
     \end{subfigure}
     \hfill
     \begin{subfigure}[b]{0.47\textwidth}
         \centering
         \includegraphics[width=\textwidth]{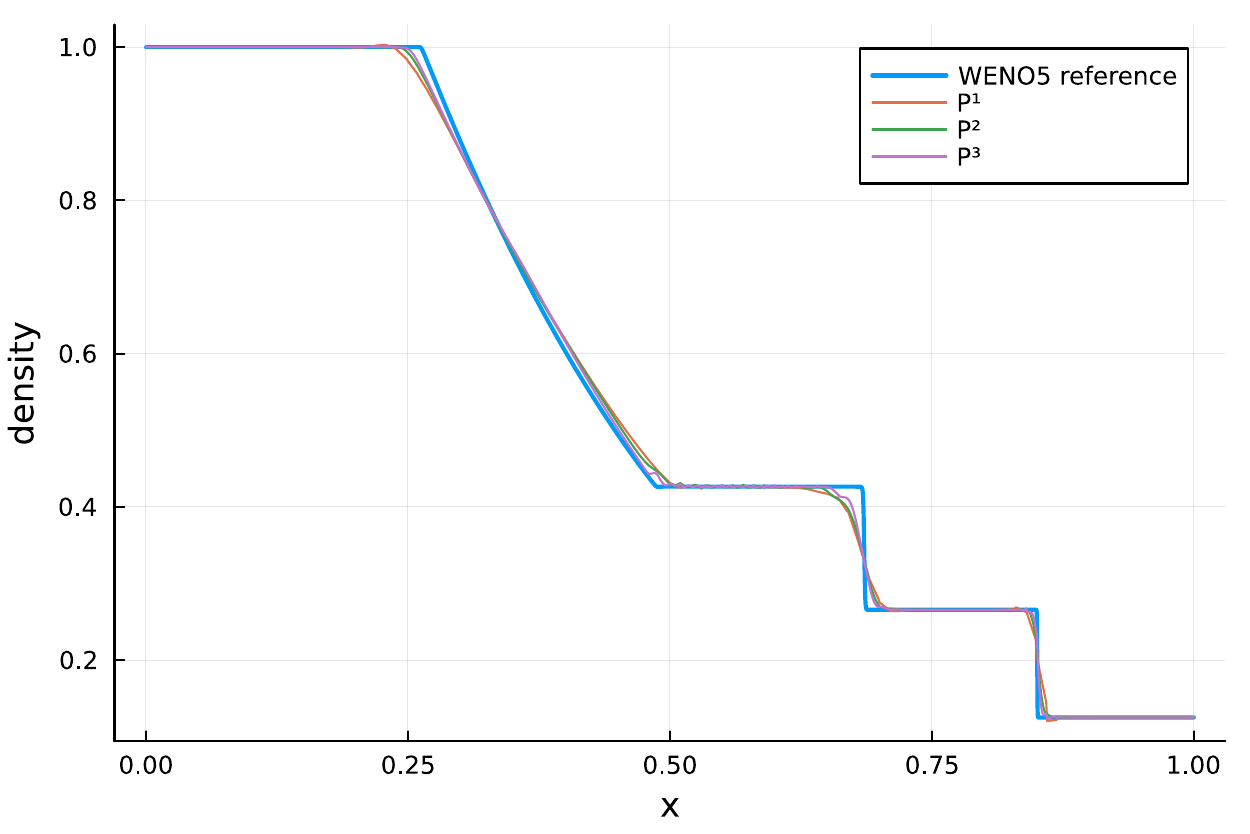}
         \caption*{Density $\rho$.}
     \end{subfigure}
     \hfill
     \begin{subfigure}[b]{0.47\textwidth}
         \centering
         \includegraphics[width=\textwidth]{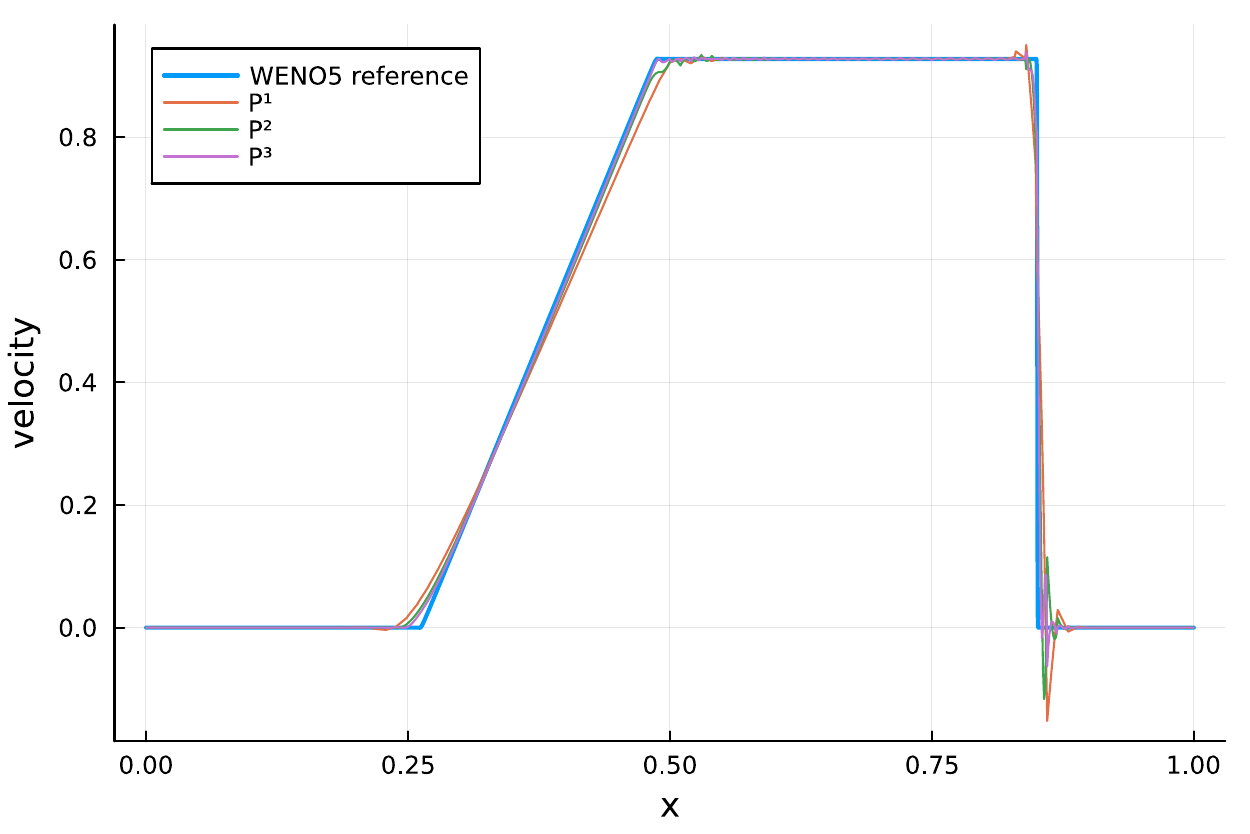}
         \caption*{Velocity $u$.}
     \end{subfigure}
     \hfill
     \begin{subfigure}[b]{0.47\textwidth}
         \centering
         \includegraphics[width=\textwidth]{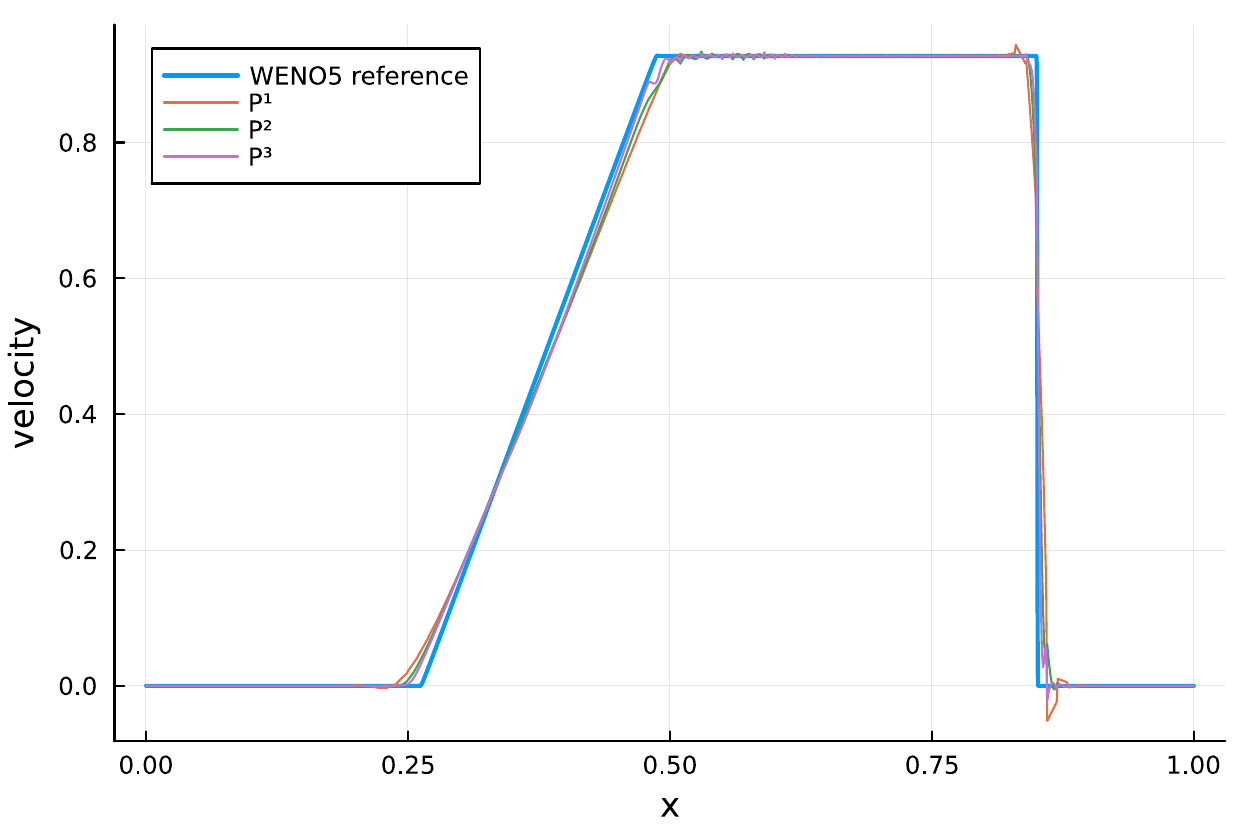}
         \caption*{Velocity $u$.}
     \end{subfigure}
     \hfill
      \begin{subfigure}[b]{0.47\textwidth}
         \centering
         \includegraphics[width=\textwidth]{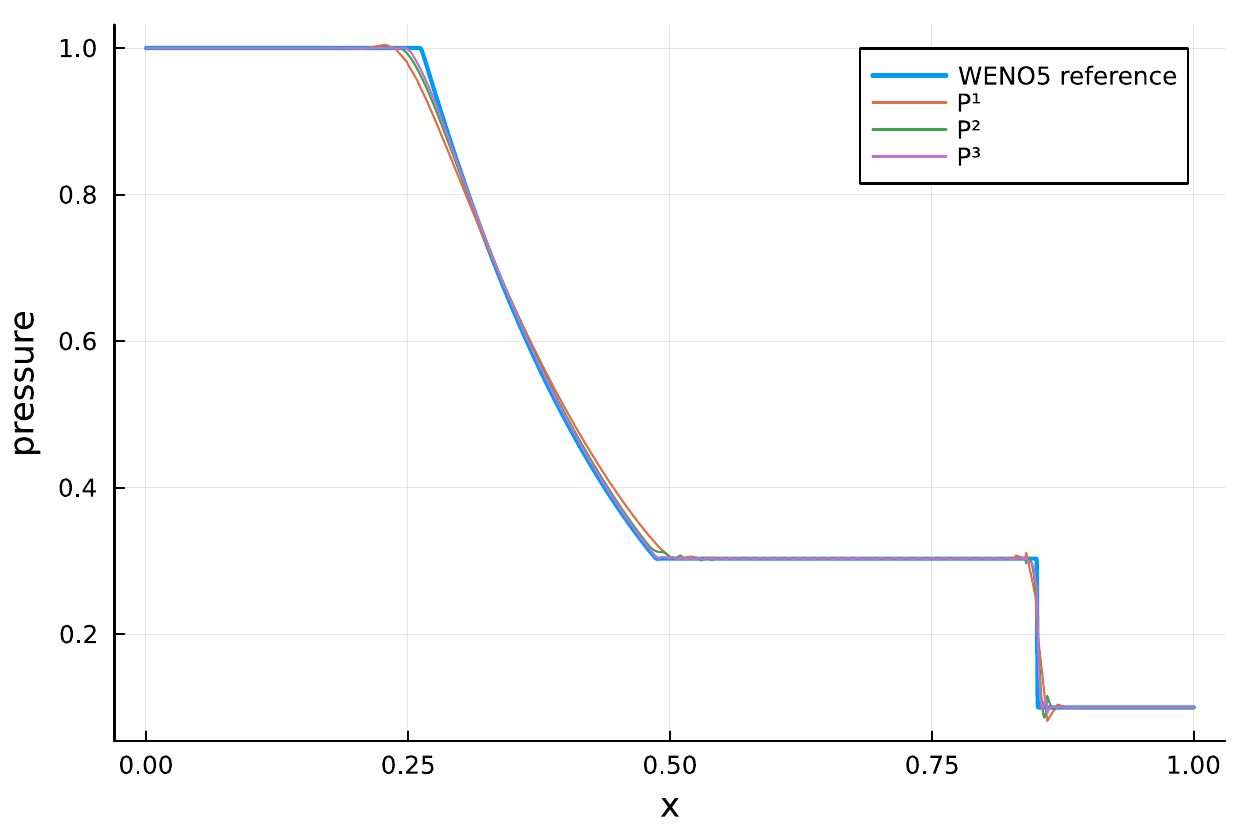}
         \caption*{Pressure $p$.}
     \end{subfigure}
     \hfill
     \begin{subfigure}[b]{0.47\textwidth}
         \centering
         \includegraphics[width=\textwidth]{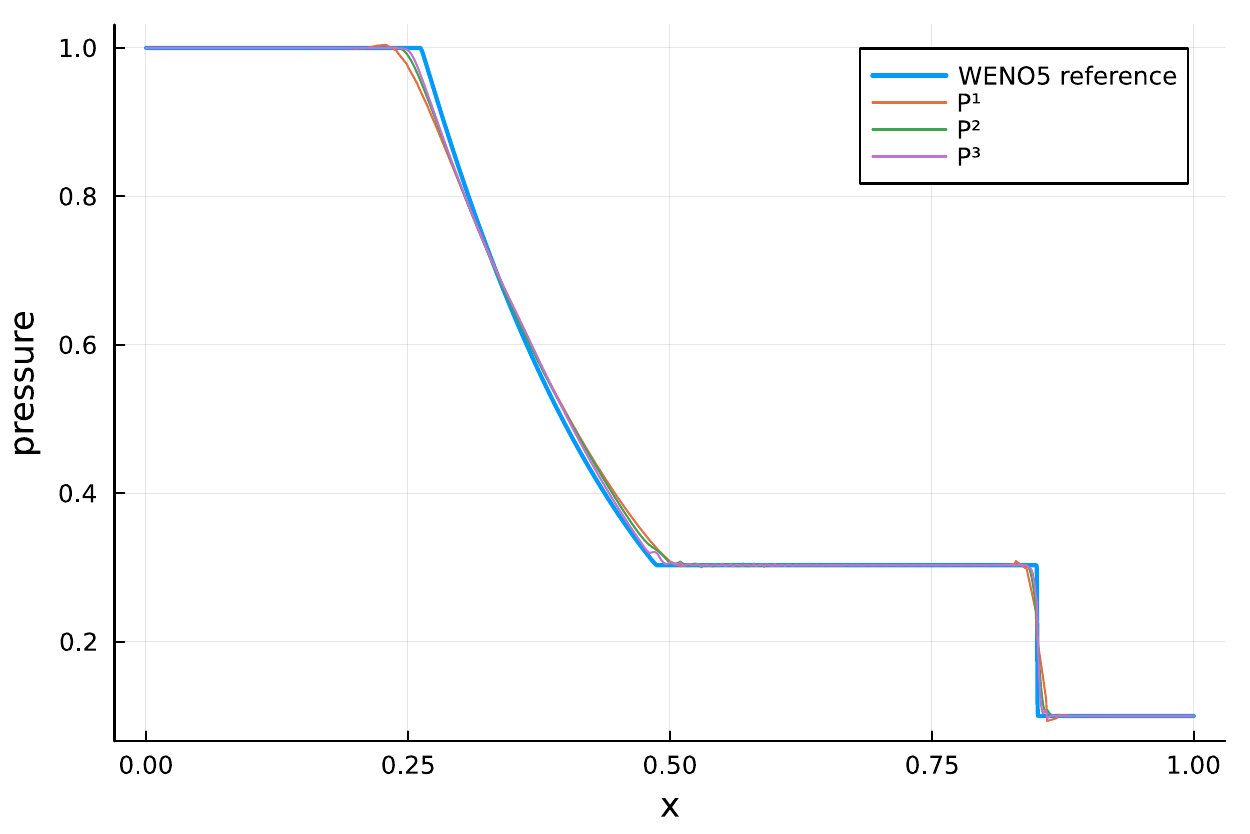}
         \caption*{Pressure $p$.}
     \end{subfigure}
     \caption{Solutions of \eqref{scheme D'} with $k=1,2,3$ and $r=k+1$. We use a uniform mesh with $N=100$. We set $\epsilon_2=10$ for all tests. The solutions on the left are obtained with $\epsilon_1=0$ for all $k=1,2,3$. On the right, we set $\epsilon_1=10$ for $k=1$, $\epsilon_1=1$ for $k=2$, and $\epsilon_1=0.01$ for $k=3.$}
     \label{fig:Sod ex2}
\end{figure}

\end{expl}

\begin{expl}

Consider the Shu-Osher problem for one-dimensional Euler equations defined on $I=[-5,5]$. The initial condition is
\begin{equation*}
    (\rho,u,p)=\begin{cases}
        (3.857143,2.629369,10.33333),& x < -4,\\
        (1+0.2 \sin (5x),0,1),& x \ge -4.
    \end{cases}
\end{equation*}
The fixed post-shock state is imposed at the left boundary, and outflow boundary conditions are used on the right. The heat capacity ratio is taken as $\gamma=1.4$. We use a uniform mesh of size $N=400$. The solution is computed up to $T=1.8$ using \eqref{scheme D'} with $k=1,2,3$ and $r=k+1$. The local Lax–Friedrichs flux and SSP--RK3 were used in all tests. We set $\epsilon_2=10$ for all tests. In Figure \ref{fig:Shu-Osher ex3}, the solutions on the left are obtained with $\epsilon_1=0$ for all $k=1,2,3$. On the right, we set $\epsilon_1=2$ for $k=1$, $\epsilon_1=1$ for $k=2$, and $\epsilon_1=0.01$ for $k=3.$ Again, no huge difference can be noted between $\epsilon_1=0$ and $\epsilon_1>0$. The reference solution is computed by the fifth order finite difference WENO method with N = 4096.
\begin{figure}[htbp]
     \centering
     \begin{subfigure}[b]{0.47\textwidth}
         \centering
         \includegraphics[width=\textwidth]{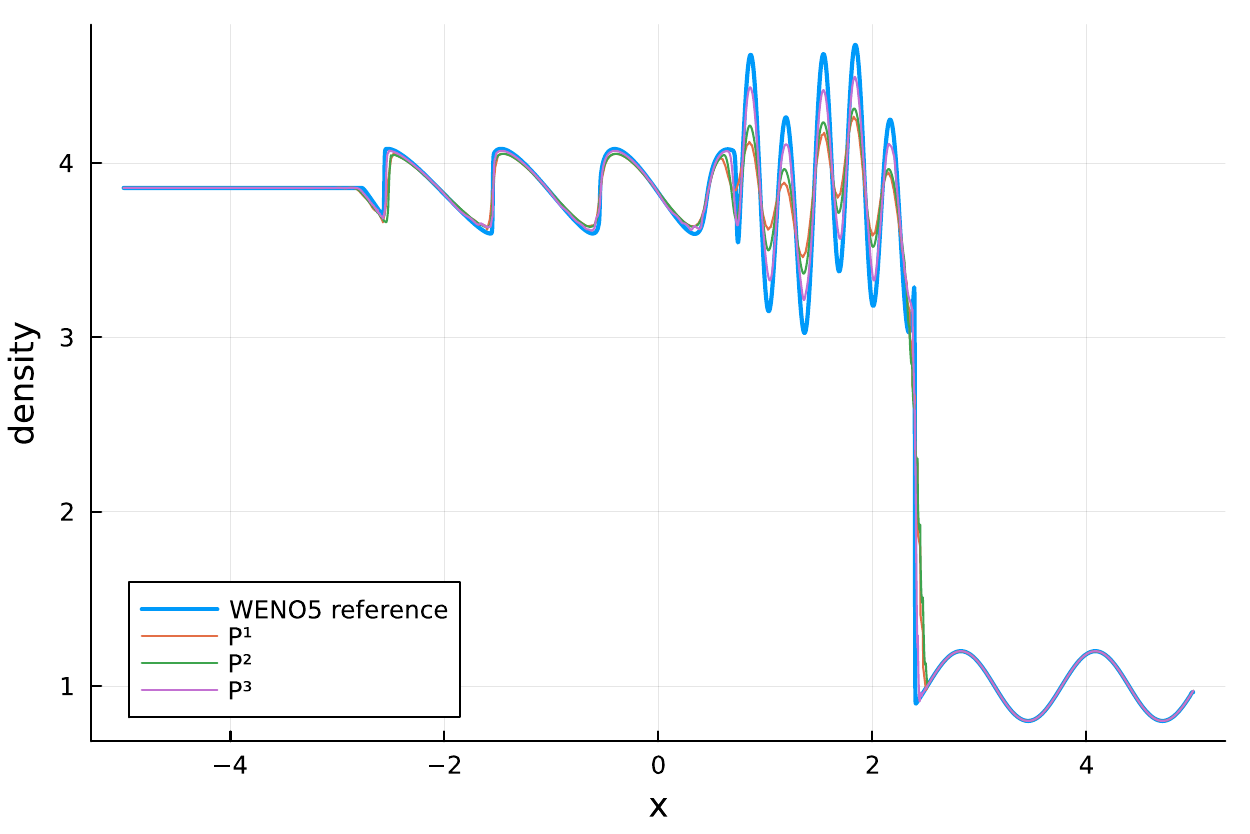}
         \caption*{Density $\rho$.}
     \end{subfigure}
     \hfill
     \begin{subfigure}[b]{0.47\textwidth}
         \centering
         \includegraphics[width=\textwidth]{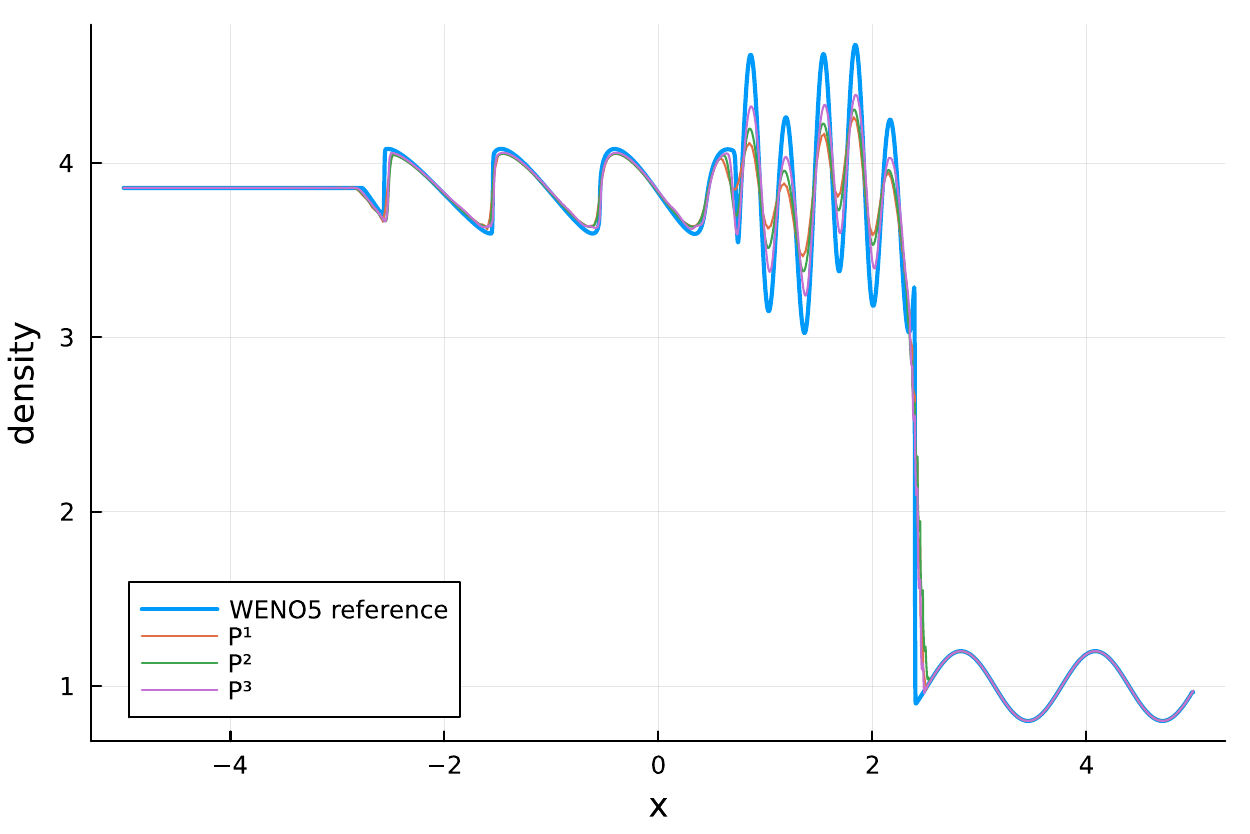}
         \caption*{Density $\rho$.}
     \end{subfigure}
     \hfill
     \begin{subfigure}[b]{0.47\textwidth}
         \centering
         \includegraphics[width=\textwidth]{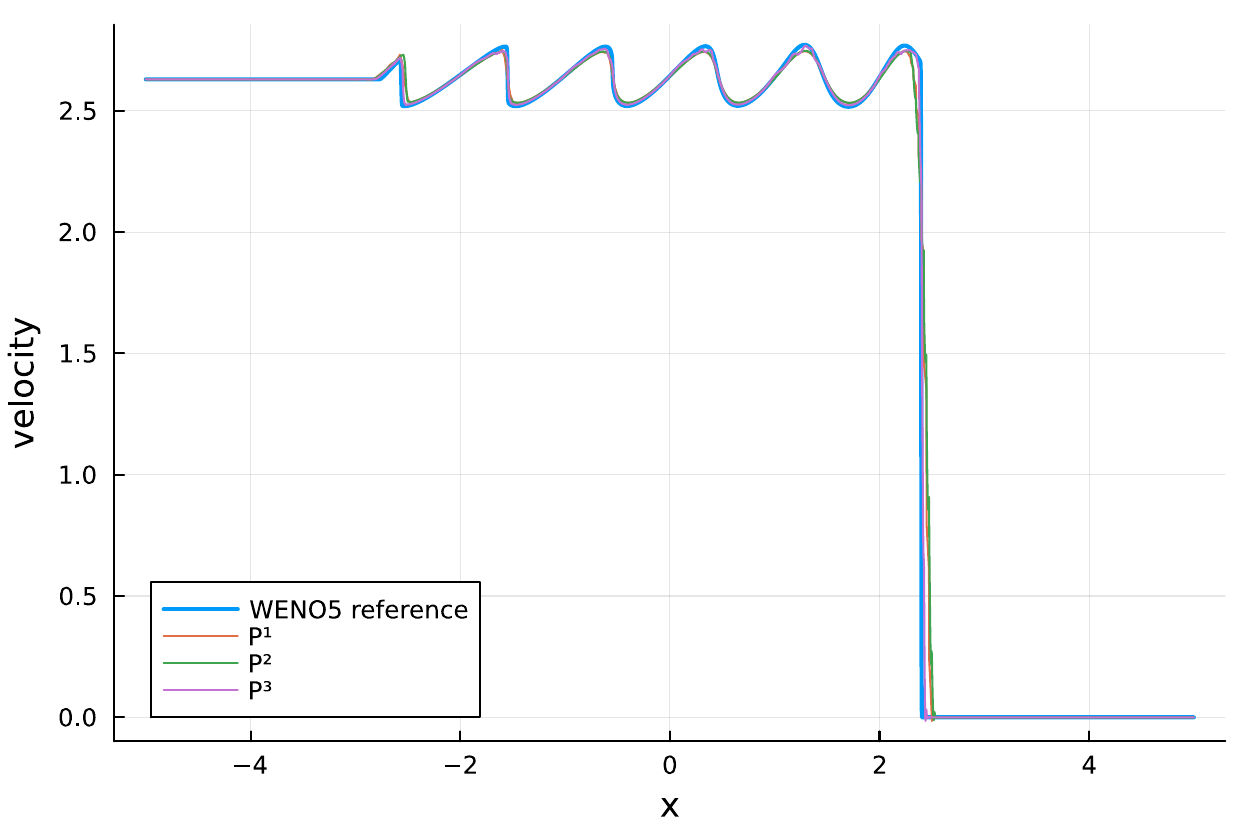}
         \caption*{Velocity $u$.}
     \end{subfigure}
     \hfill
     \begin{subfigure}[b]{0.47\textwidth}
         \centering
         \includegraphics[width=\textwidth]{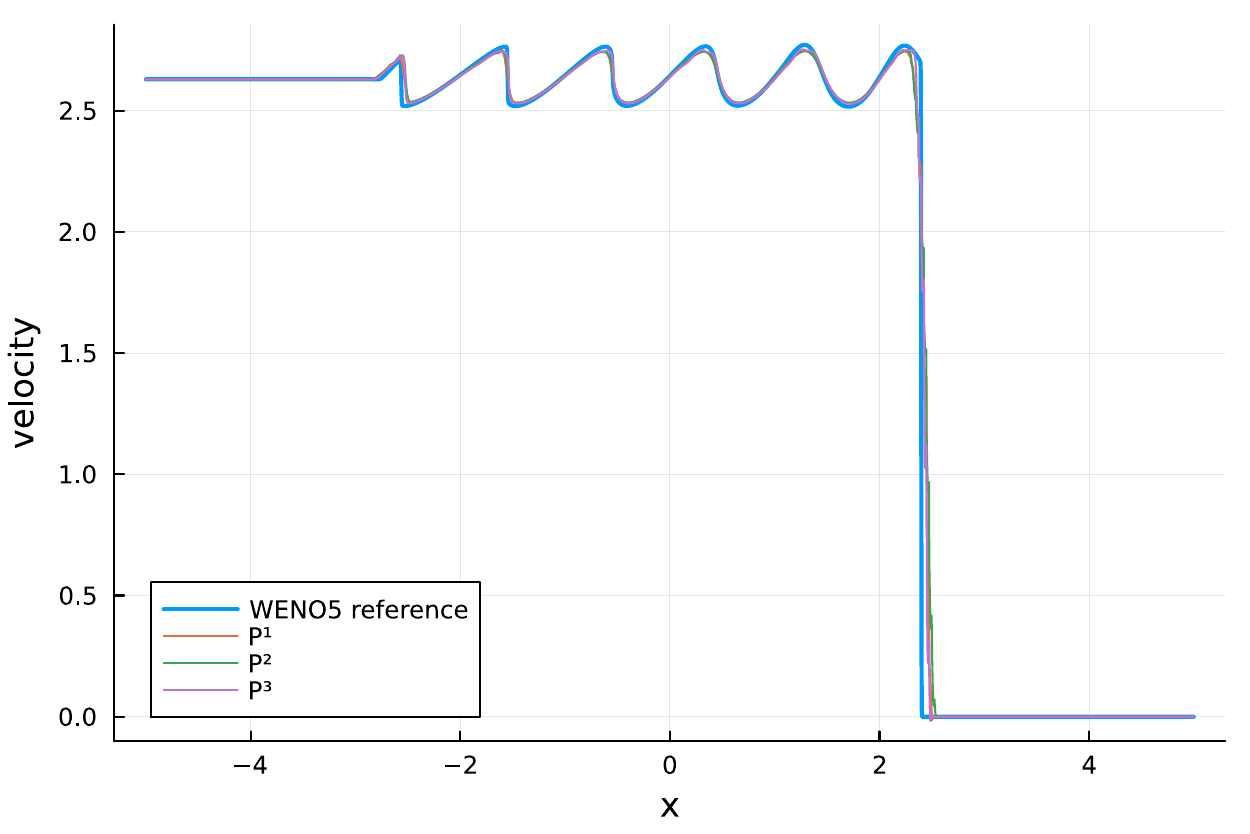}
         \caption*{Velocity $u$.}
     \end{subfigure}
     \hfill
      \begin{subfigure}[b]{0.47\textwidth}
         \centering
         \includegraphics[width=\textwidth]{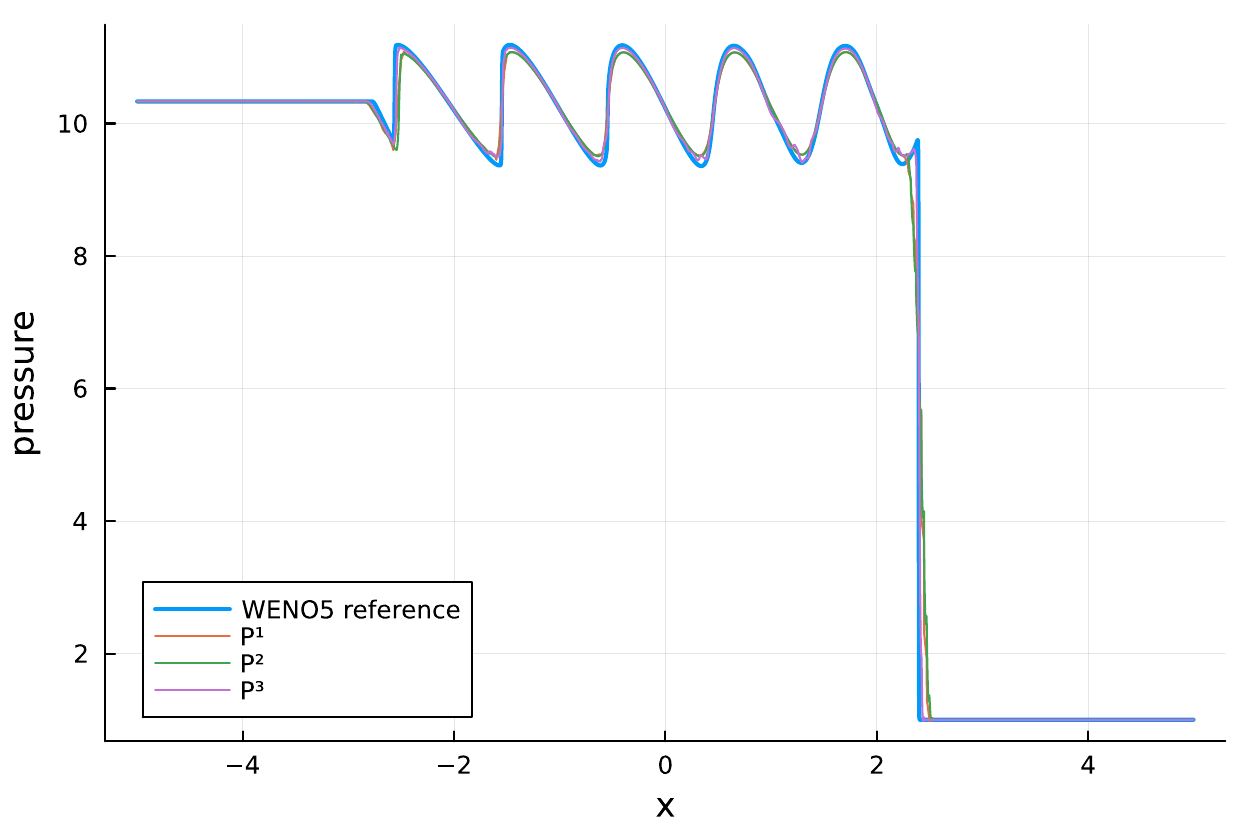}
         \caption*{Pressure $p$.}
     \end{subfigure}
     \hfill
     \begin{subfigure}[b]{0.47\textwidth}
         \centering
         \includegraphics[width=\textwidth]{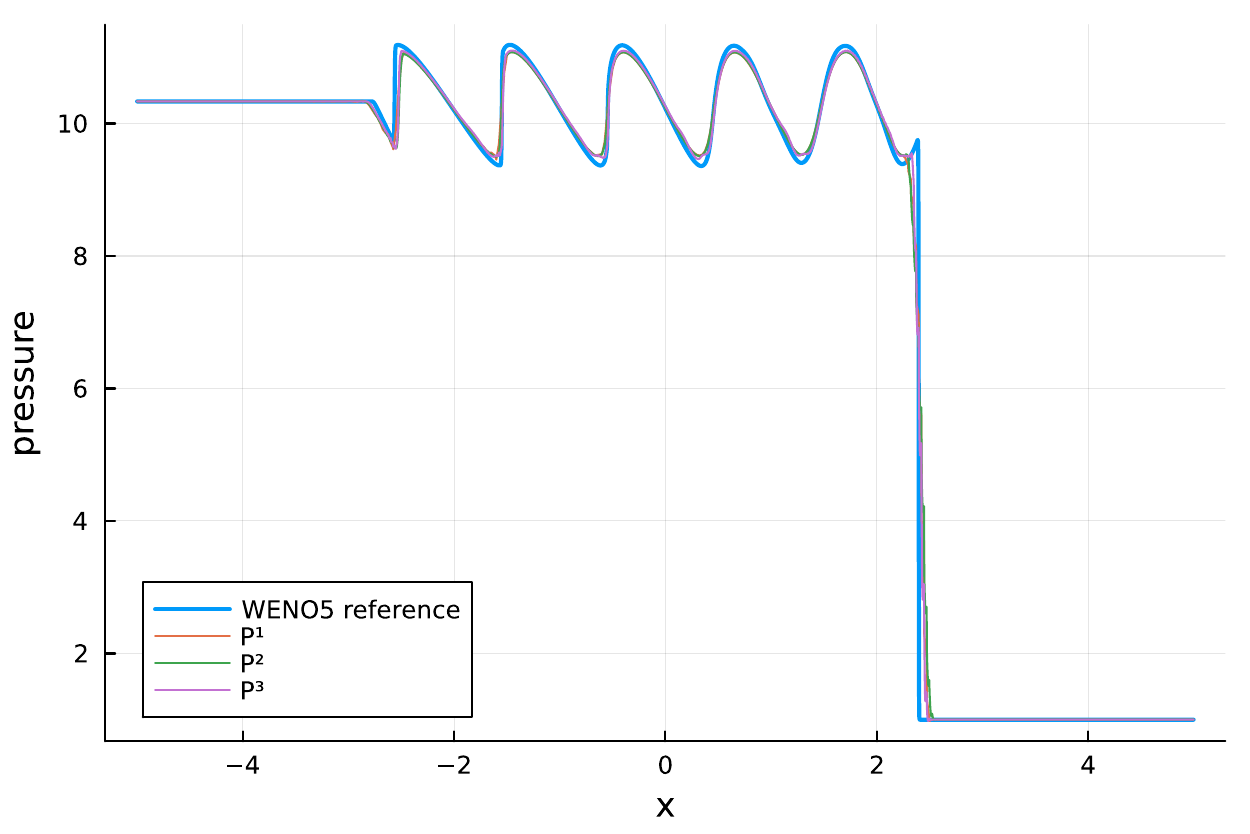}
         \caption*{Pressure $p$.}
     \end{subfigure}
     \caption{Solutions of \eqref{scheme D'} with $k=1,2,3$ and $r=k+1$. We use a uniform mesh with $N=400$. We set $\epsilon_2=10$ for all tests. The solutions on the left are obtained with $\epsilon_1=0$ for all $k=1,2,3$. On the right, we set $\epsilon_1=2$ for $k=1$, $\epsilon_1=1$ for $k=2$, and $\epsilon_1=0.01$ for $k=3.$}
     \label{fig:Shu-Osher ex3}
\end{figure}

\end{expl}

\begin{expl}
    In this test, we compute the double Mach reflection problem designed by \cite{WOODWARD1984115}. It involves a Mach 10 shock impinging on a reflecting wall at a $60^\circ$ angle. The computational domain is $\Omega=[0,4]\times[0,1]$, with the initial condition
\begin{equation*}
    (\rho_0,u_0,v_0,p_0)=\begin{cases}
  (8,\frac{33\sqrt{3}}{8},-\frac{33}{8},116.5), & \text{if } y > \sqrt{3}(x-\frac{1}{6}) \quad \text{(post-shock)}\\
  (1.4,0,0,1), & \text{if } y < \sqrt{3}(x-\frac{1}{6})\quad \text{(pre-shock)}.
\end{cases}
\end{equation*}
The inflow and outflow boundary conditions are prescribed at the left and right boundary, respectively. The top boundary values are set to equal the exact solution, which is time-dependent. On the bottom boundary, the (constant) post shock values are prescribed on the segment $\{0\le x < \frac{1}{6},y=0\}$, while reflective boundary conditions are imposed on the segment $\{\frac{1}{6}<x\le4,y=0\}$. We use a rectangular mesh of size $h_x=h_y=\frac{1}{240}$. The numerical solution is computed up to $T=0.2$ using \eqref{scheme D''} with $\epsilon_1=0$, $\epsilon_2=0.4$. SSP--RK3 is used for time integration. No positivity-preserving limiter is needed. The computed values of $\rho$ and $p$ are shown in Figure \ref{fig:ex4}. One can see that the scheme already has enough viscosity even without the flux derivative term in the damping coefficient $D''_K$. In addition, the results are very similar to those of the OFDG scheme from \cite{OFDGsystem}, although the damping coefficients are not exactly the same.

\FloatBarrier
\begin{figure}[htbp]
     \centering
     \begin{subfigure}[b]{\textwidth}
         \centering
         \includegraphics[width=\textwidth]{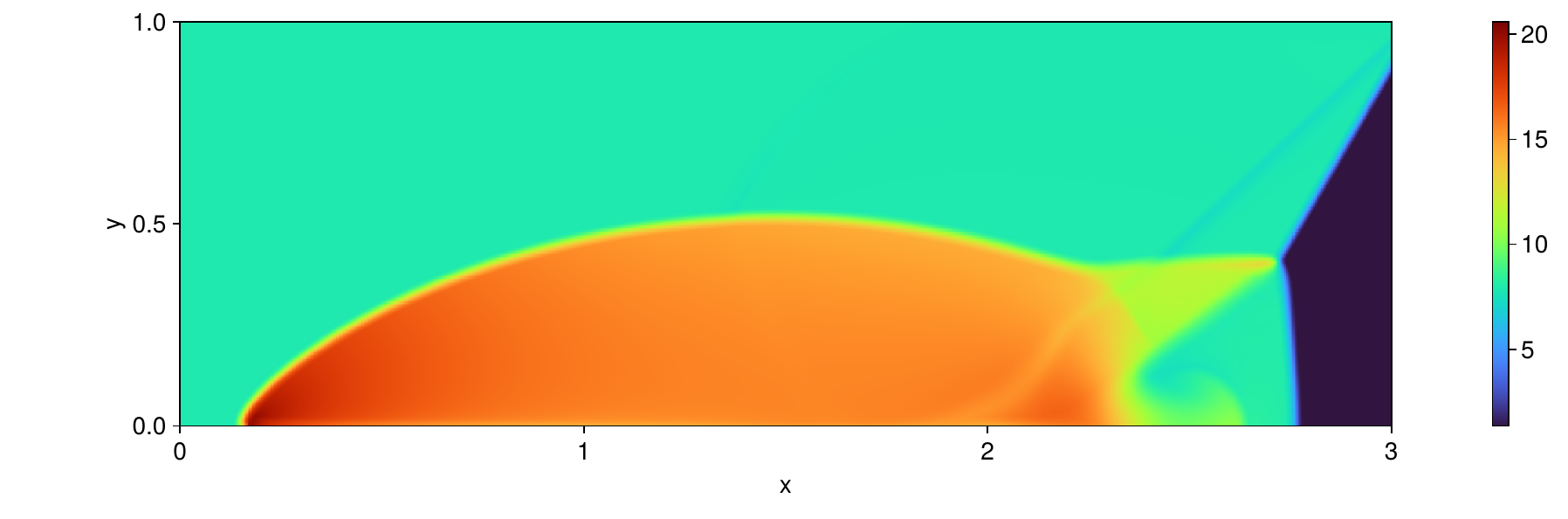}
         \caption{Density $\rho$.}
     \end{subfigure}
     \hfill
     \begin{subfigure}[b]{\textwidth}
         \centering
         \includegraphics[width=\textwidth]{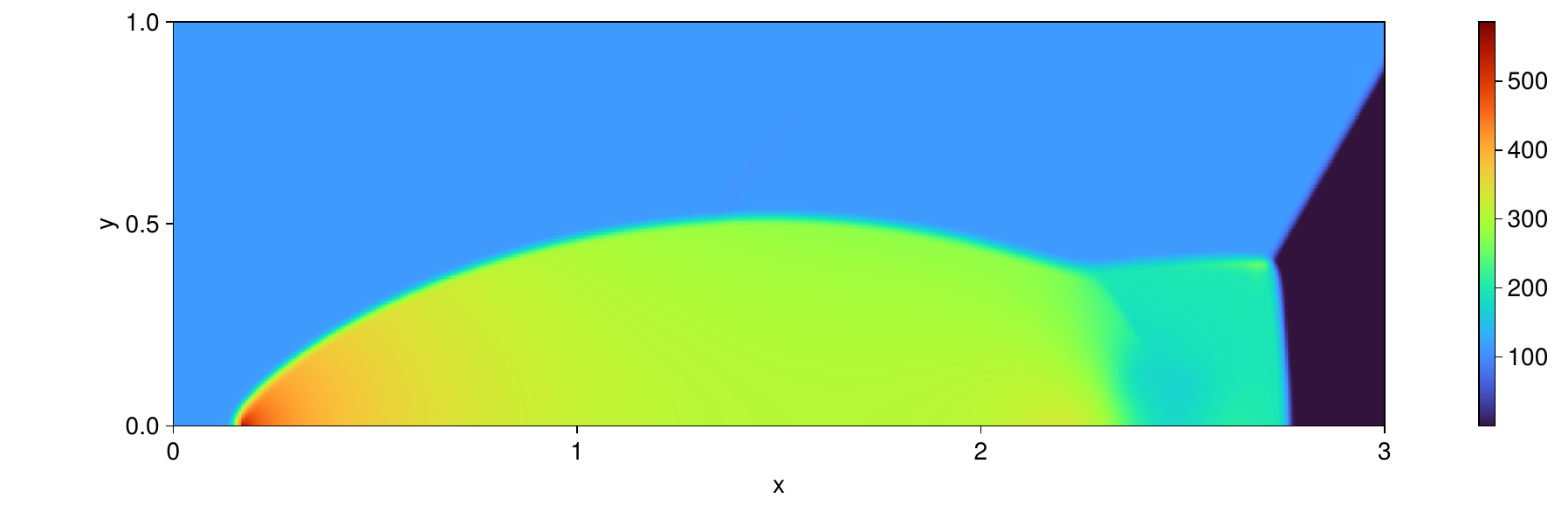}
         \caption{Pressure $p$.}
     \end{subfigure}
     \caption{The numerical solution at $T=0.2$ using \eqref{scheme D''} with $\epsilon_1=0$, $\epsilon_2=0.4$.}
     \label{fig:ex4}
\end{figure}

\end{expl}

\FloatBarrier
\begin{expl}
In this test, we compute a Mach 10 shock diffracting at a $120^\circ$ degree \cite{zhang2012maximum,CHEN2017427}. The computational domain and the unstructured triangular mesh with $h = \frac{1}{4}$ are presented in Figure \ref{fig:ex5.0}. The shock is initially located at $x=3.4$ and $6 \le y \le 11$, moving into undisturbed air with a density of 1.4 and a pressure of 1. Boundary conditions are inflow at the left/top boundary (in accordance with the exact shock motion), and outflow at the right/bottom boundary. The numerical solution is computed up to $T=0.9$ using \eqref{scheme D''} with $\epsilon_1=0$, $\epsilon_2=0.4$ and $h=\frac{1}{40}$. We use the third-order modified exponential Runge–Kutta method from \cite{OFDGsystem}. No positivity-preserving limiter is needed. The computed values of $\rho$ and $p$ are shown in Figure \ref{fig:ex5}. Again, the scheme already has enough viscosity without the flux derivative term in the damping coefficient $D''_K$, and the performance is very close to the OFDG scheme from \cite{OFDGsystem}.

\begin{figure}[htbp]
     \centering
         \includegraphics[width=0.6\textwidth]{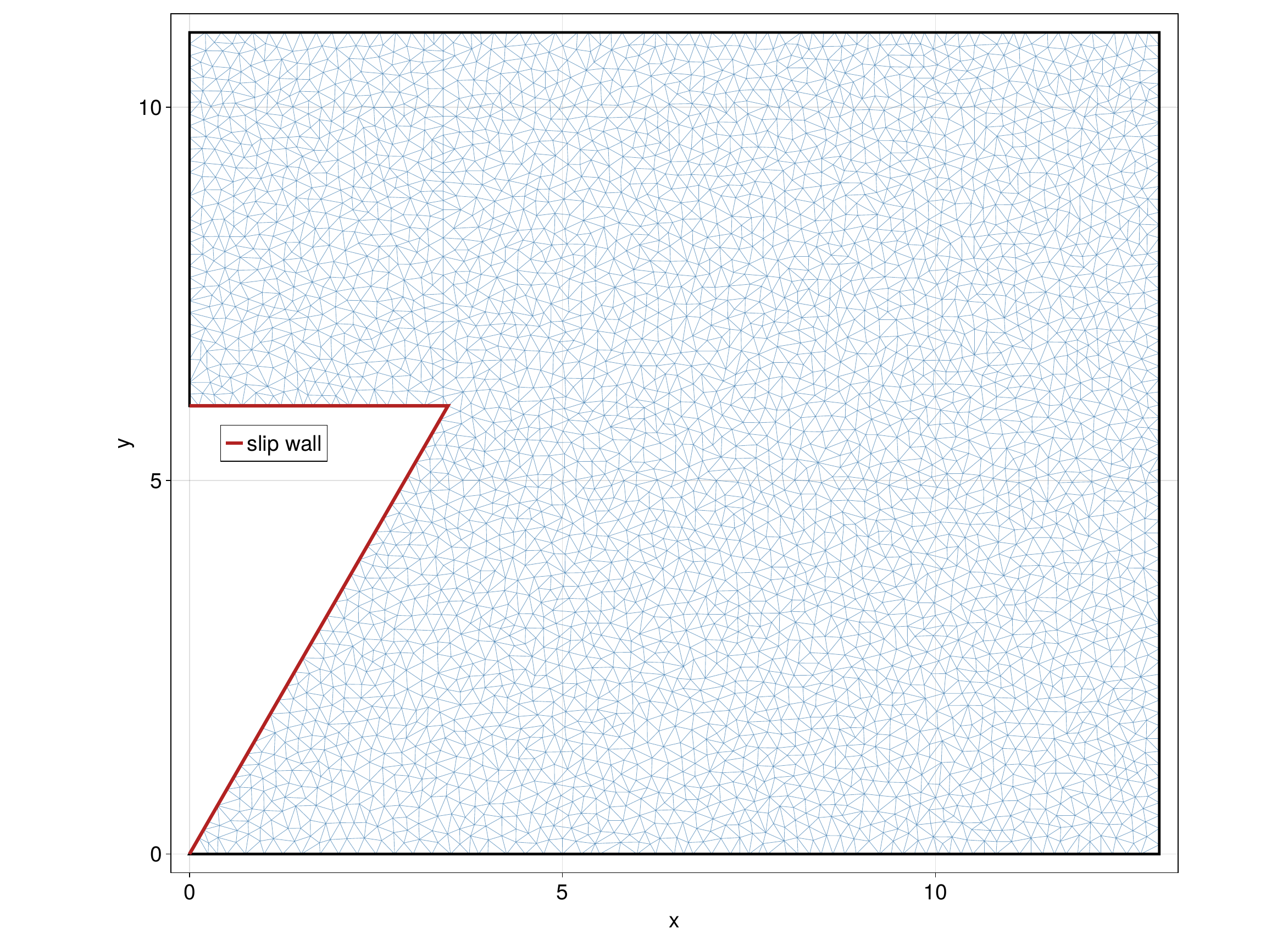}
     \caption{Illustration of the computational domain and the triangular mesh with $h=\frac{1}{4}$.}
     \label{fig:ex5.0}
\end{figure}

\begin{figure}[htbp]
     \centering
     \begin{subfigure}[b]{0.47\textwidth}
         \centering
         \includegraphics[width=\textwidth]{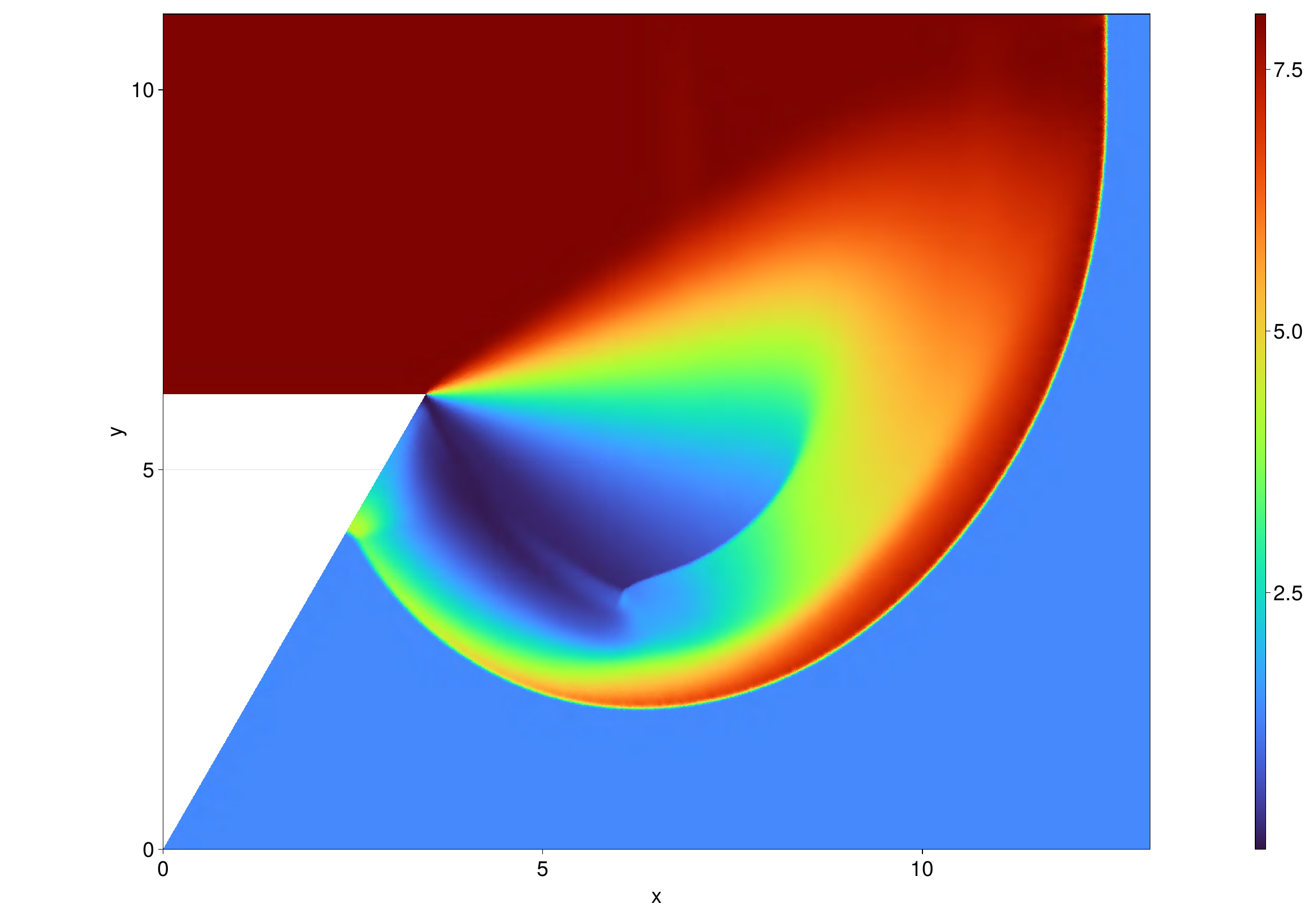}
         \caption{Density $\rho$.}
     \end{subfigure}
     \hfill
     \begin{subfigure}[b]{0.47\textwidth}
         \centering
         \includegraphics[width=\textwidth]{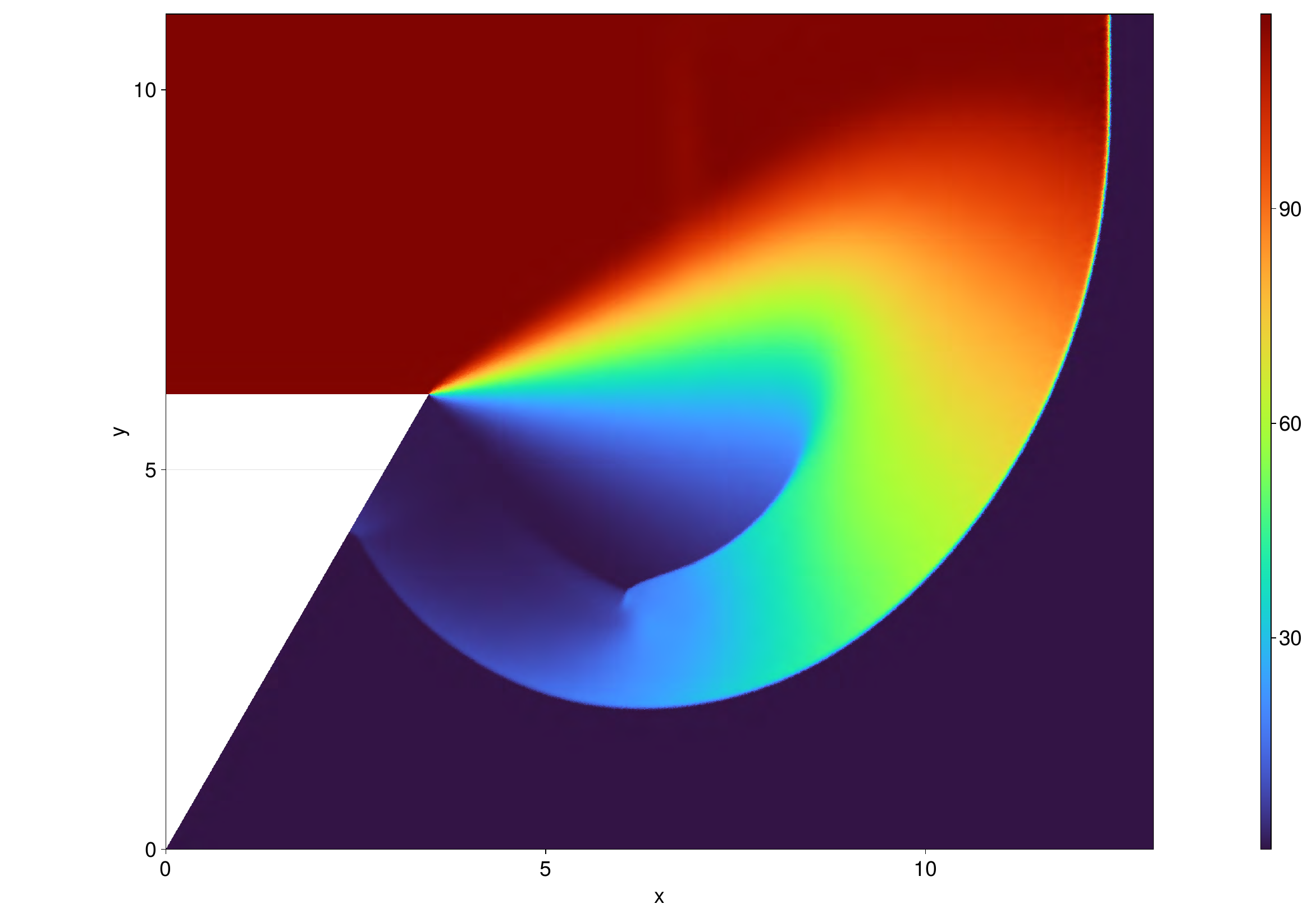}
         \caption{Pressure $p$.}
     \end{subfigure}
     \caption{The numerical solution at $T=0.9$ using \eqref{scheme D''} with $\epsilon_1=0$, $\epsilon_2=0.4$ and $h=\frac{1}{40}$.}
     \label{fig:ex5}
\end{figure}

\end{expl}

\FloatBarrier
\section{Concluding remarks}\label{section: concluding remarks}

In this paper, we have introduced an operator framework for constructing high-order DG methods for nonlinear hyperbolic conservation laws. The framework first gives an alternative derivation of the classical semi-discrete DG method from the infinitesimal action of an evolution operator followed by the $L^2$-projection. Replacing the usual linear projection by an entropy-contracting nonlinear projection then produces the Operator DG methods studied here. The resulting correction to the classical DG scheme has a particularly simple form: it damps only the nonconstant polynomial modes within each cell and therefore preserves cell averages and conservation, much like the OFDG schemes. This construction led to the schemes \eqref{scheme B}--\eqref{scheme D} for one-dimensional scalar equations, their counterparts \eqref{scheme B'}--\eqref{scheme D'} for one-dimensional systems, and the multidimensional schemes \eqref{scheme B''} and \eqref{scheme D''} on general convex, shape-regular meshes.

For 1D scalar conservation laws, the damping coefficients can be chosen so that a single numerical solution satisfies infinitely many local entropy inequalities when an E-flux is used. For systems, the same mechanism is compatible with a prescribed strictly convex entropy and an entropy-stable numerical flux. The analysis also shows that the additional entropy dissipation does not destroy high-order accuracy: for smooth 1D scalar solutions, we proved theoretically the optimal $L^2$ error bound of order $k+1$ with an upwind flux and order $k+\frac12$ with a general monotone flux. For discontinuous solutions of 1D scalar conservation laws with a strictly convex flux, the entropy dissipation supplies the compactness estimates needed in the compensated-compactness argument. Under the stated uniform boundedness assumption, the numerical solutions converge strongly to the unique entropy solution. Thus, the damping term is not merely a device for suppressing oscillations; it also provides the analytical mechanism that links the high-order approximation to the correct entropy solution.

The numerical experiments support these conclusions and clarify the roles of the two parts of the practical damping coefficient $D_j$. The smooth one- and two-dimensional tests exhibit the expected high-order accuracy. For the non-convex scalar problems, using only the interface-jump contribution can fail to approximate the correct entropy solution when the amount of damping is moderate, whereas restoring a suitable flux-derivative contribution recovers the correct solution. At the same time, the Sod, Shu--Osher, double Mach reflection, and shock-diffraction computations show that the simplified system schemes can robustly resolve strong shocks and fine-scale structures. In the multidimensional Euler tests, the interface contribution alone was sufficient for the chosen parameters, and no positivity-preserving limiter was required.

Several questions remain open. A fully discrete entropy and convergence theory that includes the time integrator would complement the present semi-discrete analysis. It would also be desirable to relax the uniform $L^\infty$ assumption in the compensated-compactness argument, to extend strong-convergence and optimal-error results to systems and multiple space dimensions, and to understand entropy selection in order to ensure the convergence for general non-convex conservation laws. Nonlinear projections other than $\tilde \pi$ need to be explored for generating schemes with additional properties; the operator framework also needs to be extended to other types of PDEs. On the computational side, efficient evaluation or approximation of the flux-derivative term, and adaptive selection of $\epsilon_1$ and $\epsilon_2$ are natural directions for making the Operator DG schemes more robust and economical for large-scale applications.

\section*{Disclosure on the use of artificial intelligence}
During the preparation of this work, ChatGPT was used for (1) improving the readability and language quality; (2) assisting with the proof of compensated compactness presented in Section \ref{section: strong convergence}, in particular the proof of Lemma \ref{lemma: SC Hminus1}; (3) programming the Julia codes for numerical tests; (4) computing the optimal coefficients $C_k$, $C(k,r)$ and $\mathcal{A}_k$ recorded in Tables \ref{tab:Ck}, \ref{tab:Ckr} and \ref{tab:improved beta}. All mathematical derivations, results, and proofs were independently verified and validated by the human authors.

\begin{funding}
This research was supported in part by NSF grant DMS-2309249.
\end{funding}

\nocite{*}
\bibliographystyle{amsplain}
\bibliography{ref.bib}

\begin{appendices}
\appendix

\numberwithin{table}{section}

\section{Operator for generating the classical DG method with HLL flux}\label{section:HLL classic}

Although Theorem \ref{theorem 1} holds only for the Godunov flux, the statement can actually be generalized to the class of HLL fluxes, which includes the local Lax–Friedrichs flux. Previously, we assumed $S_{\Delta t}$ as the true solution operator, which would produce the exact Godunov flux in the expressions \eqref{theorem 1 eq 3} and \eqref{theorem 1 eq 4}. Now, the idea is to construct a modified version of $S_{t}$, denoted by $S^{\mathrm{HLL}}_{t}$, such that it satisfies the following three properties:
\begin{description}
    \item[\textbf{(a.1)}] For sufficiently small $\Delta t>0$, $S^{\mathrm{HLL}}_{\Delta t}u=S_{\Delta t}u$ on the interior of each $I_j$, except for two small regions at the end-points with length $\mathcal{O}(\Delta t)$.
    \item[\textbf{(a.2)}] \eqref{theorem 1 eq 3} and \eqref{theorem 1 eq 4} now gives the expressions $\hat f^{\mathrm{HLL}}_{j-\frac{1}{2}}-f(u)^+_{j-\frac{1}{2}}$ and $-\hat f^{\mathrm{HLL}}_{j+\frac{1}{2}}+f(u)^-_{j+\frac{1}{2}}$, where $\hat f^{\mathrm{HLL}}$ is the HLL flux.
    \item[\textbf{(a.3)}] Under periodic or compactly supported boundary conditions, $S^{\mathrm{HLL}}_{\Delta t}$ decreases all entropies: for any convex function $U(\cdot)$, \[\int_I U(S^{\mathrm{HLL}}_{\Delta t} u(x))\ud x \le \int_I U(u(x))\ud x,\quad \forall \Delta t \ge 0,\quad u\in \mathbb{V}^k.\] This is a property shared by $S_{\Delta t}$.
    
\end{description}
Clearly, if such an $S^{\mathrm{HLL}}_{\Delta t}$ can be found, the proof of Theorem \ref{theorem 1} would still go through with $S_{\Delta t}$ replaced by $S^{\mathrm{HLL}}_{\Delta t}$ and the Godunov flux replaced by $\hat f^{\mathrm{HLL}}$. Moreover, if $S^{\mathrm{HLL}}_{\Delta t}$ is locally defined, we may expect the resulting scheme to satisfy certain local entropy inequalities on account of \textbf{(a.3)}. We now proceed to construct such an $S^{\mathrm{HLL}}_{\Delta t}$.

As in the proof of Theorem \ref{theorem 1}, suppose $u\in \mathbb{V}^k$. Let $S_{\Delta t} u$ denote the true solution of the conservation law with the initial condition $u$ after evolving a time $\Delta t>0$ sufficiently small, such that the evolutions of discontinuities at different cell interfaces don't interact, and no new discontinuity in $u$ and its derivatives is formed. Let $\lambda^-$ and $\lambda^+$ denote the lower and upper bounds for all wave speeds of the Riemann solution with left and right states $u^-$ and $u^+$. The HLL flux is defined by:
\begin{equation}\label{HLL}
    \hat f^{\mathrm{HLL}} (u^-,u^+):=\begin{cases}
    \displaystyle
        f(u^-)=f(u^+), & a^-=a^+=0 \\[3mm]
        \frac{a^+ f(u^-) - a^- f(u^+)+ a^- a^+ (u^+ - u^-)}{a^+ - a^-}, & \text{otherwise},
    \end{cases}
\end{equation}
where $a^-=\min\{0,\lambda^-\}\le 0$ and $a^+=\max\{0,\lambda^+\} \ge 0$. The local Lax–Friedrichs flux is a special case by choosing $-\lambda^-=\lambda^+$. Let $I_{j}^\pm(\Delta t)$ be the same as in \eqref{polluted regions}, and define
\begin{equation}\label{HLL averaging interval scalar}
    I_{j-\frac{1}{2}}(\Delta t):= I_{j-1}^-(\Delta t)\cup I_{j}^+(\Delta t).
\end{equation}
In each $I_{j-\frac{1}{2}}(\Delta t)$, we replace the value of $S_{\Delta t} u$ by its average over $I_{j-\frac{1}{2}}(\Delta t)$. We call this modified function $S^{\mathrm{HLL}}_{\Delta t} u$. Namely,
\begin{equation}\label{def tilde S}
    S^{\mathrm{HLL}}_{\Delta t} u(x):=
    \begin{cases}
    \displaystyle
  \frac{1}{\left|I_{j-\frac{1}{2}}(\Delta t)\right|}\int_{I_{j-\frac{1}{2}}(\Delta t)} S_{\Delta t} u(x) \ud x, & \text{if } |I_{j-\frac{1}{2}}(\Delta t)| > 0 \text{ and } x \in I_{j-\frac{1}{2}}(\Delta t) \\[3mm]
  S_{\Delta t} u(x), & \text{otherwise.}
\end{cases}
\end{equation}
Since the intervals \eqref{HLL averaging interval scalar} are disjoint for sufficiently small $\Delta t$, this defines $S^{\mathrm{HLL}}_{\Delta t} u$ unambiguously. 

With $S^{\mathrm{HLL}}_{\Delta t}$ thus defined, \textbf{(a.1)} is easily verified, since we only modified $S_{\Delta t} u$ in the region $I_{j-\frac{1}{2}}(\Delta t)$ with length $\mathcal{O}(\Delta t)$ around each interface. $S^{\mathrm{HLL}}_{\Delta t}$ also satisfies \textbf{(a.3)}, since taking average can only decrease the total entropy by Jensen's inequality. Moreover, $S^{\mathrm{HLL}}_{\Delta t}$ satisfies \textbf{(a.2)}, as can be seen in the proof of the following analogue of Theorem \ref{theorem 1}:
\begin{theorem}\label{theorem 2}
For any $u\in \mathbb{V}^k$, we have
\begin{equation}\label{beautiful 2}
    \lim_{\Delta t\rightarrow 0}\frac{\pi S^{\mathrm{HLL}}_{\Delta t}u-u}{\Delta t}=F[u],
\end{equation}
where $F[\cdot]\in \mathbb{V}^k$ is from the right-hand side of \eqref{DG ODE} with $\hat f^{\mathrm{HLL}}$.
\end{theorem}
\begin{proof}
Fix an arbitrary $j$. Following the proof of Theorem \ref{theorem 1}, we first obtain
\begin{align}
\lim_{\Delta t\rightarrow 0} \int_{I_j} \frac{\pi S^{\mathrm{HLL}}_{\Delta t} u(x) - u(x)}{\Delta t}\cdot v(x) \ud x = &-\int_{I_j} f(u)_x\cdot v(x) \ud x + v_{j-\frac{1}{2}}^+\lim_{\Delta t\rightarrow 0} \int_{I_j^+(\Delta t)} \frac{S^{\mathrm{HLL}}_{\Delta t} u(x) - u(x)}{\Delta t} \ud x \nonumber \\
&+ v_{j+\frac{1}{2}}^-\lim_{\Delta t\rightarrow 0} \int_{I_j^-(\Delta t)} \frac{S^{\mathrm{HLL}}_{\Delta t} u(x) - u(x)}{\Delta t}\ud x \label{theorem 2 eq 1}
\end{align}
for any $v \in \mathbb{V}^k$. Up to $\mathcal{O}(\Delta t)$, the value of $S_{\Delta t} u$ near $x_{j-\frac{1}{2}}$ can be approximated by the Riemann solution $S_{\Delta t} \tilde u$ defined in \eqref{theorem 1 Riemann}. Let $\lambda^-_{j-\frac{1}{2}}$ and $\lambda^+_{j-\frac{1}{2}}$ denote the lower and upper bounds for all wave speeds of this Riemann solution. Let $a^-=\min\{0,\lambda^-_{j-\frac{1}{2}}\}\le 0$ and $a^+=\max\{0,\lambda^+_{j-\frac{1}{2}}\} \ge 0$. If $(a^+-a^-) \neq 0$, by definition \eqref{def tilde S} we have
\begin{equation*}
    \int_{I_j^+(\Delta t)} S^{\mathrm{HLL}}_{\Delta t} u(x) \ud x = \frac{a^+}{a^+-a^-}\int_{I_{j-\frac{1}{2}}(\Delta t)} S_{\Delta t} u(x) \ud x = \frac{a^+}{a^+-a^-} \int_{I_{j-\frac{1}{2}}(\Delta t)} S_{\Delta t} \tilde u(x) \ud x + \mathcal{O}(\Delta t^2),
\end{equation*}
where $\tilde u(x)$ is defined in \eqref{theorem 1 Riemann}. Hence,
\begin{align}
    \int_{I_j^+(\Delta t)} \frac{S^{\mathrm{HLL}}_{\Delta t} u(x) - u(x)}{\Delta t} \ud x &= \frac{a^+}{a^+-a^-}\int_{I_{j-\frac{1}{2}}(\Delta t)} \frac{S_{\Delta t} \tilde u(x) - u^+_{j-\frac{1}{2}}}{\Delta t} \ud x + \mathcal{O}(\Delta t) \nonumber \\
    &=\frac{a^+}{a^+-a^-} \int_{a^-}^{a^+} R(\xi) - u^+_{j-\frac{1}{2}} \ud \xi + \mathcal{O}(\Delta t). \label{theorem 2 eq 2}
\end{align}
where $R(\xi)$ is from \eqref{self-similar R}. Integrating the conservation law for the self-similar Riemann solution \eqref{Riemann eq} yields
\begin{align*}
\int_{a^-}^{a^+} R(\xi) \ud \xi = a^+ u^+_{j-\frac{1}{2}} - a^- u^-_{j-\frac{1}{2}}-f(u^+_{j-\frac{1}{2}})+f(u^-_{j-\frac{1}{2}}).
\end{align*}
Applying this relation in \eqref{theorem 2 eq 2}, we obtain
\begin{equation}\label{theorem 2 eq 3}
\lim_{\Delta t\rightarrow 0} \int_{I_j^+(\Delta t)} \frac{S^{\mathrm{HLL}}_{\Delta t} u(x) - u(x)}{\Delta t} \ud x= \hat f^{\mathrm{HLL}}_{j-\frac{1}{2}}-f(u)^+_{j-\frac{1}{2}},
\end{equation}
where $\hat f^{\mathrm{HLL}}_{j-\frac{1}{2}}=\hat f^{\mathrm{HLL}} (u^-_{j-\frac{1}{2}},u^+_{j-\frac{1}{2}})$ is the HLL flux from \eqref{HLL}. If $(a^+-a^-) = 0$, namely $a^-=a^+=0$, then \eqref{theorem 2 eq 3} holds trivially since both sides are zero. In the same way, we can also prove
\begin{equation}\label{theorem 2 eq 4}
    \lim_{\Delta t\rightarrow 0} \int_{I_j^-(\Delta t)} \frac{S_{\Delta t}^{\mathrm{HLL}} u(x) - u(x)}{\Delta t} \ud x= -\hat f^{\mathrm{HLL}}_{j+\frac{1}{2}}+f(u)^-_{j+\frac{1}{2}}.
\end{equation}
The theorem then follows from applying \eqref{theorem 2 eq 3} and \eqref{theorem 2 eq 4} in \eqref{theorem 2 eq 1} and integrate by parts. As a remark, \eqref{theorem 2 eq 3} and \eqref{theorem 2 eq 4} imply that $S^{\mathrm{HLL}}_{\Delta t}$ satisfies \textbf{(a.2)}, which enables us to extend the result of Theorem \ref{theorem 1} to the HLL flux. 
\end{proof}
By property \textbf{(a.3)}, Corollary \ref{L2 stable Godunov} still holds with Godunov flux replaced by $\hat f^{\mathrm{HLL}}$, without any change to the proof. In addition, since on each cell $I_j$, the scheme only uses local information from $I_j$ and its two neighboring cells, the local square-entropy inequality also holds for $\hat f^{\mathrm{HLL}}$, see \cite{J} or Theorem \ref{theorem: Local Entropy Inequality B}. Finally, Corollary \ref{operator limit} can also be shown to hold in this case.

\section{Improved estimate of \texorpdfstring{$\beta_k$}{beta_k} from Lemma \ref{estimate2}}\label{section:estimate beta_k}
In this section, we improve the estimate on the constant $\beta_k$ from Lemma \ref{estimate2}. The bound $\beta_k\le (k+1)^4$ obtained in the proof of Lemma \ref{estimate2} is the product of two separate sharp inequalities, but the corresponding optimizers do not coincide. A better estimate is obtained by using one single inequality.

Let $\xi\in[-1,1]$ be the reference coordinate on $I_j$, namely
\begin{equation*}
    x=x_j+\frac{h_j}{2}\xi,
\end{equation*}
and let $P_i$ denote the standard Legendre polynomial of degree $i$, normalized by $P_i(1)=1$. The reproducing kernel for evaluation at the left endpoint of the $L^2$-projection onto $\mathbb P^k([-1,1])$ is
\begin{equation}\label{eq:left projection kernel}
    K_k^-(\xi):=\sum_{i=0}^k \frac{2i+1}{2}(-1)^iP_i(\xi).
\end{equation}
Indeed, for every $g\in L^2(-1,1)$,
\begin{equation}\label{eq:projection endpoint kernel}
    (\hat\pi_k g)(-1)=\int_{-1}^1 K_k^-(\xi)g(\xi)\ud\xi,
\end{equation}
where $\hat\pi_k$ is the $L^2$-projection onto $\mathbb P^k([-1,1])$.

We define
\begin{equation}\label{eq:def Ak beta}
    \mathcal{A}_k:=\sup_{\substack{p\in\mathbb P^k([-1,1])\\
    \int_{-1}^1p(\xi)\ud\xi=0,\ p\ne0}}
    \frac{\displaystyle\int_{-1}^1|K_k^-(\xi)|\,|p(\xi)-p(-1)|^2\ud\xi}
    {\displaystyle\int_{-1}^1|p(\xi)|^2\ud\xi}.
\end{equation}
Since the admissible space in \eqref{eq:def Ak beta} is finite dimensional, the supremum is attained.

\begin{proposition}\label{prop:improved beta}
For $k\ge1$, the constant in Lemma \ref{estimate2} may be chosen so that
\begin{equation}\label{eq:improved beta bound}
    \beta_k\le 2\mathcal{A}_k.
\end{equation}
The same constant applies to the left and right endpoints.
\end{proposition}

\begin{proof}
We prove the estimate at $x^+_{j-\frac12}$; the other endpoint follows by symmetry. Set
\begin{equation*}
    p(\xi):=u\left(x_j+\frac{h_j}{2}\xi\right)-\bu_j.
\end{equation*}
Then $p\in\mathbb P^k([-1,1])$ and
\begin{equation*}
    \int_{-1}^1p(\xi)\ud\xi=0.
\end{equation*}
Moreover,
\begin{equation*}
    u(x)-u^+_{j-\frac12}=p(\xi)-p(-1).
\end{equation*}
Let
\begin{equation*}
    M_j:=\max_{x\in I_j}|U'''(u(x))|.
\end{equation*}
Taylor's theorem gives, pointwise on $I_j$,
\begin{align*}
\left|U'(u)-U'(u^+_{j-\frac12})
-U''(u^+_{j-\frac12})(u-u^+_{j-\frac12})\right| \le \frac{M_j}{2}|u-u^+_{j-\frac12}|^2.
\end{align*}
Again, we use the fact that $U''(u^+_{j-\frac12})(u-u^+_{j-\frac12})\in\mathbb P^k(I_j)$ and vanishes at $x^+_{j-\frac12}$, so its projected value at that endpoint is zero. After mapping to $[-1,1]$ and using \eqref{eq:projection endpoint kernel}, we therefore obtain
\begin{align*}
\left|\pi\circ[U'(u)-U'(u^+_{j-\frac12})](x^+_{j-\frac12})\right|\le \frac{M_j}{2}\int_{-1}^1
|K_k^-(\xi)|\,|p(\xi)-p(-1)|^2\ud\xi \le \frac{M_j\mathcal{A}_k}{2}\int_{-1}^1|p(\xi)|^2\ud\xi.
\end{align*}
The change of variables gives
\begin{equation*}
    \int_{-1}^1|p(\xi)|^2\ud\xi
    =\frac{2}{h_j}\int_{I_j}|u(x)-\bu_j|^2\ud x.
\end{equation*}
Consequently,
\begin{equation*}
\left|\pi\circ[U'(u)-U'(u^+_{j-\frac12})](x^+_{j-\frac12})\right|
\le \frac{\mathcal{A}_k}{h_j}M_j
\int_{I_j}|u(x)-\bu_j|^2\ud x.
\end{equation*}
Comparison with the normalization in Lemma \ref{estimate2} yields $\beta_k\le2\mathcal{A}_k$.

For the right endpoint, the corresponding kernel is
\begin{equation*}
    K_k^+(\xi)=\sum_{i=0}^k\frac{2i+1}{2}P_i(\xi)=K_k^-(-\xi).
\end{equation*}
The transformation $p(\xi)\mapsto p(-\xi)$ preserves both the zero-average condition and the $L^2$ norm, and maps the right-endpoint quotient onto the quotient in \eqref{eq:def Ak beta}. Hence the same constant $\mathcal{A}_k$ applies.
\end{proof}

The constant $\mathcal{A}_k$ can be computed from a $k\times k$ generalized eigenvalue problem. Since $p$ has zero average, write
\begin{equation*}
    p(\xi)=\sum_{m=1}^k a_mP_m(\xi),
    \qquad \mathbf{a}=(a_1,\ldots,a_k)^T.
\end{equation*}
Then
\begin{equation*}
    p(\xi)-p(-1)
    =\sum_{m=1}^k a_m\left(P_m(\xi)-(-1)^m\right).
\end{equation*}
Define the symmetric matrices $M^{(k)}$ and $G^{(k)}$ by
\begin{align}
M^{(k)}_{mn}
&:=\int_{-1}^1 |K_k^-(\xi)|
\left(P_m(\xi)-(-1)^m\right)
\left(P_n(\xi)-(-1)^n\right)\ud\xi,
\label{eq:matrix M beta}\\
G^{(k)}_{mn}
&:=\frac{2}{2m+1}\delta_{mn},
\qquad 1\le m,n\le k.
\label{eq:matrix G beta}
\end{align}
By Legendre orthogonality,
\begin{equation*}
    \int_{-1}^1|p(\xi)|^2\ud\xi=\mathbf{a}^TG^{(k)}\mathbf{a},
\end{equation*}
whereas the numerator in \eqref{eq:def Ak beta} is $\mathbf{a}^TM^{(k)}\mathbf{a}$. Therefore
\begin{equation}\label{eq:Ak generalized eigenvalue}
    \mathcal{A}_k=\max_{\mathbf{a}\ne0}\frac{\mathbf{a}^TM^{(k)}\mathbf{a}}{\mathbf{a}^TG^{(k)}\mathbf{a}}
    =\lambda_{\max}\!\left(M^{(k)},G^{(k)}\right).
\end{equation}
Thus $2\lambda_{\max}(M^{(k)},G^{(k)})$ is an explicit computable upper bound for $\beta_k$. Solving the symmetric generalized eigenvalue problem \eqref{eq:Ak generalized eigenvalue} then gives the values in Table \ref{tab:improved beta}.

\begin{table}[htbp]
\centering
\begin{tabular}{c|c|c|c}
\hline
$k$ & $\mathcal{A}_k$ & $2\mathcal{A}_k$ & $(k+1)^4$ \\
\hline
1 & $\frac{59}{27}\approx 2.1851851852$ & $4.3703703704$ & $16$ \\
2 & $6.0975144965$ & $12.1950289930$ & $81$ \\
3 & $14.3482787131$ & $28.6965574262$ & $256$ \\
4 & $26.3960313043$ & $52.7920626087$ & $625$ \\
5 & $44.2424749310$ & $88.4849498620$ & $1296$ \\
6 & $67.2523038234$ & $134.5046076468$ & $2401$ \\
7 & $97.2694197715$ & $194.5388395430$ & $4096$ \\
\hline
\end{tabular}
\caption{Improved computable bounds $\beta_k\le2\mathcal{A}_k$ from \eqref{eq:improved beta bound}, compared with the bound $(k+1)^4$ used in Lemma \ref{estimate2}.}
\label{tab:improved beta}
\end{table}

\end{appendices}

\end{document}